\documentclass[a4paper,11pt,reqno]{amsart}
\usepackage{geometry} 
\usepackage{mathtools}
\usepackage{graphicx}
\usepackage{paralist}
\usepackage[utf8]{inputenc}
\usepackage[T1]{fontenc}
\usepackage{csquotes}
\usepackage[all]{xy}
\usepackage{chngcntr}
\usepackage{amsthm}
\usepackage{amsmath}
\usepackage{amsfonts}
\usepackage{amssymb}
\usepackage{mathrsfs}
\usepackage{MnSymbol}
\usepackage{tikz-cd}
\usepackage{tikz, tikz-cd}
\usepackage{mathdots}
\usepackage{stmaryrd}
\usepackage{hyperref}
\usepackage{cleveref}
\usepackage{fancyhdr}
\usepackage{xcolor}

\newtheorem{theorem}{Theorem}
\newtheorem{thm}{Theorem}[section]
\newtheorem{lem}[thm]{Lemma}
\newtheorem{prop}[thm]{Proposition}
\newtheorem{cor}[thm]{Corollary}
\theoremstyle{remark}

\theoremstyle{claim}
\newtheorem*{claim}{Claim}
\newtheorem*{thm*}{Theorem}

\theoremstyle{definition}
\newtheorem{definition}[thm]{Definition}

\theoremstyle{example}

\theoremstyle{convention}
\newtheorem{convention}[thm]{Convention}

\counterwithin{equation}{section}

\theoremstyle{convention}
\newtheorem{rem}[thm]{Remark}

\theoremstyle{convention}
\newtheorem{notation}[thm]{Notation}

\newcommand{\bbN}{\mathbb{N}}
\newcommand{\bbZ}{\mathbb{Z}}

\newcommand{\bbR}{\mathbb{R}}

\newcommand{\bbH}{\mathbb{H}}

\newcommand{\diam}{\mathrm{diam}}
\newcommand{\vol}{\mathrm{vol}}
\newcommand{\Ric}{\mathrm{Ric}}
\newcommand{\inj}{\mathrm{inj}}

\newcommand{\cal}[1]{\mathcal{#1}}
\newcommand{\dist}{\operatorname{dist}}

\calclayout

\begin{document}
\title[Stability of Einstein metrics and effective hyperbolization]
{Stability of Einstein metrics and effective hyperbolization in large Hempel distance}
\author{Ursula Hamenst\"adt and Frieder J\"ackel}
\thanks
{
Both authors were supported by the DFG priority program "Geometry at infinity"}
\date{November 22, 2022}

\begin{abstract}
Extending earlier work of Tian, we show that if a manifold admits a metric that is almost hyperbolic in a suitable sense, then there exists an Einstein metric that is close to the given metric in the \(C^{2,\alpha}\)-topology. In dimension \(3\) the original manifold only needs to have finite volume, and the volume can be arbitrarily large. Applications include a new proof of the hyperbolization of \(3\)-manifolds of large Hempel distance yielding some new geometric control on the hyperbolic metric, and an analytic proof of Dehn filling and drilling that allows the filling and drilling of arbitrary many cusps and tubes.
\end{abstract}

\maketitle

\tableofcontents

\section{Introduction}\label{Sec: Introduction}

\subsection{Statement of the main results} 
The search for \emph{Einstein metrics} on a closed manifold $M$ has a long and fruitful history. Such metrics can be 
found using the Ricci flow, perhaps with surgery. This approach was used by Perelman to prove the 
so-called geometrization conjecture for 3-manifolds, embarking from the easy fact that in dimensions 2 and 3, 
Einstein metrics have constant curvature. 

An older method for the construction of Einstein metrics consists in starting from a metric $\bar g$ 
which is 
almost Einstein in a suitable sense, and construct a nearby Einstein metric as a perturbation of the 
given metric. The perturbation can be done using once again the Ricci flow, as for example in 
\cite{MO90}. The cross curvature flow is another tool for evolving metrics on 3-manifolds towards an Einstein metric
\cite{KY09}. One may also use compactness properties
for Riemannian manifolds with a 
uniform upper bound on the diameter, 
a uniform lower bound on the volume and a
suitable curvature control, like an $L^p$-bound on the norm of the curvature
tensor, to establish the existence of Einstein metrics which are 
close to a metric $\bar{g}$ with these properties and 
for which in addition the $L^p$-norm of 
$\Ric(\bar{g})-\lambda \bar{g}$ is sufficiently small (see \cite[Corollary 1.6]{PetersenWei1997} and also 
\cite{Anderson2006} and \cite{Bamler2012}).

Much more recently, Fine and Premoselli \cite{FP20} 
constructed Einstein metrics using a gluing method which can 
be described as follows. Starting from a manifold $M$ which is the union of two open submanifolds $U,V$ 
admitting each an Einstein metric whose restrictions to $U\cap V$ are close to each other in a controlled way, one can 
glue these metrics on $U\cap V$ and try to use
an implicit function theorem for the so-called \emph{Einstein operator} at the glued metric 
to find a nearby Einstein metric. This method 
depends on stability of Einstein metrics near the given metric, which means 
that locally, 
if there is an Einstein metric in an a priori chosen neighborhood of the glued
metric, then this metric is unique up to scaling and pull-back by 
diffeomorphisms. No 
a priori volume bound is necessary. 

In the setting we are interested in, stability is guaranteed by curvature control. Namely, 
a classical result states that on compact manifolds of dimension \(n \geq 3\), 
Einstein metrics with negative sectional curvature are isolated in the moduli space of Riemannian structures (see \cite[Corollary 12.73]{Besse1987}, \cite[Theorem 3.3]{Koiso1978}). 
For manifolds of infinite volume, this is no longer true. We refer to \cite{Biquard2000} for more information and for references. 
We also note that for positive sectional curvature there are much stronger rigidity results. For example, 
Berger showed that if an Einstein metric has positive strictly \(\frac{n-2}{n-1}\)-pinched sectional curvature, then the sectional curvature is constant (see \cite{Berger1966}).

The main goal of this article is to develop a systematic approach for the construction of Einstein metrics 
by a perturbation of metrics whose sectional curvature is close to $-1$. The first result we prove 
is a general existence result for Einstein metrics in this setting. 
It requires a uniform lower bound on the injectivity radius, but no volume bounds. 
Its formulation is similar to the main result of an unpublished preprint of Tian \cite{Tian}, and the proof we give follows his outline. 
\Cref{Sec: Proof of pinching with inj radius bound} contains a stronger but also more technical version of this result. 


\begin{theorem}[Stability of Einstein metrics with a lower 
injectivity radius bound]\label{Pinching with inj radius bound - introduction} 
  For any \(n \geq 3, \, \alpha \in (0,1), \, \Lambda \geq 0,\)
  and \(\delta \in (0,2\sqrt{n-2})\) there exist constants 
  \(\varepsilon_0=\varepsilon_0(n,\alpha,\Lambda,\delta)>0$
  and $C=C(n,\alpha,\Lambda,\delta)>0\) with the following property.
Let \(M\) be a closed \(n\)-manifold that admits a Riemannian metric \(\bar{g}\) satisfying the following conditions for some \(\varepsilon \leq \varepsilon_0\):
\begin{enumerate}[i)]
\item \(-1-\varepsilon \leq \mathrm{sec}_{(M,\bar{g})} \leq -1+\varepsilon\);
\item \(\mathrm{inj}(M,\bar{g})\geq 1\);
\item \(|| {\nabla} \Ric(\bar{g})||_{C^0(M,\bar{g})} \leq \Lambda\);
\item It holds
\[
  \int_Me^{-(2\sqrt{n-2}-\delta)r_x(y)}
| \Ric(\bar{g})+(n-1)\bar{g}|_{\bar{g}}^2(y) \, d\vol_{\bar{g}}(y) \leq \varepsilon^2
\]
for all \(x \in M\), where $r_x(y)=d(x,y)$. 
\end{enumerate}
Then there exists an Einstein metric \(g_0\) on \(M\) so that \(\Ric(g_0)=-(n-1)g_0\) and 
\[
	||g_0-\bar{g}||_{C^{2,\alpha}(M,\bar{g})} \leq C \varepsilon^{1-\alpha}.
\]
Moreover, if additionally \(\Ric(\bar{g})=-(n-1)\bar{g}\) outside a region \(U\), and if 
\[
	\int_U | \Ric(\bar{g})+(n-1)\bar{g}|_{\bar{g}}^2 \, d\vol_{\bar{g}} \leq \varepsilon^2,
\]
then
\[
	| g_0-\bar{g}|_{C^{2,\alpha}}(x) \leq C \varepsilon^{1-\alpha} e^{-(\sqrt{n-2}-\frac{1}{2}\delta)\dist_{\bar{g}}(x,U)}	
\]
for all \(x \in M\).
\end{theorem}

We refer to Section \ref{Section - Hölder Norms} for a detailed explanation of the notion of Hölder norm used here.

As in \cite{FP20}, \cite{Anderson2006} and \cite{Bamler2012}, 
the proof of \Cref{Pinching with inj radius bound - introduction} is based on an application of the implicit function theorem 
to the Einstein operator \(\Phi\) (see \Cref{Subsec: Einstein operator} for the definition of $\Phi$). 
The novelty of our approach consists in the use of Banach spaces for tensor fields whose construction is adapted to the specific geometric situation.  
These Banach spaces are defined by 
\textit{hybrid norms} which are a combination of Hölder- and weighted Sobolev norms. 


For the main applications we have in mind, Theorem \ref{Pinching with inj radius bound - introduction} 
is not strong enough due to the assumption of a uniform positive lower bound on the injectivity radius. 
But Theorem \ref{Pinching with inj radius bound - introduction} does not extend in a 
straightforward way to finite volume manifolds without such a lower injectivity radius bound. 
We illustrate this in 
\Cref{Sec: Counterexamples} by constructing for any $L>1$ 
a metric on a closed $3$-manifold which fulfills all assumptions of 
Theorem \ref{Pinching with inj radius bound - introduction}
with the exception of a uniform lower bound on the injectivity radius, with an arbitrarily small control constant $\varepsilon$,
which is not $L$-bilipschitz equivalent to an 
Einstein metric
 
 Finding the correct assumptions for $3$-dimensional manifolds of finite volume 
 for which an extension of Theorem  \ref{Pinching with inj radius bound - introduction}
 without the hypothesis of a uniform positive lower bound on the injectivity radius holds true is the main technical 
 result of this article. For its formulation, recall that a complete manifold of bounded negative sectional curvature admits
 a decomposition into its \emph{thick} part $M_{\rm thick}$ consisting of points where the injectivity
 radius is bigger than a fixed \emph{Margulis constant} for the curvature bounds, 
 and its complement $M_{\rm thin}$.
 

 \begin{theorem}[Stability of Einstein metrics in dimension $3$]
 \label{Pinching without inj radius bound - introduction} For all \(\alpha \in (0,1)\), \(\Lambda \geq 0\), \(\delta \in (0,2)\), \(b > 1\) and \(\eta > 2\) there
   exist $\varepsilon_0=\varepsilon_0(\alpha,\Lambda,\delta,b,\eta)>0$
   and $C=C(\alpha,\Lambda,\delta,b,\eta) >0$
 with the following property. Let \(M\) be a \(3\)-manifold that admits a complete Riemannian metric \(\bar{g}\) satisfying the following conditions for some \(\varepsilon \leq \varepsilon_0\):
\begin{enumerate}[i)]
\item \(\vol(M,\bar{g}) < \infty\);
\item \(-1-\varepsilon \leq \mathrm{sec}_{(M,\bar{g})} \leq -1+\varepsilon\);
\item It holds
\[
  \max_{\pi \subseteq T_xM}|\mathrm{sec}(\pi)+1|, \,
 | \nabla R|(x), \, | \nabla^2R|(x) \leq \varepsilon e^{-\eta d(x,M_{\rm thick})}
\]
for all \( x \in M_{\rm thin}\);
\item \(|| {\nabla} \Ric(\bar{g})||_{C^0(M,\bar{g})} \leq \Lambda\);
\item It holds
\[
  e^{bd(x, M_{\rm thick})}\int_Me^{-(2-\delta)r_x(y)}
| \Ric(\bar{g})+2\bar{g}|_{\bar{g}}^2(y) \, d\vol_{\bar{g}}(y) \leq \varepsilon^2
\]
for all \(x \in M\), where $r_x(y)=d(x,y)$.
\end{enumerate}
Then there exists a hyperbolic metric \(g_{\rm hyp}\) on \(M\) so that
\[
	||g_{\rm hyp}-\bar{g}||_{C^{2,\alpha}(M,\bar{g})} \leq C \varepsilon^{1-\alpha}.
\]
Moreover, if additionally \(\bar{g}\) is already hyperbolic outside a region \(U \subseteq M\), and if 
\[
	\int_{U}|\Ric(\bar{g})+2\bar{g}|_{\bar{g}}^2 \, d\vol_{\bar{g}} \leq \varepsilon^2,
\]
then for all \(x \in M_{\rm thick}\) it holds
\[
|g_{\rm hyp}-\bar{g}|_{C^{2,\alpha}}(x) \leq C\varepsilon^{1-\alpha}e^{-(1-\frac{1}{2}\delta)\mathrm{dist}_{\bar{g}}(x,\,U \, \cup \, \partial M_{\rm thick})}.
\]
\end{theorem}

Here \(R\) denotes the Riemann curvature endomorphism. Section \ref{Sec: Proof of pinching without inj radius bound} contains
a slightly stronger but also more technical version of this result. 

We believe that there is a version of Theorem \ref{Pinching without inj radius bound - introduction} 
with similar assumptions which 
holds true in all dimensions. However, it turns out that the well
known difference 
between geometric properties of the thin parts of negatively curved  3-manifolds
and the thin parts of negatively curved manifolds
in higher dimensions require a 
modification of the strategy we use, and we do not attempt to establish 
such an extension of Theorem \ref{Pinching without inj radius bound - introduction} in this article.

In view of a good understanding of large scale geometric properties of 
asymptotically hyperbolic manifolds of infinite volume as
considered for example in \cite{Biquard2000} and \cite{HQS12} we also expect that there are extensions of 
Theorem \ref{Pinching with inj radius bound - introduction} to
negatively curved 
manifolds of infinite volume with finitely generated fundamental group and asymptotically hyperbolic 
infinite volume ends. 


Theorem \ref{Pinching without inj radius bound - introduction} 
makes it possible to glue finite volume hyperbolic 
metrics which are defined on open submanifolds $U,V$ of a given 3-manifold $M$ 
along the intersection $U\cap V$ and deform the glued metrics
to a hyperbolic metric. As a fairly immediate application,
we obtain an analytic approach to 
hyperbolic Dehn filling and Dehn drilling in dimension 3,
without the use of the deformation theory of hyperbolic cone manifolds.
An earlier analytic proof of hyperbolic Dehn filling
under a uniform
upper bound for the volume which also is based on an implicit function
theorem is due to Anderson \cite{Anderson2006} and
Bamler \cite{Bamler2012}.
We refer to \cite{FPS19b}, \cite{HK08}
for an overview on what is known to date about Dehn filling and Dehn drilling.

For the statement of our drilling result, recall that a \emph{Margulis tube} in 
a negatively curved manifold $M$ is a tubular neighborhood of a 
closed geodesic $\beta$ which is a connected component of the thin 
part of $M$. The \emph{radius} of the Margulis tube is the distance 
between the core geodesic of the tube and its boundary. 

\begin{theorem}[The drilling theorem]\label{drilling - introduction}
For any \(\varepsilon >0\), \(\kappa \in (0,1)\) and \(m>0\) there exists
a number $R=R(\varepsilon,\kappa,m) >0$
with the following property. 
Let $M$ be a finite volume hyperbolic 3-mani\-fold, and let 
$T_1,\dots,T_k$ be a family of Margulis tubes in $M$. Let $R_i>0$ be the radius of the 
tube $T_i$, and let $\beta_i$ be its core geodesic. 
If for each $r>0$ and each $x\in M$ we have
$\# \{i\mid {\rm dist}(x,T_i)\leq r\}\leq me^{\kappa r }$ and if $R_i\geq R$ for all $i$, 
then the manifold
obtained from $M$ by drilling each of the 
geodesics $\beta_i$ admits a complete hyperbolic metric of 
finite volume, and the restriction of this hyperbolic metric to the complement
of the cusps obtained from the drilling is $\varepsilon$-close 
in the $C^2$-topology 
to the restriction of the metric on $M$.
\end{theorem}

The same argument which allows for drilling closed geodesics in finite volume hyperbolic 
3-manifolds can also be used to Dehn fill cusps. This is formulated in our next result. 
Recall that 
the \emph{meridian} of a solid torus $T$ is a simple closed curve on the 
boundary torus $\partial T$ of $T$ which is homotopic to zero in $T$.
A \emph{torus cusp} in a hyperbolic 3-manifold is a cusp diffeomorphic
to $T^2\times[0,\infty)$ where $T^2$ denotes the 2-torus.
Any cusp in a finite volume orientable hyperbolic 3-manifold is a torus cusp.

\begin{theorem}[The filling theorem]\label{filling - introduction}
For any \(\varepsilon >0\), \(\kappa \in (0,1)\) and \(m>0\) there exists
a number $L=L(\varepsilon,\kappa,m) >0$
with the following property. Let $M$ be a finite volume hyperbolic 3-mani\-fold, $C_1,\dots,C_k\subseteq M$ be a finite collection of 
torus cusps, and assume that  
for each $r>0$ and each $x\in M$ we have
$\# \{i\mid {\rm dist}(x,C_i)\leq r\}\leq me^{\kappa r}$.
For each $i\leq k$ let $\alpha_i$ be a flat simple closed geodesic in  
$\partial C_i$ of length $L_i\geq L$. 
Then the manifold obtained from $M$ by filling the cusps $C_i$, 
with meridian $\alpha_i$,  
is hyperbolic, and the restriction of its metric to the complement
of the Margulis tubes obtained from the filling is $\varepsilon$-close to the metric on 
$M-\cup_iC_i$. 
\end{theorem}

Although unlike \cite{FPS19b},
we do not obtain effective constants for Dehn filling and drilling, 
Theorem \ref{drilling - introduction} and Theorem \ref{filling - introduction} 
allows 
to drill or fill an arbitrary number of 
tubes or cusps with fixed meridional length 
as long as the boundaries of the tubes 
or cusps are sufficiently sparsely distributed in the manifold. 
We also obtain a more precise geometric
control on the complement of the drilled or filled tubes or cusps
which we discuss in Section \ref{Sec: drillingand}.

The proofs of Theorem \ref{drilling - introduction} and Theorem \ref{filling - introduction} use  
Theorem \ref{Pinching without inj radius bound - introduction} 
to deform a metric glued from 
the metric on the given manifold and hyperbolic metrics on tubes and cups 
to a hyperbolic 
metric on the drilled or filled manifold.
This strategy can also be used
to construct hyperbolic metrics on manifolds glued in a controlled way from hyperbolic 
pieces as long as the gluing regions are sufficiently sparsely distributed. An example
of such a construction can be found in the article \cite{BMNS16}. We do however 
not discuss such a potential application of our main result here.

Apart from a new approach to drilling and filling, we also obtain results
towards what sometimes is called \emph{effective Mostow rigidity}, a
program which lead among others to the solution of the so-called \emph{ending
lamination conjecture} (see \cite{M10}, \cite{BCM12}). 
The idea is as follows. Due to the groundbreaking work of Thurston and 
Perelman, a closed aspherical atoroidal 3-manifold admits a 
hyperbolic metric, which is moreover unique up to isotopy 
by Mostow rigidity. 
Thus topological information gives rise to geometric invariants,
and some of these invariants, like for example 
the injectivity radius or the volume, should be 
recoverable from suitably chosen topological data. 
Even more ambitious, it may be possible to construct a
bi-Lipschitz model
for the hyperbolic metric from topological information as in
\cite{BCM12}.

To implement this program, one may try to decompose 
a closed 3-manifold $M$ into pieces which are equipped with hyperbolic metrics 
constructed from the knowledge of the pieces and knowledge on how these 
pieces glue together to $M$.
This program is well suited for an application 
of Theorem \ref{Pinching without inj radius bound - introduction}. 

For the formulation of our last main result, recall that 
a \emph{handlebody} of genus $g\geq 1$
is a compact 3-manifold with boundary which is diffeomorphic to the
connected sum of $g$ solid tori. 
The boundary $\partial H$ of such a handlebody
$H$ is a closed oriented surface $\partial H$ of genus $g$. 
Any closed 3-manifold can be realized as
the gluing $M_f=H_1\cup_fH_2$ of 
two handlebodies $H_1,H_2$ of the same genus $g\geq 1$
along a diffeomorphism
$f:\Sigma=\partial H_1\to \partial H_2$ of the boundaries.
The manifold
$M_f$ only depends on a double coset of the mapping class of
$\Sigma$ defined by $f$. 
The boundaries $\partial H_1,\partial H_2$ of $H_1,H_2$ 
contain collections ${\cal D}_1,{\cal D}_2$ 
of curves, the simple closed curves in $\partial H_1,\partial H_2$ 
which bound disks in $H_1,H_2$. We call these curves the \emph{disk sets} of $H_1,H_2$.
Using the identification of $\partial H_1$ and $\partial H_2$ via $f$,
these disk sets define subsets
in the \emph{curve graph} ${\cal C\cal G}(\Sigma)$ of the boundary surface
$\Sigma$ of $H_1,H_2$. The vertices of this graph are isotopy classes of simple closed curves 
on $\Sigma$, and two such curves are connected by an edge of length one if they
can be realized disjointly. 

The \emph{Hempel distance} of the manifold 
$M_f$ is the distance in ${\mathcal C\mathcal G}(\Sigma)$ between 
the image of the disk set ${\mathcal D}_1$ of $H_1$ under the map $f$ and the disk
set ${\mathcal D}_2$ of $H_2$. Hempel \cite{He01}
showed that if the 
Hempel distance of $M_f$ is at least $3$, then $M_f$ is aspherical
atoroidal and hence hyperbolic by the work of Perelman.

The following result can be viewed as a first step towards an 
effective geometrization of 3-manifolds. 
In its formulation,
we denote by $d_{\cal C\cal G}$ the distance in the curve graph of 
the boundary surface $\Sigma$.


\begin{theorem}[Effective hyperbolization in large Hempel distance]\label{hyperbolizationfinal - introduction}
For every $g\geq 2$ there exist
numbers $R=R(g)>0$ and $C=C(g)>0$  with the following property. 
Let $M_f$ be a closed 3-manifold with Heegaard surface $\Sigma$
of genus $g$ and gluing map $f$, 
and assume that $d_{\cal C\cal G}({\cal D}_1,{\cal D}_2)\geq R$. 
Then $M_f$ admits a hyperbolic metric, and the volume of $M_f$ for this 
metric is at least $Cd_{\cal C\cal G}({\cal D}_1,{\cal D}_2)$.
\end{theorem}

This result does not rely on the work of Perelman and does not use the Ricci flow, 
that is, we give a new proof of hyperbolization under the assumption of large Hempel distance.
The lower bound on the Hempel distance we need is not
effective, however  we obtain some explicit information on the hyperbolic metric.
An earlier proof of hyperbolization of random 3-manifolds without the use of the Ricci flow
can be found in \cite{FSV19}.

The first geometric information which is new is the lower volume
bound in terms of the Hempel distance stated in
Theorem \ref{hyperbolizationfinal - introduction}.
Note that up to a universal constant,
the volume of a closed hyperbolic 3-manifold
coincides with its simplicial volume. Work of Brock (see \cite{Br03} 
and \cite{HV22} for more details)
shows that this simplicial volume is bounded
from above by a fixed multiple of 
the smallest distance in the so-called \emph{pants graph} of $\Sigma$ 
between a pants decomposition 
consisting of pants curves in ${\cal D}_1$, and a pants decomposition
consisting of pants curves in ${\cal D}_2$.
We conjecture that this upper estimate 
computes the volume up to a universal multiplicative constant. 
This conjecture holds true for random 3-manifolds \cite{V21}. 
Our lower volume bound is expected to be far from sharp.

The second geometric information we obtain applies to manifolds $M_f$ for which
there is a sufficiently long segment of a minimal geodesic in the curve graph of
$\Sigma$ connecting ${\cal D}_1$ to ${\cal D}_2$ which has \emph{bounded
combinatorics}. In this case we obtain that the hyperbolic metric is
uniformly close to a metric obtained by gluing two convex cocompact
hyperbolic metrics on handlebodies near the boundary.

Motivated by \cite{M00},
in the absence of such a segment, and  
assuming that the manifold $M_f$ is equipped with a hyperbolic metric, 
we prove an a priori length bound for
closed geodesics in $M_f$ which arise in the following way.
For a proper essential subsurface $Y$ of $\Sigma$, denote by
${\rm diam}_Y({\cal D}_1,{\cal D}_2)$ the diameter in
the curve graph of $Y$  
of the \emph{subsurface projections} of the disk sets ${\cal D}_1,{\cal D}_2$ 
into $Y$. Furthermore, for a multicurve $c$ in $\Sigma$,
let $\ell_f(c)$ be  the length of the geodesic representative 
of $c$ in the manifold $M_f$ (which may be zero if $c$ is compressible). 

\begin{theorem}[A priori length bounds]\label{lengthbound - introduction}
  Given $\Sigma$, there exists a number
$p=p(\Sigma)\geq 3$, and for every $\varepsilon>0$,
 there exists a number $k=k(\Sigma,\varepsilon)>0$ with the 
following property.  Let $M_f$ be a hyperbolic $3$-manifold of Heegaard genus $g$ and 
Hempel distance at least $4$ and let $Y\subset \Sigma$ be a proper
essential subsurface of $\Sigma$ such that 
$d_{\cal C\cal G}(\partial Y,{\cal D}_1\cup {\cal D}_2)\geq p$. If  
${\rm diam}_Y({\cal D}_1\cup {\cal D}_2)\geq
k$, then $\ell_f(\partial Y)\leq \varepsilon$.
\end{theorem}

Note that a priori, the statement of Theorem \ref{lengthbound - introduction} 
relates a property 
which only depends on the manifold $M_f$, namely the hyperbolic length 
of a system of essential closed curves, which is well defined by Mostow rigidity, 
to a property which depends 
on choices, namely the choice of a Heegaard splitting for $M_f$.
However, as the diameter of the subsurface projection of 
the disk sets of $M_f$ into $Y$ is required to be large, it follows from \cite{JMM10}
that up to isotopy, $Y$ appears as a subsurface of every Heegaard surface
of sufficiently small genus.

\subsection{Organization of the article and outline of the proofs}
The article is roughly divided into four parts which can be read independently.
The first part, contained in Sections 2-5, is devoted to the proof of 
Theorem \ref{Pinching with inj radius bound - introduction}, and
it is the only part containing results on manifolds of dimension different from $3$.
Section \ref{Sec: Preliminaries} organizes the basic set-up and collects some technical results
used later on. It also introduces the conventions and notations we are going to use. 
In particular, we introduce the Einstein operator $\Phi$, and we 
formulate and prove a general statement which allows to obtain 
$C^0$-estimates for a solution of its linearization from suitable integral bounds.
The results in this section are variations of results available in the literature, adjusted 
to our needs. 

The goal is to use an implicit function theorem
to construct solutions of the equation 
$\Phi(\bar g+h)=0$ 
(see \Cref{Zeros of Phi are Einstein}).
To this end it is necessary to invert the linearization of $\Phi$ and 
obtain a good norm control on this inverse,
acting on suitably chosen Banach spaces of 
sections of the bundle of symmetric $(0,2)$-tensors over the manifold $M$. 
This is carried out in Section \ref{Section - Integral Inequalities}. 
As the Einstein operator is closely related to the Laplacian, we begin
with establishing 
a uniform $L^2$-Poincare inequality for the manifolds
with small pinched negative curvature.
We then introduce the weighted Sobolev spaces which are our primary tool. 
They enter in the definition of the \emph{hybrid norms} in Section \ref{Section - Invertibiliy of L}, 
which are combinations 
of H\"older and weighted Sobolev norms.
These norms are used in 
an a priori estimate for the linearized Einstein equation leading to 
an invertibility statement for the linearized equation. In Section \ref{Sec: Proof of pinching with inj radius bound} we then show that this 
is sufficient for an application of the implicit function theorem 
which completes the proof of 
Theorem \ref{Pinching with inj radius bound - introduction}.

A crucial point in this proof is a uniform $C^0$-bound for solutions of the linearization of the 
Einstein equation, which depends on a uniform lower bound on the injectivity radius. 
In \Cref{Sec: Counterexamples} we show that there is no straightforward way of dropping this assumption 
by exhibiting a family of metrics on closed 3-manifolds obtained by Dehn filling a 
finite volume hyperbolic 3-manifold with a single cusp 
and slowly changing the conformal structure 
of the level tori for the distance to the core geodesics 
of the filled cusp  
while keeping the monodromy of the core curve fixed. 
This construction does not alter the metric in the thick part of the manifold, and we show that 
it can be done in such a way that its curvature is arbitrarily close to to $-1$, the 
weighted \(L^2\)-norm of $\Ric(g)+2g$ is arbitrarily small,
while the length ratio of the core curves of the
modified metric and the hyperbolic metric can be made arbitrarily large. 

The second part of this article is devoted to overcoming this difficulty for
finite volume 3-manifolds satisfying suitable curvature assumptions. We begin with analyzing 
the case when the injectivity radius may be arbitrarily small, but an extension of the strategy
used in the proof of Theorem \ref{Pinching with inj radius bound - introduction} is possible. 
Namely, in Section \ref{Sec:thethin} we define the \emph{small part} of a finite volume negatively curved 
$3$-manifold to consist of points of small injectivity radius and such that moreover the diameter
of the distance torus or horotorus containing the points is bounded from above by a universal
constant. In the case of hyperbolic 3-manifolds,
the diameter of a component of the small part may be arbitrarily large.
Using a counting argument for preimages of points in the thin but not small part of a negatively
curved 3-manifold which are contained in a fixed size ball in the universal covering, in 
\Cref{A-priori estimate away from the small part}
we extend the main $C^0$-estimate in the proof of 
Theorem \ref{Pinching with inj radius bound - introduction}
to the thin but not small part of the 3-manifold.

Motivated by work of Bamler \cite{Bamler2012}, to deal with the small part of the manifold we take advantage
of the fact that on the small part of a hyperbolic Margulis tube or cusp, solutions of the linearized 
Einstein equations can be controlled with an ODE. The main task is then to use the geometric assumptions
to construct a hyperbolic model metric for the small parts of tubes and cusps in Section \ref{Sec:model metrics} 
and to use the ODE for the hyperbolic model metric to 
analyze the solutions of the linearized Einstein equation. We construct Banach spaces adapted to our needs
which control the growth of solutions in the small part, and we use these Banach spaces 
to invert the linearized Einstein operator with uniformly controlled norm in Section \ref{Sec: L without bound on inj}. 
This then leads to the proof of 
\Cref{Pinching without inj radius bound - introduction} in \Cref{Sec: Proof of pinching without inj radius bound}. 

In Section \ref{Sec: drillingand} which contains the third part of the article, we apply 
Theorem \ref{Pinching without inj radius bound - introduction} to Dehn filling and 
Dehn drilling as formulated in Theorem \ref{drilling - introduction} and Theorem \ref{filling - introduction}.  

The last part of this article is devoted to the proof of Theorem \ref{hyperbolizationfinal - introduction}. We begin in 
Section \ref{effective} with showing that Theorem \ref{Pinching without inj radius bound - introduction}
together with a gluing result taken from \cite{HV22} can fairly immediately be used
to construct a hyperbolic metric close to a model metric on 
a closed 3-manifold $M_f=H_1\cup_f H_2$ which has the following property. 
The Hempel distance of $M_f$ is large, and a minimal geodesic in 
the curve graph ${\cal C\cal G}(\Sigma)$ of the boundary surface $\Sigma=\partial H_1=\partial H_2$ 
connecting the disk set ${\cal D}_1$ to the disk set ${\cal D}_2$  
contains a sufficiently long segment whose endpoints have \emph{bounded combinatorics}. 

This statement is not sufficient for the proof of Theorem \ref{hyperbolizationfinal - introduction}
as it uses an assumption which is not be fulfilled for an arbitrary 3-manifold of large
Hempel distance. To complete the proof of Theorem \ref{hyperbolizationfinal - introduction} we use
instead an approach which has some resemblance to the work \cite{FSV19}. 

Namely, we first establish Theorem \ref{lengthbound - introduction}
which gives a length bound on closed geodesics in a \emph{hyperbolic} 
manifold $M_f$ arising as boundary curves of proper essential 
subsurfaces of $\Sigma$ with
large subsurface projections of the disk sets. Then we use this
a priori length bound for two distinct curves $c_1,c_2$
arising from two different such subsurfaces 
together with Dehn surgery and Thurston's hyperbolization result for
pared acylindrical manifolds 
to find a hyperbolic metric on
$M_f$. This construction is carried out in Section \ref{convexcocompact}. 

Counting Margulis tubes and segments of a geodesic in the curve graph of
$\Sigma$ connecting the disk sets with 
no large subsurface projection then yields the lower volume bound stated in 
Theorem \ref{hyperbolizationfinal - introduction}.

\bigskip\noindent
{\bf Acknowledgements:}
This work is largely inspired by the unpublished preprint
\cite{Tian}. We are grateful to Yair Minsky for making this preprint available to us.
The first author is moreover grateful to Richard Bamler for helpful discussions 
which took place during her visit of the MSRI in Berkeley in fall 2019. 
Thanks to Yair Minsky for pointing out the reference \cite{C96}, and
to Ken Bromberg and Yair Minsky for valuable information regarding Dehn surgery.

\section{The basic set-up}\label{Sec: Preliminaries}

In this section we introduce the notations we are going to use, and collect 
some technical tools needed later on. 

\subsection{Notation}\label{Subsec: Notation}

We start by stating our notational conventions of various operations on tensors. We use the sign convention 
\[
	R(x,y)z:=\nabla_x\nabla_y z-\nabla_y\nabla_x z-\nabla_{[x,y]}z
\]
for the Riemannian curvature endomorphism. The Ricci tensor \({\rm Ric}(g)\) of a Riemannian metric \(g\) is the \((0,2)\)-tensor given by 
\[
	{\rm Ric}(g)(x,y)={\rm tr}_g\big(z \mapsto R(z,x)y\big)=\sum_{i=1}^n\langle R(e_i,x)y,e_i\rangle,
\]
where \((e_i)_{1 \leq i \leq n}\) is a local orthonormal frame. The associated \((1,1)\)-tensor is denoted by
\[
	{\rm Ric}_g(x)=\sum_{i=1}^nR(x,e_i)e_i.
\]
The \((1,1)\)-tensor \({\rm Ric}_g\) induces the \textit{Weitzenböck curvature operator}, also denoted by \({\rm Ric}_g\), that acts on \((0,2)\)-tensors \(h\) by
\[
	\Ric_g(h)(x,y)=h\big(\Ric_g(x),y\big)+h\big(x,\Ric_g(y)\big)-2 \,{\rm tr}_g h\big(\cdot, R(\cdot,x)y \big).
\]
For the covariant differentiation of tensors we use the last input as the direction of differentiation, that is,
\[
	(\nabla h)(y_1,...,y_k,x)=(\nabla_x h)(y_1,...,y_k).
\]
The \textit{adjoint of the covariant derivative} is 
\[
	(\nabla^\ast h)(y_1,...,y_{k-1}):=-\sum_{i=1}^n (\nabla h)(y_1,...,y_{k-1},e_i,e_i),
\]
where \((e_i)_{1 \leq i \leq n}\) is a local orthonormal frame. The \textit{Connection Laplacian} and \textit{Lichnerowicz Laplacian} of a \((0,2)\)-tensor \(h\) are 
\[
	\Delta h:=\nabla^\ast\nabla h \quad \text{and} \quad \Delta_L h:=\Delta h + \Ric_g(h).
\]
Similarly, we define the Laplacian of a function \(u:M \to \bbR\) as 
\[
	\Delta u:=\nabla^\ast du.
\]
The \textit{divergence} \(\delta_g\) of a \((0,2)\)-tensor \(h\) is defined as
\[
	\delta_g h:=-\sum_{i=1}^n(\nabla_{e_i}h)(\cdot, e_i)
\]
Finally, the \textit{Bianchi operator} \(\beta_g\) of a metric \(g\) acts on \((0,2)\)-tensors \(h\) by
\[
	\beta_g(h):=\delta_g(h)+\frac{1}{2}d{\rm tr}_g(h).
\]
We will often drop the metric \(g\) from the notation, and hope that this leads to no confusion with the notation of Weitzenböck cruvature operator and the Ricci tensor. More precisely, if \(g\) is a background metric, then for a generic tensor \(h\) we simply write \({\rm Ric}(h)\) for \({\rm Ric}_g(h)\), while if \(g^\prime\) is another metric, then \({\rm Ric}(g^\prime)\) is the Ricci tensor of \(g^\prime\), and not \({\rm Ric}_g(g^\prime)\).

Throughout the article we shall use the following convention regarding the use of constants appearing in analytic estimates.

\begin{convention}\label{convention constants}\normalfont In a chain of inequalities, constants denoted by the same symbol may change from line to line, and 
may depend on varying sets of parameters. In short, 
the letter \(C\) does \textit{not} always refer to the same constant.
\end{convention} 

We also use the following convention for the \(O\)-notation. Here \(X\) shall be an arbitrary set.

\begin{notation}\label{big O notation}\normalfont
For functions \(u, \varphi_1,...,\varphi_m: X \to \bbR\) we write \(u=\sum_{k=1}^m O(\varphi_k)\) if there are \textit{universal} constants \(c_k\) such that \(|u(x)| \leq \sum_{k=1}^mc_k \varphi_k(x)\) for all \(x \in X\).
\end{notation}

Moreover, we always assume the following.

\begin{convention}\label{convention orientable}\normalfont
Unless otherwise stated, all manifolds are assumed to be connected and orientable.
\end{convention}

\subsection{The Einstein operator}\label{Subsec: Einstein operator}

As mentioned in the introduction,
 we shall construct the Einstein metric by an application of the implicit function theorem for the so-called \emph{Einstein operator} (see \cite[Section I.1.C]{Biquard2000}, \cite[page 228]{Anderson2006} for more information).
This operator is defined as follows. 

Consider the operator $\Psi:g\to \Ric(g)+(n-1)g$. As the diffeomorphism group ${\rm Diff}(M)$ of the manifold $M$ acts on metrics by pull-back and $\Psi$ is equivariant for this action,
the linearization of $\Psi$ is not elliptic. To remedy this problem, for a given background metric \(\bar{g}\) one defines the \textit{Einstein operator} $\Phi_{\bar{g}}$ by 
\begin{equation}\label{Def of Phi}
	\Phi_{\bar{g}}(g):=\Ric(g)+(n-1)g+\frac{1}{2}\mathcal{L}_{(\beta_{\bar{g}}(g))^\sharp}(g),
\end{equation}
where the musical isomorphism 
\(\sharp\) is with respect to the metric \(g\). 
Using the formula for the linearisation of \(\Ric\) (\cite[Proposition 2.3.7]{topping_2006}), one shows that the linearisation of \(\Phi_{\bar{g}}\) at \(\bar{g}\) is
\[
	(D\Phi_{\bar{g}})_{\bar{g}}(h)=\frac{1}{2}\Delta_L h +(n-1)h.
\]
Hence \((D\Phi_{\bar{g}})_{\bar{g}}\) is an elliptic operator. This opens up the possibility for an application of the implicit function theorem.

It has been observed many times in the literature that the Einstein operator can detect Einstein metrics. The following observation 
can for example be found in \cite{Anderson2006} (Lemma 2.1).


\begin{lem}\label{Zeros of Phi are Einstein}
Let \((M,\bar{g})\) be a complete Riemannian manifold, and let \(g\) be another metric on \(M\) so that 
\[
	\sup_{x \in M}|\beta_{\bar{g}}(g)|(x)<\infty  \quad \text{and} \quad \Ric(g) \leq \lambda g   \,\text{ for some } \lambda < 0,
\]
where \(\beta_{\bar{g}}(\cdot)=\delta_{\bar{g}}(\cdot)+\frac{1}{2}d {\rm tr}_{\bar{g}}(\cdot)\) is the Bianchi operator of the background metric \(\bar{g}\). Denote by \(\Phi=\Phi_{\bar{g}}\) the Einstein operator defined in (\ref{Def of Phi}). Then
\begin{equation*}
 \Phi(g)=0 \quad \text{if and only if } \quad g  \text{ solves the system} \quad \begin{cases} \Ric(g)=-(n-1)g \\
	\beta_{\bar{g}}(g)=0
	\end{cases}.
\end{equation*}
\end{lem}


\subsection{Hölder norms}\label{Section - Hölder Norms}

To apply the implicit function theorem to the Einstein operator \(\Phi\), we have to study its linearization \((D\Phi)_{\bar{g}}\) at the initial metric \(\bar{g}\),
acting on a suitably chosen Banach space of sections of the symmetric tensor product ${\rm Sym}^2(T^*M)={\rm Sym}(T^*M \otimes T^*M)$. The Banach norms we shall use are hybrids of two rather classical Banach norms: 
$C^{k,\alpha}$-norms, defined locally using charts, 
and weighted $L^2$-norms.  

It is important for our main results that H\"older estimates arising from 
Schauder theory for the Einstein operator on the manifold $(M,\bar g)$ only depend on local geometric information: 
The injectivity radius, and a bound on 
$\Vert  {\rm Ric}(\bar g)\Vert_{C^1(M,\bar g)}$. Since we were unable to find a suitable reference in the literature, 
we summarize what we need in the following proposition. The existence of $C^{k,\alpha}$-norms 
with the stated properties is part of the claim and will be established below. Similar statements
can for example be found in \cite{Anderson2006} (page 230 for his definition of Hölder norms, and inequality (3.16) for the estimate).


\begin{prop}[Schauder estimate for tensors]\label{Schauder for tensors}Let \((M,g)\) be an $n$-dimensional Riemannian manifold satisfying
\[
		||\Ric(g)||_{C^1(M,g)} \leq \Lambda \quad \text{and} \quad \mathrm{inj}(M) \geq i_0.
\]
Let 
$\mathcal{S} \in \mathrm{End}\big(T^{(0,2)}M\big)$ and let $\mathcal{R} \in \mathrm{Hom}\big(T^{(0,3)}M,T^{(0,2)}M\big)$ be such that 
\[||\mathcal{R}||_{C^{0,\alpha}(M)}, \quad ||\mathcal{S}||_{C^{0,\alpha}(M)} \leq \lambda.\]
For   $f \in C^{0,\alpha}\big(T^{(0,2)}M \big)$  let $h \in C^{2,\alpha}\big(T^{(0,2)}M \big)$ be a solution of the equation 
\[
	\Delta h + \mathcal{R}(\nabla h) +\mathcal{S}(h)=f.
\]
Then it holds
\[
	||h||_{C^{2,\alpha}(M)}\leq C\left(||f||_{C^{0,\alpha}(M)}+||h||_{C^0(M)} \right)
\]
and
\[
	||h||_{C^{1,\alpha}(M)}\leq C\left(||f||_{C^0(M)}+||h||_{C^0(M)} \right)
\]
for some \(C>0\) only depending on \(n,\alpha,\lambda,\Lambda,i_0\).
\end{prop}


The remainder of this subsection is devoted to 
the construction of the Hölder norms and a sketch of the proof of Proposition \ref{Schauder for tensors}. 

\begin{proof}[Proof of Proposition \ref{Schauder for tensors}]
In local coordinates, the Connection Laplacian \(\Delta\) has the form
\[
	(\Delta h)_{ij}=-g^{kl}\partial_{kl}^2(h_{ij})+\text{Lower Order Terms}
\]
and the coefficients of the lower order terms can involve up to two derivatives of \(g_{ij}\). 
Therefore, to import the Schauder estimates from \(\bbR^n\) by writing the equation in local coordinates, we need coordinates \(\varphi\) that have the following properties:
\begin{itemize}
\item The matrix \((g_{\varphi}^{ij})\) is uniformly elliptic;
\item \(||g_{ij}^\varphi||_{C^{2,\alpha}}\) is bounded by a universal constant;
\item The coordinates are defined on a metric ball of a priori size.
\end{itemize}
The existence of such charts is guaranteed if we have a lower bound on the injectivity radius and an upper bound on \(||\Ric(g)||_{C^1(M,g)}\).
Namely, Anderson proved the following (see \cite{JK82}, \cite[Main Lemma 2.2]{Anderson1990}, \cite[page 230]{Anderson2006} and \cite[Proposition I.3.2]{Biquard2000}):

For any \(n \in \bbN\), \(\alpha \in (0,1)\), \(\Lambda \geq 0\), \(i_0 > 0\) there exist \(\rho=\rho(n,\alpha,\Lambda,i_0)>0\) and 
\(C=C(n,\alpha,\Lambda,i_0)\) with the following property. Let \((M,g)\) be a Riemannian \(n\)-manifold satisfying
\[
	||\Ric(g)||_{C^{1}(M,g)} \leq \Lambda \quad \text{and} \quad \inj(M) \geq i_0.
\]
Then for all \(p \in M\) there exists a \textit{harmonic} chart \(\varphi:B(p,2\rho) \subseteq M \to \bbR^n\) centered at $p$ so that 
\begin{equation}\label{properties of admissable harmonic charts 1}
	e^{-Q}|v|_g \leq |(D\varphi)(v)|_{\rm eucl.} \leq e^Q |v|_g
\end{equation} 
for all \(v \in TB(p,2\rho)\), and
\begin{equation}\label{properties of admissable harmonic charts 2}
	||g_{ij}^\varphi||_{C^{2,\alpha}} \leq C
\end{equation}
for all \(i,j=1,...,n\). Here \(Q > 0\) is a very small fixed constant, and \(||\cdot||_{C^{2,\alpha}}\) is the usual Hölder norm of the coefficient functions in \(\varphi(B(p,2\rho) )\subseteq \bbR^n\).

Assume from now on that the Riemannian manifold \((M,g)\) satisfies
\begin{equation*}\label{assumptions for Hölder norm}
	||\Ric(g)||_{C^1(M,g)} \leq \Lambda \quad \text{and} \quad \mathrm{inj}(M) \geq i_0
\end{equation*}
for some \(\Lambda \geq 0\) and \( i_0 > 0\). 

\begin{rem}\normalfont
The \(C^1\) bound on \(\Ric(g)\) is only used at two places in this article. First, we need it for the construction of our notion of Hölder norm. Second, it is used to obtain an upper bound on \(||\Ric(\bar{g})+(n-1)\bar{g}||_{C^{0,\alpha}}\) (see the proof of \Cref{Pinching with inj radius bound - full version}).
\end{rem}

We now state the definition of the Hölder norms. For \(k=1,2\), a \(C^k\)-tensor field \(T\) and \(p \in M\) we define the \(C^{k,\alpha}\)-norm of \(T\) at \(p\) as
\[
	|T|_{C^{k,\alpha}}(p):=\max_{i,j}||T_{ij}^\varphi||_{C^{k,\alpha}},
\]
where \(\varphi\) is some harmonic chart satisfying (\ref{properties of admissable harmonic charts 1}) and (\ref{properties of admissable harmonic charts 2}), 
and \(||\cdot||_{C^{k,\alpha}}\) is the usual Hölder norm in \(\varphi ( B(p,\frac{\rho}{2})) \subseteq \bbR^n\).  Similarly, we define the \(C^{0,\alpha}\)-norm of \(T\) at \(p\) to be
\[
	|T|_{C^{0,\alpha}}(p):=\max_{i,j}||T_{ij}^\varphi||_{C^{0,\alpha}},
\]
where \(\varphi\) is some harmonic chart satisfying (\ref{properties of admissable harmonic charts 1}) and (\ref{properties of admissable harmonic charts 2}), 
and \(||\cdot||_{C^{0,\alpha}}\) is the usual H\"older norm in \(\varphi ( B(p,\rho)) \subseteq \bbR^n\). For all \(k=0,1,2\) we also define 
\[||T||_{C^{k,\alpha}(M)}:=\sup_{p \in M}|T|_{C^{k,\alpha}}(p).\] 
Note that 
the euclidean domain in which the Hölder norms \(C^{0,\alpha}\) 
are computed is bigger than for the \(C^{1,\alpha}\) and \(C^{2,\alpha}\)-norm.

By (\ref{properties of admissable harmonic charts 1}), the matrix \((g^{ij}_{\varphi})\) is uniformly elliptic, 
and we have 
\[	\diam \big(\varphi(B(p,2\rho))\big)\leq 4e^Q \rho 
	\, \text{ and } \,
	\dist_{\bbR^n}\big(\varphi(B(p,\frac{\rho}{2})),\partial \varphi(B(p,\rho)) \big) \geq \frac{1}{2}e^{-Q}\rho.\] 
Therefore, the classical interior Schauder estimates imply
\begin{equation}\label{Schauder pointwise - C^2}
	|h|_{C^{2,\alpha}}(p)\leq C\Big(|\mathcal{L}h|_{C^{0,\alpha}}(p)+\sup_{B(p,\rho)}|h| \Big)
\end{equation}
and
\begin{equation}\label{Schauder pointwise - C^1}
	|h|_{C^{1,\alpha}}(p)\leq C\Big(\sup_{B(p,\rho)}|\mathcal{L}h|+\sup_{B(p,\rho)}|h| \Big)
\end{equation}
for a constant \(C=C(n,\alpha,\lambda,\Lambda,i_0)\), where \(\mathcal{L}\) is the elliptic operator from \Cref{Schauder for tensors}. This immediately yields \Cref{Schauder for tensors}.
\end{proof}

Apart from Schauder estimates, there is one more basic property that an elliptic operator \(\mathcal{L}\) should satisfy, namely the continuity property \(||\mathcal{L}h||_{C^{0,\alpha}(M)} \leq C||h||_{C^{2,\alpha}}\). In fact, an elliptic operator \(\mathcal{L}\) as in \Cref{Schauder for tensors} satisfies 
\begin{equation}\label{Continuity of elliptic operator - pointwise}
	|\mathcal{L}h|_{C^{0,\alpha}}(p) \leq \sup_{q \in B(p,\rho)}|h|_{C^{2,\alpha}}(q)
\end{equation}
for some \(C=C(n,\alpha,\lambda,\Lambda,i_0)\). 

\begin{proof}[Proof of (\ref{Continuity of elliptic operator - pointwise})]Fix \(p \in M\) and let \(\varphi:B(p,2\rho) \to \bbR^n\) be a harmonic chart satisyfing (\ref{properties of admissable harmonic charts 1}) and (\ref{properties of admissable harmonic charts 2}). Let \(q \in B(p,\rho)\) be arbitrary and choose a harmonic chart \(\psi:B(q,2\rho) \to \bbR^n\) satisfying (\ref{properties of admissable harmonic charts 1}) and (\ref{properties of admissable harmonic charts 2}). It suffices to show that the \(C^{3,\alpha}\)-norm of the coordinate change 
\[
	\psi \circ \varphi^{-1}: B(\varphi(q),\frac{1}{4}e^{-Q}\rho)\subseteq \bbR^n \to \psi(B(q,\frac{1}{2}\rho)) \subseteq \bbR^n
\]
is bounded by a universal constant. Note that this coordinate change is well-defined since \(\varphi\) is a \(e^Q\)-biLipschitz equivalence by (\ref{properties of admissable harmonic charts 1}). In fact, this coordinate change is even defined on \(B(\varphi(q),\frac{1}{2}e^{-Q}\rho)\).  Abbreviate \(B_1:=B(\varphi(q),\frac{1}{4}e^{-Q}\rho)\), \(B_2:=B(\varphi(q),\frac{1}{2}e^{-Q}\rho)\) and \(F:=\psi \circ \varphi^{-1}\). As \(\psi\) is a harmonic chart, we have \(\Delta_g\psi^m=0\) for every coordinate function \(\psi^m\) of \(\psi\). Also \(-\Delta_g u=g_\varphi^{ij}\frac{\partial^2 u}{\partial x^i \partial x^j}\) for any \(C^2\) function \(u\) because \(\varphi\) is harmonic. Hence for every \(m=1,...,n\) we get
\[
	g_\varphi^{ij}\frac{\partial^2 F^m}{\partial x^i \partial x^j}=0 \quad \text{in }\, B_2.
\]
Invoking the classical interior Schauder estimates yields 
\[
	||F^m||_{C^{3,\alpha}(B_1)} \leq C||F^m||_{C^0(B_2)}
\]
for a universal constant \(C\). We may without loss of generality assume that \(\psi(q)=0 \in \bbR^n\). Then \(\mathrm{Im}(\psi) \subseteq B(0,2e^Q\rho) \subseteq \bbR^n\). In particular, \(||F^m||_{C^0(B_2)} \leq 2e^Q \rho\). This completes the proof.
\end{proof}

This argument also shows that (up to equivalence) the choice of harmonic chart for the definition of the pointwise Hölder norm is irrelevant (as long as the chart satisfies (\ref{properties of admissable harmonic charts 1}) and (\ref{properties of admissable harmonic charts 2})).

\begin{rem}\label{Hölder norms without inj rad bound}\normalfont
For \Cref{Pinching without inj radius bound - introduction}, we have to consider manifolds without a lower injectivity radius bound. However, these manifolds are negatively curved, 
and hence their universal covers have infinite injectivity radius. So our notion of Hölder norms applies to the universal cover. We then define \(|T|_{C^{m,\alpha}}(p):=|\tilde{T}|_{C^{m,\alpha}}(\tilde{p})\) where $\tilde T$ is the pull-back of $T$ to the universal cover.
\end{rem}

\subsection{$C^0$-estimates}
To obtain $C^0$-estimates for the linearization of the Einstein operator, we use once again a standard tool, the  
De Giorgi–Nash–Moser estimates on manifolds in the following form. In its formulation \(\rho=\rho(n,\alpha,\Lambda,i_0)>0\) shall denote the constant appearing in the definition of the \(C^{k,\alpha}\)-norms (see the proof of \Cref{Schauder for tensors}). 

\begin{lem}[$C^0$-estimates]\label{Nash-Moser} For \(\alpha \in (0,1)\), \(\Lambda \geq 0\), \(i_0>0\), \(q>n\), \(\lambda >0$, and \(r \in (0,\rho)\)
there exists a constant $C=C(n,\alpha,\Lambda,i_0, q,\lambda, r)>0$ with the following property. 
Let \((M,g)\) be a Riemannian \(n\)-manifold satisfying 
\[
		||\Ric(g)||_{C^1(M,g)} \leq \Lambda \quad \text{and} \quad \mathrm{inj}(M) \geq i_0.
\] 
Let $X$ be a continuous vector field, and let $c$ be a continuous function on 
$M$ so that $||X||_{C^0(M)}, ||c||_{C^0(M)}\leq \lambda$. If $u \in C^2(M)$ and $f \in C^0(M)$ satisfy
\[
	-\Delta u + \langle X,\nabla u \rangle +cu \geq f,
\]
then for all \(x \in M\) it holds
\begin{equation}\label{c0} 
	u(x) \leq C\Big( ||u||_{L^2(B(x,r))}+||f||_{L^{q/2}(B(x,r))}\Big).
\end{equation}
Moreover, if \(-\Delta u + \langle X,\nabla u \rangle +cu = f\), then the same upper bounds holds for \(|u|(x) \). 
\end{lem}

\begin{rem}\normalfont
It will be apparent from the proof that assuming \(||\Ric(g)||_{C^0(M,g)} \leq \Lambda\) is sufficient for the statement of \Cref{Nash-Moser}. 
\end{rem} 

Note that when dealing with negatively curved manifolds without a lower injectivity radius bound, 
the terms on the right hand side of inequality (\ref{c0}) 
has to be replaced with the corresponding term in the universal cover.

\begin{proof}Fix \(x_0 \in M\) and pick a harmonic chart \(\varphi:B(x_0,2\rho) \to \bbR^n\) satisfying (\ref{properties of admissable harmonic charts 1}) and (\ref{properties of admissable harmonic charts 2}). Also fix \(r \in (0,\rho)\), and abbreviate \(\Omega:=\varphi\big(B(x_0,r) \big) \subseteq \bbR^n\). In the local coordinates given by \(\varphi\) the differential inequality reads
\[
g_{\varphi}^{ij}\partial_i\partial_j (u \circ \varphi^{-1})+X^i\partial_i (u \circ \varphi^{-1})+c(u \circ \varphi^{-1}) 
	\geq (f \circ \varphi^{-1}) \quad \text{in } \, \Omega \subseteq \bbR^n.\]
By (\ref{properties of admissable harmonic charts 1}) \((g_{\varphi}^{ij})\) is uniformly elliptic. Note that as \(\varphi\) is an \(e^Q\)-bi-Lipschitz equivalence onto its image, it holds \(B(\varphi(x_0),2r^\prime) \subseteq \Omega\) for \(r^\prime=\frac{1}{2}e^{-Q}r\). 
The classical De Giorgi–Nash–Moser estimates (see \cite[Theorem 8.17]{Gilbarg2001}) yield that there is
\(C=C(\lambda,q,r^\prime,n,\alpha,\Lambda,i_0)\) so that
\[
	\sup_{B(\varphi(x_0),r^\prime)} (u \circ \varphi^{-1})  \leq C\Big( ||u\circ \varphi^{-1}||_{L^2(\Omega)}+||f\circ \varphi^{-1}||_{L^{q/2}(\Omega)}\Big).
\]
Since \(\varphi^{-1}:\Omega  \to B(p,r)\) is an \(e^Q\)-biLipschitz equivalence, this completes the proof.
\end{proof}

\section{Integral inequalities}\label{Section - Integral Inequalities}

Our goal is to invert the linearisation \((D\Phi)_{\bar{g}}(h)=\frac{1}{2}\Delta h+\frac{1}{2}\Ric(h)+(n-1)h\) of the Einstein operator \(\Phi\).

\subsection{Poincaré inequalities}
As a first step, we establish   
that the Laplacian \(\Delta\) acting on the space of \((0,2)\)-tensors has a spectral gap, that is, that it satisfies a \textit{Poincaré inequality}. This is expressed by the following proposition,
which is taken from \cite{Tian} (Corollary 1 of Section 3).

\begin{prop}\label{Poincare inequality}
For every $n\geq 2$ there exist numbers \(\varepsilon(n)>0\) and \(c=c(n)>0\) with the following property. 
  Let \(M^n\) be a Riemannian manifold with
\(
	|\sec +1|\leq \varepsilon \leq \varepsilon(n);
\)
then
\[
	||h||_{L^2(M)}^2\leq \frac{1}{n-c \cdot \varepsilon}||\nabla h||_{L^2(M)}^2
\]
for all \(h \in C_c^2\big({\rm Sym}^2(T^*M)\big)\) with \({\rm tr}(h)=0\).
\end{prop}

The proof of this proposition is divided into two steps. We begin with some
purely algebraic control of the curvature tensor $R$ of a Riemannian manifold
$(M,\langle \cdot \,,\cdot \rangle)$ whose sectional curvature is close to $-1$, summarized in Lemma \ref{Estimate for (Ric(h),h)}
below. 
The second step consists in
controlling $L^2$-norms of covariant derivatives. 

For the algebraic control, 
recall that a Riemannian manifold \((M\) has constant sectional curvature \(\kappa\) if and only if the curvature endomorphism satisfies 
\(R(x,y)z=\kappa(\langle y,z\rangle x - \langle x,z\rangle y )\). Motivated by this, we define for \(\kappa \in \bbR\) the tensors \(R^\kappa\) and \(Rm^\kappa\) by
\begin{equation}\label{def of R^k and Rm^k}
	R^\kappa(x,y)z:=\kappa(\langle y,z\rangle x - \langle x,z\rangle y) \quad \text{and} \quad Rm^\kappa(x,y,z,w):=\langle R^\kappa(x,y)z, w\rangle.
\end{equation}
More specifically, we denote \(R^{hyp}:=R^{-1}\) and \(Rm^{hyp}:=Rm^{-1}\), where \textit{hyp} stands for \textit{hyperbolic}. We also remark that bounding \( |Rm-Rm^\kappa|\) is equivalent to bounding \(|\mathrm{sec}-\kappa|\). More precisely, there is a constant \(c(n)>0\) such that
\begin{equation}\label{Rm-Rm^k small iff sec almost k}
	\sup_{\sigma}\big|\mathrm{sec}(\sigma)-\kappa\big| \leq |Rm-Rm^\kappa|\leq c(n) \cdot \sup_{\sigma}\big|\mathrm{sec}(\sigma)-\kappa\big|,
\end{equation}
where the supremum is taken over all planes \(\sigma \subseteq T_pM\)
and \(p \in M\) is arbitrary.
This is clear because curvature operators are determined by their sectional curvatures
through an explicit formula (see \cite[Exercise 3.4.29]{Petersen2016}). 

With these notations we observe

\begin{lem}\label{Estimate for (Ric(h),h)} Let \(M\) be a Riemannian \(n\)-manifold and \(\kappa \in \bbR\). Then the pointwise estimate
\[
	\left|\frac{1}{2}\langle \Ric(h),h\rangle -\kappa\big(n|h|^2-{\rm tr}(h)^2\big) \right| \leq (1+\sqrt{n})|Rm-Rm^\kappa| |h|^2
\]
holds for all symmetric \((0,2)\)-tensors \(h\).
\end{lem}

By the irreducible decomposition of the curvature tensor (see Section G of Chapter 1 in \cite{Besse1987}) \(Rm\) decomposes as
\[
	Rm= \frac{{\rm scal}}{2n(n-1)} g \owedge g + \frac{1}{n-2} \left({\rm Ric}(g)-\frac{{\rm scal}}{n}g \right) \owedge g +W,
\]
where \(\owedge\) is the Kulkarni-Nomizu product, \({\rm scal}={\rm tr}({\rm Ric}(g))\) is the scalar curvature, and \(W\) is the Weyl tensor. Using the Kulkarni-Nomizu one can write \(Rm^\kappa=\frac{\kappa}{2}g \owedge g\). Therefore, if \(g\) is almost Einstein in the sense that \(|{\rm Ric}(g)-(n-1)\kappa g|\) is small, then \(|Rm-Rm^\kappa|\) is small if and only if the norm of the Weyl tensor \(|W|\) is small.

\begin{proof}For this proof we introduce the following abbreviations (here $(e_i)_{1\leq i\leq n}$ is an orthonormal frame): 
\begin{itemize}
\item \(\rho(h)(x,y)=\frac{1}{2}\Big( h\big(\Ric(x),y\big)+h\big(x,\Ric(y)\big)\Big) \);
\item \(L(h)(x,y):={\rm tr} h\big(\cdot, R(\cdot,x)y\big)\);
\item \(\Ric^\kappa(x,y)={\rm tr} Rm^\kappa(x,\cdot,\cdot,y)\) and \(\Ric^\kappa(x)=\sum_{i}R^\kappa(x,e_i)e_i\);
\item \(\rho^\kappa(h)(x,y)=\frac{1}{2}\Big( h\big(\Ric^\kappa(x),y\big)+h\big(x,\Ric^\kappa(y)\big)\Big) \);
\item \(L^\kappa(h)(x,y):= {\rm tr} h\big(\cdot, R^\kappa(\cdot,x)y\big)\).
\end{itemize}
\begin{claim}[1]For every symmetric \((0,2)\)-tensor \(h\), it holds
\[
	\big|\langle L(h)-L^\kappa(h), h\rangle\big| \leq |Rm-Rm^\kappa| |h|^2 \quad \text{and} \quad \big|\langle \rho(h)-\rho^\kappa(h),h \rangle\big|\leq \sqrt{n}\,|Rm-Rm^\kappa||h|^2.
\]	
\end{claim}
\begin{proof}[Proof of Claim 1]
  The first inequality follows by writing the expression in
  an orthonormal basis and invoking the Cauchy-Schwarz inequality.

For the second inequality, choose an orthonormal basis \((e_i)_{1\leq i \leq n}\) so that \(h(e_k,e_l)=0\) for \(k\neq l\). Writing the expression in this frame and using \(|\Ric(x,y)-\Ric^\kappa(x,y)| \leq \sqrt{n}|Rm-Rm^\kappa||x||y|\) (which holds since \(|{\mathrm {tr}}(\cdot)|\leq \sqrt{n} |\cdot|\)) yields the second inequality. 
\end{proof}
\begin{claim}[2]For every symmetric \((0,2)\)-tensor \(h\), we have
\[
	\langle L^\kappa(h), h\rangle =\kappa\big({\rm tr}(h)^2-|h|_g^2\big) \quad \text{and} \quad \langle \rho^\kappa(h),h \rangle=\kappa (n-1)|h|_g^2.
\]	
\end{claim}
\begin{proof}[Proof of Claim 2]
  Note that \(\Ric^\kappa(x)=\kappa(n-1)x\). So the second equality is clear.
  Choose an orthonormal basis \((e_i)_{1\leq i \leq n}\) for the metric $g$ 
  so that \(h(e_k,e_l)=0\) for \(k\neq l\). Then
\[
	L^\kappa(h)(e_l,e_l)=\sum_{i,j}h(e_i,e_j)\kappa (\langle e_i, e_j\rangle -\langle e_i, e_l\rangle\langle e_j, e_l\rangle)=\sum_{i}h_{ii}\kappa(1-\delta_{il})=\kappa\sum_{i \neq l}h_{ii}
\]
and thus
\[
	\langle L^\kappa(h), h\rangle=\sum_{l}L^\kappa(h)_{ll}h_{ll}=\kappa\sum_{l}\sum_{i \neq l}h_{ll}h_{ii}=\kappa\big((\sum_{i}h_{ii})^2-\sum_{i}h_{ii}^2\big)=
	\kappa\big({\rm tr}(h)^2-|h|_g^2 \big).
\]
This finishes the proof of the second claim.
\end{proof}
As \(\frac{1}{2}\langle \Ric(h),h\rangle =\langle \rho(h)-L(h),h \rangle\), the two claims immediately imply the desired result.
\end{proof}

The following $L^2$-identity is the second auxiliary result we need.

\begin{lem}\label{1st integral inequality}
Let \(M\) be a Riemannian manifold. Then it holds
\[
	0 \leq ||\nabla h||_{L^2(M)}^2 + \frac{1}{2}\big(\Ric(h),h \big)_{L^2(M)}
\]
for all \(h \in C_c^2\big({\rm Sym}^2(T^*M)\big)\).
\end{lem}

\begin{proof}This follows immediately from the equation on the bottom of page 355 of \cite{Besse1987}. Indeed, this equation states
\[
	(\delta^\nabla d^\nabla + d^\nabla \delta^\nabla )h=\nabla^\ast \nabla h - \overset{\circ}{R}h + h \circ r.
\]
By \(\overset{\circ}{R}h\) Besse denotes \((\overset{\circ}{R}h)(x,y)={\rm tr}h\big(\cdot,R(\cdot,x)y) \big)\) (see \cite[page 51]{Besse1987} and note that \cite{Besse1987} uses the sign convention opposite to ours). Moreover, \(r\) is the \((0,2)\)-Ricci tensor \({\rm Ric}(g)\), and \(\circ\) denotes the symmetric product, that is, the pairing that under the isomorphism \({\rm Sym}^2(T^\ast M) \cong {\rm Sym}\big({\rm End}(TM)\big)\) corresponds to the symmetric product \((L,L^\prime) \to \frac{1}{2}(L \circ L^\prime + L^\prime \circ L)\), where \(L \circ L^\prime\) is the composition of endomorphisms. So \((h \circ r)(x,y)=\frac{1}{2}\big(h({\rm Ric}(x),y)+h(x,{\rm Ric}(y)) \big)\), and hence the right hand side in the equation above is just \(\nabla^\ast \nabla h +\frac{1}{2}{\rm Ric}(h)\).

Moreover, \(d^\nabla\) is the exterior differential on \(T^\ast M\)-valued 1-forms induced by the Levi-Civita connection \(\nabla\), and \(\delta^\nabla\) denotes its dual (note that a symmetric \((0,2)\)-tensor can be thought of as a \(T^\ast M\)-valued 1-form). Therefore, 
\begin{align*}
	0 \leq & ||d^\nabla h||^2_{L^2(M)}+||\delta^\nabla h||^2_{L^2(M)} \\
	= & \big((\delta^\nabla d^\nabla + d^\nabla \delta^\nabla )h,h \big)_{L^2(M)} \\
	= & (\nabla^\ast \nabla h,h)_{L^2(M)}+\frac{1}{2}\big({\rm Ric}(h),h \big)_{L^2(M)} \\
	= & ||\nabla h||^2_{L^2(M)}+\frac{1}{2}\big({\rm Ric}(h),h\big)_{L^2(M)}.
\end{align*}
This completes the proof.
\end{proof}

We are now ready for the proof of Proposition \ref{Poincare inequality}.

\begin{proof}[Proof of \Cref{Poincare inequality}]
  We copy the proof from \cite[Corollary 1 in Section 3]{Tian}.
  By \Cref{1st integral inequality}, we have
  \(0 \leq ||\nabla h||_{L^2(M)}^2+\frac{1}{2}\big(\Ric(h),h\big)_{L^2(M)}\). As
 \(|\sec+1| \leq \varepsilon\), it holds \(|Rm-Rm^{hyp}|\leq \varepsilon c_0\) for some constant \(c_0=c_0(n)\) due to (\ref{Rm-Rm^k small iff sec almost k}). Hence Lemma \ref{Estimate for (Ric(h),h)} yields \(-\frac{1}{2}\langle \Ric(h),h\rangle \geq  \big(n- \varepsilon c_0(1+\sqrt{n})\big)|h|^2\). Combining these two inequalities completes the proof in case \(\varepsilon\) is small enough that \(n- \varepsilon c_0(1+\sqrt{n}) > 0\).
\end{proof}

\subsection{Weighted $L^2$-norms}
The main tool for obtaining 
a priori \(C^0\)-estimates for the differential equation \((D\Phi)_{\bar{g}}(h)=f\)
which are independent of \(\vol(M)\) is the use of \textit{hybrid norms}
that are a mixture of Hölder norms and weighted Sobolev norms (see \Cref{subsection - hybrid norm}).
The next proposition establishes the required a priori estimates for weighted \(L^2\)-norms.
It follows Corollary 2 in Section 3 of \cite{Tian}.

\begin{prop}\label{weighted integral estimate - smooth}
Let \(M\) be a complete Riemannian \(n\)-manifold of finite volume, let
\(f \in C^0\big({\rm  Sym}^2(T^*M)\big)\cap L^2(M)\) and let \(h \in C^2\big({\rm  Sym}^2(T^*M)\big)\cap L^2(M)\)
be a solution of
\[
	\frac{1}{2}\Delta_L h+(n-1)h=f.
\]
Let \(\varphi \in C^\infty(M)\) be so that \(\varphi h, \, \varphi f \in L^2(M)\). Then it holds
\begin{align*}
	(n-2)\int_M \varphi^2|h|^2 \, d\vol & \leq  2 \int_M \varphi^2\langle h,f\rangle \, d\vol +\int_M |\nabla\varphi|^2|h|^2 \,d\vol \\
	&+ (1+\sqrt{n})\int_M  \varphi^2|Rm-Rm^{hyp}| |h|^2 \, d\vol.
\end{align*}
\end{prop}

The weights appearing in our hybrid norms will \textit{not} be smooth.
For convenience, 
we state the precise version of Proposition \ref{weighted integral estimate - smooth}.

\begin{cor}\label{weighted integral estimate - lipschitz} Let \(M,h,f\)
  be as in \Cref{weighted integral estimate - smooth}, and
  let \(\rho:\mathbb{R}\to \mathbb{R}\) be
a smooth function so that \(\rho\) and \(\rho^\prime\) are bounded on \(\bbR_{\geq 0}\). Then for every \(x \in M\) we have
\begin{align*}
	(n-2)\int_M \rho(r_x)^2|h|^2 \, d\vol  & \leq   2\int_M \rho(r_x)^2\langle h,f\rangle \, d\vol +\int_{M} \big(\rho^\prime(r_x)\big)^2|h|^2 \,d\vol \\
	&+ (1+\sqrt{n})\int_M  \rho(r_x)^2|Rm-Rm^{hyp}| |h|^2 \, d\vol,
\end{align*}
where \(r_x=d_M(\cdot,x )\).
\end{cor}

\begin{proof}Theorem 1 in \cite{Azagra2007} shows that for every \(\epsilon>0\) there is a Lipschitz function \(r_\epsilon \in C^\infty(M)\) such that \(||r_\epsilon-r_x||_{C^0(M)}<\epsilon\) and \(\mathrm{Lip}(r_\epsilon)\leq 1+\epsilon \). Consider \(\varphi_\epsilon:=\rho\circ r_\epsilon \), and note that \(\varphi_\varepsilon h,\varphi_\varepsilon f \in L^2(M)\) as \(\rho\) is bounded. Applying \Cref{weighted integral estimate - smooth} to \(\varphi_\epsilon:=\rho\circ r_\epsilon \) and taking \(\epsilon \to 0\) implies the desired result.
\end{proof}

\begin{proof}[Proof of \Cref{weighted integral estimate - smooth}]
 We adapt the proof from \cite[Corollary 2 in Section 3]{Tian} and include further details. 

We first consider the case that \(\varphi\) has compact support. Abbreviate \(\tilde{h}=\varphi h\). A direct computation yields
\[
	\Delta \tilde{h}=(\Delta\varphi)h-2{\rm tr}^{(1,4)}(\nabla\varphi \otimes \nabla h)+\varphi \Delta h.
\]
This implies
\[
	\Delta \tilde{h}=(\Delta\varphi)h-2{\rm tr}^{(1,4)}(\nabla\varphi \otimes \nabla h)+2\varphi f-2(n-1)\tilde{h}-\Ric(\tilde{h})
\]
as \(f=\frac{1}{2}\Delta h +\frac{1}{2}\Ric(h)+(n-1)h\). Consider the \(1\)-form \(\omega=\varphi |h|^2 d\varphi\). A straightforward calculation in a local orthonormal frame shows
\[
	-\nabla^\ast \omega=|\nabla \varphi|^2|h|^2+2\langle {\rm tr}^{(1,4)}(\nabla\varphi \otimes \nabla h),\tilde{h}\rangle - \langle (\Delta \varphi) h,\tilde{h}\rangle.
\]
Note that \(\int_M \nabla^\ast \omega \, d\vol=0\). Indeed, \(-\nabla^\ast \omega={\rm div}(\omega^\sharp)\) because the musical isomorphisms commute with covariant differentiation, and \(\int_M {\rm div}(\omega^\sharp) \, d\vol=0\) by the divergence theorem since \(\omega^\sharp\) is compactly supported.
 
Thus \(\big((\Delta\varphi) h,\tilde{h}\big)_{L^2}-2\big({\rm tr}^{(1,4)}(\nabla\varphi \otimes \nabla h),\tilde{h} \big)_{L^2}=||\,|\nabla \varphi|\,  h ||_{L^2}^2\). Together with \Cref{1st integral inequality} this shows
\begin{align*}
	0\leq & \big(\Delta\tilde{h},\tilde{h} \big)_{L^2}+\frac{1}{2}\big(\Ric(\tilde{h}),\tilde{h} \big)_{L^2} \\
	=&|| \,|\nabla \varphi|\, h||_{L^2}^2+2\big(\varphi f,\tilde{h} \big)_{L^2}-2(n-1)||\tilde{h}||_{L^2}-\frac{1}{2}\big(\Ric(\tilde{h}),\tilde{h} \big)_{L^2}.
\end{align*}
Lemma \ref{Estimate for (Ric(h),h)} implies
\begin{align*}
	-\frac{1}{2}\langle \Ric(\tilde{h}),\tilde{h} \rangle\leq & \big(n|\tilde{h}|^2-{\rm tr}(\tilde{h})^2\big)+ (1+\sqrt{n})|Rm-Rm^{hyp}| |\tilde{h}|^2 \\
	\leq & n|\tilde{h}|^2+ (1+\sqrt{n})|Rm-Rm^{hyp}| |\tilde{h}|^2.
\end{align*}
Hence
\[
	(n-2)||\tilde{h}||_{L^2}^2 \leq 2 \big( \varphi f, \tilde{h}\big)_{L^2}+|| \,|\nabla \varphi|\, h||_{L^2}^2+(1+\sqrt{n})\int_{M} |Rm-Rm^{hyp}| |\tilde{h}|^2 \, d\vol.
\]
Since \(\tilde{h}=\varphi h\) this finishes the case that \(\varphi\) is compactly supported.

We now consider the general case. Choose a pointwise non-decreasing sequence of bump functions \((\psi_k)_{k \in \bbN} \subseteq C_c^\infty(M)\) so that \(0 \leq \psi_k \leq 1\), \(||\nabla \psi_k||_{C^0(M)}\leq \frac{1}{k}\), and \(\psi_k \to 1\) pointwise. Since \((a+b)^2\leq (1+k)a^2+(1+\frac{1}{k})b^2\),
we have
\[|\nabla (\psi_k \varphi)|^2 \leq (1+k)|\nabla \psi_k|^2\varphi^2+\left(1+\frac{1}{k}\right)\psi_k^2 |\nabla \varphi|^2 \leq \frac{1+k}{k^2}\varphi^2+\left(1+\frac{1}{k}\right) |\nabla \varphi|^2.\]
Applying the result for the case of compact support to \(\varphi_k:=\psi_k\varphi\) gives
\begin{align*}
	(n-2)\int_M \psi_k^2\varphi^2|h|^2 \, d\vol  &
	\leq  2\int_M \psi_k^2\varphi^2\langle h,f\rangle \, d\vol \\
	& +\frac{1+k}{k^2}||\varphi h||_{L^2}^2 + \left(1+\frac{1}{k}\right)\int_M |\nabla \varphi|^2|h|^2\, d\vol\\
	& + (1+\sqrt{n})\int_M  \psi_k^2\varphi^2|Rm-Rm^{hyp}| |h|^2 \, d\vol.
\end{align*}
Taking \(k \to \infty\) implies the desired result. Indeed, the second summand on the right hand side converges to \(0\) since \(\varphi h \in L^2(M)\), and the first and fourth summand converge by dominated convergence because \(\varphi h, \varphi f \in L^2(M)\). Also we may assume \(|\nabla \varphi|h \in L^2(M)\) since otherwise the desired inequality trivially holds. So the third summand on the right hand side also converges. 
\end{proof}

\section{Invertibility of \(\mathcal{L}h=\frac{1}{2}\Delta_Lh+(n-1)h\)}\label{Section - Invertibiliy of L}

In order to apply the implicit function theorem
it is necessary to invert the linearisation of the Einstein operator \(\Phi\) at the original metric \(\bar{g}\). This linearisation is given by 
\[
	(D\Phi)_{\bar{g}}(h)=\frac{1}{2}\Delta_Lh+(n-1)h.
\]
For simplicity of notation we abbreviate this operator by \(\mathcal{L}\). It is of utmost importance that \(||\mathcal{L}^{-1}||_{\rm op}\) is bounded by some constant
that is independent of \(\vol(M)\). To achieve this we consider special \textit{hybrid norms} that are defined in \Cref{subsection - hybrid norm}. The a priori estimate is then proven in \Cref{Subsec: a priori estimate}.

\subsection{The hybrid norm}\label{subsection - hybrid norm}

Bounding the operator norm of the inverse \(\mathcal{L}^{-1}\) boils down to proving an a priori estimate \(||h||_{\rm source} \leq C ||\mathcal{L}h||_{\rm target}\). As \(\mathcal{L}\) is an elliptic operator, it is natural to work with Hölder norms and use the Schauder estimates
established in \Cref{Schauder for tensors}. To obtain constants that are
independent of \(\vol(M)\) we define norms that are a combination of Hölder
and weighted Sobolev norms.
The basic reason for this is that \(C^0\)-bounds can be deduced
from \(L^2\)-bounds by De Giorgi–Nash–Moser estimates (\Cref{Nash-Moser}). 

Let \(M\) be a Riemannian manifold of dimension $n\geq 3$. For \(\bar{\epsilon}>0, \delta \in (0,2\sqrt{n-2})\) and \(r_0 \geq 1\) define 
\begin{equation}\label{def of set E}
	E:=E(M;\bar{\epsilon},\delta,r_0):=\left\{x \in M \, \left| \, \int_{B(x,2r_0)\setminus B(x,r_0)}e^{-(2\sqrt{n-2}-\delta)r_x(y)} \, d\vol(y) \leq \bar{\epsilon} \right.\right\},
\end{equation}
where \(r_x(y)=d_M(x,y)\). 
Inside \(E\) we will be able to bound \(||h||_{C^0}\) in terms of \(||f||_{C^0}\). Outside of \(E\)
we will use \Cref{Nash-Moser} to bound \(||h||_{C^0}\)
in terms of \(||f||_{L^2}\). The norms we define are
supposed to capture \(L^2\)-information outside of \(E\).

We now come to the precise definition, which we take from Section 5 of \cite{Tian}. 

\begin{definition}\label{hybridnorm}For \(\alpha \in (0,1)\), \(\bar{\epsilon}>0\), \(\delta \in (0,2\sqrt{n-2})\) and \(r_0 \geq 1\)
the \textit{hybrid norms} \(|| \cdot ||_k \) on \(C^{k,\alpha}\big({\rm Sym}^2(T^*M)\big)\) are defined as
\begin{equation}\label{Hybrid 2-norm}
	||h||_2:=\max \left\{||h||_{C^{2,\alpha}(M)}\, , \, \sup_{x \nin  E}\left(\int_Me^{-(2\sqrt{n-2}-\delta)r_x(y)}\big(|h|^2+|\nabla h|^2+|\Delta h|^2\big) \, d\vol(y)\right)^{\frac{1}{2}}\right\}
\end{equation}
and
\begin{equation}\label{Hybrid 0-norm}
	||f||_0:=\max \left\{||f||_{C^{0,\alpha}(M)} \, , \, \sup_{x \nin  E}\left(\int_Me^{-(2\sqrt{n-2}-\delta)r_x(y)}|f|^2 \, d\vol(y)\right)^{\frac{1}{2}}\right\},
\end{equation}
where \(E=E(M;\bar{\epsilon},\delta,r_0)\) is the set defined in (\ref{def of set E}).
\end{definition} 

Strictly speaking these norms also depend on a choice of constants \(\Lambda \geq 0\) and \(i_0>0\) for which it holds \(||\nabla \Ric(g)||_{C^0(M,g)}\leq \Lambda\) and \(\inj(M,g)\geq i_0\). This is because our notion of Hölder norm depends on this geometric information (see \Cref{Schauder for tensors} and its proof).

The reason why we use weights of the form \(e^{-(2\sqrt{n-2}-\delta)r_x}\)  (and not \(e^{-a r_x}\) for arbitrary big \(a > 0\)) is that in order to obtain weighted \(L^2\)-estimates we can only use weights \(e^{2\omega}\) for functions 
\(\omega\) satisfying \(|\nabla \omega| < \sqrt{n-2}\). This is because in the estimate of \Cref{weighted integral estimate - smooth} the factor \(n-2\) on the left hand side needs to be able to absord the factor \(|\nabla \omega|^2\) on the right hand side.

For later purposes it will be useful to mention the following equivalent version of the norm \(||\cdot||_2\).

\begin{rem}\label{alternative hybrid norm}\normalfont Let \(M\) be a closed Riemannian \(n\)-manifold with \(|\mathrm{sec}| \leq 2\). The norm \(||\cdot||_{2}\) is equivalent to the norm
\[
		||h||_{2}^\prime:=\max \left\{||h||_{C^{2,\alpha}(M)} \, , \, \sup_{x \nin  E}\left(\int_Me^{-(2\sqrt{n-2}-\delta)r_x(y)}\big(|h|^2+|\nabla h|^2+|\nabla^2 h|^2\big) \, d\vol(y)\right)^{\frac{1}{2}}\right\}.
\]
and the equivalence constants can be chosen to depend only on \(n\).
\end{rem}

\begin{proof}[Sketch of proof]By approximation it suffices to show that for any smooth Lipschitz function \(\varphi:M \to \bbR\) and any smooth \((0,2)\)-tensor \(h\) it holds
\[
	\int_M e^{\varphi}|\nabla^2h|^2 \, d\vol \leq c\int_Me^{\varphi}\big(|h|^2+|\nabla h|^2+|\Delta h|^2\big) \, d\vol
\]
for a constant \(c\) depending only on \(n\) and \({\rm Lip}(\varphi)\) (see \cite[Theorem 1]{Azagra2007} for the approximation of Lipschitz functions). It is a well-known fact that for \(u \in C_c^\infty(\bbR^n,\bbR)\) it holds \(\int_{\bbR^n}|\nabla^2 u| \, dx=\int_{\bbR^n}|\Delta u|^2 \, dx\) (see (3) on page 326 of \cite{EvansPDE}). In fact, this only needs integration by parts and the fact that second order partial derivatives commute. A similar calculation applies in the present context. The only difference is that now second order partial derivatives only commute up to a term involving the curvature tensor. But since \(|\sec|\leq 2\) we have uniform control on any term involving the curvature.
\end{proof}

\subsection{The a priori estimate}\label{Subsec: a priori estimate}

After having introduced the hybrid norms \(||\cdot||_2\) and \(||\cdot||_0\) in the previous subsection, we now prove that 
\[
	\mathcal{L}: \Big( C^{2,\alpha}\big( {\rm Sym}^2(T^*M)\big), ||\cdot||_2 \Big) \longrightarrow \Big( C^{0,\alpha}\big( {\rm Sym}^2(T^*M)\big), ||\cdot||_0 \Big)
\]
satisfies an a priori estimate with a constant independent of \(\mathrm{vol}(M)\). This is Proposition 5.1 in \cite{Tian}.

\begin{prop}\label{A-priori estimate for L}For all \(n \geq 3\), \(\alpha \in (0,1)\), \(\Lambda \geq 0\), \(\delta \in (0,2\sqrt{n-2})\) and \(r_0 \geq 1\) there exist constants \(\varepsilon_0\), \(\bar{\epsilon}_0\) and \(C >0\) with the following property. Let \(M\) be a closed Riemannian \(n\)-manifold with
\[
	 |\,\mathrm{sec}+1 \, |\leq \varepsilon_0, \quad \mathrm{inj}(M) \geq 1 \quad \text{and} \quad ||\nabla \Ric||_{C^0(M)}\leq \Lambda.
\]
Then for all \(h \in C^{2,\alpha}\big( {\rm Sym}^2(T^*M)\big)\) it holds
\[
	||h||_2 \leq C ||\mathcal{L}h||_0,
\]
where \(||\cdot||_2\) and \(||\cdot||_0\) are the norms defined in (\ref{Hybrid 2-norm}) and (\ref{Hybrid 0-norm}) with respect to any \(\bar{\epsilon} \leq \bar{\epsilon}_0\).
\end{prop}

\begin{proof}Our proof follows that in \cite{Tian}. We add further details and at times give alternative arguments. Abbreviate \(f:=\mathcal{L}h\). 

\medskip
  \textbf{Step 1 (Integral estimate):}
  Let \(c(n,\delta):=n-2-(\sqrt{n-2}-\delta/2)^2\) and
  \(\epsilon_0 \leq \frac{c(n,\delta)}{2(1+\sqrt{n})} \). 
Assume that \(||Rm-Rm^{hyp}||_{C^0(M)} \leq \varepsilon_0\).
An application of  \Cref{weighted integral estimate - lipschitz} with \(\rho(t)=e^{-(\sqrt{n-2}-\delta/2)t}\) gives
\begin{align*}
	c(n,\delta)\int_M e^{-(2\sqrt{n-2}-\delta)r_x}|h|^2 \, d\vol  & \leq 2\int_M e^{-(2\sqrt{n-2}-\delta)r_x}\langle f, h \rangle \, d\vol \\
	& + (1+\sqrt{n})\varepsilon_0 \int_M e^{-(2\sqrt{n-2}-\delta)r_x}|h|^2 \, d\vol,
\end{align*}
and consequently 
\[
	\frac{c(n,\delta)}{2}\int_M e^{-(2\sqrt{n-2}-\delta)r_x}|h|^2 \, d\vol \leq \int_M e^{-(2\sqrt{n-2}-\delta)r_x}\left(\frac{c(n,\delta)}{4}|h|^2+\frac{4}{c(n,\delta)}|f|^2 \right) \, d\vol
\]
by the Cauchy-Schwarz inequality
and the inequality between the arithmetic and the geometric mean. Hence 
\begin{equation}\label{int1 - weight}
	\int_M e^{-(2\sqrt{n-2}-\delta)r_x}|h|^2 \, d\vol \leq \frac{16}{c(n,\delta)^2}\int_M e^{-(2\sqrt{n-2}-\delta)r_x}|f|^2 \, d\vol.
\end{equation}
Note that \(|\Ric(h)| \leq c(n)|h|\) because \(|\mathrm{sec}|\leq 2\). Testing the equation \(f-\frac{1}{2}\Ric(h)-(n-1)h=\frac{1}{2}\Delta h \) with \(e^{-(2\sqrt{n-2}-\delta)r_x}\Delta h\) and using the inequality between the weighted arithmetic and the weighted geometric mean implies
\begin{equation}\label{int2 - weight}
	\int_M e^{-(2\sqrt{n-2}-\delta)r_x}|\Delta h|^2 \, d\vol \leq C(n)\int_M e^{-(2\sqrt{n-2}-\delta)r_x}\big(|h|^2+|f|^2\big) \, d\vol.
      \end{equation}
Set \(\varphi:=(2\sqrt{n-2}-\delta)r_x\). We may act as if \(\varphi\) were smooth. Indeed, by Theorem 1 in \cite{Azagra2007} there exists a sequence \((\varphi_\varepsilon)_{\varepsilon > 0} \subseteq C^\infty(M)\) so that \(\lim_{\varepsilon \to 0}||\varphi-\varphi_\varepsilon||_{C^0(M)}=0\) and \({\rm Lip}(\varphi_\varepsilon)\leq 2\sqrt{n-2}\). Then all the arguments below apply to \(\varphi_\varepsilon\), so that (\ref{integral1}) will hold for \(\varphi_\varepsilon\) instead of \(\varphi\). But then taking \(\varepsilon \to 0\) will yield (\ref{integral1}) for \(\varphi\).

Using       
\(\frac{1}{2}\Delta (|h|^2)=\langle \Delta h,h \rangle - |\nabla h|^2 \leq \frac{1}{2}|\Delta h|^2+\frac{1}{2}|h|^2-|\nabla h|^2\) we obtain
\begin{equation}\label{inttt}
	\begin{split}
	\int_M e^{-\varphi}|\nabla h|^2 \, d\vol \leq & \frac{1}{2}\int_M e^{-\varphi}|h|^2 \, d\vol+\frac{1}{2}\int_M e^{-\varphi}|\Delta h|^2 \, d\vol \\
	&- \frac{1}{2}\int_M e^{-\varphi}\Delta (|h|^2) \, d\vol. 
	\end{split}
\end{equation}
Integration by parts shows that 
\[
	- \frac{1}{2}\int_M e^{-\varphi}\Delta (|h|^2) \, d\vol=-\frac{1}{2}\int_M \langle \nabla (|h|^2),\nabla (e^{-\varphi}) \rangle \, d\vol, 
      \]
  moreover \(\left|\frac{1}{2}\langle\nabla (|h|^2),\nabla \varphi \rangle\right| \leq |h||\nabla h||\nabla \varphi| \leq \frac{1}{2}|\nabla \varphi|^2|h|^2+\frac{1}{2}|\nabla h|^2 \). 
  Absorbing \(\frac{1}{2}e^{-\varphi}|\nabla h|^2\) to the left hand side of inequality (\ref{inttt}) 
  and using \(|\nabla \varphi| \leq 2\sqrt{n-2}\) yields
\begin{equation}\label{int3 - weight}
\begin{split}
	\int_M e^{-\varphi}|\nabla h|^2 \, d\vol \leq & \int_M e^{-\varphi}|h|^2 \, d\vol+\int_M e^{-\varphi}|\Delta h|^2 \, d\vol \\
	&+ 4(n-2)\int_M e^{-\varphi}|h|^2 \, d\vol.
\end{split}
\end{equation}
Therefore 
\begin{equation}\label{integral1}
	\int_M e^{-(2\sqrt{n-2}-\delta)r_x}\big(|h|^2+|\nabla h|^2+|\Delta h|^2\big) \, d\vol \leq C(n,\delta) \int_M e^{-(2\sqrt{n-2}-\delta)r_x}|f|^2 \, d\vol
\end{equation}
by combining (\ref{int1 - weight}), (\ref{int2 - weight}) and (\ref{int3 - weight}). This completes the integral estimates.

\medskip
\textbf{Step 2 (\(C^0\)-estimate):}
It remains to estimate \(||h||_{C^{2,\alpha}(M)}\).
By \Cref{Schauder for tensors}, it suffices to bound \(||h||_{C^0(M)}\). We reduce the \(C^0\)-estimate to an \(L^2\)-estimate. Namely, we show that there is a constant \(C=C(n,\alpha,\Lambda)\) so that for each $x\in M$ it holds
\begin{equation}\label{C^0 from L^2}
	|h|(x) \leq C \big(||h||_{L^2(B(x,\rho))}+||f||_{C^0(B(x,\rho))} \big),
\end{equation}
where \(\rho\) is the constant appearing in the definition of the Hölder norms. This will follow from the De Giorgi–Nash–Moser estimates of \Cref{Nash-Moser}. The problem is that De Giorgi–Nash–Moser estimates only hold for scalar equations, but not for systems. For this reason we can not directly apply \Cref{Nash-Moser} to \(\mathcal{L}h=f\). To remedy this, we show that \(|h|\) satisfies an elliptic partial differential inequality.

Recall \(f=\frac{1}{2}\Delta h+\frac{1}{2}\Ric(h)+(n-1)h \). Using \(\frac{1}{2}\Delta (|h|^2)=\langle \Delta h,h \rangle -|\nabla h|^2\) and the estimate on \(\frac{1}{2}\langle\Ric(h),h\rangle\) from Lemma \ref{Estimate for (Ric(h),h)} we get (assuming  \(\varepsilon_0 \leq \frac{1}{1+\sqrt{n}}\))
\begin{align}\label{intttt}
	-\frac{1}{2}\Delta (|h|^2)&=-2\langle f,h\rangle + \langle \Ric(h),h\rangle +2(n-1)|h|^2+ |\nabla h|^2 \notag \\
	& \geq -2|f||h| -2(n+\varepsilon_0(1+\sqrt{n}))|h|^2+2{\rm tr}(h)^2 +2(n-1)|h|^2 + |\nabla h|^2 \notag\\
	& \geq -2|f||h| -4|h|^2 + |\nabla h|^2.
\end{align}
Suppose for the moment that \(h \neq 0\) everywhere. Then \(|h|\) is a nowhere vanishing \(C^2\) function. Observe
\[
	|\nabla(|h|)|\leq |\nabla h| \quad \text{and} \quad -\frac{1}{2}\Delta(|h|^2)=-|h|\Delta(|h|)+|\nabla(|h|)|^2.
\]
Combining this with inequality (\ref{intttt}) and dividing by \(|h|\) shows
\begin{equation}\label{differential inequality for |h|}
	-\Delta(|h|)\geq -2|f|-4|h|.
\end{equation}
Applying \Cref{Nash-Moser} to (\ref{differential inequality for |h|}) yields (\ref{C^0 from L^2}).

Recall that we assumed \(h \neq 0\) everywhere. We will now show that this assumption can be dropped. 
Namely, (\ref{C^0 from L^2}) is stable under \(C^2\)-convergence, that is,
if (\ref{C^0 from L^2}) holds for a sequence of \(h_i\) and if \(h_i \to h\) in the \(C^2\)-topology, then (\ref{C^0 from L^2}) also holds for \(h\). Therefore, it suffices to construct a sequence \(h_i\) converging to \(h\) in the \(C^2\)-topology so that \(h_i \neq 0\) everywhere.

Let \(h \in C^{2,\alpha}\big( {\rm Sym}^2(T^*M)\big)\) be arbitrary. Then $h$ can be approximated in the $C^2$-topology by symmetric 
$(0,2)$-tensors $h_i$ $(i\geq 1)$ which are transverse 
to the zero-section of ${\rm Sym}^2(T^*M)$. For reasons of dimension, such a section is disjoint from the zero-section, in other words, the tensors $h_i$ vanish nowhere. Therefore, the estimate (\ref{C^0 from L^2}) holds for all \(h \in C^{2,\alpha}\big( {\rm Sym}^2(T^*M)\big)\) and \(x \in M\).

Fix \(h \in C^{2,\alpha}\big( {\rm Sym}^2(T^*M)\big)\). Choose \(x \in M\) so that \(|h|(x) \geq \frac{1}{2}||h||_{C^0(M)}\). Then (\ref{C^0 from L^2}) implies
\begin{equation}\label{C^0 from L^2 - global}
	\frac{1}{2}||h||_{C^0(M)} \leq C \big(||h||_{L^2(B(x,\rho))}+||f||_{C^0(M)} \big)
\end{equation}
for some \(C=C(n,\alpha,\Lambda)\). We can without loss of generality assume that the \(\rho\) from the definition of Hölder norms is at most \(1\). So it suffices to bound \(||h||_{L^2(B(x,r_0))}\) as \(r_0 \geq 1 \geq \rho\). To this end we distinguish two cases.

We first consider the case that \(x \nin E \). By (\ref{int1 - weight})
\begin{align*}
	\int_{B(x,r_0)}|h|^2 \, d\vol &\leq e^{2\sqrt{n-2}r_0}\int_M e^{-(2\sqrt{n-2}-\delta)r_x}|h|^2 \, d\vol \\
	&\leq e^{2\sqrt{n-2}r_0}C(n,\delta)\int_M e^{-(2\sqrt{n-2}-\delta)r_x}|f|^2 \, d\vol \\
	&\leq C(n,\delta,r_0) ||f||_0^2,
\end{align*}
where in the last line we used (\ref{Hybrid 0-norm}) and \(x \nin E \). This finishes the case \(x \nin E\). 

Now consider the case \(x \in E\). Choose \(\eta:\bbR \to \bbR \) smooth such 
that \(\eta=1 \) on \((-\infty,1] \), \(\eta=0 \) on \([2,\infty) \) and \(-\frac{3}{2} \leq \eta^\prime \leq 0\). 
Let \(\rho(t)=\eta(\frac{t}{r_0})e^{-(\sqrt{n-2}-\delta/2)t}\). Compute 
\[
	\rho^\prime(t)=e^{-(\sqrt{n-2}-\delta/2)t}\left(\frac{1}{r_0}\eta^\prime(t/r_0)-\eta(t/r_0)(\sqrt{n-2}-\delta/2) \right).
\]
Abbreviate \(\sigma:=n-2-(\sqrt{n-2}-\delta/2)^2-(1+\sqrt{n})\varepsilon_0\). As \(||Rm-Rm^{hyp}||\leq \varepsilon_0\), \Cref{weighted integral estimate - lipschitz} implies 
\begin{align}\label{int4 - weight}
	\int_M e^{-(2\sqrt{n-2}-\delta)r_x}\eta^2 & \big(\sigma|h|^2-2\langle h,f\rangle \big)\, d\vol \notag \\
	&\leq \int_M \left(\eta (-\eta^\prime)\frac{2\sqrt{n-2}-\delta}{r_0}+(\frac{\eta^\prime}{r_0})^2\right)|h|^2e^{-(2\sqrt{n-2}-\delta)r_x} \, d\vol \notag \\
	&\leq \int_{B(x,2r_0)\setminus B(x,r_0)}\left(\frac{3\sqrt{n-2}}{r_0}+\frac{9}{4r_0^2}\right)|h|^2e^{-(2\sqrt{n-2}-\delta)r_x} \, d\vol \notag \\
	&\leq C(n,r_0)||h||_{C^0(M)}^2 \int_{B(x,2r_0)\setminus B(x,r_0)}e^{-(2\sqrt{n-2}-\delta)r_x} \, d\vol \notag \\
	&\leq \bar{\epsilon} C(n,r_0)||h||_{C^0(M)}^2,
\end{align}
where we used that \(\eta^\prime(\frac{r_x}{r_0})=0\) outside \(B(x,2r_0)\setminus B(x,r_0)\) in the second inequality, and \(x \in E \) together with the definition (\ref{def of set E}) of \(E\) in the last inequality. Moreover,
\begin{equation*}
	\int_M e^{-(2\sqrt{n-2}-\delta)r_x}\eta^2|f|^2\, d\vol \leq e^{4\sqrt{n-2}r_0}\int_{B(x,2r_0)}e^{-(2\sqrt{n-2}-\delta)r_x}|f|^2 \, d\vol \leq c(n,r_0)||f||_0^2
\end{equation*}
since \(\eta(\frac{r_x}{r_0}) =0\) outside \(B(x,2r_0)\). Combining this with (\ref{int4 - weight}) yields
\[
	\sigma\int_M e^{-(2\sqrt{n-2}-\delta)r_x}\eta^2 \big|h-\frac{1}{\sigma}f\big|^2 \, d\vol \leq \bar{\epsilon} C(n,r_0)||h||_{C^0(M)}^2 +\frac{c(n,r_0)}{\sigma}||f||_0^2.
\]
Note \(\sigma=c(n,\delta)-(1+\sqrt{n})\varepsilon_0\). Assume \(\varepsilon_0 \leq \frac{c(n,\delta)}{2(1+\sqrt{n})}\). Then  \(\sigma\geq \frac{c(n,\delta)}{2}\). Hence
\begin{align*}
	\int_{B(x,r_0)}\big| h-\frac{1}{\sigma}f\big|^2 \, d\vol &\leq e^{2\sqrt{n-2}r_0}\int_M e^{-(2\sqrt{n-2}-\delta)r_x}\eta^2 \big| h-\frac{1}{\sigma}f\big|^2 \, d\vol\\
	&\leq C(n,\delta,r_0)\big( \bar{\epsilon}||h||_{C^0(M)}^2+||f||_0^2 \big).
\end{align*}
Using the triangle inequality we get
\[
	||h||_{L^2(B(x,r_0))} 
	\leq C(n,\delta, r_0) \big(\bar{\epsilon}^{\frac{1}{2}}||h||_{C^0(M)}+||f||_0\big). 
\]
Combining this with (\ref{C^0 from L^2 - global}) yields
\[
	\frac{1}{2}||h||_{C^0(M)}\leq C\big(\bar{\epsilon}^\frac{1}{2}||h||_{C^0(M)}+||f||_0 \big)
\]
for some \(C=C(n,\alpha,\delta,r_0,\Lambda)\). Thus for \(\bar{\epsilon} \leq \frac{1}{16C^2}\)
\[
	\frac{1}{2}||h||_{C^0(M)}\leq \frac{1}{4}||h||_{C^0(M)}+C||f||_0.
\]
This implies the desired \(C^0\)-estimate.
\end{proof}


We now make some further remarks concerning this proof. First, we point out the following estimate that can be extracted from the proof. In fact, it follows immediately from the estimates (\ref{int1 - weight}) and (\ref{C^0 from L^2}). This estimate will be the key ingredient to obtain the exponential decay estimate in \Cref{Pinching with inj radius bound - introduction}.

\begin{rem}\label{moreover estimate 1}\normalfont Let \(f \in C^{0,\alpha}\big({\rm Sym}^2(T^*M)\big)\) and let \(h \in C^{2,\alpha}\big({\rm Sym}^2(T^*M)\big)\) be a solution of \(\mathcal{L}h=f\). Then there is a constant \(C=C(n,\alpha,\delta,\Lambda)\) so that 
\[
	|h|(x) \leq C \left(||f||_{C^0(B(x,\rho))} +\left(\int_M e^{-(2\sqrt{n-2}-\delta)r_x}|f|^2 \, d\vol \right)^{\frac{1}{2}}\right)
\]
for all \(x \in M\). Here \(\rho>0\) is the constant appearing in the definition of the Hölder norms.
\end{rem}


For the proof of \Cref{Pinching without inj radius bound - introduction} we have to deal with manifolds that may no longer be compact (but have finite volume), and do \textit{not} have a positive lower bound on the injectivity radius. The next two remarks explain to what extend the arguments from the proof of \Cref{A-priori estimate for L} are still valid in that situation.

\begin{rem}\label{integration by parts works in finite volume case}\normalfont 
Let \(M\) be a finite volume manifold that satisfies all the assumptions from \Cref{A-priori estimate for L} except the compactness assumption and the lower bound on the injectivity radius. If \(h \in C^2\big({\rm Sym}^2(T^*M)\big) \cap H^2(M)\) and if $\mathcal{L}(h)=f$, then the inequality (\ref{integral1}) is still valid. 
Here \(h \in C^2\big({\rm Sym}^2(T^*M)\big) \cap H^2(M)\) just means that \(h\) is \(C^2\) and that \(\int_M \big( |h|^2+|\nabla h|^2+|\nabla^2 h|^2 \big) \, d\vol < \infty\).
\end{rem}

\begin{proof}
The proof of inequality (\ref{integral1}) carries over without change provided we can verify the equality 
\begin{equation}\label{integral2}
	\int_M e^{-\varphi}\Delta (|h|^2) \, d\vol=\int_M \langle \nabla (e^{-\varphi}),\nabla (|h|^2)\rangle \, d\vol,
\end{equation}
which involved an integration by parts. Here \(\varphi=(2\sqrt{n-2}-\delta)r_x\). As in the proof of \Cref{A-priori estimate for L} we may act as if \(\varphi\) were smooth. 

Consider the vector field \(X:=e^{-\varphi}\nabla(|h|^2)\). As \(h \in H^2(M)\) and because \(\varphi\) is Lipschitz and bounded from below, it is easy to see that \(X \in L^1(M)\) and \({\rm div}(X)=\langle \nabla (e^{-\varphi}),\nabla (|h|^2)\rangle -e^{-\varphi}\Delta(|h|^2) \in L^1(M)\). Therefore, the main result of \cite{Gaffney1954} shows \(\int_M {\rm div}(X) \, d\vol=0\).
\end{proof}

\begin{rem}\label{pointwise C^0 estimate given inj rad bound}\normalfont 
Let \(M\) be a finite volume manifold that satisfies all the assumptions from \Cref{A-priori estimate for L} except the compactness assumption and the lower bound on the injectivity radius. Then there exist \(\bar{\epsilon}_0=\bar{\epsilon}_0(n,\alpha, \Lambda,\delta, r_0)>0\) and \(C=C(n,\alpha, \Lambda,\delta, r_0)\) with the following property. Let \(h \in C^2\big({\rm Sym}^2(T^*M)\big) \cap H^2(M)\), and assume there is \(x_0 \in M_{\rm thick}\) with \(|h|(x_0) \geq \frac{1}{2}||h||_{C^0(M)}\). Then it holds
\[
	||h||_{C^0} \leq C ||\mathcal{L}h||_0,
\] 
where \(||\cdot||_0\) is the norm defined in (\ref{Hybrid 0-norm}) with respect to  any \(\bar{\epsilon}\leq \bar{\epsilon}_0\).
\end{rem}

\begin{proof}A priori (\ref{C^0 from L^2}) only holds in the universal cover. But for \(x \in M_{\rm thick}\), the norm \(||\cdot||_{L^2(B(x,\rho))}\) is the same in the universal cover and in the base manifold (here we assume without loss of generality that the universal radius \(\rho\) used to define Hölder norms is smaller than a chosen Margulis constant). So (\ref{C^0 from L^2}) holds for \(x \in M_{\rm thick}\), and the rest of the argument of \Cref{A-priori estimate for L} applies without change. 
\end{proof}


We finish this section by showing that the operator \(\mathcal{L}\) is invertible. This follows from \Cref{A-priori estimate for L} by 
standard techniques.


\begin{prop}\label{Invertibility of L}For all \(n \geq 3\), \(\alpha \in (0,1)\), \(\Lambda \geq 0\), \(\delta \in (0,\sqrt{n-2})\) and \(r_0 \geq 1\) there exist constants \(\varepsilon_0\), \(\bar{\epsilon}_0\) and \(C>0\) with the following property. Let \(M\) be a closed Riemannian \(n\)-manifold with
\[
	 |\,\mathrm{sec}+1 \, |\leq \varepsilon_0, \quad \mathrm{inj}(M) \geq 1 \quad \text{and} \quad ||\nabla \Ric||_{C^0(M)}\leq \Lambda.
\]
Then the operator
\[
	\mathcal{L}: \Big( C^{2,\alpha}\big( {\rm Sym}^2(T^*M)\big), ||\cdot||_2 \Big) \longrightarrow \Big( C^{0,\alpha}\big( {\rm Sym}^2(T^*M)\big), ||\cdot||_0 \Big)
\]
is invertible, and 
\[
	||\mathcal{L}||_{\rm op}, ||\mathcal{L}^{-1}||_{\rm op} \leq C,
\]
where \(||\cdot||_2\) and \(||\cdot||_0\) are the norms defined in (\ref{Hybrid 2-norm}) and (\ref{Hybrid 0-norm}) with respect to any \(\bar{\epsilon} \leq \bar{\epsilon}_0\).
\end{prop}

\begin{proof} By \Cref{A-priori estimate for L}, it remains to show that \(\mathcal{L}\) is surjective. We split up the equation into its trace and its trace-free part. Namely, note that any \((0,2)\)-tensor \(f\) can be written as \(f=f^\circ+\varphi g\), where \(f^\circ\) has vanishing trace, \(\varphi\) is a function, and \(g\) is the given Riemannian metric of \(M\). 

Note that \(\mathcal{L}(ug)=\big(\frac{1}{2}\Delta u+(n-1)u\big)g\). The bilinear form \(a_0:H^1(M) \times H^1(M) \to \bbR\) associated to the equation \(\frac{1}{2}\Delta u+(n-1)u=\varphi\) is given by 
\[
	a_0(u,v)=\int_M \left(\frac{1}{2}\langle \nabla u,\nabla v\rangle +(n-1)uv \right) \, d\vol.
\] 
This is clearly bounded and coercive. 
Thus by Lax-Milgram and the Weyl Lemma, for any \(\varphi \in C^\infty(M)\) there is \(u \in C^\infty(M)\) so that \(\mathcal{L}(u g)=\varphi g\).

Let \(E\to M\) be the vector bundle of symmetric \((0,2)\)-tensors with vanishing trace. The bilinear form \(a:H^1(E) \times H^1(E) \to \bbR\) associated to \(\mathcal{L}\) is given by
\[
	a(h,h^\prime)=\int_M \left(\frac{1}{2}\langle \nabla h,\nabla h^\prime \rangle +\frac{1}{2}\langle \Ric(h),h^\prime\rangle + (n-1)\langle h,h^\prime \rangle \right)\, d\vol.
\]
By \Cref{Poincare inequality} the Poincaré inequality holds for tensors with vanishing trace. Together with the estimate from Lemma \ref{Estimate for (Ric(h),h)} we get
\[
	a(h,h) \geq \frac{1}{2}||\nabla h||_{L^2(M)}^2-(1+(1+\sqrt{n})\varepsilon_0)||h||_{L^2(M)}^2 \geq \left(\frac{1}{2}-\frac{1+(1+\sqrt{n})\varepsilon_0}{n-c(n)\varepsilon_0} \right)||\nabla h||_{L^2(M)}^2
\]
for all \(h \in H^1(E)\), so that for \(\varepsilon_0>0\) small enough, the form \(a\) is coercive on \(E\).
So again by Lax-Milgram and the Weyl Lemma, for any \(f^\circ \in C^\infty(E)\) there is \(h \in C^\infty(E)\) so that \(\mathcal{L}h=f^\circ\).

Therefore, splitting any \(f\) up into its trace part \(\varphi g\) and its trace-free part \(f^\circ\), we obtain that for any \(f \in C^\infty\big({\rm Sym}^2(T^*M)\big)\) there is \(h \in C^\infty\big({\rm Sym}^2(T^*M)\big)\) so that \(\mathcal{L}h=f\). 

Recall the well-known fact that for any \(u \in C^{0,\alpha}(\bbR^n)\) there is a sequence \((u_\varepsilon)_{\varepsilon>0} \subseteq C^\infty(\bbR^n)\) so that \(\lim_{\varepsilon \to 0}||u_\varepsilon - u||_{C^{0,\beta}(\bbR^n)}=0\) for any \(\beta \in (0,\alpha)\). Moreover, if \(u\) has compact support in some open set \(\Omega \subseteq \bbR^n\), then \(u_\varepsilon\) can be assumed to have compact support in \(\Omega\) too. Now let \(f \in C^{0,\alpha}\big({\rm Sym}^2(T^*M)\big)\) be arbitrary. Applying this approximation result locally, we obtain a sequence \((f_i)_{i \in \bbN}\) in \(C^\infty\big({\rm Sym}^2(T^*M)\big)\) 
converging to \(f\) with respect to the \(C^{0,\frac{\alpha}{2}}\)-norm. Let \(h_i\) be the solutions 
of \(\mathcal{L}h_i=f_i\). Note that the norms \(||\cdot||_{C^{0,\frac{\alpha}{2}}}\) and \(||\cdot||_0\) are equivalent on \(C^{0,\frac{\alpha}{2}}\big({\rm Sym}^2(T^\ast M)\big)\) (but with a non-universal constant). It follows from \Cref{A-priori estimate for L} (applied with \(\frac{\alpha}{2}\)) 
\[
	||h_i-h_j||_{C^2}\leq C||f_i-f_j||_{C^{0,\frac{\alpha}{2}}} \to 0 \quad \text{ as }\, i,j \to \infty
\]
for a (non-universal) constant \(C\). So \((h_i)_{i \in \bbN} \subseteq C^{2}\big({\rm Sym}^2(T^*M)\big)\) is a Cauchy sequence. Denote the limit tensor field by \(h\). Clearly \(h\) solves \(\mathcal{L}h=f\). Finally, \(h \in C^{2,\alpha}\big({\rm Sym}^2(T^*M)\big)\) by 
elliptic regularity theory. Therefore, \(\mathcal{L}\) is bijective. The bound on \(||\mathcal{L}^{-1}||_{\rm op}\) follows from \Cref{A-priori estimate for L}, and the one for \(||\mathcal{L}||_{\rm op}\) is obvious.
\end{proof}

Recall that by \Cref{convention orientable} we assume all manifolds to be orientable. Nonetheless, we have the following.

\begin{rem}\label{Invertibility of L - non-orientable}\normalfont
\Cref{Invertibility of L} also holds when \(M\) is not orientable.
\end{rem}

\begin{proof}\Cref{Invertibility of L} holds for the orientation cover \(\hat{M}\) of \(M\). Moreover, since the non-trivial decktransformation \(\tau:\hat{M} \to \hat{M}\) is an isometry, \Cref{Invertibility of L} shows that the elliptic operator \(\mathcal{L}\) on \(\hat{M}\) restricts to an isomorphism between subbundles of \(\tau\)-invariant Hölder sections of symmetric \((0,2)\)-tensors. But \(\tau\)-invariant Hölder sections on \(\hat{M}\) are nothing else than Hölder sections on \(M\). 
\end{proof}

\section{Proof of the pinching theorem with lower injectvity radius bound}\label{Sec: Proof of pinching with inj radius bound}

We start by stating a more precise formulation of \Cref{Pinching with inj radius bound - introduction}.

\begin{thm}\label{Pinching with inj radius bound - full version} For any \(n \geq 3, \, \alpha \in (0,1), \, \Lambda \geq 0, \, \delta \in (0,2\sqrt{n-2})\) and \(r_0 \geq 1\) there exist constants \(\varepsilon_0\) and \(C>0\) with the following property.
Let \(M\) be a closed \(n\)-manifold that admits a Riemannian metric \(\bar{g}\) satisfying the following conditions for some \(\varepsilon \leq \varepsilon_0\):
\begin{enumerate}[i)]
\item \(-1-\varepsilon \leq \mathrm{sec}_{(M,\bar{g})} \leq -1+\varepsilon\);
\item \(\mathrm{inj}(M,\bar{g})\geq 1\);
\item \(|| {\nabla} \Ric(\bar{g})||_{C^0(M,\bar{g})} \leq \Lambda\);
\item It holds
\[
	\int_Me^{-(2\sqrt{n-2}-\delta)r_x(y)}|\Ric(\bar{g})+(n-1)\bar{g}|_{\bar{g}}^2(y) \, d\vol_{\bar{g}}(y) \leq \varepsilon^2
\]
for all \(x \in M\) with 
\[\int_{B(x,2r_0)\setminus B(x,r_0)}e^{-(2\sqrt{n-2}-\delta)r_x(y)} \, d\vol_{\bar{g}}(y)>\varepsilon_0,\] where \(r_x(y)=d_{\bar{g}}(x,y)\).
\end{enumerate}
Then there exists an Einstein metric \(g_0\) on \(M\) so that \(\Ric(g_0)=-(n-1)g_0\) and 
\[
	||g_0-\bar{g}||_{2} \leq C \varepsilon^{1-\alpha},
\]
where \(||\cdot||_{2}\) is the norm defined in (\ref{Hybrid 2-norm}) with respect to  
the metric  \(\bar{g}\) and the constants \(\epsilon_0,\delta,r_0\).

Moreover, if for some \(\beta \leq 2\sqrt{n-2}-\delta\) and \(U \subseteq M\) it holds
\[
	\int_M e^{-(2\sqrt{n-2}-\delta)r_x(y)}|{\rm Ric}(\bar{g})+(n-1)\bar{g}|^2(y) \, d\vol(y) \leq \varepsilon^{2(1-\alpha)} e^{-2\beta \dist_{\bar{g}}(x,U)} \quad \text{ for all } \, x \in M,
\]
then 
\[
	|g_0-\bar{g}|_{C^{2,\alpha}}(x) \leq C \varepsilon^{1-\alpha} e^{-\beta\dist_{\bar{g}}(x,U)}	\quad \text{ for all }\, x \in M.
\]
In particular, if \(\Ric(\bar{g})=-(n-1)\bar{g}\) outside a region \(U\), and if 
\[
	\int_U |\Ric(\bar{g})+(n-1)\bar{g}|^2 \, d\vol_{\bar{g}} \leq \varepsilon^2,
\]
then
\[
	|g_0-\bar{g}|_{C^{2,\alpha}}(x) \leq C \varepsilon^{1-\alpha} e^{-(\sqrt{n-2}-\frac{1}{2}\delta)\dist_{\bar{g}}(x,U)}	\quad \text{ for all }\, x \in M.
\]
\end{thm}

As mentioned previously, we will prove this using the implicit function theorem. 
The linearisation \((D\Phi)_{\bar{g}}\) of the Einstein operator $\Phi$ 
at the initial metric was studied in \Cref{Section - Invertibiliy of L}, and we showed that it is invertible, 
with controlled norm of its inverse. 
To control the size of a
neighborhood of \(\Phi(\bar{g})\) in which \(\Phi\) is invertible requires an estimate of the Lipschitz constant of the mapping \(g \mapsto (D\Phi)_g\). This will follow from the next lemma.

\begin{lem}\label{Pointwise continuity of differential} For all \(n \geq 2\), \(\alpha \in (0,1)\), \(\Lambda \geq 0\), \(i_0>0\) there exist \(\varepsilon=\varepsilon(n,\alpha,\Lambda,i_0)> 0\) and \(C=C(n,\alpha,\Lambda,i_0)\) with the following property. Let \((M,\bar{g})\) be a Riemannian \(n\)-manifold with
\[
	||\Ric(\bar{g})||_{C^1(M,\bar{g})}\leq \Lambda \quad \text{and} \quad \mathrm{inj}(M,\bar{g}) \geq i_0
\] 
and let \(g \in C^{2,\alpha}\big( {\rm Sym}^2(T^*M)\big)\) be another Riemannian metric so that \(||g-\bar{g}||_{C^{2,\alpha}(M,\bar{g})} \leq \varepsilon\). 
Then the linearization of the 
Einstein operator \(\Phi=\Phi_{\bar{g}}\) defined in (\ref{Def of Phi}) satisfies the pointwise estimates
\[
	|(D\Phi)_g(h)-(D\Phi)_{\bar{g}}(h)|_{C^{0,\alpha}}(x) \leq C \max_{y \in B(x,\rho)}|g-\bar{g}|_{C^{2,\alpha}}(y)|h|_{C^{2,\alpha}}(y)
\]
and
\[
	|(D\Phi)_g(h)-(D\Phi)_{\bar{g}}(h)|_{C^{0}}(x) \leq C |g-\bar{g}|_{C^{2}}(x)|h|_{C^{2}}(x)
\]
for all \(h \in C^{2,\alpha}\big( {\rm Sym}^2(T^*M)\big)\), where all norms are taken w.r.t. the background metric \(\bar{g}\).
\end{lem}

The term \(\max_{y \in B(x,\rho)}\) comes from the fact that \(C^{0,\alpha}\)-norms are defined in bigger local charts than \(C^{2,\alpha}\)-norms 
(also see inequality (\ref{Continuity of elliptic operator - pointwise})). Also note that since all norms are taken w.r.t. \(\bar{g}\), we do \textit{not} need an upper bound on \(||\nabla {\rm Ric}(g)||_{C^0}\).

\begin{proof}The linearisation of the operator \(g \to {\rm Ric}(g)\) is given by (see \cite[Proposition 2.3.7]{topping_2006})
\[
	(D{\rm Ric})_g(h)=\frac{1}{2}\Delta_Lh-\frac{1}{2}\mathcal{L}_{(\beta_g(h))^{\sharp_g}}(g),
\]
where \(\beta_g=\delta_g(\cdot)+\frac{1}{2}d{\rm tr}_g(\cdot)\) is the Bianchi operator of \(g\), \(\sharp_g:T^\ast M \to TM\) is the musical isomorphism associated to \(g\), and \(\mathcal{L}_X(\cdot)\) is the Lie derivative in direction \(X\). Therefore, it follows from the definition (\ref{Def of Phi}) of \(\Phi\) and the product rule that
\[
	(D\Phi)_g(h)=\Delta_Lh+(n-1)h+\frac{1}{2}\mathcal{L}_{(\beta_{\bar{g}}(g))^{\sharp_g}}(h)+\frac{1}{2}\mathcal{L}_{(\beta_{\bar{g}}(h)-\beta_g(h))^{\sharp_g}}(g)+\frac{1}{2}\mathcal{L}_{(D_h\sharp)_g(\beta_{\bar{g}}(g))}(g),
\]
where \((D_h\sharp)_g:T^\ast M \to TM\) is the linearisation of \(g \to \sharp_g\) in direction \(h\). Denote by \(\flat_g:TM \to T^\ast M, v \mapsto g(v,\cdot)\) the inverse of \(\sharp_g\). Differentiating the identity \({\rm id}_{T M}=\sharp_g \circ \flat_g\)  in direction \(h\), and applying \(\flat_g\) yields \((D_h\sharp)_g=-\sharp_g \circ \flat_h \circ \sharp_g\), where \(\flat_h(v)=h(v,\cdot)\). In local coordinates this reads \(\big((D_h\sharp)_g(\omega)\big)^m=-g^{mk}h_{kj}g^{ji}\omega_i\). Therefore, \Cref{Pointwise continuity of differential} can be checked by a straightforward (albeit tedious) calculation in local coordinates. We will not carry this out in more detail.
\end{proof}

\Cref{Pointwise continuity of differential} and \Cref{alternative hybrid norm} immediately imply the following corollary.

\begin{cor}\label{global continuity of differential} For any \(n \geq 3, \, \alpha \in (0,1), \, \Lambda \geq 0\) there exist \(\varepsilon=\varepsilon(n,\alpha,\Lambda)>0\) 
and \(C=C(n,\alpha,\Lambda)\) with the following property. Let \((M,\bar{g})\) be a closed Riemannian \(n\)-manifold with
\[
	|\mathrm{sec}| \leq 2, \quad \mathrm{inj}(M,\bar{g})\geq 1 \quad \text{and} \quad || {\nabla}\Ric(\bar{g})||_{C^0(M,\bar{g})} \leq \Lambda
\]
and let \(g \in C^{2,\alpha}\big( {\rm Sym}^2(T^*M)\big)\) be another Riemannian metric so that \(||g-\bar{g}||_{C^{2,\alpha}(M,\bar{g})} \leq \varepsilon\). Then the operator \(\Phi=\Phi_{\bar{g}}\) defined in (\ref{Def of Phi}) satisfies 
\[
	||(D\Phi)_g(h)-(D\Phi)_{\bar{g}}(h)||_{0} \leq C ||g-\bar{g}||_{2}||h||_{2}
\]
for all \(h \in C^{2,\alpha}\big( {\rm Sym}^2(T^*M)\big)\), where \(||\cdot||_{2}\) and \(||\cdot||_{0}\) are the norms defined in (\ref{Hybrid 2-norm}) and (\ref{Hybrid 0-norm}) with respect to the metric \(\bar{g}\) and any \(\bar{\epsilon}, \delta,r_0\).
\end{cor}

We now come to the proof of \Cref{Pinching with inj radius bound - full version}.

\begin{proof}[Proof of \Cref{Pinching with inj radius bound - full version}]
In this proof, \(||\cdot||_2\) resp. \(||\cdot||_0\) shall denote the norms defined in (\ref{Hybrid 2-norm}) and (\ref{Hybrid 0-norm}) with respect to the metric \(\bar{g}\) and the constants \(\varepsilon_0,\delta,r_0\), and \( C^{2,\alpha}\big({\rm Sym}^2(T^*M)\big)\) is understood to be equipped with \(||\cdot||_2\). 
Metric balls $B(h,R)$ of radius $R$ about a section $h$ 
are taken with respect to that norm.

Define the operator
\[
	\Psi: B(0,1) \subseteq C^{2,\alpha}\big({\rm Sym}^2(T^*M)\big) \to  C^{2,\alpha}\big({\rm Sym}^2(T^*M)\big)
\]
by
\[
	 \Psi(h):=h-\mathcal{L}^{-1}\big( \Phi(\bar{g}+h)\big).
\]
Here \(\Phi=\Phi_{\bar{g}}\) is the Einstein operator defined in (\ref{Def of Phi}),
and \(\mathcal{L}=(D\Phi)_{\bar{g}}\). 
 
By \Cref{A-priori estimate for L} and \Cref{global continuity of differential},
there is a constant \(C=C(n,\alpha,\delta,r_0,\Lambda)\) 
such that
 \(||\mathcal{L}^{-1}||_{\rm op} \leq C\) and \(\mathrm{Lip}\big((D\Phi)_{\bullet}\big)\leq C\).
 Thus 
 it follows from \((D\Psi)_h=\mathcal{L}^{-1}\circ\big( \mathcal{L}-(D\Phi)_{\bar g+h}\big)\) that for \(R=R(n,\alpha,\delta,\Lambda)>0\) small enough, 
 the restriction of \(\Psi\) to the closed ball \(\bar{B}(0,R)\) is \(\frac{1}{2}\)-Lipschitz. Moreover, since \(||\cdot||_{C^{0,\alpha}} \leq C ||\cdot||_{C^0}^{1-\alpha}||\cdot||_{C^1}^{\alpha}\) and \(\Phi(\bar{g})=\Ric(\bar{g})+(n-1)\bar{g}\), the assumptions \(i)\) and \(iii)\) imply that \(||\Phi(\bar{g})||_{C^{0,\alpha}} \leq C\varepsilon^{1-\alpha}\). Together with condition \(iv)\), this shows
\[
	||\Phi(\bar{g})||_{0} \leq C \varepsilon^{1-\alpha}
\]
due to the definition of the norm \(||\cdot||_{0}\). 
As a consequence, for \(\varepsilon_0=\varepsilon_0(n,\alpha,\delta,\Lambda)>0\) small enough, we have
\(||\Psi(0)||_2 \leq \frac{R}{2}\) and hence \(\Psi\) restricts to a map \(\bar{B}(0,R) \to \bar{B}(0,R)\). By the Banach fixed point theorem there exists a fixed point \(h_0\) of \(\Psi\). By definition of \(\Psi\) this means \(\Phi(\bar{g}+h_0)=0\), and hence \(g_0=\bar{g}+h_0\) is an Einstein metric due to \Cref{Zeros of Phi are Einstein}. 
Moreover, as \(\Psi\) is \(\frac{1}{2}\)-Lipschitz, it holds
\[
	||h_0||_2=||\Psi(h_0)||_2 \leq ||\Psi(h_0)-\Psi(0)||_2+||\Psi(0)||_2 \leq \frac{1}{2}||h_0||_2+C\varepsilon^{1-\alpha}.
\]
This implies \(||h_0||_2 \leq C\varepsilon^{1-\alpha}\). 

It remains to show the improved estimate on \(|g_0-\bar{g}|_{C^{2,\alpha}}(x)\). Let \(\beta \leq 2\sqrt{n-2}-\delta\) and \(U \subseteq M\) be as in the statement of \Cref{Pinching with inj radius bound - full version}.

Recall that the pointwise Hölder norm \(|\cdot|_{C^{0,\alpha}}(x)\) is computed in a local chart defined on the ball \(B(x,\rho)\), where \(\rho=\rho(n,\alpha,\Lambda)>0\) is a universal constant (see the proof of \Cref{Schauder for tensors} for more details). In particular, \(||\cdot||_{C^0(B(x,\rho))} \leq C|\cdot|_{C^{0,\alpha}}(x)\). So \Cref{moreover estimate 1} and the pointwise Schauder estimate (\ref{Schauder pointwise - C^2}) show that there is a universal constant \(C_0\) so that for all \(h \in  C^{2,\alpha}\big({\rm Sym}^2(T^*M)\big) \) it holds
\begin{equation}\label{improved estimate - eq 1}
	|h|_{C^{2,\alpha}}(x) \leq C_0 \left( |\mathcal{L}h|_{C^{0,\alpha}}(x)+\left(\int_M e^{-(2\sqrt{n-2}-\delta)r_x(y)}|\mathcal{L}h|^2(y) \, d\vol_{\bar{g}}\right)^{\frac{1}{2}} \right).
\end{equation}
Choose \(C_0\) large enough so that the a priori estimate from \Cref{A-priori estimate for L}, and the weighted integral estimates from Step 1 of the proof of \Cref{A-priori estimate for L} hold, that is, 
\begin{equation}\label{improved estimate - eq 2}
	||h||_{2} \leq C_0||\mathcal{L}h||_0 \quad \text{and} \quad ||h||_{H^2(M;\omega_x)} \leq C_0||\mathcal{L}h||_{L^2(M;\omega_x)} \, \text{ for all }\, x \in M,
\end{equation}
where \(||\cdot||_{H^2(M,\omega_x)}:=\left(\int_M e^{-(2\sqrt{n-2}-\frac{1}{2}\delta)r_x(y)}|\cdot|_{C^2}(y) \, d\vol_{\bar{g}}(y) \right)^{\frac{1}{2}}\) is the weighted \(H^2\)-norm, and analogously \(||\cdot||_{L^2(M;\omega_x)}\) shall denote the weighted \(L^2\)-norm. Moreover, we assume that \(C_0\) is large enough so that \(||\Ric(\bar{g})+(n-1)\bar{g}||_{C^{0,\alpha}(M)} \leq C_0\varepsilon^{1-\alpha}\).

Define \(C_1:=2C_0^2e^{\rho\sqrt{n-2}}+2C_0\), and consider the set
\begin{equation}\label{U - def}
	\mathcal{U}:=\Big\{h \in \mathrm{Dom}(\Psi) \, | \, h \text{ satisfies the inequalities (\ref{U - eq 1}), (\ref{U - eq 2}) for all }x \in M\Big\},
\end{equation}
where the inequalities (\ref{U - eq 1}) and (\ref{U - eq 2}) appearing in the definition of \(\mathcal{U}\) are
\begin{equation}\label{U - eq 1}
	|h|_{C^{2,\alpha}}(x)\leq C_1 \varepsilon^{1-\alpha}e^{-\beta \dist_{\bar{g}}(x,U)} 
\end{equation}
and
\begin{equation}\label{U - eq 2}
	||h||_{H^2(M;\omega_x)}\leq C_1 \varepsilon^{1-\alpha}e^{-\beta \dist_{\bar{g}}(x,U)}.
\end{equation}
We will show that \(\Psi(\mathcal{U}) \subseteq \mathcal{U}\). This implies the desired estimate, because the fixed point \(h_0\) is then necessarily contained in \(\mathcal{U}\).

To prove \(\Psi(\mathcal{U}) \subseteq \mathcal{U}\) we first observe that for all \(h\in \mathrm{Dom}(\Psi)\) it holds
\begin{equation}\label{improved estimate - eq 3}
	\Psi(h)-\Psi(0)=\int_0^1 \mathcal{L}^{-1}\big(\mathcal{L}h-(D\Phi)_{\bar{g}+th}h \big)\, dt
\end{equation}
by the Fundamental Theorem of Calculus. Denote by \(C_2:=\mathrm{Lip}\big((D\Phi)_{\bullet} \big)\) the universal continuity constant given by \Cref{Pointwise continuity of differential}, so that it holds
\begin{equation}\label{improved estimate - eq 4}
	|\mathcal{L}h-(D\Phi)_{\bar{g}+th}h|_{C^{0,\alpha}}(x) \leq C_2\sup_{y \in B(x,\rho)}|h|_{C^{2,\alpha}}^2(y)
\end{equation}
and
\begin{equation}\label{improved estimate - eq 5}
	|\mathcal{L}h-(D\Phi)_{\bar{g}+th}h|_{C^0}(x) \leq C_2|h|_{C^2}^2(x).
\end{equation}

Now let \(h \in \mathcal{U}\) be arbitrary. We start by showing that \(\Psi(h)\) satisfies (\ref{U - eq 2}). Combining the Jensen-inequality, (\ref{improved estimate - eq 2}), (\ref{improved estimate - eq 3}), and (\ref{improved estimate - eq 5}) yields
\begin{align}\label{improved estimate - eq 7}
	||\Psi(h)-\Psi(0)||_{H^2(M;\omega_x)}^2 \stackrel{(\ref{improved estimate - eq 3})}{\leq} & \int_0^1 \left|\left|\mathcal{L}^{-1}\big(\mathcal{L}h-(D\Phi)_{\bar{g}+th}h\big)\right|\right|_{H^2(M;\omega_x)}^2 \, dt\notag \\
	\stackrel{(\ref{improved estimate - eq 2})}{\leq} & C_0^2\int_0^1 \left|\left|\mathcal{L}h-(D\Phi)_{\bar{g}+th}h\right|\right|_{L^2(M;\omega_x)}^2 \, dt \notag \\
	\stackrel{(\ref{improved estimate - eq 5})}{\leq} & C_0^2C_2^2\int_Me^{-(2\sqrt{n-2}-\delta)r_x(y)}|h|_{C^2}^4(y) \, d\vol_{\bar{g}}(y).
\end{align}
Note \(||h||_{C^2(M)}\leq C_1\varepsilon^{1-\alpha}\) by (\ref{U - eq 1}) and since \(h \in \mathcal{U}\). Together with (\ref{U - eq 2}) and (\ref{improved estimate - eq 7}) this implies
\begin{align}\label{improved estimate - eq 6}
	||\Psi(h)-\Psi(0)||_{H^2(M;\omega_x)}^2 \stackrel{(\ref{improved estimate - eq 7})}{\leq} & C_0^2C_2^2\int_M e^{-(2\sqrt{n-2}-\delta)r_x(y)}|h|_{C^2}^4(y) \, d\vol(y) \notag \\
	\stackrel{\quad \, \, \, \, \, \,}{\leq} & C_0^2C_2^2 C_1^2\varepsilon^{2(1-\alpha)}\int_M e^{-(2\sqrt{n-2}-\delta) r_x(y)}|h|_{C^2}^2(y) \, d\vol(y) \notag \\
	\stackrel{(\ref{U - eq 2})}{\leq} & C_0^2C_2^2 C_1^2\varepsilon^{2(1-\alpha)} C_1^2\varepsilon^{2(1-\alpha)}e^{-2\beta \dist_{\bar{g}}(x,U)}.
\end{align}
Note \(\Phi(\bar{g})={\rm Ric}(\bar{g})+(n-1)\bar{g}\). Hence \(||\Phi(\bar{g})||_{L^2(M;\omega_x)} \leq \varepsilon^{1-\alpha} e^{-\beta\dist_{\bar{g}}(x,U)}\) by assumption, so that
\(
	||\Psi(0)||_{H^2(M;\omega_x)}\leq C_0 \varepsilon^{1-\alpha} e^{-\beta \dist_{\bar{g}}(x,U)}
\)
by (\ref{improved estimate - eq 2}). Applying (\ref{improved estimate - eq 6}) and the triangle inequality yields
\[
	||\Psi(h)||_{H^2(M;\omega_x)} \leq \left(C_0+ C_0C_2 C_1^2\varepsilon^{(1-\alpha)}\right)\varepsilon^{(1-\alpha)}e^{-\beta \dist_{\bar{g}}(x,U)}.
\]
As \(C_0+ C_0 C_2 C_1^2\varepsilon^{(1-\alpha)} \leq C_1\) for \(\varepsilon>0\) small enough, we conclude that \(\Psi(h)\) satisfies (\ref{U - eq 2}) for all \(x \in M\).

It remains to show that \(\Psi(h)\) satisfies (\ref{U - eq 1}) for all \(x \in M\), because then \(\Psi(h) \in \mathcal{U}\) by the definition (\ref{U - def}) of \(\mathcal{U}\). Combining (\ref{improved estimate - eq 1}), (\ref{improved estimate - eq 3}), (\ref{improved estimate - eq 4}), and (\ref{improved estimate - eq 5}) shows
\begin{align}\label{improved estimate - eq 8}
	|\Psi(h)-\Psi(0)|_{C^{2,\alpha}}(x) \stackrel{\quad (\ref{improved estimate - eq 3})\quad}{\leq}& \int_0^1 \left|\mathcal{L}^{-1}\big(\mathcal{L}h-(D\Phi)_{\bar{g}+th}h\big)\right|_{C^{2,\alpha}}(x) \, dt \notag \\
	\stackrel{\quad(\ref{improved estimate - eq 1})\quad}{\leq}& \int_0^1 C_0|\mathcal{L}h-(D\Phi)_{\bar{g}+th}h|_{C^{0,\alpha}}(x) \, dt \notag \\
	+&C_0\int_0^1 ||\mathcal{L}h-(D\Phi)_{\bar{g}+th}h||_{L^2(M;\omega_x)} \, dt \notag \\
	\stackrel{(\ref{improved estimate - eq 4}),(\ref{improved estimate - eq 5})}{\leq}&  C_0C_2\sup_{y \in B(x,\rho)}|h|_{C^{2,\alpha}}(y) \notag \\
	+&C_0C_2\left(\int_M e^{-(2\sqrt{n-2}-\delta)r_x(y)}|h|_{C^2}^4(y)\, d\vol_{\bar{g}}(y)\right)^\frac{1}{2}.
\end{align}
Note that the last summand is estimated in (\ref{improved estimate - eq 6}). Using (\ref{U - eq 1}) to estimate \(|h|_{C^{2,\alpha}}(y)\), and remembering \(\beta \leq \sqrt{n-2}\) we get
\begin{align*}
	|\Psi(h)-\Psi(0)|_{C^{0,\alpha}}(x) \stackrel{\quad(\ref{improved estimate - eq 8})\quad}{\leq} & C_0 C_2\sup_{y \in B(x,\rho)}|h|_{C^{2,\alpha}}^2(y) \\
	&+ C_0C_2\left(\int_M e^{-(2\sqrt{n-2}-\delta)r_x(y)}|h|_{C^2}^4(y) \, d\vol_{\bar{g}}(y) \right)^{\frac{1}{2}} \\
	\stackrel{(\ref{U - eq 1}),(\ref{improved estimate - eq 6})}{\leq} & C_0C_2 C_1^2\varepsilon^{2(1-\alpha)}\sup_{y \in B(x,\rho)}e^{-2\beta \dist_{\bar{g}}(y,U)} \\
	&+ C_0C_2C_1^2\varepsilon^{2(1-\alpha)} e^{-\beta \dist_{\bar{g}}(x,U)} \\
	\stackrel{\quad \quad \quad \quad}{\leq} & \left(2C_0C_2 C_1^2e^{2 \rho \sqrt{n-2}}\varepsilon^{(1-\alpha)}\right)\varepsilon^{(1-\alpha)} e^{-\beta \dist_{\bar{g}}(x,U)}.
\end{align*}
Recall \(||\Phi(\bar{g})||_{L^2(M;\omega_x)} \leq \varepsilon^{1-\alpha} e^{-\beta \dist_{\bar{g}}(x,U)}\) and \(\Psi(0)=-\mathcal{L}^{-1}\Phi(\bar{g})\). Thus applying (\ref{improved estimate - eq 1}) shows \(|\Psi(0)|_{C^{2,\alpha}}(x) \leq C_0\big(|\Phi(\bar{g})|_{C^{0,\alpha}}(x)+\varepsilon^{1-\alpha} e^{-\beta\dist_{\bar{g}}(x,U)}\big)\). Since \(|\cdot|_{C^{0,\alpha}}(x)\) is computed in a local chart defined on \(B(x,\rho)\), it holds \(|\Phi(\bar{g})|_{C^{0,\alpha}}(x)=0\) if \(\dist_{\bar{g}}(x,U)> \rho\). For \(x \in M\) with \(\dist_{\bar{g}}(x,U) \leq \rho\) it holds \(|\Phi(\bar{g})|_{C^{0,\alpha}}(x) \leq ||\Phi(\bar{g})||_{C^{0,\alpha}(M)} \leq C_0 \varepsilon^{1-\alpha} \leq C_0e^{\beta \rho} \varepsilon^{1-\alpha}e^{-\beta \dist_{\bar{g}}(x,U)}\). All in all we conclude (again remembering \(\beta \leq \sqrt{n-2}\))
\[
	|\Psi(h)|_{C^{2,\alpha}}(x) \leq \left(2C_0C_2 C_1^2e^{2 \rho \sqrt{n-2}}\varepsilon^{(1-\alpha)}+C_0+C_0^2e^{\rho \sqrt{n-2}} \right)\varepsilon^{1-\alpha}e^{-\beta \dist_{\bar{g}}(x,U)}
\]
for all \(x \in M\). Recall that \(C_1=2C_0^2e^{\rho \sqrt{n-2}}+2C_0\). As \(2C_0C_2 C_1^2e^{2\beta \rho}\varepsilon^{(1-\alpha)} \leq C_0\) for \(\varepsilon>0\) small enough, we conclude that \(\Psi(h)\) satisfies (\ref{U - eq 1}) for all \(x \in M\). Therefore, \(\Psi(h) \in \mathcal{U}\). Since \(h \in \mathcal{U}\) was arbitrary, we obtain \(\Psi(\mathcal{U}) \subseteq \mathcal{U}\). This completes the proof.
\end{proof}

Since the proof of \Cref{Pinching with inj radius bound - full version} is merely an application of the Banach fixed point theorem based on \Cref{Invertibility of L}, the next remark is immediate due to \Cref{Invertibility of L - non-orientable}.

\begin{rem}\label{Pinching with inj rad bound - non-orientable}\normalfont
\Cref{Pinching with inj radius bound - full version} also holds when \(M\) is non-orientable.
\end{rem}

\section{Counterexamples}\label{Sec: Counterexamples}

For our most important applications, we need a version of \Cref{Pinching with inj radius bound - full version} which does not 
require a lower injectivity radius bound. The purpose of this section is to show that at least in dimension 3,
such a result can not be obtained as a straightforward extension of \Cref{Pinching with inj radius bound - full version}
by providing examples which show that such
straightforward extensions do not hold true. The mechanism behind
these examples lies in the fact that hyperbolic metrics on 
\emph{Margulis tubes}  or\emph{ cusps}  admit nontrivial hyperbolic deformations,
and such deformations can be used to construct families of metrics on closed hyperbolic
3-manfolds violating the a priori stability estimates which are essential
for an application of the implicit function theorem.
The geometric feactures of these examples motivate our approach towards our second main result 
\Cref{Pinching without inj radius bound - introduction} which is valid for 3-dimensional manifolds 
without the assumption of a lower injectivity radius bound.

\begin{prop}\label{Counterexample} For any $\varepsilon >0$, $\lambda \in (0,2)$ and any $C>0$ there exists 
a closed $3$-manifold $M_\epsilon$ and a Riemannian metric $g$ on $M_\varepsilon$ with the following properties.
\begin{enumerate}[i)]
\item The sectional curvature of $g$ is contained in the interval $[-1-\varepsilon,-1+\varepsilon]$.
\item For each component $A$ of the thin part of $(M_\varepsilon,g)$, we have
\[\int_A \frac{1}{(\inj)^{2-\lambda}}\vert {\rm Ric}(g)+2g\vert_g^2d{\rm vol}\leq \varepsilon^2.\]
\item The volume of $(M_\varepsilon,g)$ is bounded from above by a constant independent of $\varepsilon$.
\item There is no constant curvature metric
  \(g_{\rm const}\) on $M_\epsilon$ with the property that the identity
  \(\mathrm{id}_M:(M_\epsilon,g) \to (M_\epsilon,g_{\rm const})\) is a $C$-bilipschitz equivalence.
\end{enumerate}
\end{prop}

\begin{rem}\normalfont
It will be apparent from the proof that the geometric properties we use for the construction 
of the examples are special to dimension 3 and, by a result of Gromov \cite{G78}, do not extend to 
dimension at least 4. Although we expect a result similar to our second main theorem
to hold true in all dimensions, such an extension may require a new strategy for the proof. 
\end{rem}

For the construction of the manifolds $M_\varepsilon$ we start with an 
orientable hyperbolic $3$-manifold 
$M$ of finite volume, with a single cusp. 
For example, the figure 8 knot complement will do. 
The cusp has a neighborhood $B=\Gamma\backslash H$ which is the quotient of 
a horoball $H$ in $\mathbb{H}^3$ by
an abelian subgroup $\Gamma=\mathbb{Z}^2$ of parabolic isometries. The group 
$\Gamma$ preserves each of the 
horospheres which foliate $H$, and the quotient of each horosphere under 
$\Gamma$ is a flat two-torus $T^2$.

Let us write the cusp neighborhood $B$
and its hyperbolic metric $g$ in the form
\[B=T^2\times [0,\infty), \quad g=e^{-2t}g_0 + dt^2\]
where $g_0$ is a fixed flat metric on $T^2$.
In other words, we have $(T^2,g_0)=\Gamma_0\backslash \mathbb{R}^2$ where
$\Gamma_0$ is a group of translations of $\mathbb{R}^2$ isomorphic
to $\mathbb{Z}^2$. 

We now look at deformations of the metric $g_0$ of the following form.
Let $(e_1,e_2)$ be any orthonormal basis of $\mathbb{R}^2$. For $s\in \mathbb{R}$ 
consider the matrix
\[A(s)=\begin{pmatrix} 1 & 0\\
0 & e^s \end{pmatrix}. \]
It acts as a linear isomorphism on $\mathbb{R}^2$ not preserving the euclidean metric.
Let $g_s$ be the pull-back of the euclidean metric by $A(s)$. This pullback metric
preserves orthogonality of the vectors $e_1,e_2$, preserves the length of $e_1$ and 
scales the length of $e_2$ by the factor $e^s$. 

For numbers $\delta >0,R>0$ let $f_{\delta,R}:[0,\infty)\to [0,\infty)$ be a smooth function with the following 
properties.
\begin{enumerate}
\item $f_{\delta,R}$ is supported in $[0,\frac{R}{\delta}+4]$.
\item $\vert f_{\delta,R}\vert \leq \delta$. 
\item $\vert f^{\prime}_{\delta,R}\vert \leq \delta$.
\item $\vert f^{\prime\prime}_{\delta,R}\vert \leq 1$.
\item $\int_0^\infty f_{\delta,R}(s)ds =R.$
\end{enumerate}

Let $g_s=g(s,\delta,R,e_1,e_2)$ be the pull-back of the euclidean metric on the torus $T^2$ 
by the linear isomorphism $A(\int_0^sf_{\delta,R}(u)du)$.
Then for each $s$, the metric $g_s$ is a  
flat metric on $T$ which is the pullback
of the standard metric by an affine 
automorphism and as such 
determined by the property that
the vectors $e_1,e_2/e^{\int_0^sf_{\delta,R}(u)du}$ are orthonormal. 

Since the curvature of a Riemannian metric is computed by second derivatives of the metric, 
and since furthermore for any of the flat metrics $g_s$ on $T^2$, the metric 
$e^{-2t}g_s+dt^2$ on $B$ is hyperbolic, that is, of constant curvature $-1$, the following is a consequence
of the construction.

\begin{lem}\label{counterexamples - curvature}
For any $\varepsilon >0$ there exists a number $\delta(\varepsilon) >0$ such that for any $\delta <\delta(\varepsilon)$,
$m>0$ 
the curvature of each of the 
metrics 
\[e^{-2t}g(t-m,\delta,R,e_1,e_2)+ dt^2\] on 
$B=T^2\times [0,\infty)$ is contained in the interval $[-1-\varepsilon,-1+\varepsilon]$.
\end{lem}

\begin{proof}
As curvature computation is local, we can carry this out in the universal covering.
Thus for a fixed point $y\in T^2\times \{t\}\subset B$, we compute in the universal covering 
$\mathbb{R}^2\times [0,\infty)$, which we can identify with the 
horoball $H=\{x_3\geq c\}$ for some $c>0$ 
in hyperbolic $3$-space $\mathbb{H}^3=\{(x_1,x_2,x_3)\in \mathbb{R}^3\mid x_3>0\}$.
We also may assume that 
the horosphere $S=\{x_3=1\}\subset H$ contains
a preimage $\tilde y$ of $y$. Furthermore, we consider the standard hyperbolic metric 
$h$ on $H$, where the normalization is such that the standard flat metric on 
$\{x_3=1\}$ is the preimage of the flat  
metric on the slice $T^2\times \{t\}$ containing $y$. 

For simplicity of notation, write $u=t-s$.
With this description, the hyperbolic metric near $\tilde y$ determined 
by the flat metric on $T^2\times \{s\}$ is the 
warped product metric $h=e^{-2u}h_0+ du^2$ on $H$ where $u=\log x_3$, and the lift
$\hat g$  of the 
metric $e^{-2u}g(s-m,\delta,R,e_1,e_2)$ near $\tilde y$ is of the form 
\[\hat g=e^{-2u}A(\beta(u))^*h_0 + du^2\] 
where $\beta(u)$ is a smooth function on an interval containing $0$
which satisfies $\beta(0)=0$ and 
whose first and second derivatives near $0$ are smaller than $2\delta$ in absolute value. 

Thus in standard coordinates, the Christoffel symbols for the metric $\hat g$ 
and their first derivatives are uniformly
near the Christoffel symbols and their first derivatives for the hyperbolic metric. This implies that 
for any $\varepsilon >0$ we
can find a number $\delta(\varepsilon) >0$ 
so that the statement of the lemma holds true for this $\epsilon$.
\end{proof}

We now use this construction as follows.
Consider as before a finite volume hyperbolic $3$-manifold $M$ with a single 
cusp $B$. We shall Dehn-fill the cusp and use a deformation of the Dehn filled metric 
to achieve our goal. The geometric control we are looking for is obtained
from geometric information of Margulis tube of 
the Dehn filled manifold corresponding to the cusp of $M$.

To set up this construction, note that a Margulis tube in 
a hyperbolic 3-manifold is given as the quotient of a tubular neighborhood 
$N(\tilde \gamma,R)$ of 
radius $R>0$ of a geodesic $\tilde \gamma\subseteq \mathbb{H}^3$ by an infinite cyclic group
$\langle \phi \rangle\subseteq SO(3,1)$ of isometries, 
generated by an element $\phi$ which preserves $\tilde{\gamma}$ and 
acts on $\tilde{\gamma}$ as a translation. Then $\phi$ can be represented as a product of a transvection
preserving $\tilde \gamma$ and an isometry $\psi$ which fixes $\tilde \gamma$ pointwise and
acts as a rotation on the orthogonal complement of $\tilde \gamma^\prime\subseteq T_{\tilde \gamma}{\mathbb H}^3$.

Parameterize $\tilde \gamma$ by arc length and 
write $x_0= \tilde \gamma(0)$. There is a totally geodesic hyperbolic plane $\mathbb{H}^2
\subseteq \mathbb{H}^3$ orthogonal to $\tilde \gamma^\prime$ which passes through $x_0$.
The quotient by $\phi$ of the tubular neighborhood $N(\tilde \gamma,R)$ intersects 
$\mathbb{H}^2$ in a hyperbolic disk whose boundary is an embedded circle in the two-torus 
$T^2=\partial N(\tilde \gamma,R)/\phi$ of length $2\pi \sinh(R)$.
This circle is the meridian of the solid torus 
$B=N(\tilde \gamma,R)/\phi$. 

Choose a totally geodesic hyperbolic plane $H_0\subseteq \mathbb{H}^3$ which contains 
$\tilde \gamma$. If $\tau >0$ is the translation length of $\phi$, then $H_0$ intersects the fundamental 
domain $\{\exp Y\mid Y\in T_{\tilde \gamma(u)}\mathbb{H}^3, 0\leq u\leq \tau,Y\perp \tilde \gamma^\prime\}$ 
for $\phi$ in a strip bounded by two geodesics which are orthogonal to $\tilde \gamma$, and 
the intersection of this strip with $\partial N(\tilde \gamma,R)$ contains an arc which descends to 
a straight line segment on the 
boundary torus $T^2$ of length $\tau \cosh(R)$. In particular, the translation length $\tau$ of 
$\phi$, which equals the length of the closed geodesic in the free homotopy class defined by 
$\phi$ in the quotient manifold $\mathbb{H}^3/\langle \phi \rangle$, 
can explicitly be computed from the length of the meridian
on $T^2$ and the length of a straight line segment orthogonal to the meridian which connects
two points on the meridian and does not contain an intersection point with the meridian in its 
interior.

A Dehn filling of the finite volume hyperbolic manifold 
$M$ is determined by the choice of a simple closed 
geodesic $\zeta$ on the boundary $T^2$ of the cusp, which is a flat torus.
The Dehn filling along $\zeta$ is obtained from $M$ by removal of the
cusp and gluing a solid torus along the boundary whose meridian is glued to 
$\zeta$. If $\zeta$ is sufficiently long in the flat metric on $T^2$,
then the filled manifold is hyperbolic
(see for example \cite{HK08} or \Cref{Sec: drillingand} of this article which does not 
depend on this section). 
Furthermore, as the lengths of such  simple closed curves tend to infinity, 
the Dehn filled manifolds, equipped with their unique hyperbolic metrics, 
 will be almost isometric to $M$ on larger and larger neighborhoods of the
complement of the cusp (see for example \cite{BP92} or \Cref{Sec: drillingand}). 
As a consequence, for each $\ell>0$ we can find such a
Dehn filling with the property that the hyperbolic manifold obtained by this Dehn filling
contains a copy of the $\ell$-neighborhood of $T^2$ in the cusp 
$B\subset M$ up to a change of the metric
in the $C^2$-topology 
which is as close to zero as we wish. 

Fix a number $\lambda \in (0,2)$. 
The modification of the metric on such a Dehn filling of $M$ will be carried out in 
a region of the form $T^2\times [m,m+\frac{R}{\delta}+4]$ in standard coordinates on the cusp where \(\delta < \delta(\varepsilon)\) (for \(\delta(\varepsilon)\) given by \Cref{counterexamples - curvature}), and $m>0$ is a number which is sufficiently large that 
\[
	D_0^2\int_m^\infty e^{-\frac{\lambda}{2} t} dt \leq 1,
\] 
where \(D_0>0\)
is the (intrinsic) diamater of the boundary $T^2$
of the cusp (see \Cref{counterexamples - volume}). 

After choosing $m$ with this property, the length of the 
meridian for the Dehn filling is chosen large enough so that the Dehn filled metric is 
arbitrarily near the metric of the cusp on the neighbhood of radius $\ell=m+\frac{R}{\delta}+4$ 
of the thick part of $M$. 
The deformation is chosen so that it stretches the
direction orthogonal to the meridian $\zeta$. 
For an arbitrarily chosen constant $C>0$, 
the deformed metric $g$ on the Dehn filled hyperbolic manifold $(M_\zeta,g_0)$ 
has the following properties.
\begin{enumerate}
\item 
The metric coincides with the hyperbolic metric $g_0$ on a neighborhood of radius $m$
of the thick part of the hyperbolic metric. 
\item The ratio of the lengths of the closed geodesics for the metrics $g$ and $g_0$
in the filled manifold which are freely homotopic to the core curve of the tube is at least $C$.
\end{enumerate}

\begin{lem}\label{finished}
For all sufficiently large $m$ 
there is no constant curvature 
metric \(g_{\rm const}\) so that
\(\mathrm{id}_M:(M_\zeta,g) \to (M_\zeta,g_{\rm const})\) is a \(\sqrt{C}/2\)-bilipschitz equivalence.
\end{lem}

\begin{proof}
Choose $m$ sufficiently large that there exists a closed geodesic $\beta$ in 
the union of the thick part of $M$ with the $m$-neighborhood of the boundary torus of the cusp.
The length of this geodesic in the Dehn filled manifold $M_\zeta$, equipped with the hyperbolic 
metric $g_0$, is almost identical 
to the length of $\beta$. Furthermore, $\beta$ is a closed geodesic for the 
deformed metric $g$ since $g$ coincides with the hyperbolic metric 
near $\beta$.

Assume there is a constant curvature metric \(g_{\rm const}\) so that \(\mathrm{id}_{M_\zeta}:(M_\zeta,g) \to (M_\zeta,g_{\rm const})\) is a \(\sqrt{C}/2\)-bilipschitz equivalence. 
By Mostow Rigidity, there is then a number \(c > 0\) and a diffeomorphism \(\phi\) homotopic to the identity 
so that \(\phi^\ast g_{\rm const}=c^2g_0\). Define \(\hat{g}:=\phi^\ast g\). Then \(\mathrm{id}_{M_\zeta}:(M_\zeta,\hat{g}) \to (M_\zeta,c^2g_0)\) is a \(\sqrt{C}/2\)-bilipschitz equivalence. Denote by \(\gamma_0\) the core geodesic of the distinguished Margulis tube of the Dehn filled hyperbolic manifold \((M_\zeta,g_0)\). 
Note that \(\gamma_0\) also is the core geodesic for \((M_\zeta,g)\). 
As a consequence, \(\gamma:=\phi^{-1}(\gamma_0)\) is the unique \(\hat{g}\)-geodesic in its free homotopy class. 
Moreover, $\gamma$ and $\gamma_0$ are freely homotopic 
since \(\phi \simeq \mathrm{id}\). Therefore, 
\begin{align*}
 C\ell_{g_0}(\gamma_0)& \leq   \ell_g(\gamma_0) && \text{(by the construction of } g) \\
  & =  \ell_{\hat{g}}(\gamma)&& \text{(by the definition of } \hat{g} \text{ and } \gamma) \\
  & \leq \ell_{\hat{g}}(\gamma_0) && \text{(because } \gamma \text{ is a } \hat{g}\text{-geodesic and }\gamma \simeq \gamma_0) \\
  & \leq \frac{1}{2}\sqrt{C}\ell_{c^2g_0}(\gamma_0) && \text{(by the bilipschitz equivalence)} \\
  & =\frac{1}{2}c\sqrt{C}\ell_{g_0}(\gamma_0). &&
\end{align*}
Hence \(c \geq 2\sqrt{C}\). Similarly, we have
\begin{align*}
 \ell_{g_0}(\beta)& =  \ell_g(\beta) && \text{(by the first paragraph)} \\
  & =  \ell_{\hat{g}}(\phi^{-1}(\beta))&& \text{(by the definition of } \hat{g}) \\
  & \geq \frac{2}{\sqrt{C}}\ell_{c^2g_0}(\phi^{-1}(\beta)) && \text{(by the bilipschitz equivalence)}\\
  & \geq \frac{2c}{\sqrt{C}}\ell_{g_0}(\beta). && \text{(because } \beta \text{ is a } g_0\text{-geodesic and }\phi^{-1}(\beta) \simeq \beta) 
\end{align*}
Thus \(c \leq \frac{1}{2}\sqrt{C}\). This is a contradiction.
\end{proof}

It follows from the following lemma that the constructed manifolds satisfy the curvature assumption \(ii)\) of \Cref{Counterexample}.

\begin{lem}\label{counterexamples - volume}
Let \(M\) be a closed \(3\)-manifold, \(T\) a Margulis tube of \(M\) with core geodesic \(\gamma\), and \(\lambda \in (0,2)\). Assume that \(-1-\varepsilon\leq \sec(M)\leq -1+\varepsilon \) for some \(\varepsilon \leq \frac{1}{8}\lambda\), and that for some \(m>0\) the metric is hyperbolic outside the region
\[
	\Big\{y \in T \, | \, m \leq \dist(y,M_{\rm thick}) \leq {\rm Rad}-1 \Big\},
\]
where \({\rm Rad}\) is the Radius of \(T\). Then for some universal constant \(c>0\) it holds
\[
	\int_M \frac{1}{\inj(y)^{2-\lambda}}|\Ric(g)+2g|^2(y) \, d\vol(y) \leq cD_0^2\varepsilon^2 \int_m^{{\rm Rad}-1} e^{-\frac{\lambda}{2}r} \, dr,
\]
where \(D_0:=\diam(\partial T)\) is the (intrinsic) diameter of \(\partial T\).
\end{lem}

\begin{proof}For \(r \geq 0\) denote by \(T(r)\) the torus in the Margulis tube \(T\) all whose points have distance \(r\) to \(\partial T\). It follows from standard Jacobi field estimates that for some universal constant \(c>0\) it holds
\[
	{\rm area}\big(T(r)\big)\leq c e^{-2(1-\varepsilon)r}{\rm area}(\partial T) \leq cD_0^2e^{-2(1-\varepsilon)r}
\]
for all \(r \in [0,{\rm Rad}-1]\). Similarly, a comparison argument shows that for all \(y \in T(r)\) with \(r \in [0,{\rm Rad}-1]\) it holds
\[
	\frac{1}{\inj(y)} \leq ce^{(1+\varepsilon)r}
\]
for some universal constant \(c>0\) (see the proof of \Cref{counting preimages II} for more details). Thus
\[
	\int_{T(r)}\frac{1}{\inj(y)^{2-\lambda}} \, d\vol_2(y) \leq cD_0^2e^{\left((2-\lambda)(1+\varepsilon)-2(1-\varepsilon)\right)r}
\]
for all \(r \in [0,{\rm Rad}-1]\). Note that \((2-\lambda)(1+\varepsilon)-2(1-\varepsilon)=\varepsilon (4-\lambda)-\lambda \leq -\frac{1}{2}\lambda\) since by assumption \(\varepsilon \leq \frac{1}{8}\lambda\). Therefore, the desired estimate follows from the fact that the curvature assumption \(\sec(M)\in [-1-\varepsilon,-1+\varepsilon]\) implies
\(
	|\Ric(g)+2g|^2\leq 3(2\varepsilon)^2.
\)
\end{proof}

We quickly review the construction of the counterexamples constructed in this section and point out what should be taken away from these examples. 
We started with a hyperbolic metric. The new metric was defined by slowly changing the conformal structure on the horotori (quotients of horospheres). But this change only started 
deep in the thin part of the manifold. More precisely, the change of the conformal structure only starts at horotori 
that have distance at least \(m\) to \(M_{\rm thick}\), where for some arbitrary constant $\lambda \in (0,2)$, the number 
\(m>0\) was chosen so large that \(D_0^2\int_m^\infty e^{-\frac{\lambda}{2} r} dr \leq 1\), where \(D_0\) is the (intrinsic) diameter of the boundary of the filled Margulis tube \(T\). In particular, the change of the conformal structure only occours on tori \(T(r)\) whose diameter is bounded by some universal constant. Here  \(T(r):=\{y \in T \, | \,  d(y,\partial T)=r\}\). Indeed, for \(r \geq m\) it holds
\begin{align*}
	\diam(T(r)) \leq & \, \diam(T(m)) && (\text{monotonicity})\\
	\leq & \, cD_0e^{-m} && (\text{the metric is hyperbolic up to }T(m))\\
	\leq & \, cD_0^2 \int_m^\infty e^{-r}\, dr && (\diam(\partial T)\geq \inj(\partial T)=\mu)\\
	\leq & \, cD_0^2 \int_m^\infty e^{-\frac{1}{2}\lambda r}\, dr && (\lambda < 2)\\
	\leq & \, c &&(\text{definition of }m),
\end{align*}
where \(c>0\) is a universal constant, and \(\mu\) is a Margulis constant. With the terminology of \Cref{Sec:thethin} we can express this by saying that the change of the conformal structure only happens in the \textit{small part} of the Margulis tube.


These geometric facts (which however are inherent to dimension 3, see \cite{G78}) prevent the establishment of the $C^0$-estimate 
required in the proof of \Cref{Pinching with inj radius bound - full version}, and this problem can not be resolved by enriching the hybrid norms \(||\cdot||_2\) and \(||\cdot||_0\) with weighted \(L^2\)-norms whose weight involves \(\inj(y)\).

Our second main result \Cref{Pinching without inj radius bound - introduction} overcomes this difficulty by constructing new hybrid Banach spaces and imposing stronger geometric control in the thin 
part of a negatively curved 3-manifold $M$. This is done in two steps. In a first step, carried out in \Cref{Sec:thethin}, we locate the region in the thin 
part of $M$ for which we can obtain $C^0$-estimates sufficient for our goal with a direct adaptation of the proof of 
\Cref{Pinching with inj radius bound - full version}. In a second step, we control its complement, which we call the \emph{small part} of the manifold, with an 
ODE-Ansatz motivated by the work \cite{Bamler2012}.


\section{The thin but not small part of a negatively curved 3-manifold}
\label{Sec:thethin}

In the proof of \Cref{Pinching with inj radius bound - full version}, the assumption
on a lower bound for the injectivity radius was used to establish a
$C^0$-estimate for a symmetric $(0,2)$-tensor field $h$ from knowledge of
${\mathcal L}(h)$. The examples in \Cref{Sec: Counterexamples} show that
without such a bound, we can not expect that such an a priori estimate holds true.

In the remainder of this section, 
\(M\) denotes a finite volume Riemannian 
3-manifold of sectional curvature in $[-4,-1/4]$, 
with universal cover 
\(\tilde{M}\). We know that
each cusp is diffeomorphic to $T^2\times [0,\infty)$ where $T^2$ is the 2-torus.
This is true since by \Cref{convention orientable} we assume that $M$ is orientable.
In \Cref{small} 
we introduce a region $M_{\rm small}$ in $M$, called \emph{small part} of $M$, and 
we establish some of its basic properties. We show in \Cref{co2} that 
if $M$ satisfies the hypothesis on the curvature stated in
\Cref{Pinching with inj radius bound - full version}, then in $M\setminus M_{\rm small}$, 
a modified version of such a $C^0$-estimate   
holds true in spite of the fact
that the injectivity radius may be arbitrarily small. The proof rests on 
a counting result for preimages in $\tilde M$ 
of points in $M-M_{\rm small}$ 
which is proved in \Cref{count}.

\subsection{The small part of $M$}\label{small}

Choose once and for all a  Margulis constant $\mu\leq 1$ for 3-manifolds $M$
with curvature in the interval $[-4,-1/4]$. This constant determines
the thin part 
\[M_{\rm thin}=\{x\in M\mid {\rm inj}(x)<\mu\}\]
of $M$. For a number $\chi\leq \mu$ let $M^{<\chi}\subset M_{\rm thin}$ be
the set of all points $x$ with ${\rm inj}(x)<\chi$.
The set $M_{\rm thin}$ is a disjoint union of \emph{cusps} and 
\emph{Margulis tubes}, 
where a Margulis tube is a tubular neighborhood of a closed geodesic 
of length smaller than $2\mu$.

From now on, we shall only consider Margulis tubes of 
radius $r\geq 3$, where the radius means the distance between 
the core curve of the tube and the boundary, determined by the constant $\mu$. 
Comparison
shows that this only excludes tubes whose core curves have length
bounded from below by a fixed positive constant. In other words, tubes of radius at most three 
can be thought of belonging to the thick part of $M$ for a properly 
adjusted Margulis constant.

Consider a Margulis tube \(T\) of \(M\) 
and let \(\gamma\) be its  core geodesic. For \(r>0\) denote by 
\[
	T(r):= \big\{ x \in {M} \, | \, d(x,{\gamma})=r\big\}
\]
the torus of distance \(r\) to \({\gamma}\). The \textit{small part} of the Margulis tube \(T\) is
\[
	T_{\rm small}:= \Big\{x \in T \, | \, d(x,\gamma)\leq 2 \text{ or } 
	\diam\big(T(r_{\gamma}(x))\big) \leq D \Big\},
\]
where \(r_{\gamma}=d(\cdot,\gamma)\), the diameter is with respect to the intrinsic metric on the torus, 
and 
\(D>0\) is a universal constant which will be determined later.

The \emph{small part} \(C_{\rm small}\) of a rank 2 cusp \(C\) is 
defined similarly. The only difference is that we have to consider Busemann functions (instead of \(r_\gamma(\cdot)\)) to define the level tori \(T(r)\). Fix a rank 2 cusp \(C\), and let \(\xi \in \partial_{\infty} \tilde{M}\) be a point corresponding to \(C\). Choose a Busemann function \(b_\xi:\tilde{M} \to \bbR\) associated to \(\xi\) (see \cite{BGS85} for more information on Busemann functions). This induces a Busemann function \(\bar{b}_\xi:C \to \bbR\). For \(r \in \bbR\) we define \(T(r):=\{x \in C \, | \, \bar{b}_\xi(x)=r\}\).
As in the case of a tube, we define the \textit{small part} of the cusp \(C\) to be
\[
	C_{\rm small}:= \bigcup_r \Big\{T(r) \mid \diam ( T(r))\leq D\Big\}, 
\]
where \(D\) is the same universal constant as before that will be determined later, and 
the diameter is the intrinsic diameter. 

Finally, the \textit{small part} \(M_{\rm small}\) of the manifold \(M\) is the union
\[
	M_{\rm small}:=\bigcup_{T}T_{\rm small} \cup \bigcup_{C}C_{\rm small},
\]
where \(T\) ranges over all Margulis tubes \(T\) of radius at least 3 
and \(C\) ranges over all rank 2 cusps of \(M\).

\begin{rem}\label{nohighdim}
\normalfont
Although the definition of the small part of a negatively curved manifold makes sense in 
all dimensions, it follows from \cite{G78} that the small part of a closed 
negatively curved manifold $M$ of dimension $n\geq 4$ is the union of 
tubular neighborhoods of short geodesics of uniformly bounded radius. 
The fact that this is not true in dimension 3 (see Section \ref{Sec: Counterexamples}) 
is the main reason for introducing the small part of a negatively curved 3-manifold.
\end{rem} 

\begin{rem}\label{smallconnected}\normalfont
Since we use intrinsic diameters for the definition of the small part of $M$,
standard Jacobi field estimates show that the small part of a Margulis tube 
or cusp is a connected subset of the tube or cusp.
\end{rem} 

\begin{rem}\label{small part for non-orientable}\normalfont
If \(M\) is non-orientable, we define \(M_{\rm small}:=\pi(\hat{M}_{\rm small})\), where \(\pi:\hat{M} \to M\) is the orientation cover of \(M\).
\end{rem}

Choose \[D<\min\{\mu,1/4\}\] sufficiently
small so that the one-neighborhood of the $D$-thin part $M^{D}$ 
is contained in the $\mu/2$-thin part $M^{<\mu/2}$ of $M$. Using comparison of Jacobi fields, one observes that such a 
constant \(D\) only depends on the choice of $\mu$ and the curvature bounds. 



%


\begin{lem}\label{containedinthin}
$M_{\rm small}\subset M^{<D}$ and hence if  
\(x \in M_{\rm small}\), 
then the torus $T(r)$ containing $x$ 
is a $C^1$-submanifold of $M$ contained in \({\rm int}(M_{\rm thin})\) whose distance to \(M_{\rm thick}\) is at least \(\mu/2\).
\end{lem}


\begin{proof}
Let \(x \in M_{\rm small}\) be arbitrary. We only present the case that \(x\) is contained in a Margulis tube \(T\), the case of a cusp being similar. 

Let \(\gamma\) denote the core geodesic of \(T\). We first consider the case that \(r_\gamma(x)=d(x,\gamma) \leq 2\). 
As we only consider Margulis tubes of radius at least three, the distance of \(x\) to \(M_{\rm thick}\) is at least \(1 > \mu/2\). 

So we may assume \(r:=r_\gamma(x)>2\). By the 
definition of $M_{\rm small}$, the diameter of the distance torus $T(r)$ containing $x$ is at most $D$. 
Let $\alpha\subseteq T(r)$  be a shortest closed geodesic for the induced metric on $T(r)$ which is not contractible as a curve in $T(r)$. The length of $\alpha$ is at most $2D$. 
If $\alpha$ is contractible in $M$, then 
$\alpha$ is a meridian in $T(r)$. Thus comparison of Jacobi fields shows that the length of \(\alpha\) is at least \(4\pi\sinh(r/2)\). 
As \(r>2\) and \(D \leq \frac{1}{4}\), this length is at least \(4\pi \sinh(1)>1>2D\), which is a contradiction. 

As a consequence, 
$\alpha$ is not contractible in $M$ 
and hence defines an essential loop in $M$
of length at most $2D$. But then $x$ is contained in the $D$-thin part of 
$M$ and hence the $D$-neighborhood of $x$ (which contains 
$T(r)$) is contained in the $\mu/2$-thin part of $M$ by the choice of $D$. 
In particular, $T(r)$ is the projection to $M$ of the level set of
a function in the universal covering $\tilde M$ of $M$ which either
is the distance function to a geodesic line or a Busemann function.
Such functions are known to be of class $C^1$ and non-singular away from 
their minimum
\cite{BGS85}. This completes the proof.
\end{proof}


\subsection{The $C^0$-estimate}\label{co2}

The main reason for introducing the small part of $M$ is
that for points in $M_{\rm thin} \setminus M_{\rm small}$,
we can prove a \(C^0\)-estimate which is weaker than the estimate
established in the proof of 
\Cref{Pinching with inj radius bound - full version}
but sufficient for an analogous conclusion. 
To obtain \(C^0\)-estimates for points in \(M_{\rm small}\), we shall use an ODE-Ansatz 
inspired by \cite{Bamler2012} (see \Cref{Subsec: Growth estimates} for more details). 
As before, \(\mathcal{L}\) denotes the elliptic differential operator given by \(\mathcal{L}h=\frac{1}{2}\Delta_Lh+2h\).

\begin{prop}\label{A-priori estimate away from the small part}For all \(\alpha \in (0,1)\), \(\Lambda \geq 0\), \(\delta \in (0,2)\), and \(b > 1\) there exist \(\varepsilon_0=\varepsilon_0(\delta,b)>0\) and \(C=C(\alpha,\Lambda,\delta,b) >0\) with the following property. Let \(M\) be a Riemannian \(3\)-manifold of finite volume so that
\[
	|\sec +1| \leq \varepsilon_0 \quad \text{and} \quad ||\nabla \Ric||_{C^0(M)} \leq \Lambda.
\]
Then for all \(h \in C^{2}\big( {\rm Sym}^2(T^*M)\big)\cap H^2(M)\) and \(x \in M_{\rm thin} \setminus M_{\rm small}\) it holds
\[
	|h|(x) \leq C\left(||\mathcal{L}h||_{C^0(M)}+e^{\frac{b}{2}d(x,M_{\rm thick})}\left(\int_M e^{-(2-\delta)r_x(y)}|\mathcal{L}h|^2(y) \, d\vol(y) \right)^\frac{1}{2} \right).
\]
\end{prop}

The estimate in \Cref{A-priori estimate away from the small part} 
motivates the definition of the norm \(||\cdot||_{0,\lambda}\) that will be introduced 
in \Cref{Subsec: various norms}. The proof of \Cref{A-priori estimate away from the small part}
is based on a counting result for the number of preimages of points
in $M\setminus M_{\rm small}$ 
in the universal 
covering $\tilde M$ of $M$ contained in a ball in $\tilde M$ 
of fixed size which is the main result of \Cref{count}.  
We denote by \(N_r(M \setminus M_{\rm small})\) (\(r>0\)) the $r$-neighborhood of $M \setminus M_{\rm small}$. 

\begin{prop}[Counting preimages]\label{Counting preimages - general case} There is a constant 
\(C>0\) so that for every 
\(x \in M^{<D}\cap 
N_{1/4}(M \setminus M_{\rm small})\) and every lift \(\tilde{x} \in \tilde{M}\) it holds
\[
	\# \big(\pi^{-1}(x) \cap B(\tilde{x},D) \big) \leq C \frac{1}{\inj(x)},
\]
where \(\pi:\tilde{M} \to M\) is the universal covering projection.
\end{prop}


The complete proof of \Cref{Counting preimages - general case} is is a bit technical. For this reason we postpone it to \Cref{count}. However, in the special case that \(\sec \equiv -1\) in \(M_{\rm thin}\), the proof is very simple and it already contains the core ideas for the general case. Moreover, this special case is sufficient for our applications to drilling and filling, and effective hyperbolization.

\begin{proof}[Proof of \Cref{Counting preimages - general case} when the thin part is hyperbolic]Fix some \(x_0 \in M^{<\mu^\prime}\cap 
N_{1/4}(M \setminus M_{\rm small})\), and denote by \(T:=T(r(x_0))\) the distance torus or horotorus containing \(x_0\). Since the intrinsic geometry of \(T\) is uniformly bilipschitz to its extrinsic geometry (see \Cref{controlgeometry} for a detailed formulation), it suffices to show
\[
	\# \big(\pi_T^{-1}(x_0) \cap B(\tilde{x}_0,D) \big) \leq C \frac{1}{\inj(T)},
\]
where \(\pi_T:\bbR^2 \to T\) is the universal covering projection of \(T\), and \(B(\tilde{x}_0,D) \subseteq \bbR^2\).

Since radial projections are uniformly Lipschitz (see \Cref{controlgeometry}), 
it follows from the definition of the small part of $M$ 
that \(\diam(T) \geq D^\prime\) for some universal constant \(D^\prime\). 
Note that \(T\) is a flat torus since \(\sec \equiv -1\) in \(M_{\rm thin}\). Thus 
 by a simple volume counting argument 
for balls in $\mathbb{R}^2$, it suffices to assume that \(\inj(T)\leq \frac{1}{2}D^\prime\).

 By the results of Section 2.24 in \cite{GHL}, there is a fundamental region \(\big\{tv_1+sv_2 \, | \, s,t \in [0,1]\big\} \subseteq \bbR^2\) for \(T\) with 
\[
	|v_1|=2\inj(T) \quad \text{and} \quad \theta:=\measuredangle (v_1,v_2) \in \left[\frac{\pi}{3},\frac{2\pi}{3}\right].
\]
Clearly \(D^\prime \leq \diam(T) \leq \frac{1}{2}\big(|v_1|+|v_2|\big)\), and thus \(|v_2| \geq D^\prime\) since we assume \(\inj(T) \leq \frac{1}{2}D^\prime\). For any \(\tilde{x} \in \pi_{T}^{-1}(x_0)\) consider 
\[
	\mathcal{R}_{\tilde{x}}:=\tilde{x}+\left\{  t v_1 + s\frac{v_2}{|v_2|} \, \Big| \, t \in [0,1/2], s \in [0,D^\prime/2]\right\}, 
\]
and note that they are pairwise disjoint. All of them have area
\[
	{\rm area}(\mathcal{R}_{\tilde{x}})=\sin(\theta)\frac{|v_1|}{2}\frac{D^\prime}{2} \geq \frac{\sqrt{3}D^\prime}{4}\inj(T).
\]
Also \(\bigcup_{\tilde{x} \in B(\tilde{x}_0,D)}\mathcal{R}_{\tilde{x}} \subseteq B(\tilde{x}_0,D+2D^\prime)\) since 
\(
	\diam(\mathcal{R}_{\tilde{x}}) \leq \frac{1}{2}\big(|v_1|+D^\prime\big) < 2D^\prime.
\)
As the sets \(\mathcal{R}_{\tilde{x}}\) are pairwise disjoint, volume counting implies
\[
	\# \big(\pi_T^{-1}(x_0) \cap B(\tilde{x}_0,D) \big) \leq \frac{{\rm area}\big(B(\tilde{x}_0,D+2D^\prime)\big)}{{\rm area}(\mathcal{R})} \leq \frac{C}{\inj(T)}.
\]
This completes the proof of the special case.
\end{proof}


In order to deduce \Cref{A-priori estimate away from the small part} from \Cref{Counting preimages - general case} we have to replace \(\frac{1}{\inj(x)}\) by a function that is easier to control. 

\begin{cor}\label{counting preimages II} There is a universal constant \(C>0\) with the following property.
Let \(M\) be a Riemannian 3-manifold with 
\[
	-b^2 \leq \sec_M \leq -1/4
\]
for some \(1\leq b \leq 2\). Then for all \(x \in N_{1/4}(M \setminus M_{\rm small})\) and every lift $\tilde x$ of $x$ to $\tilde M$, it holds
\[
	\#\big(\pi^{-1}(x) \cap B(\tilde{x},D) \big) \leq C e^{b  d(x,M_{\rm thick})}.
\]
\end{cor}

The proof of \Cref{counting preimages II} is also contained in \Cref{count}. We now show how \Cref{counting preimages II} can be used to prove \Cref{A-priori estimate away from the small part}.

\begin{proof}[Proof of \Cref{A-priori estimate away from the small part}]Abbreviate \(f:=\mathcal{L}h\). 
It holds \(\mathcal{L}\tilde{h}=\tilde{f}\) in the universal cover. By the argument which led to (\ref{C^0 from L^2}), we have
\begin{equation}\label{not small - NashMoser}
	|\tilde{h}|(\tilde{x}) \leq C \Big(||\tilde{h}||_{L^2(B(\tilde{x},D/2))}+||\tilde{f}||_{C^0(\tilde{M})} \Big)
\end{equation}
for a constant \(C=C(n,\alpha,\Lambda)\). Therefore, it suffices to bound \(||\tilde{h}||_{L^2(B(\tilde{x},D/2))}\). 
To this end, we invoke the following basic claim, which states that an integral in the universal cover can be estimated by a weighted integral in the manifold when the weight is an upper bound for the number of preimages.

\begin{claim}Let \(x\in M\) and \(\rho:M \to \bbR\) be a function so that
\[
	\# \left(\pi^{-1}(y) \cap B(\tilde{y},D) \right) \leq \rho(y)
\]
holds for all \(y \in B(x,D/2) \subseteq M\). Let \(u: M \to \bbR_{\geq 0}\) be a non-negative integrable function and denote by \(\tilde{u}:=u \circ \pi\) its lift to the universal cover. Then
we have 
\[
	\int_{B(\tilde{x},D/2)} \tilde{u}(\tilde{y}) \, d\vol_{\tilde{g}}(\tilde{y}) \leq  \int_{B(x,D/2)} \rho(y) u(y) \, d\vol_{g}(y). 
\]
\end{claim}

\begin{proof}[Proof of the claim] By the triangle inequality, if $\tilde y\in B(\tilde x,D/2)$ then 
$ B(\tilde x,D/2) \subseteq B(\tilde y,D)$. Thus by assumption, 
a point $y\in B(x,D/2)$ has at most $\rho(y)$ preimages in $B(\tilde x,D/2)$. 
Hence the claim holds true for the indicator function \(u=\chi_U\) of a small open 
subset \(U \subseteq B(x,D/2)\). 
By linearity and monotonicity the result follows for all non-negative simple functions. A standard approximation argument completes the proof.
\end{proof}

For \(\varepsilon_0=\varepsilon_0(b)>0\) small enough, 
\(|\sec+1| \leq \varepsilon_0\) implies \(-b^2\leq \mathrm{sec} \leq -1/4\). Hence \Cref{counting preimages II} shows that for any \(x \in M \setminus M_{\rm small}\) the function \(\rho(y)=Ce^{bd(y,M_{\rm thick})}\) satisfies the assumption of the claim. Thus
\[
	\int_{B(\tilde{x},D/2)} |\tilde{h}|^2(\tilde{y}) \, d\vol_{\tilde{g}}(\tilde{y}) \leq C \int_{B(x,D/2)} e^{bd(y,M_{\rm thick})}|h|^2(y) \, d\vol_{g}(y).
\]
As \(d(y,M_{\rm thick}) \leq d(x,M_{\rm thick})+D/2\) for \(y \in B(x,D/2)\), 
\begin{equation}\label{not small - int1}
	\int_{B(\tilde{x},D/2)} |\tilde{h}|^2(\tilde{y}) \, d\vol_{\tilde{g}}(\tilde{y}) \leq Ce^{b D/2} e^{bd(x,M_{\rm thick})}\int_{B(x,D/2)}|h|^2(y) \, d\vol_{g}(y).
\end{equation}
Moreover,
\begin{equation}\label{not small - int2}
	\int_{B(x,D/2)}|h|^2(y) \, d\vol_{g}(y) \leq e^{D}\int_{B(x,D/2)}e^{-(2-\delta)r_x(y)}|h|^2(y) \, d\vol_{g}(y) .
\end{equation}
By \Cref{integration by parts works in finite volume case} the integral estimate (\ref{integral1}) from the proof of \Cref{A-priori estimate for L} is still valid. In particular, for \(\varepsilon_0=\varepsilon_0(\delta)>0\) small enough it holds
\begin{equation}\label{not small - int3}
	\int_M e^{-(2-\delta)r_x(y)}|h|^2(y) \, d\vol_g(y) \leq C\int_M e^{-(2-\delta)r_x(y)}|f|^2(y) \, d\vol_g(y)
\end{equation}
for a constant \(C=C(\delta)\). Combining (\ref{not small - NashMoser})-(\ref{not small - int3}) yields the desired estimate.
\end{proof}

\subsection{Counting preimages}\label{count} 
This  subsection is concerned with the proof of \Cref{Counting preimages - general case}. 
Before we come to the more technical details, we begin with a short overview of the proof.
Let $x\in M^{<\mu^\prime}\cap N_{1/4}(M\setminus M_{\rm small})$ and 
let as before ${\rm inj}(x)$ be the injectivity radius of $M$ at $x$. 
There is a geodesic loop of length at most $2{\rm inj}(x)$ based at $x$. 
This loop can be homotoped with fixed endpoints 
to a loop \(c_1\) lying entirely in the distance torus containing \(x\) of
controlled comparable length. 
 Using the assumption that $x\in N_{1/4}(M\setminus M_{\rm small})$, 
we then show that any closed curve on the 
torus whose homotopy class is not a multiple of 
the class of \(c_1\) has length at least \(\ell\) where
$\ell >0$ is a fixed constant. Namely, we show that otherwise 
the torus has a small diameter, 
contradicting that \(x\) is contained in a small neighbourhood of \(M\setminus M_{\rm small}\). Therefore, in the universal cover of the torus,
a preimage $\tilde x$ of \(x\) either lies on
the lift \(\tilde{c_1}\) of $c_1$ through $\tilde x$, 
or it has distance at least \(\ell\) from $\tilde x$.
A volume counting argument then completes the proof. 

The main step in the implementation  of this argument lies 
in obtaining sufficient geometric control on the tori so that the 
volume counting argument used in the case when the thin part is 
hyperbolic can be applied. 
The following proposition summarizes geometric properties of 
distance tubes and horospheres in simply connected 
manifolds of pinched negative curvature which 
are used in the proof of \Cref{Counting preimages - general case}. 
Note that although a priori Busemann functions are only of class $C^2$,
the Gau\ss{} equations show that their sectional curvature is 
defined and continuous. 
%

\begin{prop}\label{controlgeometry}
For every $n\geq 2$ 
there exists numbers \(A=A(n)\geq 1\) and \(B=B(n)>0\) such that for any 
simply connected complete $n$-manifold $\tilde M$ of curvature 
${\rm sec}_{\tilde M}\subseteq [-4,-1/4]$ the following holds true.
Let $\tilde \gamma\subseteq \tilde M$ be a geodesic. 
\begin{enumerate}[i)]
\item For
$r\geq 3/4$, the sectional curvature of the level sets
\[
	\{d(\tilde \gamma,\cdot)=r\}
\] 
with respect to the induced metric is contained in $[-A^2,A^2]$, and the injectivity
radius is at least $B$.
\item For
$r\geq 3/4$, the radial 
projection 
\[\{r-1/4 \leq d(\tilde \gamma,\cdot)\leq r+1/4\}\to 
\{d(\tilde \gamma,\cdot )=r\}\] is $A$-Lipschitz.
\item For points $x,y$ on distance level sets 
\(\{d(\tilde \gamma,\cdot ) =r\}\) \((r\geq 3/4)\) 
in $\tilde M$, the distance between $x,y$ with respect to 
 the intrinsic metric on the level set 
 is at most $Ad_{\tilde{M}}(x,y)$ provided that 
 $d_{\tilde M}(x,y)\leq 1/4$.
\end{enumerate}
Analogous properties also hold true for horospheres in $\tilde M$, with the same
constants $A>1,B>0$.
\end{prop}



 \begin{proof}[Proof of \Cref{controlgeometry}]
We sketch an argument for the first part of the proposition and 
refer to \cite{E87} and \cite{HIH77} for more information about the remaining parts. 

In simply connected manifolds of constant sectional curvature \(\kappa < 0\),
  the distance cylinders of distance \(r\) about geodesics 
  have principal curvatures \(\sqrt{-\kappa} \tanh(\sqrt{-\kappa}r)\)
  and \(\sqrt{-\kappa} \coth(\sqrt{-\kappa}r)\). By standard comparison results for
  solutions of the Riccati equation, 
  the principal curvatures \(\lambda\) of the distance tori in \(M\) are bounded by
  the maximal resp. minimal principal curvatures of the distance tori in the spaces of constant curvature \(-4\) resp. \(-1/4\).
  Thus \(\frac{1}{2}\tanh\left(\frac{1}{2}r\right) \leq \lambda \leq 2\coth(2r)\). 
   
  There is a constant \(c>1\) so that \(\frac{1}{c}\leq \frac{1}{2}\tanh\left(\frac{1}{2}r\right), 2\coth(2r)\leq c\) for all \(r \geq 1\). So the principal curvatures \(\lambda\) of the distance tori in \(M\)
  are contained in \([1/c,c]\). 
 Thus there are uniform bounds for the shape operator of the level sets, and since the curvature 
 of the ambient manifold is contained in $[-4,-1/4]$ by assumption, the curvature of the level sets 
is uniformly bounded by the Gau\ss{} equations
(see  Chapter 6 of \cite{doCarmoRG}).
 This completes the proof of the 
 curvature control stated in the proposition. 
   
 To establish a uniform lower bound on the injectivity radius of the 
 level sets \(Z=\{d(\tilde \gamma,\cdot)=r\}\) (\(r \geq 3/4\)),  note first 
 that for $\delta\leq 1/4$, the ball $B^Z({\tilde x},\delta)$ of radius $\delta$ about  a point 
$\tilde x\in Z$ for the intrinsic metric contains the radial projection of the 
intersection with 
$Z$ of the ball $B(\tilde x,\delta/A)$ 
of radius $\delta/A$ in $\tilde M$ about $\tilde x$ (this uses \(ii)\)). 
This implies that $B(\tilde x,\delta/A)$ is contained in the preimage of 
$B^Z(\tilde x,\delta)$ under the restriction of the radial projection to 
the $1/4$-neighborhood of $Z$ in $\tilde M$. 

Since the radial projections are uniformly Lipschitz continuous, 
Fubini's theorem implies that the volume of $B^Z(\tilde x,\delta)$ is bounded 
from below by $C_0{\rm vol}(B(\tilde x,\delta/A))$, and the latter is bounded from 
below by $C_1(\delta/A)^{n}$ where $C_0,C_1$ only depend on the curvature
bounds of $\tilde{M}$.

As a consequence, for $\delta=1/4$ fixed, the volume of $B^Z(\tilde x,1/4)$ is bounded 
from below by a universal constant not depending 
on $\tilde x$ or $Z$.  Since the sectional curvature 
of the distance hypersurface $Z$ is bounded from above by a universal constant, 
this implies that its injectivity radius is bounded from below 
by a universal constant $B>0$ by a result of Cheeger, Gromov and Taylor (see \cite[Theorem 4.7]{CGT82}).   
 
All remaining statements follow in a similar way, and their proofs will be omitted.
\end{proof}



\begin{proof}[Proof of \Cref{Counting preimages - general case}]\textbf{Step 1 (Large injectivity radius):} 
Due to the definition of $D\leq 1/2$,
 the ball of radius $1$ about a point in $M^{\leq \mu^\prime}$ is contained in $M_{\rm thin}$. Therefore, 
 it suffices to consider the covering $\hat M=\tilde M/\Gamma$ of $M$ where 
 $\Gamma$ is the fundamental group of the component of $M_{\rm thin}$ containing $x$. 


Let \(A=A(3)>1\) and \(B=B(3)>0\) be as in \Cref{controlgeometry}. 
Assume without loss of generality that $B < D/2A$ where $D>0$ is as in the definition of $M_{\rm small}$.
If the injectivity radius ${\rm inj}(x)$ of $M$ at $x$ is at least $B/8A$, then the ball of 
radius $B/8A$ about $x$ is diffeomorphic to a ball of the same radius about 
a preimage $\tilde x$ of $x$ in $\tilde M$. By the curvature bounds, the volume of this 
ball is bounded from below by a universal positive constant $c_0>0$. Similarly, the volume 
of the ball $B(\tilde x,2D)$ is bounded from above by a universal constant
$c_1>0$. As the balls of radius $B/8A<D$ about the 
preimages of $x$ in $B(\tilde x,D)$ are pairwise disjoint and contained in 
$B(\tilde x,2D)$, the number of preimages of $x$ contained in 
$B(\tilde x,\mu^\prime)$ is at most $c_1/c_0$. Thus in the sequel we 
may always assume that ${\rm inj}(x)< B/8A$.

Let \(x \in M^{< B/8A} \cap  N_{1/4}(M\setminus M_{\rm small}) \),
choose a lift $\hat x$ of $x$ to $\hat M$
and let $\tilde x\in \tilde M$ be a lift of $\hat x$.
Let $T$ be the distance torus of $\hat M$ containing $\hat x$. 
There exists a geodesic loop \(\sigma\) based at \(\hat{x}\) 
with $\ell(\sigma)=2{\rm inj}(x)$.



By the definition of $M_{\rm small}$, the distance of $\hat x$ to the core geodesic is at least $3/4$.
Thus by \Cref{controlgeometry}, 
there is a curve \(c_1\) lying entirely in \(T\) that is
homotopic to \(\sigma\) relative endpoints and that satisfies 
\(\ell(c_1) \leq A\ell(\sigma) < B/4\). 
It follows from the definition of \(\sigma\) that
the curve \(c_1\) is not contractible in \(T\). 

Assume without loss of generality that $c_1\subseteq T$ is the
shortest essential based loop at $\hat x$. 
Then $c_1$ is simple, that is, $c_1$ does not have self-intersections, and it is a geodesic
with at most one breakpoint at $\hat x$. 
Cut $T$ open along $c_1$ and let $Z$ be the 
resulting metric cylinder. Let $\partial^0Z,\partial^1 Z$  be the two distinct boundary components of $Z$.
%


\textbf{Step 2 (Loops independent from \(c_1\) are long):} 
The distance $d:=d_Z(\partial^0 Z,\partial^1 Z)$ can be realized by an embedded arc $c_2\subseteq Z$ 
connecting $\partial^0Z$ to $\partial^1Z$. 
Concatenation of $c_2$ with a subarc of $c_1$ gives a closed essential curve 
$c_{2^\prime}\subseteq T$ of length $\ell(c_{2^\prime})\leq d+\ell(c_1)/2 < d+B/4$. 

By \Cref{controlgeometry}, 
since $T\subseteq N_{1/4}(M-M_{\rm small})$, 
the diameter of $T$ with respect to the intrinsic metric is at least $D/A>2B$.
We use this to show that $d\geq  B/4$. 
To this end we argue by contradiction and we assume otherwise. 
Let \(\gamma_1\subseteq T\) be the closed 
geodesic of minimal length in the free homotopy class of \(c_1\). Its length 
$\ell(\gamma_1)$ is at most $\ell(c_1)< B /4$.

Cut \(T\) along \(\gamma_1\) and denote the resulting cylinder by \(Z^\prime\). 
The connected components of \(c_{2^\prime} \setminus (\gamma_1 \cap c_{2^\prime})\) lift to 
arcs in \(Z^\prime\) whose endpoints lie on the one of the boundary components
\(\partial^0Z^\prime\) or \(\partial^1Z^\prime\) of $Z^\prime$. 
At least one of these lifts must connect \(\partial^0Z^\prime\) and \(\partial^1Z^\prime\).
Namely,
otherwise \(c_{2^\prime}\) is freely homotopic to a multiple of \(\gamma_1\),
which contradicts that $c_1,c_{2^\prime}$ intersect in a single point.
This implies that \(d_{Z^\prime}(\partial^0Z^\prime,\partial^1Z^\prime) \leq \ell(c_{2^\prime}) < B/2\).

Let \(\tau\) be a minimal geodesic in \(Z^\prime\) from \(\partial^0Z^\prime\) to \(\partial^1Z^\prime\).
Since \(\gamma_1\) 
is a periodic geodesic, \(\tau\) intersects \(\partial^0Z^\prime\) and \(\partial^1Z^\prime\) 
perpendicularly. Cutting \(Z^\prime\) open along \(\tau\), we see that \(T\) has a rectangular fundamental region ${\cal R}$ 
in the universal covering $\tilde T$ of $T$ with 
geodesic sides of side lengths \(\ell(\gamma_1)< B/4\) and \(\ell(\tau)< B/2\), intersecting
each other perpendicularly. 

If $v\in \tilde T$ is a vertex of ${\mathcal R}$, then any point in the boundary 
$\partial {\mathcal R}$ of $\mathcal R$ 
is of distance smaller than $3B/4 $ to $v$.
As a consequence, $\partial {\mathcal R}$ is a Jordan curve embedded in 
the ball of radius $B$ about $v$. Since the injectivity radius of 
$\tilde T$ is at least $B$, this ball is diffeomorphic to a disk in $\mathbb{R}^2$.
Then $\partial {\mathcal R}$
encloses a compact disk embedded in this ball. 
On the other hand, since $\tilde T$ is diffeomorphic to $\mathbb{R}^2$, the 
disk ${\mathcal R}$ is the unique disk in $\tilde T$ bounded by 
$\partial {\mathcal R}$. Hence ${\mathcal R}$ is contained in the open disk of radius
$B$ about $v$ which yields that the diameter of ${\mathcal R}$ 
is smaller than $2B$. Consequently the diameter of $T$ is smaller than $2B< D/A$ which is a contradiction
to the assumption that $x\in N_{1/4}(M-M_{\rm small})$.

\textbf{Step 3 (Counting argument):} The main idea in this step is the following. By the result of Step 2, all preimages of \(\hat{x}\) in the universal 
covering $\pi_T:\tilde T\to T$  of the distance torus \(T\)  
that are contained in a ball of radius \(r \leq  B/4\) come from the action of \([c_1] \in \pi_1(T,\hat{x})\), and thus a volume counting argument (similar to the proof of the special case) should complete the proof.

We now make this more precise. Since the deck group of $\tilde T$ is isomorphic to $\mathbb{Z}^2$ and acts freely and isometrically, 
the union of all lifts of the simple geodesic loop $c_1$ from Step 2 above form a 
$\pi_1(T,\hat x)$-invariant countable collection ${\mathcal L}$ 
of disjoint piecewise geodesic lines in 
$\tilde T=\mathbb{R}^2$. By Step 2, 
the distance for the metric on $\tilde T$
 between any two of these lines is at least $B/4$. Furthermore, these lines contain all preimages
 of $\hat x$ in $\tilde T$.

 Now if $c_1$ is \emph{smooth}, that is, if $c_1$ does not have
a breakpoint at $\hat x$, then 
the lines in ${\mathcal L}$ are biinfinite geodesics. Since the injectivity radius of 
$\tilde T$ is at least $B$, this implies that the number of preimages of $\hat x$ which are
contained in the ball of radius $B/4$ about a fixed preimage is at most $B/4 \ell(c_1)$.
As $\ell(c_1)\geq 2{\rm inj}(x)$, we conclude that this number 
is at most $B/4\,{\rm inj}(x)$, completing the proof of the proposition in this case (note that by the same argument as in Step 1, bounds on the number of preimages in a ball of radius \(B/4\) implies bounds on the number of preimages in a ball of radius \(D\)).

In general, we can not hope that $c_1$ is smooth. We use instead 
a volume counting argument. Namely, the lines 
in the family ${\mathcal L}$ divide $\tilde T$ in a union of disjoint strips with 
boundary in ${\mathcal L}$. Let $\alpha$ be a minimal geodesic connecting 
two adjacent lines $L_1,L_2$ from ${\mathcal L}$. This is an embedded geodesic 
arc embedded in one of the strips, say the strip $S$, with endpoints on the two distinct
boundary lines $L_1,L_2$. 
The infinite cyclic subgroup of $\pi_1(T,\hat x)$ which 
is generated by the class $\varphi$ of $c_1$ preserves the lines in ${\mathcal L}$ and the strip $S$, and 
its maps $\alpha$ to a geodesic $\varphi(\alpha)$ disjoint from $\alpha$. 
Namely, if they did intersect, they intersect transversely,
and thus an elementary variational argument yields that one can find a curve connecting \(L_1\) to \(L_2\) of strictly shorter length than \(\alpha\).
This contradicts the minimality of \(\alpha\). As a consequence, the subsegments $a_i$ of $L_i$ connecting 
the endpoints of $\alpha$ and $\varphi(\alpha)$ bound together with $\alpha$ and 
$\varphi(\alpha)$ a rectangular region ${\mathcal R}$ in $\tilde T$ which is a fundamental domain
for the action of the deck group of $T$.

As we assume that $L_1,L_2$ are not smooth, each of the lines $L_1,L_2$ contains countably many
breakpoints of the same breaking angle. For an orientation of $\tilde T$ and the induced orientation of 
$L_1,L_2$ as oriented boundary of $S$, for one of the boundary lines, say the line $L_1$,
all internal angles at the breakpoints are strictly bigger than $\pi$. In other words, $L_1$ is a locally concave
boundary component of $S$. Since $\alpha$ and $\varphi(\alpha)$ are minimal geodesics 
connecting $L_1$ to $L_2$, either they meet $L_1$ at a singular point of $L_1$ and the arc $a_1$ does not
contain a singular point in its interior, or they meet $L_1$ orthogonally at a smooth point, and 
$a_1$ contains a unique singular point in its interior. In case \(\alpha\) and \(\varphi(\alpha)\) meet \(L_1\) at a singular point, the angle they form with the smooth subsegments of \(a_1\) exiting the breakpoint is at least \(\pi/2\).

We shall show that the fundamental region \(\mathcal{R}\) contains embedded rectangles of width at least \(\ell(c_1)/2\) and height \(B/4\).
We only consider the case that \(\alpha\) and \(\varphi(\alpha)\) meet \(L_1\) at a smooth point, the other case being similar (even a bit easier).
Throughout we use the fact that since the injectivity radius
of $\tilde T$ is at least \(B\) (and since \(B<D/2A<\pi/2A\)) the convexity radius is at least \(B/2\) (see \cite[Theorem 5.14]{CheegerEbin}). In particular, this implies the following.
Let \(p \in \tilde{T}\), \(\beta_1,\beta_2\) be two geodesics
segments emanating from \(p\) whose endpoints are connected by a
geodesic segment
\(c\). If $\beta_1\cup \beta_2\cup c\subset B(p,B/2)$ and if $c$
meets $\beta_1$ orthogonally, 
then the interior angle at the endpoint of $\beta_2$ of the triangle
with sides $\beta_1,\beta_2,c$ 
is strictly smaller than \(\pi/2\).

Recall from Step 1 that $\ell(a_1)=\ell(c_1)< B/4$. Let $\hat a_1:[0,\delta]\to \tilde T$ be a subsegment of 
$a_1$ of length at least $\ell(c_1)/2$ which connects the cone point
$\hat a_1(0)\in a_1\cap \pi_T^{-1}(\hat x)$ to the endpoint 
$\hat a_1(\delta)=\alpha\cap L_1$ of 
$a_1$. Let $t\to \nu(t)$ be the unit normal field along $\hat a_1$ pointing inside of the strip $S$. 
We claim that the restriction of the normal exponential 
map to  the set $\{s\nu(t)\mid 0\leq s\leq B/4,0\leq t\leq \delta\}$ is an embedding into ${\mathcal R}$.

First, observe that this is an embedding into the strip $S$.
Indeed, if two distinct orthogonal segments \(s \to \exp(s\nu(t_1))\) and \(s \to \exp(s\nu(t_2))\) $(0\leq t_1<t_2\leq \delta)$ 
intersect in a point \(p\),
then they are sides of a triangle with edge opposite to $p$ is the arc 
\(\hat{a}_1|_{[t_1,t_2]}\). This triangle is
contained in \(B(p,B/2)\) and has two right angles, contradicting convexity.

If the image intersects $S-{\mathcal R}$, then since the arc 
$\{\exp(s \nu(\delta))\mid 0\leq s\leq B/4\}$
is contained in the side $\alpha$ of ${\mathcal R}$, the arc 
$\beta:=\{\exp (s\nu(0)) \mid 0\leq s\leq B/4\}$ has to intersect the
geodesic segment $\varphi(\alpha)$ in some point \(p\) (perhaps after replacing $\varphi$ with $\varphi^{-1}$). 
Thus we obtain a triangle whose sides are the subarc of
$\beta$ connecting $\hat a_1(0)$ to $p$,
the subarc of $\phi(\alpha)$ connecting $\phi(\alpha)\cap L_1$ to $p$ and
the subarc of $a_1$ connecting $\hat a_1(0)$ to $\phi(\alpha)\cap L_1$
(see \Cref{convexity2}).
This triangle is contained in $B(p,B/2)$, and it has a right angle
at $\phi(\alpha)\cap L_1$ and an angle $\geq \pi/2$ at $\tilde x$
since the interior angle at the breakpoint $\tilde x$ is
strictly bigger than $\pi$. As before, this violates convexity.
This finishes the proof that restriction of the normal exponential 
map to the set $\{s\nu(t)\mid 0\leq s\leq B/4,0\leq t\leq \delta\}$ is an embedding into ${\mathcal R}$.

\begin{figure}
\begin{center}
\centering
	\def\svgwidth{0.75\textwidth}
\[
	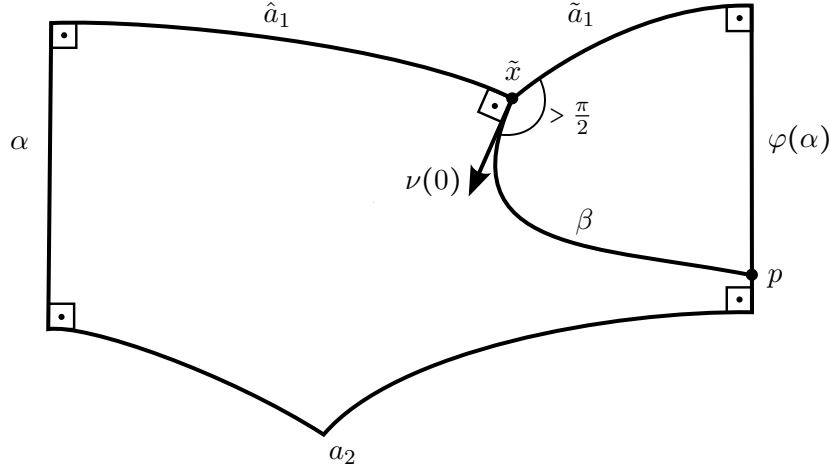
\] 
\caption{The argument by contradiction showing that the image of the normal exponential map is contained in \(\mathcal{R}\).}
\label{convexity2}
\end{center}
\end{figure}



Using once more the lower bound on the injectivity radius of $\tilde T$ and the upper
bound on the Gau\ss{} curvature, we conclude from \(\ell(\hat{a}_1)\geq \ell(c_1)/2\) that
\[
	{\rm area}(\exp\{s\nu(t)\mid 0\leq s\leq B/4, 0\leq t\leq \delta\})\geq \kappa \ell(c_1),
\]
where $\kappa >0$ is a universal constant.

Since the images of the rectangle ${\mathcal R}$ under the action of the 
deck group of $T$ have pairwise disjoint interiors, 
we conclude that the area of the $B/4$-neighborhood of the ball 
of radius $B$ about any point $z\in \pi_T^{-1}(\hat x)$ is at least 
\[
	\# \big(\pi_T^{-1}(\hat x)\cap B^{\tilde T}(z,B)\big) \cdot \kappa \ell(c_1)
      \]
      (here the notation $B^{\tilde T}(\tilde x,B)$ makes is precise
      that we take a ball in $\tilde T$).
      On the other hand, this area is bounded from above by a
  universal constant.
      Since moreover  for any $\tilde x\in \pi_{\hat M}^{-1}(\hat x)$ we have
      \(\pi_M^{-1}(\hat{x})\cap B^{\tilde{M}}(\tilde{x},B/A) \subseteq \pi_T^{-1}(\hat x)\cap B^{\tilde{T}}(\tilde x,B)\) (note that this is an abuse of notation
      since $\pi_{\hat M}^{-1}(T)$ is an infinite cylinder if
      $\Gamma$ is an infinite cyclic group of hyperbolic isometries, that is, if the component of $M_{\rm thin}$ with fundamental group $\Gamma$ is a Margulis
      tube),   
and since $2\ell(c_1)\geq {\rm inj}(x)$, this shows that 
\[
  \#\big(\pi_M^{-1}(\hat x)\cap B^{\tilde{M}}(\tilde x,B/A)\big)
  \leq \kappa^\prime/ {\rm inj}(x).
\]
Again, as in Step 1, estimates on the number of preimages in a ball of radius \(B/A\) implies bounds on the number of preimages in a
ball of radius \(D\). This completes the proof.
\end{proof}

\Cref{counting preimages II} is now an easy consequence of \Cref{Counting preimages - general case}.

\begin{proof}[Proof of \Cref{counting preimages II}]It suffices to prove the estimate for those \(x \in N_{1/4}(M\setminus M_{\rm small})\)
 with \(\inj_M(x) \leq D\). By \Cref{Counting preimages - general case}
this follows if for 
those \(x\) it holds \(\inj(x) \geq C e^{-bd(x,M_{\rm thick})}\) for a universal constant \(C>0\).

Thus let $x\in N_{1/4}(M\setminus M_{\rm small}$ and 
let \(x^\ast\) be the first point on the radial
geodesic through \(x\) that lies in \(\partial M_{\rm thick}\).
Abbreviate \(R=d(x,x^\ast)=d(x,M_{\rm thick})\).

Let \(\Gamma\) be the fundamental group of the component of \(M_{\rm thin}\) containing \(x\). Consider the intermediate cover \(\hat{M}:=M/\Gamma\), and choose lifts \(\hat{x}\) and \(\hat{x}^\ast\) of \(x\) and \(x^\ast\) with \(d(\hat{x},\hat{x}^\ast)=R\).
It holds \(\inj_{\hat{M}}(\hat{x})=\inj_M(x)\) and \(\inj_{\hat{M}}(\hat{x}^\ast)=\inj_M(x^\ast)\). Let $\sigma\subset \hat M$ be an essential based loop at
$\hat x$ of minimal length.
Radially project \(\sigma\) to a curve \(\sigma^\ast\) based at \(\hat{x}^\ast\). Note that if \(x\) is contained in a Margulis tube, then \(x\) has distance at least \(3/4\) from the core geodesic due to the definition of \(M_{\rm small}\). Hence Jacobi field comparison shows \(\ell(\sigma^\ast) \leq C e^{bR}\ell(\sigma)\) for a universal constant \(C\). As \(\inj_{\hat{M}}(\hat{x}^\ast)=\inj_M(x^\ast)=\mu\), we have \(\ell(\sigma^\ast)\geq 2\mu\), where \(\mu\) is the chosen Margulis constant for manifolds with sectional curvature contained in \([-4,-1/4]\). Thus \(Ce^{bR}\ell(\sigma)\geq 2\mu\). The definition of \(\sigma\), and the fact that \(\inj_{\hat{M}}(\hat{x})=\inj_M(x)\), imply \(Ce^{bR}\inj_M(x) \geq \mu\). Since \(R=d(x,M_{\rm thick})\), this completes the proof.
\end{proof}

\subsection{Generalisations} In \Cref{Subsec: a priori estimates in non-compact case} we will prove a global \(C^0\)-estimate in terms of a new hybrid norm. For this we will need a slightly more general version of \Cref{A-priori estimate away from the small part}. This is based on the following more general version of \Cref{Counting preimages - general case}. Recall that for \(r \geq 0\) we denote by \(N_{r}(M \setminus M_{\rm small})\) the \(r\)-neighbourhood of \(M \setminus M_{\rm small}\). 

\begin{lem}\label{Counting preimages - finite distance inside small part}For all \(\bar{R} \geq 0\) there exists a constant \(C(\bar{R})>0\) with the following property. Let \(x \in M^{\leq \mu^\prime} \cap N_{\bar{R}+1/4}(M \setminus M_{\rm small})\), and if \(x\) is contained in a Margulis tube,
assume in addition that \(N_{\bar{R}}(M \setminus M_{\rm small})\) is disjoint from the one-neighbourhood of the core geodesic. Then it holds
\[
	\# \big(\pi^{-1}(x) \cap B(\tilde{x},D) \big) \leq C(\bar{R}) \frac{1}{\mathrm{inj}(x)},
\]
where \(\pi:\tilde{M} \to M\) is the universal covering projection.
\end{lem}

\begin{proof}We quickly review the proof of \Cref{Counting preimages - general case}. We choose the constants \(A\) and \(B\) from \Cref{controlgeometry}, and we assumed without loss of generality that \(B < \frac{D}{2A}\). The reason for chosing \(\frac{D}{A}\) is the following. For all \(x \in N_{1/4}(M \setminus M_{\rm small})\) the level torus \(T\) containing \(x\) satisfies \(\diam(T)>D/A\) (see Step 2). 

We now explain how to adjust the argument from the proof of \Cref{Counting preimages - general case}. Standard Jacobi field estimates show that for any \(\bar{R} \geq 0\) there exists \(\bar{D}(\bar{R})>0\) with the following property. For any \(x\) as stated in \Cref{Counting preimages - finite distance inside small part} it holds \(\diam(T) > \bar{D}(\bar{R})\) for the level torus \(T\) containing \(x\). Choose some \(B(\bar{R}) < \min\{B(3),\frac{1}{2}\bar{D}(R)\}\). The proof of \Cref{Counting preimages - general case} goes through without change when replacing \(B\) by \(B(\bar{R})\) (and \(A=A(3)\) still given by \Cref{controlgeometry}).
\end{proof}

The next result is the generalisation of \Cref{A-priori estimate away from the small part} that we need for the global \(C^0\)-estimate in \Cref{Subsec: a priori estimates in non-compact case}. It follows from \Cref{Counting preimages - finite distance inside small part} analogous to how \Cref{A-priori estimate away from the small part} followed from \Cref{Counting preimages - general case}. We omit the details.

\begin{lem}\label{A-priori estimate with finite distance away from the small part}
For all \(\alpha \in (0,1)\), \(\Lambda \geq 0\), \(\delta \in (0,2)\), \(b > 1\), and \(\bar{R}\geq 0\) there exist \(\varepsilon_0=\varepsilon_0(\delta,b)>0\) and \(C(\bar{R})=C(\bar{R},\alpha,\Lambda,\delta,b) >0\) with the following property. Let \(M\) be a Riemannian \(3\)-manifold of finite volume so that
\[
	|\sec +1| \leq \varepsilon_0 \quad \text{and} \quad ||\nabla \Ric||_{C^0(M)} \leq \Lambda.
\]
Let \(x \in N_{\bar{R}}(M \setminus M_{\rm small})\), and if \(x\) is contained in a Margulis tube, assume in addition that \(N_{\bar{R}}(M \setminus M_{\rm small})\) is disjoint from the one-neighbourhood of the core geodesic. Then for all \(h \in C^{2}\big( {\rm Sym}^2(T^*M)\big)\cap H^2(M)\) it holds
\[
	|h|(x) \leq C(\bar{R})\left(||\mathcal{L}h||_{C^0(M)}+e^{\frac{b}{2}d(x,M_{\rm thick})}\left(\int_M e^{-(2-\delta)r_x(y)}|\mathcal{L}h|^2(y) \, d\vol(y) \right)^\frac{1}{2} \right).
\]
\end{lem}


\section{Model metrics in tubes and cusps}\label{Sec:model metrics}


The examples of \Cref{Sec: Counterexamples} show that \Cref{Pinching with inj radius bound - introduction} no longer holds 
true without the assumption of a uniform lower bound on the injectivity radius. 
At the end of \Cref{Sec: Counterexamples} we pointed out that the counterexamples model a deformation of hyperbolic structures
on a fixed Margulis tube, obtained  
by slowly changing the conformal structure of the tori \(T(r)\) contained in \(M_{\rm small}\). 
The goal of this section is to formulate a geometric condition for the 
tubes and cusps, controlled asymptotic hyperbolicity, which rules out such examples. 
This section can be skipped by readers who are mainly interested in the applications to 
drilling, filling and hyperbolization. For these applications, it suffices to consider metrics which have constant
curvature in the thin parts of the manifold. 

Asymptotically hyperbolic metrics on non-compact manifolds have been widely studied in the literature, however
mainly in the context of manifolds with flaring ends. We refer to \cite{HQS12} for an overview of some related results.

In the sequel, $\eta >1$ is a constant fixed once and for all. 
Let \(M\) be a complete Riemannian \(3\)-manifold of finite volume that satisfies the following 
curvature decay condition:
\begin{equation}\label{curvaturedecay}
	\max_{\pi \subseteq T_xM}|\mathrm{sec}(\pi)+1|, \, |\nabla R|(x), \, |\nabla^2R|(x) \leq \varepsilon_0 e^{-\eta d(x,\partial M_{\rm small})} \quad \text{for all } \, x \in M_{\rm small}.
      \end{equation}
Here \(R\) denotes the Riemann curvature endomorphism.      
           
As before, we know that all cusps are diffeomorphic to $T^2\times [0,\infty)$ since by \Cref{convention orientable} we assume that \(M\) is orientable.
We construct in this section a hyperbolic model metric in the small part of 
cusps and the complements of the $1$-neigbhborhood of the core curves of the small part of tubes. 
These auxiliary metrics are used in \Cref{Subsec: various norms} to construct Banach spaces geared at controlling
solutions of the equation ${\mathcal L}(h)=f$ in the small part of $M$.
 

Let as before $T^2$ be a two-torus. 
Call a metric \(g\) on \(T^2 \times I\) (where \(I\) is an interval) a \textit{cusp metric} if it is of the form
\[
	g=e^{-2r}g_{Flat}+dr^2,
\]
where \(g_{Flat}\) is some flat metric on \(T^2\) and \(r\) is the \(I\)-coordinate. Let \(T\) be a Margulis tube and \(C\) a rank 2 cusp of \(M\). 
Note that \(C_{\rm small} \cong T^2 \times [0,\infty)\) and \(T_{\rm small} \setminus N_{1}(\gamma) \cong T^2 \times [0,R-1]\), where \(R\) is the 
\emph{radius} of \(T_{\rm small}\), that is, the distance of the boundary of $T_{\rm small}$ to the core curve of $T$, 
and \(\cong\) stands for diffeomorphic.
The given metric on \(M\) will in general not be a cusp metric on these sets. 

The following two statements are the main results of this section.

\begin{prop}\label{tubes are almost cusps - pinched curvature}For any \(\eta > 1\) there exists \(\varepsilon_0=\varepsilon_0(\eta)>0\) with the following property. 
Let \(M\) be an Riemannian 3-manifold satisfying the curvature decay condition (\ref{curvaturedecay}) and let \(T\) be a Margulis tube of \(M\) with core geodesic \(\gamma\). Then there exists a cusp metric \(g_{cusp}\) on \(T_{\rm small}\setminus N_{1}(\gamma)\) so that for all \(x \in T_{\rm small}\setminus N_{1}(\gamma)\) it holds
\[
	|g-g_{cusp}|_{C^2}(x)=O\big(e^{-2r_{\gamma}(x)}+\varepsilon_0 e^{-\eta r_{\partial T}(x)}\big),
\]
where \(r_{\partial T}(x)=d(x, \partial T_{\rm small})\), and \(r_{\gamma}(x)=d(x,\gamma)\).
\end{prop}

See \Cref{big O notation} for our convention of the \(O\)-notation. For cusps we have a slightly better estimate.

\begin{prop}\label{Existence of approximate cusp metric}For any \(\eta > 1\) there exists \(\varepsilon_0=\varepsilon_0(\eta)>0\) with the following property. 
Let \(M\) be a Riemannian 3-manifold satisfying 
the curvature decay condition (\ref{curvaturedecay}) and let \(C\) be a rank 2 cusp of \(M\). Then there exists a cusp metric \(g_{cusp}\) on \(C_{\rm small}\) so that for all \(x \in C_{\rm small}\) it holds
\[
	|g-g_{cusp}|_{C^2}(x)=O(\varepsilon_0 e^{-\eta r(x)}),
\]
where \(r(x)=d(x,\partial C_{\rm small})\).
\end{prop}

The remainder of this section is devoted to the proof for \Cref{tubes are almost cusps - pinched curvature} and \Cref{Existence of approximate cusp metric}. The main idea for the proof is to compare the Jacobi equation in \(M\) with the one in the comparison space \(\bar{M}=\bbH^3\). To do so we require the following stability estimate for linear ODEs. 

\begin{lem}\label{Stability of ODEs}Let \(A,\bar{A}:[0,T] \to {\rm End}(\bbR^n)\) and \(b,\bar{b}:[0,T] \to \bbR^n\) be continuous, and assume that the following conditions are satisfied:
\begin{enumerate}[i)]
\item \(||A(t)||_{\rm op}\leq a\) and \(||\bar{A}(t)||_{\rm op}\leq \bar{a}\) for all \(t \in [0,T]\);
\item \(||A(t)-\bar{A}(t)||_{\rm op}=O\big(\varepsilon e^{\eta (t-T)} \big)\) for some \(\eta > a - \bar{a}\);
\item \(|\bar{b}(t)|=O\big(\bar{\beta} e^{\bar{\mu}t}\big)\) for some \(\bar{\mu} > \max\{a,\bar{a}\}\) and \(\bar{\beta} \geq 0\);
\item \(|b(t)-\bar{b}(t)|=O\big(\beta e^{\mu t} \big)\) for some \(\mu > a\) and \(\beta \geq 0\).
\end{enumerate}
Then the solutions \(y,\bar{y}:[0,T] \to \bbR^n\) of the ODEs
\[
	y^\prime(t)=A(t)y(t)+b(t) \quad \text{and} \quad \bar{y}^\prime(t)=\bar{A}(t)\bar{y}(t)+\bar{b}(t)
\]
with initial conditions \(y(0)=y_0\) and \(\bar{y}(0)=\bar{y}_0\) satisfy
\[
	|\bar{y}(t)-y(t)|=O\Big(|\bar{y}_0-y_0|e^{a t}+\varepsilon |\bar{y}_0|e^{\bar{a} t}e^{\eta (t-T)}+\varepsilon \bar{\beta} e^{\bar{\mu} t}e^{\eta (t-T)}+\beta e^{\mu t} \Big).
\]
\end{lem}

The same estimates hold for second order linear ODEs \(v^{\prime \prime}(t)=R(t)v(t)\) if \(\bar{y}_0\) is replaced by \(|\bar{v}(0)|+|\bar{v}^\prime(0)|\) (similarly for \(|\bar{y}_0 - y_0|\)), and \(a\) is replaced by \(\max\{1,\max_t ||R(t)||_{\rm op} \}\) (similarly for \(\bar{a}\)). Indeed, substituting \(y=(v,v^\prime)\) the second order ODE is equivalent to
\[
	y^\prime(t)=
	\begin{pmatrix}
	0 & {\rm id}_{\bbR^n} \\
	R(t) & 0
	\end{pmatrix}
	y(t)	
	\quad \text{and} \quad y(0)=(v(0),v^\prime(0)),
\]
and it holds
\[
	\left| \left|\begin{pmatrix}
	0 & {\rm id}_{\bbR^n} \\
	R(t) & 0
	\end{pmatrix} \right| \right|_{\rm op}=\max\{1,||R(t)||_{\rm op}\}.
\]

\begin{proof}Consider a linear ODE
\[
	\chi^\prime(t)=\Sigma(t)\chi(t)+\xi(t),
\]
and assume 
\[
	\max_t||\Sigma(t)||_{\rm op} \leq \sigma \quad \text{ and } \quad |\xi(t)|\leq \sum_{i}\kappa_i e^{\lambda_i t} \, \text{ for some } \lambda_i> \sigma \text{ and }\kappa_i \geq 0.
\]
Then it holds 
\begin{equation}\label{ODE stability 1}
	|\chi(t)| \leq  |\chi(0)|e^{\sigma t}+ \sum_{i}(\lambda_i-\sigma)^{-1}\kappa_i e^{\lambda_i t}.
\end{equation}
Indeed, this is a straightforward consequence of inequality (4.9) on page 56 of \cite{ODEHartman}.

Compute 
\[
	(\bar{y}-y)^\prime(t)=A(t)(\bar{y}-y)(t)+(\bar{A}-A)(t)\bar{y}(t)+(\bar{b}-b)(t).
\]
As \(\bar{\mu} > \bar{a}\) we may apply (\ref{ODE stability 1}) to \(\bar{y}^\prime=\bar{A}\bar{y}+\bar{b}\) to obtain \(|\bar{y}|(t)=\big(|\bar{y}_0|e^{\bar{a}t}+\bar{\beta}e^{\bar{\mu}t} \big)\). Hence \(|(\bar{A}-A)\bar{y}|(t)=O\big(\varepsilon |\bar{y}_0|e^{\bar{a}t}e^{\eta (t-T)}+\varepsilon \bar{\beta} e^{\bar{\mu}t}e^{\eta (t-T)} \big)\) due to condition \(ii)\). Since \(\bar{a}+\eta,\bar{\mu}+\eta,\mu > a\) we can apply (\ref{ODE stability 1}) to the ODE satisfied by \(\bar{y}-y\). Therefore,
\[
	|\bar{y}-y|(t)=O\Big(|\bar{y}_0-y_0|e^{at}+\varepsilon |\bar{y}_0|e^{\bar{a} t}e^{\eta (t-T)}+\varepsilon \bar{\beta} e^{\bar{\mu}t}e^{\eta (t-T)}+\beta e^{\mu t} \Big).
\]
This completes the proof.
\end{proof}

We now come to the construction of \(g_{cusp}\) on \(T_{\rm small} \setminus N_{1}(\gamma)\). 
As an intermediate step we first construct another metric \(g_{tube}\) on \(T_{\rm small}\). 

Let $\gamma$ be the core curve of $T$ and 
let \(\tilde{\gamma} \subseteq \tilde{M}\) be a lift of \(\gamma\). Denote by 
\(\varphi:\tilde{M} \to \tilde{M}\) the deck transformation corresponding to \([\gamma] \in \pi_1(M)\) which preserves \(\tilde{\gamma}\) and acts on it as a translation. 
To the element $\varphi$ we associate its \emph{translation length} which is the length of $\gamma$, and 
the \emph{rotation angle}, defined by parallel transport of the orthogonal complement of $\gamma^\prime$ in $TM\vert \gamma$.  
Let \(\beta \subseteq \mathbb{H}^3\) be a geodesic, 
and let \(\psi:\mathbb{H}^3 \to \mathbb{H}^3\) be an orientation preserving loxodromic isometry, with axis $\beta$ and the 
same translation length and the same 
rotation angle as \(\varphi\). 

Define \(\hat{M}:=\tilde{M}/\langle \varphi \rangle\) and \(\hat{\mathbb{H}}^3:=\mathbb{H}^3/\langle \psi \rangle\).
Using the normal exponential maps for $\tilde \gamma$ in $\tilde M$ and for $\beta$ in 
$\mathbb{H}^3$, we see that there is 
a diffeomorphism \(\hat{M} \supseteq N_R(\hat{\gamma})  \xrightarrow{\cong} N_R(\hat{\beta}) \subseteq \mathbb{H}^3/ \langle \psi \rangle\) of the full distance tori. The projection \(\hat{M} \to M\) also induces a diffeomorphism \(\hat{M} \supseteq N_{R}(\hat{\gamma}) \xrightarrow{\cong} N_{R}(\gamma) \subseteq M\) when \(R\) is the radius of \(T_{\rm small}\) because \(T_{\rm small} \subseteq M_{\rm thin}\) by \Cref{containedinthin}. Note that \(N_R(\gamma)=T_{\rm small}\) by the definition of the radius \(R\) of \(T_{\rm small}\). The \textit{tube metric} \(g_{tube}\) on \(T_{\rm small}\) is the pullback of the hyperbolic metric on \(N_R(\hat{\beta})\) via the diffeomorphism \(T_{\rm small} \xrightarrow{\cong} N_R(\hat{\beta}) \subseteq \hat{\mathbb{H}}^3\).

The \textit{cusp metric} \(g_{cusp}\) on \(T_{\rm small} \setminus N_1(\gamma)\cong \partial T_{\rm small} \times [0,R-1]\) 
is the metric 
\[g_{cusp}=e^{-2r}g_{Flat}+dr^2,\] where \(g_{Flat}\) is the flat metric on \(\partial T_{\rm small}\) induced 
by the tube metric \(g_{tube}\). One can easily check by explicit calculations that
\begin{equation}\label{tubes are almost cusps}
	|g_{tube}-g_{cusp}|_{C^2}(x)=O\big(e^{-2r_{\gamma}(x)} \big)
\end{equation}
for all \(x \in T_{\rm small} \setminus N_1(\gamma)\). 

We next verify that $g_{cusp}$ has the properties stated in \Cref{tubes are almost cusps - pinched curvature}.


\begin{proof}[Proof of \Cref{tubes are almost cusps - pinched curvature}]
By (\ref{tubes are almost cusps}), it suffices to show that 
\begin{equation}\label{approximation by tube metric}
|g-g_{tube}|_{C^2}(x)=O\big(\varepsilon_0 e^{-\eta r_{\partial T}(x)}\big)
\end{equation}
for all \(x \in T_{\rm small}\setminus N_1(\gamma)\). Let \(\gamma\) be the core geodesic of the tube \(T\). It suffices to prove the estimates in the universal cover. 
Let \(\tilde{\gamma} \subseteq \tilde{M}\) be a lift of \(\gamma\). Choose parallel unit vector fields 
 \(\nu_1,\nu_2\) along \(\tilde{\gamma}\) so that \(\tilde{\gamma}^\prime,\nu_1,\nu_2\) is a positively oriented orthonormal frame along \(\tilde{\gamma}\). Define a map \(\varphi:\bbR \times \bbR_{\geq 0} \times \bbR \to \tilde{M}\) by
\[
	\varphi(s,t,\theta):=\exp_{\tilde{\gamma}(s)}\big(t\nu_\theta(s) \big),
\] 
where \(\nu_\theta(s):=\cos(\theta)\nu_1(s)+\sin(\theta)\nu_2(s)\). We think of \(\varphi\) as a \(2\)-dimensional variation of geodesics. The main idea is that the variational fields of \(\varphi\) solve Jacobi equations, and
so one can use \Cref{Stability of ODEs} to compare the situation with the comparison space \(\bbH^3\).

We fix some notation. Let \(V\) be a vector field along \(\varphi\), i.e., a map \(V:\bbR \times \bbR_{\geq 0}\times \bbR \to T\tilde{M}\) so that \(\pi \circ V=\varphi\). 
Denote by \(D_tV\) the vector field whose value at \((s_0,t_0,\theta_0)\) equals the covariant derivative of \(V\) along the curve \(t \to \varphi(s_0,t,\theta_0)\) at \(t=t_0\). The vector fields \(D_s V\) and \(D_\theta V\) are defined analogously. For example, for \(t=0\) the vector \(D_\theta V(s_0,0,\theta_0)\) is just the usual derivative of the curve \(\theta \to V(s_0,0,\theta)\) in \(T_{\tilde{\gamma}(s_0)}\tilde{M}\). We will also write \(()^\prime\) for \(D_t\).

Fix \(s_0\) and \(\theta_0\), and consider the geodesic \(\sigma(t)=\varphi(s_0,t,\theta_0)\) 
with \(\sigma(0)=\tilde{\gamma}(s_0)\) and \(\sigma^\prime(0)=\nu_{\theta_0}(s_0)\). Let \(E_1,E_2,E_3\) be a parallel orthonormal frame 
along \(\sigma\) so that \(E_1(0)=\tilde{\gamma}^\prime(s_0)\), \(E_2(t)=\sigma^\prime(t)\), and \(E_3(0)=\nu_{\theta_0}^\perp(s_0):=-\sin(\theta_0)\nu_1(s_0)+\cos(\theta_0)\nu_2(s_0)\). 

Let \(i\) be either the \(s\)- or the \(\theta\)-coordinate. The restriction of the variational field \(J_i:=\partial_i\varphi\) to \(\sigma\) is a Jacobi field, that is, it solves the Jacobi equation
\[
	J_i^{\prime \prime}(t)+R(t)J_i(t)=0,
\]
where \(R(t)=R(\, \cdot \, ,\sigma^\prime(t))\sigma^\prime(t)\) and \(J_i^{\prime \prime}=D_tD_tJ_i\).

We do the same set-up in the comparison space \(\bar{M}=\bbH^3\). Using the orthonormal frames \((E_i)_{i=1}^3\) we can think of \(J_i\) resp. \(R\) as a curve resp. a family of symmetric matrices in \(\bbR^3\). Similarly, \((\bar{E}_i)_{i=1}^3\) can be used to think of \(\bar{J}_i\) resp. \(\bar{R}\) as a curve resp. a family of symmetric matrices in \(\bbR^3\). Hence it makes sense to write \(J_i(t)-\bar{J}_i(t)\) and \(R(t)-\bar{R}(t)\). The curvature decay condition (\ref{curvaturedecay}) translates to (for \(t \in [0,R]\))
\[
	||R(t)-\bar{R}(t)||_{\rm op}\leq \varepsilon_0e^{\eta (t-R)},
\]
where \(R\) is the radius of \(T_{\rm small}\), i.e., the distance of the core geodesic \(\gamma\) to \(\partial T_{\rm small}\). Observe that \(\bar{a}:=\max\{1,\max_t ||\bar{R}(t)||_{\rm op}\}=1\) and \(a:=\max\{1,\max_t ||R(t)||_{\rm op}\} \leq 1+\varepsilon_0\). So condition \(ii)\) of \Cref{Stability of ODEs} is satisfied if \(\varepsilon_0 < \eta\). 

Note that \(J_i\) and \(\bar{J}_i\) have the same initial conditions. Therefore, invoking \Cref{Stability of ODEs} with \(\bar{\beta}=\beta=0\) yields (for \(t \in [0,R]\))
\begin{equation}\label{variational field estimate - 1st order}
	|J_i(t)-\bar{J}_i(t)|, |J_i^\prime(t)-\bar{J}_i^\prime(t)|=O\big(\varepsilon_0 e^te^{\eta(t-R)}\big).
\end{equation}
Note \(|\bar{J}_i|(t), |\bar{J}_i^\prime|(t)=O( e^t )\) due to (\ref{ODE stability 1}), and by (\ref{variational field estimate - 1st order}) the same holds for \(J_i\).

Let \(j\) be either the \(s\)- or the \(\theta\)-coordinate. Consider the variational fields \(D_jJ_i\) restricted to \(\sigma\). A straightforward calculation shows that \(D_jJ_i\) solves a inhomogeneous Jacobi equation, that is, 
\[
	(D_jJ_i)^{\prime \prime}+R(t)D_jJ_i(t)=b(t)
\]
for a vector field \(b\) along \(\sigma\). In fact, one can show
\begin{align*}
	b(t)=&-(\nabla R)(J_j,\sigma^\prime,J_i,\sigma^\prime)-(\nabla R)(J_i,\sigma^\prime,\sigma^\prime,J_j) \\
	&-R(J_j^\prime,\sigma^\prime)J_i-2R(J_j,\sigma^\prime)J_i^\prime-R(J_i,J_j^\prime)\sigma^\prime-R(J_i,\sigma^\prime)J_j^\prime.
\end{align*}
The analogous statements hold in the comparison space. We again use the parallel orthonormal frames \((E_i)_{i=1}^3\) and \((\bar{E}_i)_{i=1}^3\) to view \(D_jJ_i\), \(\bar{D}_j\bar{J}_i\), \(b\), and \(\bar{b}\) as curves in \(\bbR^3\). Then (\ref{curvaturedecay}), (\ref{variational field estimate - 1st order}) and the growth estimates \(|J_i^{(\prime)}|(t),|\bar{J}_i^{(\prime)}|(t)=O(e^t)\) can be used to show \(|\bar{b}(t)|=O(e^{2t})\) and \(|b(t)-\bar{b}(t)|=O\big(\varepsilon_0 e^{2t}e^{\eta (t-R)}\big)\). One also computes \(D_jJ_i(0)=0=\bar{D}_j\bar{J}_i(0)\), and \(|(D_jJ_i)^\prime(0)-(\bar{D}_j\bar{J}_i)^\prime(0)|=O(e^{- \eta R})\) by using the curvature decay condition (\ref{curvaturedecay}). Then one can again use \Cref{Stability of ODEs} to obtain
\begin{equation}\label{variational field estimate - 2nd order}
	|D_jJ_i(t)-\bar{D}_j\bar{J}_i(t)|, |(D_jJ_i)^\prime(t)-(\bar{D}_j\bar{J}_i)^\prime(t)|=O\big(\varepsilon_0 e^{2t}e^{\eta(t-R)}\big).
\end{equation}
Again note that \(|\bar{D}_j\bar{J}_i|(t),|(\bar{D}_j\bar{J}_i)^\prime|(t)=O(e^{2t})\) due to (\ref{ODE stability 1}) and \(|\bar{b}|(t)=O(e^{2t})\). By (\ref{variational field estimate - 2nd order}) the same estimate holds for \(D_jJ_i\).

Finally, let \(k\) be either the \(s\)- or \(\theta\)-coordinate. Using arguments similar as for \(D_jJ_i\) (that is, inhomogeneous Jacobi equation, and \Cref{Stability of ODEs}) one can show
\begin{equation}\label{variational field estimate - 3rd order}
	|D_kD_jJ_i(t)-\bar{D}_k\bar{D}_j\bar{J}_i(t)|, |(D_kD_jJ_i)^\prime(t)-(\bar{D}_k\bar{D}_j\bar{J}_i)^\prime(t)|=O\big(\varepsilon_0 e^{3t}e^{\eta(t-R)}\big).
\end{equation}

The estimates (\ref{variational field estimate - 1st order}), (\ref{variational field estimate - 2nd order}), and (\ref{variational field estimate - 3rd order}) imply the desired estimate on \(|g-g_{tube}|_{C^2}\) in (\ref{approximation by tube metric}). Indeed, for \(m > 0\) define \(c:\{s,t,\theta\}^m \to \bbN\) as \(c(i_1,...,i_m):=\# \{ u \, | \, i_u \neq t\}\). Then it follows from (\ref{variational field estimate - 1st order}), (\ref{variational field estimate - 2nd order}), and (\ref{variational field estimate - 3rd order}) that
\begin{align}
	\quad \quad \quad \quad |g_{ij}-\bar{g}_{ij}|&=O\big(\varepsilon_0 e^{c(i,j)t}e^{\eta (t-R)}\big) \label{partial derivatives of difference of metrics - 0th order} \\
	\quad \quad |\partial_k g_{ij}-\partial _k \bar{g}_{ij}|&=O\big(\varepsilon_0 e^{c(i,j,k)t}e^{\eta (t-R)}\big) \label{partial derivatives of difference of metrics - 1st order} \\
		\, \, |\partial_l\partial_k g_{ij}-\partial_l\partial _k \bar{g}_{ij}|&=O\big(\varepsilon_0 e^{c(i,j,k,l)t}e^{\eta (t-R)}\big)  \label{partial derivatives of difference of metrics - 2nd order}
\end{align}
for all \(i,j,k,l \in \{s,t,\theta\}\). Note 
\(
	\bar{g}=\cosh^2(t)ds^2+dt^2+\sinh^2(t)d\theta^2
\),
and so in particular the matrix \((\bar{g}_{ij})\) is diagonalised. So for any \((0,m)\)-tensor \(T\) 
\[
	|T|_{\bar{g}}^2=\sum_{i_1,...,i_m}(T_{i_1...i_m})^2(\bar{g}_{i_1i_1})^{-1}\cdot ... \cdot (\bar{g}_{i_mi_m})^{-1}.
\]
For \(t \geq 1\), \(\cosh(t)\) and \(\sinh(t)\) agree with \(e^t\) up to a uniform multiplicative constant. So \((\bar{g}_{i_1i_1})^{-1}\cdot ... \cdot (\bar{g}_{i_mi_m})^{-1}=O\big(e^{-2c(i_1,...,i_m)t}\big)\), and thus (\ref{partial derivatives of difference of metrics - 0th order}), (\ref{partial derivatives of difference of metrics - 1st order}), and (\ref{partial derivatives of difference of metrics - 2nd order}) imply (\ref{approximation by tube metric}).
\end{proof}

We now come to the sketch of proof for \Cref{Existence of approximate cusp metric}. 
This is a bit more involved than the case of a tube. The main idea is to pull back the conformal structure of the distance tori at infinity for the construction of the cusp metric
and use the fact that up to scale, a flat metric on $T^2$ is determined by the conformal structure it defines.
To produce a cusp metric that is close to the given metric, 
we establish an effective version of the Uniformization Theorem. 

The classical Uniformization Theorem states 
that for any Riemannian metric \(g\) on the two torus \(T^2\), there exists a flat metric \(\bar{g}\) on \(T^2\) and a function \(\rho:T^2 \to \bbR\) so that \(g=e^{\rho}\bar{g}\). The flat metric is unique only up to a multiplicative constant, and hence \(\rho\) is only defined up to an additive constant. The following definition should be thought of as the choice of a canonical flat metric in the conformal class of \(g\).

\begin{definition}\label{def - associated flat metric} Let \(g\) be a metric on \(T^2\). The \textit{associated flat metric} \(g_{Flat}\) is the unique flat metric conformal to \(g\) 
such that the corresponding function \(\rho\) satisfies
\[
	\int_{T^2} \rho \, d\mathrm{vol}_g=0.
\]
\end{definition}

We are now in a position to state the \textit{effective Uniformization Theorem}. It says that if the given metric is "almost" a flat metric, 
then \(\rho\) is small. Here the function \(\rho\) is understood to define the associated flat metric.

\begin{lem}\label{effective uniformization C^0} There exist constants \(\delta_0>0\) and \(C>0\) with the following property. If \(g\) is a metric on \(T^2\) such that
\[
	\mathrm{diam}(T^2,g) \leq 1 \quad \text{and} \quad |\sec(g)| \leq \delta
\]
for some \(\delta \leq \delta_0\), then
\[
	|\rho(x)| \leq C\delta \quad \text{for all }x \in T^2.
\]
\end{lem}

We postpone the proof of \Cref{effective uniformization C^0}, and show next 
how to use \Cref{effective uniformization C^0} to obtain \Cref{Existence of approximate cusp metric}. 

\begin{proof}[Proof of \Cref{Existence of approximate cusp metric}]
For each \(r \geq 0\) we denote by \(T(r)\) the torus in \(C_{\rm small}\) all of whose points have distance \(r\) to \(\partial C_{\rm small}\). For the moment fix \(r_0>0\). Note that by the curvature decay condition (\ref{curvature decay}) it holds \(|\sec+1|,|\nabla R|,|\nabla^2 R| \leq \varepsilon_0 e^{-\eta r_0}\) for all points \(x\) that lie further inside \(C_{\rm small}\) than \(T(r_0)\). The results of \cite{Shcherbakov1983} show that under this condition, 
a Busemann function associated to the rank 2 cusp \(C\) is of class \(C^4\), with controlled norm of the derivatives, 
and the shape operator \(\mathcal{H}_{r_0}\) of the horotorus \(T(r_0)\) satisfies \(|\mathcal{H}_{r_0}-{\rm id}|_{C^2} = O\big(\varepsilon_0 e^{-\eta r_0}\big)\). This implies that \(|K|_{C^2} =O\big(\varepsilon_0 e^{-\eta r_0}\big)\) for the Gau\ss{} curvature \(K\) of \(T(r_0)\). As \(T(r_0) \subseteq C_{\rm small}\) we can invoke \Cref{effective uniformization C^0} to conclude \(|\rho|=O(\varepsilon_0 e^{-\eta r_0})\), where \(\rho\) on \(T(r_0)\) is given by \Cref{def - associated flat metric}. By Exercise 2 in Chapter 4.3 of \cite{doCarmoSurfaces} it holds \(\Delta \rho=2K\), where \(\Delta\) is the Laplace operator
with respect to the restriction
of the metric $g$ to \(T(r_0)\).
Therefore, Schauder estimates imply \(||\rho||_{C^2} = O\big(\varepsilon_0 e^{-\eta r_0}\big)\). 

Let \(g_{Flat}^{(r_0)}\) be the flat metric on \(T(r_0)\)
defined by $e^\rho g_{Flat}^{(r_0)}=g\vert T(r_0)$
(see \Cref{def - associated flat metric}).
Let \(\psi:\bbR^2 \to T(r_0)\) be the universal Riemannian covering map for
 \(g_{Flat}^{(r_0)}\). Define \(\varphi:\bbR^2 \times [0,r_0] \to C_{\rm small}\) by
\[
	\varphi(x,t):=\exp_{\psi(x)}(-t\partial_r),
\]
where \(\partial_r\) is the radial vector field in $C_{\rm small}$ 
pointing to infinity. Fix \(x_0 \in \bbR^2\), and consider the geodesic \(\sigma(t)=\varphi(x_0,t)\). 

As in the proof of \Cref{tubes are almost cusps - pinched curvature} consider the variational fields \(J_i:=\partial_{x^i}\varphi\), \(D_jJ_i\), and \(D_kD_jJ_i\) along \(\sigma\) (for the notation see the proof of \Cref{tubes are almost cusps - pinched curvature}). These satisfy (in)homogeneous Jacobi equations.
We claim that 
the initial conditions \(J_i^{(\prime)}(0)\), \((D_jJ_i)^{(\prime)}(0)\), \((D_kD_jJ_i)^{(\prime)}(0)\) are \(\varepsilon_0 e^{-\eta r_0}\)-close to those in the comparison space \(\bar{M}=\bbH^3\),
where we write \(V^{(\prime)}\) to denote \(V\) or \(V^\prime\).
We show this for \(J_i, J_i^\prime\) and \(D_jJ_i\), the other cases being similar.

As in the proof of \Cref{tubes are almost cusps - pinched curvature} we use a parallel orthonormal frame \((E_i)_{i=1}^3\) along \(\sigma\) to view all vector fields as curves in \(\bbR^3\), and all tensors as a family of tensors on \(\bbR^3\). Choose the parallel orhonormal frame so that \(E_1(0)=e^{-\rho/2}\frac{\partial}{\partial \psi^1}\), \(E_2(0)=e^{-\rho/2}\frac{\partial}{\partial \psi^2}\) and \(E_3(0)=\sigma^\prime(0)\), so that \(J_i(0)=e^{\rho/2}E_i(0)\) for \(i=1,2\). In the comparison space \(\bar{M}=\bbH^3\) it holds \(\bar{J}_i(0)=\bar{E}_i(0)\). Using \(|e^{z}-1| \leq 2|z|\) for \(|z|\) small, we obtain \(|J_i-\bar{J}_i|(0) \leq 2 |\rho|=O(e^{-\eta r_0})\). Moreover, \(J_i^\prime(0)=\nabla_{J_i(0)}\nu=\mathcal{H}_{\nu}\big(J_i(0)\big)\), where \(\mathcal{H}_{\nu}\) is the shape operator of \(T(r_0)\)
with respect to the unit normal \(\nu:=-\partial_r\).
Now the shape operator $\bar{\mathcal{H}}$ of horospheres in
$\bbH^3$ is the identity and therefore  
\(|\mathcal{H}-\bar{\mathcal{H}}|=O(e^{-\eta r_0})\)and \(|J_i^\prime-\bar{J}_i^\prime|(0)=O(e^{- \eta r_0})\).
To see that \(D_jJ_i(0)\) is close to \(\bar{D}_j\bar{J}_i(0)\) observe that 
\[
	D_jJ_i(0)={\rm I\!I}_{\nu}\big(J_i(0),J_j(0)\big)+\nabla^{T(r_0)}_{J_i(0)}J_j(0),
\]
where \(\nabla^{T(r_0)}\) is the Levi-Civita connection of \(T(r_0)\).
Note that \(J_i(0)=\partial_{\psi^i}\). The Christoffel symbol \(\Gamma_{ji}^k\) of \(\nabla^{T(r_0)}\) is \(\frac{1}{2}\big(\delta_{kj}\partial_i\rho+\delta_{ik}\partial_j\rho+\delta_{ij}\partial_k \rho \big)\) because \(g=e^{\rho} \psi_{\ast}g_{\bbR^2}\) on \(T(r_0)\). Therefore, the estimates on the shape operator and \(||\rho||_{C^2}\) imply \(|\bar{D}_j\bar{J}_i-D_jJ_i|(0)=O(e^{-\eta r_0})\).

As in the proof of \Cref{tubes are almost cusps - pinched curvature} the curvature decay condition (\ref{curvaturedecay}) reads (for \(t \in [0,r_0]\))
\[
	|R-\bar{R}|(t),\, |\nabla R-\bar{\nabla}\bar{R}|(t),\, |\nabla^2 R-\bar{\nabla}^2\bar{R}|(t) \leq \varepsilon_0 e^{\eta(t-r_0)}.
\]
Since all initial conditions are \(e^{-\eta r_0}\)-close to the ones in the comparison situation, one can, exactly as in the proof of \Cref{tubes are almost cusps - pinched curvature}, iteratively use \Cref{Stability of ODEs} to conclude that the metric \(g\) is close to the metric \(\bar{g}\) in the comparison situation (with error \(O\big(\varepsilon_0 e^{\eta (t-r_0)}\big)\)). Note that in \(\bbH^3\) the hyperbolic metric is of the form \(\bar{g}=e^{-2r}g_{Flat}+dr^2\) for a flat metric \(g_{Flat}\) on a reference horosphere.
As \(d(\sigma(t),\partial C_{\rm small})=r_0-t\) we have
\[
	\big|g-g_{cusp}^{(r_0)}\big|_{C^2}(x)=O\big( \varepsilon_0e^{-\eta r(x)}\big) \quad \text{whenever }\, r(x) \leq r_0,
\] 
where \(r(x)=d(x,\partial C_{\rm small})\), and \(g_{cusp}^{(r_0)}:=e^{-2(r-r_0)}g_{Flat}^{(r_0)}+dr^2\). 

Now choose a sequence \(r_0 \to \infty\). After passing to a subsequence we may assume \(g_{cusp}^{(r_0)} \to g_{cusp}\) in the pointed \(C^2\)-topology. This limit is the desired cusp metric.
\end{proof}

It remains to prove \Cref{effective uniformization C^0}. 
For its proof we will need the fact that for metrics \(g\) on \(T^2\) as stated in \Cref{effective uniformization C^0}, the injectivity radius of the 
universal cover \((\tilde{T}^2,\tilde{g})\) is bounded from below by a universal constant. This follows from the next lemma.

\begin{lem}\label{inj radius bound for universal cover}Let \(v > 0\), and let \(g\) be a Riemannian metric on \(T^2\) so that
\[
	\vol_g(T^2) \leq v \quad \text{and} \quad \sec(g) \leq \delta 
\]
for some \( 0 < \delta \leq \frac{2\pi}{v}\). Then
\(
	\inj(\tilde{T}^2,\tilde{g}) \geq \frac{\pi}{\sqrt{\delta}}.
\)
\end{lem}

In fact, the same proof applies to all closed orientable surfaces.

\begin{proof}Assume \(\inj(\tilde{T}^2,\tilde{g}) < \frac{\pi}{\sqrt{\delta}}\). Choose \(\tilde{x}_0 \in \tilde{T}^2\) so that \(\inj_{\tilde{g}}(\tilde{x}_0)=\inj(\tilde{T}^2,\tilde{g})\). 
By the curvature assumption, the conjugate radius is not smaller than  \(\frac{\pi}{\sqrt{\delta}}\). 
Thus there exists a periodic geodesic \(\tilde{\gamma}\) through \(\tilde{x}_0\) with \(\ell(\tilde{\gamma})=2\inj(\tilde{T}^2,\tilde{g})\) (see Proposition 2.12 in Chapter 13 of \cite{doCarmoRG}). It suffices to prove the following claim.

\begin{claim}\normalfont The projected geodesic \(\gamma:=\pi \circ \tilde{\gamma}\) has no self-intersections, that is, the restriction of the universal covering projection \(\pi:\tilde{T}^2 \to T^2\) to \(\tilde{\gamma}\) is injective.
\end{claim}

We first show how this claim can be used to finish the proof, and then we prove the claim. As \(\tilde{\gamma}\) is a closed curve in the universal cover, \(\gamma\) is null-homotopic in \(T^2\). Moreover, by the claim \(\gamma\) has no self-intersections. Therefore  by the Jordan curve theorem,
there is a closed disc \(\Omega \subseteq T^2\) with \(\partial \Omega =\gamma\). Using the local Gau\ss{}-Bonnet theorem, the curvature bound, and the volume bound we get
\[
	2\pi=2\pi\chi(\Omega)=\int_{\Omega}\sec \, d\vol_g \leq \delta \vol_g(\Omega) < \delta \vol_g(T^2) \leq \delta v,
\]
which contradicts \(\delta \leq \frac{2\pi}{v}\). For the application of the local Gau\ss{}-Bonnet theorem observe that \(\partial \Omega\) is a geodesic without corners.

So it remains to prove the claim. Arguing by contradiction, we assume that the claim is wrong. Then there exists a non-trivial deck transformation \(\varphi:\tilde{T}^2 \to \tilde{T}^2\) so that for \(\tilde{\gamma}^\varphi:=\varphi(\tilde{\gamma})\) it holds
\[
	\tilde{\gamma} \cap \tilde{\gamma}^\varphi \neq \emptyset.
\] 
Observe that \(\tilde{\gamma}\) and \(\tilde{\gamma}^\varphi\) intersect each other transversely. Indeed, otherwise we have 
\(\varphi(\tilde{\gamma})=\tilde{\gamma}^\varphi=\tilde{\gamma}\) by uniqueness of geodesics. But then all powers of \(\varphi\) fix the compact set \(\tilde{\gamma}\), which contradicts that the Deck group is torsion-free and acts properly on \(\tilde{T}^2\).

It is well known that the \(\mathrm{mod} \, 2\) intersection number \(i_{\bbZ/2\bbZ}(\cdot, \cdot)\) is homotopy-invariant. Hence \(i_{\bbZ/2\bbZ}(\tilde{\gamma},\tilde{\gamma}^\varphi)=0\) because \(\tilde{\gamma}\) and \(\tilde{\gamma}^\varphi\) are both null-homotopic as \(\tilde{T}^2\) is simply connected. So \(\tilde{\gamma}\) and \(\tilde{\gamma}^\varphi\) intersect in an even number of points. By assumption their intersection is non-trivial, and hence they intersect in at least two distinct points.

Choose two distinct points \(\bar{x},\bar{y} \in \tilde{\gamma}\cap \tilde{\gamma}^\varphi\).
Choose subarcs \(c\) of \(\tilde{\gamma}\) and \(c^\varphi\) of \(\tilde{\gamma}^\varphi\) from \(\bar{x}\) to \(\bar{y}\) in such a way that \(\ell(c),\ell(c^\varphi) \leq \frac{1}{2}\ell(\tilde{\gamma})=\frac{1}{2}\ell(\tilde{\gamma}^\varphi)\). If \(c\) and \(c^\varphi\) are both strictly shorter than \(\frac{1}{2}\ell(\tilde{\gamma})\), then \(\exp_{\bar{x}}\) is \textit{not} injective on \(B\big(0,\frac{1}{2}\ell(\tilde{\gamma})\big) \subseteq T_{\bar{x}}\tilde{T}^2\), and so \(\inj(\bar{x}) < \frac{1}{2}\ell(\tilde{\gamma})\). But this contradicts \(\ell(\tilde{\gamma})=2\inj(\tilde{T}^2,\tilde{g})\). The argument in the general case is similar. By a variational argument we construct two geodesics with the same endpoints that have strictly smaller length than \(c\) and \(c^\varphi\). The same reasoning will then lead to a contradiction.

Since \(c\) and \(c^\varphi\) intersect transversely at \(\bar{y}\),
there exists \(v \in T_{\bar{y}}\tilde{T}^2\) with \(\langle v,c^\prime\rangle< 0\)
and \(\langle v, (c^\varphi)^\prime\rangle < 0\). Choose a curve \(c_v:(-\varepsilon,\varepsilon) \to \tilde{T}^2\) with \(c_v(0)=\bar{y}\)
and \(c_v^\prime(0)=v\). Since \(\bar{y}\) is not conjugate to \(\bar{x}\) along \(c\) or \(c^\varphi\), there exist variations through geodesics \(\Gamma,\Gamma^\varphi:(-\varepsilon,\varepsilon)\times [0,1]\) of \(c\) and \(c^\varphi\) with 
\[
	\Gamma(s, 0)=\bar{x}=\Gamma^\varphi(s, 0) \quad \text{and} \quad \Gamma(s, 1)=c_v(s)=\Gamma^\varphi(s,1)
\]
for all \(s \in (-\varepsilon,\varepsilon)\) (here we reparametrize \(c\) and \(c^\varphi\) to be defined on \([0,1]\)). As \(c\) is a geodesic, the first variation formula shows
\[
  \left.\frac{d}{ds}\right|_{s=0}\ell(\Gamma(s, \cdot))=
\langle \partial_s \Gamma(0,1),\partial_t \Gamma (0,1)\rangle =\langle v,c^\prime \rangle < 0.
\]
Similarly, \(\left.\frac{d}{ds}\right|_{s=0}\ell(\Gamma^\varphi(s, \cdot))< 0\). So
\(\ell(\Gamma(s_0,\cdot)) < \ell(c)\) and \(\ell(\Gamma^\varphi(s_0,\cdot)) < \ell(c^\varphi)\) for \(s_0 > 0\) sufficiently small. Then \(\Gamma(s_0,\cdot)\) and \(\Gamma^\varphi(s_0,\cdot)\) are geodesics with the same starting point \(\bar{x}\) and the same endpoint \(c_v(s_0)\), and both are strictly shorter than \(\frac{1}{2}\ell(\tilde{\gamma})\). Therefore, \(\exp_{\bar{x}}\) is not injective on \(B\big(0,\frac{1}{2}\ell(\tilde{\gamma})\big) \subseteq T_{\bar{x}}\tilde{T}^2\), and hence \(\inj(\bar{x}) < \frac{1}{2}\ell(\tilde{\gamma})=\inj(\tilde{T}^2,\tilde{g})\) which is a contradiction. This finishes the proof of the claim.
\end{proof}

We remainder of this section is devoted to the proof of the effective Uniformization Theorem.

\begin{proof}[Proof of \Cref{effective uniformization C^0}]\textbf{Step 1 (averaged \(L^2\)-estimate):} Inequality (\(\ast \ast\)) on p. 520 of \cite{GromovMetricStructures07} states the following. If \(M\) is a closed Riemannian \(n\)-manifold 
of diameter \(D\) satisfying \(\Ric \geq -(n-1)\), then the smallest positive eigenvalue \(\lambda_1\) of the Laplace operator satisfies \(\lambda_1 \geq e^{-2(n-1)D}\). 

By definition, we have ${\rm diam}(T^2,g)\leq 1,\, \vert {\rm sec}(g)\vert \leq \delta \leq 1$ and  
 \(\int_{T^2}\rho \, d\mathrm{vol}_g=0\). So \(\rho\) is perpendicular to the space of constant functions in \(H^1(T^2)\). Thus we have the Poincaré-inequality
\[
	\int_{T^2}\rho^2 d\vol_g \leq C\int_{T^2}|\nabla \rho|^2 \, d\vol_g
\]
for \(C=e^2\). It follows from Exercise 2 in Chapter 4.3 of \cite{doCarmoSurfaces} that
\[
	\Delta \rho =2K,
\]
where \(K\) is the Gau\ss{} curvature of \(g\) and \(\Delta\) is the Laplace operator of \(g\). 
Recall that by our sign convention \(\Delta=-{\rm tr}(\nabla^2)\). 
We denote the average integral \(\frac{1}{\vol_g(T^2)}\int_{T^2} d\vol_g\) by \(\strokedint_{T^2} d\vol_g\). Testing \(\Delta \rho =2K\) with \(\rho\) gives
\begin{align*}
	\strokedint_{T^2}|\nabla \rho|^2 d\vol_g=& 2\strokedint_{T^2}\rho K d\vol_g \\
	\leq & 2\delta \strokedint_{T^2}|\rho| d\vol_g \\
	\leq & 2\delta \left(\strokedint_{T^2}|\rho|^2 d\vol_g \right)^{\frac{1}{2}} \\
	\leq & 2\delta C^{\frac{1}{2}} \left(\strokedint_{T^2}|\nabla \rho|^2 d\vol_g \right)^{\frac{1}{2}},
\end{align*}
where we used the curvature assumption, the Cauchy-Schwarz and the Poincaré inequality. Thus
\begin{equation}\label{effective unif - eq1}
	\left(\strokedint_{T^2}|\rho|^2 d\vol_g \right)^{\frac{1}{2}} \leq C^{\frac{1}{2}}\left(\strokedint_{T^2}|\nabla \rho|^2 d\vol_g \right)^{\frac{1}{2}} \leq 2\delta C.
\end{equation}

\textbf{Step 2 (\(C^0\)-estimate):} Lifting \(\Delta \rho=2K\) to the universal cover we get \(\Delta \tilde{\rho}=2\tilde{K}\), where \(\tilde{\rho}\) and \(\tilde{K}\) are the lifts of \(\rho\) and \(K\) to \(\tilde{T}^2\). Curvature bounds and diameter bounds imply upper 
volume bounds by the Bishop-Gromov volume comparison theorem. Hence it follows from \Cref{inj radius bound for universal cover} that for \(\delta_0\) small enough it holds \(\inj(\tilde{T}^2,\tilde{g}) \geq i_0\) for a universal constant \(i_0 > 0\). So we can apply \Cref{Nash-Moser}, and conclude that for all \(\tilde{x}_0 \in \tilde{T}^2\) it holds
\begin{equation}\label{effective unif - eq2}
	|\tilde{\rho}|(\tilde{x}_0) \leq C \Big(||\tilde{\rho}||_{L^2(B(\tilde{x}_0,1))}+||\tilde{K}||_{C^0(\tilde{T}^2)} \Big)
\end{equation}
for a universal constant \(C\). Choose a partition of \(\tilde{T}^2\) into fundamental regions for the 
action of the fundamental group of $T^2$ whose diameters do not exceed  $2{\rm diam}(T^2,g)$.
For each \(\tilde{x} \in \tilde{T}^2\) denote by \(\mathcal{F}_{\tilde{x}}\) the element of the partition containing \(\tilde{x}\). Fix \(\tilde{x}_0 \in \tilde{T}^2\). Define
\[
	\hat{\mathcal{F}}:=\bigcup_{\tilde{x} \in B(\tilde{x}_0,1)}\mathcal{F}_{\tilde{x}}.
\]
Observe that \(\hat{\mathcal{F}}\) is a finite union of fundamental regions, and thus
\begin{equation}\label{effective unif - eq3}
	\strokedint_{\hat{\mathcal{F}}} \tilde{\rho}^2 \, d\vol_{\tilde{g}}=\strokedint_{T^2}\rho^2 \, d\vol_g,
\end{equation}
where \(\strokedint_U\) denotes the averaged integral \(\frac{1}{\vol(U)}\int_U\). Since \(\diam(\mathcal{F}_{\tilde{x}}) \leq 2\diam(T^2,g)\leq 2\) we also have 
\(B(\tilde{x}_0,1) \subseteq \hat{\mathcal{F}} \subseteq B(\tilde{x}_0,3)\). 
From the lower bound $i_0$ for the injectivity radius of $\tilde T$ we deduce that 
\(\vol\left(\hat{\mathcal{F}}\right)\) is bounded from below and from above by a universal constant.  Remembering \(B(\tilde{x}_0,1) \subseteq \hat{\mathcal{F}}\) this implies
\begin{equation}\label{effective unif - eq4}
	||\tilde{\rho}||_{L^2(B(\tilde{x}_0,1))}^2 \leq ||\tilde{\rho}||_{L^2(\hat{\mathcal{F}})}^2=O\left(\strokedint_{\hat{\mathcal{F}}}\tilde{\rho}^2 \, d\vol_{\tilde{g}} \right).
\end{equation}
Combining (\ref{effective unif - eq1})-(\ref{effective unif - eq4}) and using the curvature assumption \(|K| \leq \delta\) yields
\[
	|\tilde{\rho}|(\tilde{x}_0) \leq C\delta
\]
for a universal constant \(C\). This finishes the proof.
\end{proof}


\section{Invertibility of \(\mathcal{L}\) without a lower injectivity radius bound}\label{Sec: L without bound on inj}



\subsection{Statement and overview}\label{Subsec: L without bound on inj - statement}


For the proof of \Cref{Pinching without inj radius bound - introduction} we need to analyze the invertibility of the elliptic operator \(\mathcal{L}=\frac{1}{2}\Delta_L+2\,\mathrm{id}\) in complete Riemannian \(3\)-manifolds of finite volume that do \textit{not} have a positive lower bound on the injectivity radius. The examples of \Cref{Sec: Counterexamples} show that \Cref{A-priori estimate for L} can no longer hold in this more general situation. Also recall from the discussion at the end of \Cref{Sec: Counterexamples} that the counterexamples were constructed by slowly changing the 
conformal structure of the level tori \(T(r)\) contained in \(M_{\rm small}\). 
To exclude these examples, we shall use the geometric control of tubes and cusps established in \Cref{Sec:model metrics} for manifolds $M$ whose sectional curvature approaches 
constant curvature $-1$ exponentially fast in $M_{\rm small}$.  

We introduce new norms $\Vert \cdot \Vert_{2,\lambda;*}$ and $\Vert \cdot \Vert_{0,\lambda}$ 
for smooth sections of the bundle ${\rm Sym}^2(T^*M)={\rm Sym}(T^*M\otimes T^*M)$ 
which are inspired by 
the work of Bamler \cite{Bamler2012}, and we use these norms to prove the following invertibility result. At this point we only mention that these norms depend on certain parameters \(\alpha,\lambda,\delta,r_0,b,\bar{\epsilon}\). Recall that \(R\) denotes the Riemann curvature endomorphism.

\begin{prop}\label{Invertibility of L - non-compact case}For all \(\alpha \in (0,1)\), \(\Lambda \geq 0\), \(\lambda \in (0,1)\), \(\delta \in (0,2)\), \(r_0 \geq 1\), \(b > 1\) and \(\eta \geq 2+\lambda\) there exist constants \(\varepsilon_0\), \(\bar{\epsilon}_0\) and \(C >0\) with the following property. Let \(M\) be a Riemannian \(3\)-manifold of finite volume that satisfies 
\[
	|\sec+1|\leq \varepsilon_0, \quad ||\nabla \Ric||_{C^0(M)} \leq \Lambda ,
\]
and
\begin{equation}\label{curvature decay}
	\max_{\pi \subseteq T_xM}|\mathrm{sec}(\pi)+1|, \, |\nabla R|(x), \, |\nabla^2R|(x) \leq \varepsilon_0 e^{-\eta d(x,\partial M_{\rm small})} \quad \text{for all } \, x \in M_{\rm small}.
\end{equation}
Then the operator
\[
	\mathcal{L}: \Big( C_{\lambda}^{2,\alpha}\big( {\rm Sym}^2(T^*M)\big), ||\cdot||_{2 ,\lambda;\ast}\Big) \longrightarrow \Big( C_{\lambda}^{0,\alpha}
	\big( {\rm Sym}^2(T^*M)\big), ||\cdot||_{0,\lambda} \Big)
\]
is invertible and 
\[
	||\mathcal{L}||_{\rm op}, ||\mathcal{L}^{-1}||_{\rm op} \leq C,
\]
where \(||\cdot||_{2,\lambda;\ast}\) and \(||\cdot||_{0,\lambda}\) are the norms defined in (\ref{Hybrid 2-norm - exponential}) and (\ref{Hybrid 0-norm - exponential}) with respect to any \(\bar{\epsilon} \leq \bar{\epsilon}_0\), and $C_\lambda^{k,\alpha}\big({\rm Sym}^2(T^*M)\big)$ is the corresponding Banach space of sections of ${\rm Sym}^2(T^*M)$.
\end{prop}

The proof of the existence of \(\varepsilon_0\) and \(\bar{\epsilon}_0\) is constructive. This is not the case for the constant 
\(C\) which involves an argument by contradiction. 

\begin{rem}\label{Invertbility of L - eta > 1}\normalfont \Cref{Invertibility of L - non-compact case} is also valid for all \(\lambda \in (0,1)\) and \(\eta > 1\). 
The proof we found is technically more involved, 
but it does not require any new insights. As we are not aware of additional applications, 
we restrict ourselves to the case \(\eta \geq 2+\lambda\) and postpone the presentation of the stronger statement
to forthcoming work.
\end{rem}

Analogous to \Cref{Invertibility of L - non-orientable} the following holds (also see \Cref{small part for non-orientable}).

\begin{rem}\label{Invertibility of L - non-compact case - non-orientable}\normalfont
\Cref{Invertibility of L - non-compact case} also holds when \(M\) is non-orientable.
\end{rem}

This section is structured as follows. 
The definition of the norms which appear in the statement of \Cref{Invertibility of L - non-compact case}
is presented in \Cref{Subsec: various norms}. In \Cref{Subsec: Growth estimates} we prove that solutions \(h\) of the equation \(\mathcal{L}h=f\) satisfy certain growth estimates in \(M_{\rm small}\). These will be used in \Cref{Subsec: a priori estimates in non-compact case} to show that \(\mathcal{L}\) satisfies an a priori estimate \(||h||_{2,\lambda;\ast}\leq C||\mathcal{L}h||_{0,\lambda}\). Finally, the surjectivity of \(\mathcal{L}\) will be established in \Cref{Subsec: Surjectivity of L}.


\subsection{Various norms}\label{Subsec: various norms}

In this subsection we present the definition of the norms \(||\cdot||_{2,\lambda;\ast}\) and \(||\cdot||_{0,\lambda}\) appearing in \Cref{Invertibility of L - non-compact case}, and state some of their properties. 
To this end we first 
define \textit{exponential norms} \(||\cdot||_{C_\lambda^{?}}\). Second, adapting a construction in \cite{Bamler2012}, 
we present \textit{decomposition norms} \(||\cdot||_{C_\lambda^{?};\ast}\). 
The \textit{exponential hybrid norms} \(||\cdot||_{2,\lambda;\ast}\) and \(||\cdot||_{0,\lambda}\) are then a combination of the decomposition norms \(||\cdot||_{C_\lambda^{?};\ast}\) and the hybrid norms \(||\cdot||_2\) and \(||\cdot||_0\) defined in \Cref{subsection - hybrid norm}.

Throughout this section we assume that \(M\) satisfies the assumptions from \Cref{Invertibility of L - non-compact case}, 
that is, \(M\) is a complete Riemannian \(3\)-manifold of finite volume
satisfying 
\(|\sec+1|\leq \varepsilon_0\), \(||\nabla \Ric||_{C^0(M)} \leq \Lambda\), and 
\[
	\max_{\pi \subseteq T_xM}|\mathrm{sec}(\pi)+1|, \, |\nabla R|(x), \, |\nabla^2R|(x) \leq \varepsilon_0 e^{-\eta d(x,\partial M_{\rm small})} \quad \text{for all } \, x \in M_{\rm small}
\]
for some \(\eta \geq 2+\lambda\). As before, \(\mathcal{L}\) denotes the elliptic operator \(\frac{1}{2}\Delta_L+2\mathrm{id}\).

We begin with the definition of the \textit{exponential norms}. Let \(C_1,...,C_{q}\) be the cusps of $M$, and 
let \(T_1,...,T_p\) be the Margulis tubes of \(M\) of radius at least 3. For any \(k=1,...,p\) let \(R_k\) be the radius of \((T_k)_{\rm small}\), 
that is, the distance of the core geodesic \(\gamma_k\) to \(\partial (T_k)_{\rm small}\). Denote by 
\[
	r(x):=d(x,M \setminus M_{\rm small} )
\]
the distance of \(x\) to the complement of \(M_{\rm small}\). For \(\lambda \in (0,1)\) define the \textit{inverse weight function} \(W_{\lambda}: M \to \bbR\) by
\[
	W_{\lambda}(x):= \begin{cases}
	e^{-\lambda r(x)} & \, \text{if } \, x \in \bigcup_{i=1}^{q}(C_i)_{\rm small} \\
	e^{-\lambda r(x)}+e^{\lambda (r(x)-R_k)} & \, \text{if } \, x \in (T_k)_{\rm small} \\
	1  & \text{otherwise } 
	\end{cases}.
\]
For any \(C^k\) or \(C^{k,\alpha}\) norm \(||\cdot||_{C^{?}}\) we define the corresponding \textit{exponential \(C^{?}\) norm} by
\begin{equation}\label{exponential norm - general definition}
	||\cdot||_{C_{\lambda}^{?}(M)}:=\sup_{x \in M}\left(\frac{1}{W_{\lambda}(x)}|\cdot|_{C^{?}}(x) \right).
\end{equation}
Observe that \(\mathrm{Lip}\left(\log \left(\frac{1}{W_{\lambda}}\right) \right) \leq \lambda\) in \(M_{\rm small}\), and that outside \(M_{\rm small}\) it holds \(W_{\lambda}(x)=1\). 
So the pointwise Schauder estimates (\ref{Schauder pointwise - C^2}) and (\ref{Schauder pointwise - C^1}) imply that there are Schauder estimates for the exponential norms, that is, there is a constant \(C=C(\alpha,\Lambda,\lambda)\) so that
\begin{equation}\label{Schauder estimate - exponential norm}
	||h||_{C_{\lambda}^{2,\alpha}(M)} \leq C \Big(||\mathcal{L}h||_{C_{\lambda}^{0,\alpha}(M)}+||h||_{C_{\lambda}^{0}(M)} \Big)
\end{equation}
and
\begin{equation}\label{Schauder estimate C^1 - exponential norm}
	||h||_{C_{\lambda}^{1,\alpha}(M)} \leq C \Big(||\mathcal{L}h||_{C_{\lambda}^{0}(M)}+||h||_{C_{\lambda}^{0}(M)} \Big).
\end{equation}
Similarly, it follows from (\ref{Continuity of elliptic operator - pointwise}) that \(||\mathcal{L}h||_{C_\lambda^{0,\alpha}(M)} \leq C||h||_{C_\lambda^{2,\alpha}(M)}\).

We continue with the definition of the \textit{decomposition norms}. Following Section 3.2 of \cite{Bamler2012}, 
we first introduce \textit{trivial Einstein variations}. In Definition \ref{def:trivialeinstein} below, 
\(T^2\) denotes a flat torus (with a fixed flat metric), and \(I \subseteq \bbR\) is an interval. Even though \(T^2\) is equipped with a metric, the product 
\(T^2 \times I\) is only meant as a topological product. Moreover, in the definition we take coordinates \((x^1,x^2,r)\) on \(T^2 \times I\), 
where \((x^1,x^2)\) are flat coordinates for \(T^2\), and \(r\) is the standard coordinate for \(I \subseteq \bbR\).

\begin{definition}\label{def:trivialeinstein}
A \((0,2)\)-tensor \(u\) on \(T^2 \times I\) is called an \textit{Einstein variation} if it is of the form
\[
	u=e^{-2r}u_{ij}dx^idx^j
\]
for some constants \(u_{ij} \in \bbR\). Moreover, \(u\) is called a \textit{trivial Einstein variation} if the trace of \(u\) with respect to 
the flat metric on \(T^2\) vanishes everywhere, that is, if \(\sum_{i}u_{ii}=0\).
\end{definition}

Here \(dx^i\) and \(dx^j\) are understood to be either \(dx^1\) or \(dx^2\), but not \(dr\). Our definition looks slightly different than that in \cite{Bamler2012}, but this difference is only due to a change of coordinates.

\begin{rem}\label{Lu=0}\normalfont The trace free condition guarantees that  \(\mathcal{L}_{cusp}u=0\) for a trivial 
Einstein variation \(u\) (see (\ref{Lh=f in a cusp in coordinates})). Here \(\mathcal{L}_{cusp}\) denotes the operator \(\frac{1}{2}\Delta_L+2\mathrm{id}\) with respect to 
the hyperbolic cusp metric \(g_{cusp}=e^{-2r}g_{Flat}+dr^2\), where \(g_{Flat}\) is the given flat metric on \(T^2\).
\end{rem}

An Einstein variation should be thought of as an infinitesimal change in the conformal structure of the torus \(T^2\), 
and a trivial Einstein variation is an infinitesimal change of the conformal structure that can \textit{not} be detected by the operator \(\mathcal{L}\). 
Therefore, to ensure that \(\mathcal{L}\) is invertible, we have to work with a norm in the source space that isolates trivial Einstein variations. Also recall from the discussion at the end of \Cref{Sec: Counterexamples}, that the counterexamples of \Cref{Counterexample} are constructed by changing the conformal structure of the horotori that are contained in the small part of the manifold. This further justifies the use of a norm that is sensitive to changes in the conformal structures. A more precise explanation for the necessity of a norm that isolates trivial Einstein variations will be given in \Cref{Use of *-norm}. 

For each \(k=1,...,p\) let \(\gamma_k\) be the core geodesic of \(T_k\). Choose cutoff functions \(\rho_k\) so that
\[
	\rho_k=0 \, \text{ in } \, \big(M \setminus (T_k)_{\rm small}\big) \cup N_1(\gamma_k) \quad \text{and} \quad \rho_k=1 \, \text{ in } \,   N_{R_k-1/4}(\gamma_k)\setminus  N_{5/4}(\gamma_k).
\]
Similarly, choose cutoff functions \(\varrho_l\) so that
\[
	\varrho_l=0 \, \text{ in } \, M \setminus (C_l)_{\rm small} \quad \text{and} \quad \varrho_l=1 \, \text{ outside} \,   N_{1}\big(M \setminus (C_l)_{\rm small} \big).
\]
Here for any subset \(X \subseteq M\), and any \(r>0\), \(N_r(X)\) denotes the set of all points that have distance less than \(r\) to \(X\). The cutoff functions 
are chosen in such a way that the Hölder norms \(||\rho_k||_{C^{2,\alpha}},||\varrho_l||_{C^{2,\alpha}}\) are bounded from above 
by a universal constant. Recall that by \Cref{containedinthin} the boundary \(\partial (T_k)_{\rm small}\) is a smooth torus, and \((T_k)_{\rm small} \setminus N_1(\gamma_k)\) is diffeomorphic to \(\partial (T_k)_{\rm small}\times [0,R_k-1]\). 
By \Cref{tubes are almost cusps - pinched curvature}, 
there is a natural choice for a flat metric on \(\partial (T_k)_{\rm small}\), namely the flat metric induced by the model metric \(g_{cusp}\) of \Cref{tubes are almost cusps - pinched curvature}. Hence it makes sense to speak of trivial Einstein variations on \((T_k)_{\rm small} \setminus N_1(\gamma_k)\). Similarly, it makes sense to speak of trivial Einstein variations on \((C_l)_{\rm small}\). 

For a continuous symmetric $(0,2)$-tensor field $h$, consider decompositions
\begin{equation}\label{decomposition with trivial Einstein variations}
	h=\bar{h}+\sum_{k=1}^p\rho_k u_k+\sum_{l=1}^{q}\varrho_l v_{l},
\end{equation}
where \(u_k\) is a trivial Einstein variation in \((T_k)_{\rm small} \setminus N_1(\gamma_k)\), and \(v_l\) is a trivial Einstein variation on \((C_l)_{\rm small}\). 
Following p. 896 of \cite{Bamler2012}, 
for any \(C^k\) or \(C^{k,\alpha}\) norm \(||\cdot||_{C^{?}}\) define the corresponding \textit{decomposition norm} by
\begin{equation}\label{Decomposition norm}
	||h||_{C_{\lambda}^{?}(M); \ast}:= \inf \Big( ||\bar{h}||_{C_{\lambda}^{?}(M)}+\max_{k=1,...,p}|u_k|+\max_{l=1,...,q}|v_l|\Big),
\end{equation}
where the infimum is taken over all decompositions as in (\ref{decomposition with trivial Einstein variations}) and \(|\cdot|\) is a norm on the finite dimensional space of trivial Einstein variations, e.g., \(|\cdot|:=||\cdot||_{C^0}\). 
We point out that our notation differs from that in \cite{Bamler2012}. 

We now state some properties of the decomposition norm. We start with a basic inequality.

\begin{lem}\label{C^0 is smaller that decomposition C^0}\normalfont It holds
\[
	||h||_{C^0} \leq 2||h||_{C_{\lambda}^0;\ast}.
\]
\end{lem}

\begin{proof}Note \(W_\lambda \leq 2\), so that \(||\bar{h}||_{C^0}\leq 2||\bar{h}||_{C_\lambda^0}\). Now the desired inequality follows from the triangle inequality and the definition of \(||\cdot||_{C_\lambda^0;\ast}\).
\end{proof}

Schauder estimates also hold for the decomposition norm (see \cite[Lemma 4.1]{Bamler2012}).

\begin{lem}\label{Schauder estimates for decomposition norm} There is a universal constant \(C\) so that 
\[
	||h||_{C_{\lambda}^{2,\alpha};\ast}\leq C\Big( ||\mathcal{L} h||_{C_{\lambda}^{0,\alpha}}+||h||_{C_{\lambda}^{0};\ast}\Big)
\]
for all \(h \in C^{2,\alpha}\big( {\rm Sym}^2(T^*M)\big)\).
\end{lem}

\begin{proof}In this proof we use \Cref{big O notation}. Let \(\mathcal{L}_{cusp}\) denote the differential operator \(\frac{1}{2}\Delta_L+2\mathrm{id}\) with respect to 
the model metrics \(g_{cusp}\)  given by \Cref{tubes are almost cusps - pinched curvature} and \Cref{Existence of approximate cusp metric}. 
By \Cref{Lu=0}, it holds \(\mathcal{L}_{cusp}u_k=0\) and \(\mathcal{L}_{cusp}v_l=0\). Hence \(||u_k||_{C^{2,\alpha}}=O(|u_k|)\) and \(||v_l||_{C^{2,\alpha}}=O(|v_l|)\) due to Schauder estimates.

The estimates on \(|g-g_{cusp}|_{C^2}\) given by \Cref{Existence of approximate cusp metric} and \Cref{tubes are almost cusps - pinched curvature} yield that 
\(|\mathcal{L}v_l|_{C^{0,\alpha}}(x)=O\big(\varepsilon_0e^{-\eta r(x)}|v_l| \big)\) and \(|\mathcal{L}u_k|_{C^{0,\alpha}}(x)=O\Big(\big(e^{2(r(x)-R_k)}+\varepsilon_0e^{-\eta r(x)}\big)|u_k| \Big)\), 
where \(\eta \) is the decay rate in the curvature decay condition (\ref{curvature decay}), and \(r(x)=d(x,M\setminus M_{\rm small})\). We 
then get \(||\mathcal{L}u_k||_{C_{\lambda}^{0,\alpha}} =O(|u_k|)\) and \(||\mathcal{L}v_l||_{C_{\lambda}^{0,\alpha}} =O(|v_l|)\) since \(\lambda < 1< \eta\). The rest of the argument carries over from \cite[Lemma 4.1]{Bamler2012}
\end{proof}

The next result and its proof is analogous to Lemma 4.2 in \cite{Bamler2012}. It gives the canonical choice of a trivial Einstein variation in a Margulis tube. For rank 2 cusps the canonical choice of a trivial Einstein variation will be given by \Cref{trivial Einstein variation and a priori estimate in a cusp}.

\begin{lem}\label{Canonical choice of trivial Einstein variation} Let \(h \in C^{0}\big({\rm Sym}^2(T^*M)\big)\). Choose points \(c_k \in T_k\) with \(r(c_k)=\frac{R_k}{2}\), where \(r=d(\cdot,M \setminus M_{\rm small})\). For each \(k\) let \(u_k\) be the trivial Einstein variation in \((T_k)_{\rm small}\setminus N_1(\gamma_k)\) such that \(|h-u_k|(c_k)\) is minimal among all trivial Einstein variations in  \((T_k)_{\rm small}\setminus N_1(\gamma_k)\). Then for some universal constant \(C\) it holds
\[
	||h||_{C_{\lambda}^{0}(M^\prime); \ast} \leq ||\bar{h}||_{C_{\lambda}^{0}(M^\prime)}+\max_{k}|u_k| \leq C ||h||_{C_{\lambda}^{0}(M^\prime); \ast},
\]
where \(\bar{h}:=h-\sum_{k}\rho_ku_k\) and \(M^\prime=M \setminus \bigcup_{l=1}^q (C_l)_{\rm small}\).
\end{lem}

Note that on \((T_k)_{\rm small}\), the weight \(\frac{1}{W_{\lambda}(r)}\) is maximal at \(r=\frac{R_k}{2}\).

\begin{proof}The proof of \cite[Lemma 4.2]{Bamler2012} goes through without modification.
  For later purpose we point out that \(|u_k|(c_k) \leq |h|(c_k)\), and \(|\bar{h}|(c_k)\leq |\bar{h}-u|(c_k)\) for any trivial Einstein variation \(u\) on \((T_k)_{\rm small}\setminus N_1(\gamma_k)\). This is because by its definition, \(u_k\) 
is the image of the orthogonal projection from \({\rm Sym}^2(T_{c_k}^*M)\) to the space of trivial Einstein variations on  \((T_k)_{\rm small}\setminus N_1(\gamma_k)\). Also note that trivial Einstein variations have constant norm (if the norm is taken with respect to the cusp metric \(g_{cusp}\)). In particular, we have \(|u_k| \leq ||h||_{C^0(M)}\).
\end{proof}

Finally, we come to the definition of the \textit{exponential hybrid norms} \(||\cdot||_{2,\lambda;\ast}\) and \(||\cdot||_{0,\lambda}\) appearing in \Cref{Invertibility of L - non-compact case}. Recall that for \(k=0\) and \(k=2\), in Definition \ref{hybridnorm}  we defined the hybrid norms (when \(n=3\)) by
\[
	||\cdot||_{k}:=\max \left\{ ||\cdot||_{C^{k,\alpha}(M)}, \, \sup_{x \nin E}\left(\int_M e^{-(2-\delta)r_x(y)}|\cdot|_{C^k}^2(y)\, d\vol(y) \right)^{\frac{1}{2}}\right\},
\]
where \(E \subseteq M\) is a subset defined by a volume growth condition and $\vert \cdot \vert_{C^k}(y)$ denotes the $C^k$-norm at the point $y$. We refer to 
\Cref{subsection - hybrid norm} for more details. For ease of notation we abbreviate
\[
	||h||_{H^2(M; \omega_x)}:=\left(\int_M e^{-(2-\delta)r_x(y)}|h|_{C^2}^2(y)\, d{\rm vol}(y) \right)^{\frac{1}{2}}
\]
and
\[
	||f||_{L^2(M; \omega_x)}:=\left(\int_M e^{-(2-\delta)r_x(y)}|f|_{C^0}^2(y)\, d{\rm vol}(y) \right)^{\frac{1}{2}}.
\]
Here \(\omega_x\) should indicate that there is a weight function involved that depends on \(x \in M\). 

\begin{definition}\label{exponentialhybrid}
For \(\alpha \in (0,1)\), \(\lambda \in (0,1)\), \(b>1\), \(\bar{\epsilon}>0\), \(\delta \in (0,2)\) and \(r_0 \geq 1\) the \textit{exponential hybrid norms} \(||\cdot||_{2,\lambda;\ast}\) and \(||\cdot||_{0,\lambda}\) are defined by
\begin{equation}\label{Hybrid 2-norm - exponential}
	||h||_{2,\lambda; \ast}:= \max \left\{ ||h||_{C_{\lambda}^{2,\alpha}(M);\ast}, \,\sup_{x \nin E}||h||_{H^2(M; \omega_x)}, \, \sup_{x \in M_{\rm thin} \setminus M_{\rm small}} e^{\frac{b}{2}d(x,M_{\rm thick})}||h||_{H^2(M; \omega_x)}\right\}
\end{equation}
and
\begin{equation}\label{Hybrid 0-norm - exponential}
	||f||_{0,\lambda}:= \max \left\{ ||f||_{C_{\lambda}^{0,\alpha}(M)}, \,\sup_{x \nin E}||f||_{L^2(M; \omega_x)}, \, \sup_{x \in M_{\rm thin} \setminus M_{\rm small}} e^{\frac{b}{2}d(x,M_{\rm thick})}||f||_{L^2(M; \omega_x)}\right\},
\end{equation}
where \(E=E(M;\bar{\epsilon},\delta,r_0)\) is the set defined in (\ref{def of set E}).
\end{definition}

In the source space we use \(||\cdot||_{C_{\lambda}^{2,\alpha}(M);\ast}\) instead of just \(||\cdot||_{C_{\lambda}^{2,\alpha}(M)}\) so that the norm is sensitive to trivial Einstein variations. We refer to \Cref{Use of *-norm} and the discussion after \Cref{Lu=0} as to why this is necessary. The last integral terms in the definition of the norms are included so that we can employ the \(C^0\)-estimate from \Cref{A-priori estimate away from the small part}. As was the case with the previously defined hybrid norms \(||\cdot||_2\) and \(||\cdot||_0\), we suppress most of the constants in the notation for the norms \(||\cdot||_{2,\lambda; \ast}\) and \(||\cdot||_{0,\lambda}\). 

Define the spaces \(C_\lambda^{0,\alpha}\big({\rm Sym}^2(T^*M)\big)\) and \(C_\lambda^{2,\alpha}\big({\rm Sym}^2(T^*M)\big)\) as
\begin{equation}\label{Def of the exponential target space}
	C_\lambda^{0,\alpha}\big({\rm Sym}^2(T^*M)\big):=\left\{\left. f \in C^{0,\alpha}\big({\rm Sym}^2(T^*M)\big) \, \, \right| \,\, ||f||_{C_\lambda^{0,\alpha}(M)} < \infty \right\}
\end{equation}
and
\begin{equation}\label{Def of the exponential source space}
	C_\lambda^{2,\alpha}\big({\rm Sym}^2(T^*M)\big):=\left\{\left. h \in C^{2,\alpha}\big({\rm Sym}^2(T^*M)\big) \, \, \right| \,\, ||h||_{C_\lambda^{2 ,\alpha}(M);\ast} < \infty \right\}.
\end{equation}
The latter is the space of sections $h$ with the property that 
\(\mathcal{L}h \in C_\lambda^{0,\alpha}\big({\rm Sym}^2(T^*M)\big)\) (see \Cref{Regularity of L}).


\subsection{Growth estimates}\label{Subsec: Growth estimates}

In \Cref{Subsec: a priori estimates in non-compact case} we shall prove the a priori estimate of \Cref{Invertibility of L - non-compact case}. An intermediate step towards this goal is \Cref{C^0 estimate from exponential hybrid norm}, in which we prove a global \(C^0\)-estimate \(||h||_{C^0(M)}\leq C||\mathcal{L}h||_{0,\lambda}\). For points in \(M_{\rm thick}\) we can use the arguments from the proof of \Cref{A-priori estimate for L} to obtain such an estimate (see \Cref{pointwise C^0 estimate given inj rad bound}). Moreover, \Cref{A-priori estimate away from the small part} provides the desired estimate for points in \(M_{\rm thin} \setminus M_{\rm small}\). Therefore, it remains to obtain \(C^0\)-estimates in \(M_{\rm small}\). The main ingredient to obtain \(C^0\)-estimates in \(M_{\rm small}\) are certain \textit{growth estimates} that solutions of \(\mathcal{L}h=f\) satisfy. These estimates are contained in the following \Cref{trivial Einstein variation and a priori estimate in a cusp} and \Cref{Growth estimate in a tube - pinched curvature} which are the main results of this section.

We begin with the growth estimate in a cusp. Besides the growth estimate, this result also states that there is a canonical choice of trivial Einstein variation inside a cusp (for tubes the canonical choice of trivial Einstein variation was given by \Cref{Canonical choice of trivial Einstein variation}).

\begin{prop}[Growth estimate in a cusp]\label{trivial Einstein variation and a priori estimate in a cusp} For all \(\alpha \in (0,1)\), \(\Lambda \geq 0\), \(\lambda \in (0,1)\), \(b > 1\), \(\delta \in (0,2)\), and \(\eta \geq 2+\lambda\) there exists \(\varepsilon_0>0\) with the following property. 

Let \(M\) be a finite volume \(3\)-manifold that satisfies 
\[
	|\sec+1| \leq \varepsilon_0, \quad  ||\Ric(g)||_{C^0(M)} \leq \Lambda,
\]
and
\[
	\max_{\pi \subseteq T_xM}|\mathrm{sec}(\pi)+1|, \, |\nabla R|(x), \, |\nabla^2R|(x) \leq \varepsilon_0 e^{-\eta d(x,\partial M_{\rm small})} \quad \text{for all } \, x \in M_{\rm small}.
\] 
Let \(f \in C_\lambda^{0,\alpha}\big({\rm Sym}^2(T^*M)\big)\), and let \(h \in C^{2}\big({\rm Sym}^2(T^*M)\big)\) with \(||h||_{C^0(M)} < \infty\) be a solution of 
\[
	\mathcal{L}h=f.
\]
Fix a cusp \(C\) of \(M\). Then there exists a unique trivial Einstein variation \(v\) in \(C_{\rm small}\) satisfying
\[
	||h-v||_{C_{\lambda}^0(C_{\rm small})} < \infty,
\]
and we have
\[
	|v| =O(||f||_{0,\lambda}).
\]
Moreover, if \(||h||_{C^0(M)},||f||_{C^{0,\alpha}(M)} \leq 1\), then for all \(x \in C_{\rm small}\) it holds
\begin{equation}\label{growth estimate in a cusp - eq}
	|h|(x) = O\Big(||f||_{0,\lambda}+e^{-r(x)}\Big)
\end{equation}
and
\begin{equation}\label{exponential growth estimate for h-v in cup}
	e^{\lambda r(x)}|h-v|(x) =  O\Big(||f||_{0,\lambda}+e^{-(1-\lambda)r(x)}\Big),
\end{equation}
where \(r(x)=d(x,\partial C_{\rm small})\).
\end{prop}

We refer to \Cref{big O notation} for our convention of the \(O\)-notation. The component \(\sup_{x \nin E}||f||_{L^2(M;\omega_x)}\) of the norm \(||f||_{0,\lambda}\) is \textit{not} needed for this estimate. For this reason we don't have to include constants \(\bar{\epsilon}>0\) and \(r_0 \geq 1\) in the formulation of \Cref{trivial Einstein variation and a priori estimate in a cusp} as these only enter the definition of the set \(E\). 

The estimate in a tube is very similar, but it additionally involves the distance to the core geodesic.

\begin{prop}[Growth estimate in a tube]\label{Growth estimate in a tube - pinched curvature}Let all the constants and the manifold \(M\) be as in \Cref{trivial Einstein variation and a priori estimate in a cusp}. Let \(f \in C_\lambda^{0,\alpha}\big({\rm Sym}^2(T^*M)\big)\) with \(||f||_{0,\lambda} \leq 1\), and 
let \(h \in C^{2,\alpha}\big({\rm Sym}^2(T^*M)\big)\) with \(||h||_{C^0(M)} \leq 1\) be a solution of 
\[
	\mathcal{L}h=f.
\]
Fix a Margulis tube \(T\) of \(M\), and denote its core geodesic by \(\gamma\). For all \(x \in T_{\rm small} \setminus N_1(\gamma)\) it then holds
\begin{equation}\label{Growth estimate in small part}
	|h|(x) = O \left(||f||_{0,\lambda}+ e^{- r_{\partial T}(x)}+e^{-\frac{3}{2}r_{\gamma}(x)} \right),
\end{equation}
where \(r_{\partial T}(x)=d(x,\partial T_{\rm small})\), and \(r_{\gamma}(x)=d(x,\gamma)\).
\end{prop}

The main idea to obtain these growth estimates is as follows. On \(M_{\rm small}\) consider the model metric \(g_{cusp}\) obtained in \Cref{Sec:model metrics}. 
Following \cite[p. 901]{Bamler2012}, we define
with respect to this model metric an \textit{averaging operator} that assigns to each tensor \(h\) another tensor \(\hat{h}\) that only depends on \(r=d(\cdot,M \setminus M_{\rm small})\) (we will say in a moment what this means exactly). This averaging operator commutes with the differential operator \(\mathcal{L}\) (if \(\mathcal{L}\) is taken with respect to 
the model metric). So \(\mathcal{L}\hat{h}=\hat{f}\). A key point is that as \(\hat{h}\) and \(\hat{f}\) both only depend on \(r\), the equation 
\(\mathcal{L}\hat{h}=\hat{f}\) actually just is a linear system of ODEs with constant coefficients whose fundamental solutions can be written down explicitly. Using standard ODE arguments we thus obtain growth estimates for \(\hat{h}\). These will yield growth estimates for \(h\) since \(|h-\hat{h}|(x)\) decays exponentially in \(r(x)\). 

We now explain these ideas in more detail. We start with some terminology. Recall from \Cref{Sec:model metrics} that we call a metric \(g\) on \(T^2 \times I\) (where \(I\) is an interval) a \textit{cusp metric} if it is of the form
\[
	g=e^{-2r}g_{Flat}+dr^2,
\]
where \(g_{Flat}\) is some flat metric on \(T^2\), and \(r\) is the standard coordinate on \(I \subseteq \bbR\). 
We call a covering \(\varphi:\bbR^2 \times I \to T^2 \times I\) \textit{cusp coordinates} if it is of the form \(\varphi(x_1,x_2,r)=(\psi(x_1,x_2),r)\) for some local isometry \(\psi: \bbR^2 \to (T^2,g_{Flat})\). We say that a tensor \(h\) on \(T^2\times I\) \textit{only depends on \(r\)} if its coefficients \(h_{ij}\) in cusp coordinates only depend on \(r\). This can also be stated without reference to local coordinates as follows. Note that there is an isometric \(\bbR^2\)-action on \((T^2\times I,g)\) preserving the level tori \(\{r=\mathrm{const}\}\). Then \(h\) only depends on \(r\) if it is invariant under this isometric \(\bbR^2\)-action. 

Next we explain the averaging operator that assigns to each tensor \(h\) on \(T^2 \times I\) another tensor \(\hat{h}\) that only depends on \(r\). 
The average \(\hat{u}\) of a function \(u:T^2 \times I \to \bbR\) is
\[
	\hat{u}(x):=\frac{1}{\vol_2(T(r))}\int_{T(r)}u(y) \, d\vol_2(y),
\]
where \(r=r(x)\), and \(T(r):=T^2\times \{r\}\). For a \((0,2)\)-tensor \(h\) we define \(\hat{h}\) componentwise, that is,
\[
	(\hat{h})_{ij}(x):=\widehat{h_{ij}}(x)=\frac{1}{\vol_2(T(r))}\int_{T(r)}h_{ij}(y) \, d\vol_2(y),
\]
where the coefficients are with respect to cusp coordinates.
It is clear that this definition is independent of the choice
of cusp coordinates. The averaging
for tensors of any type is defined in exactly the same way.
We collect the properties of this averaging operation in the following lemma. 

\begin{lem}\label{Properties of averaging operator} Let \(T^2\times I\) be equipped with a cusp metric. The averaging operation \(\hat{\cdot}\) has the following properties:
\begin{enumerate}[i)]
\item \(\hat{h}\) only depends on \(r\);
\item There is a universal constant \(c>0\) so that 
\[
	|\hat{h}|(x) \leq c \frac{1}{\vol_2(T(r))}\int_{T(r)}|h|(y) \, d\vol_2(y).
\]
In particular,  \(|\hat{h}|(x)\leq  c\max_{T(r(x))}|h|\) for a universal constant \(c\);
\item If \(h\) is of class \(C^1\), then the same holds true for \(\hat{h}\), and \(\widehat{\nabla h}=\nabla \hat{h}\);
\item \(\hat{\cdot}\) commutes with taking the trace,
 that is, \({\rm tr}(\hat{h})=\widehat{{\rm tr}(h)}\);
\item If \(h\) is \(C^1\), then 
\[
	|h-\hat{h}|(x)\leq cDe^{- r(x)}\max_{T(r(x))}|h|_{C^1},
\]
where \(D:=\diam(T^2,g_{Flat})\) and \(c\) is a universal constant.
\end{enumerate}
\end{lem}

Property \(v)\) is what allows us to deduce estimates on \(|h|\) from estimates on \(|\hat{h}|\).

\begin{proof}These properties are all straighforward to check. The main fact to observe is that in cusp coordinates the \(g_{ij}\) (and hence the Christoffel symbols) only depend on \(r\), and so one can take \(g_{ij}(y)\) inside the integral \(\int_{T(r)}\) out of the integral.
\end{proof}

Another crucial point is that if \(h\) and \(f\) are \((0,2)\)-tensors that only depend on \(r\), then the equation \(\mathcal{L}h=f\) is a linear system of ODEs with constant coefficients (here \(\mathcal{L}\) is taken with respect to  the cusp metric). Namely, it holds
\begin{align*}
	-2(\mathcal{L}h)(\partial_3,\partial_3)&=(h_{33})^{\prime \prime}-2(h_{33})^{\prime}-4h_{33} ; \\
	-2(\mathcal{L}h)(\partial_i,\partial_3)&=(h_{i3})^{\prime \prime}-4h_{i3}; \\
	-2(\mathcal{L}h)(\partial_i,\partial_j)&=(h_{ij})^{\prime \prime}+2(h_{ij})^{\prime} -2\delta_{ij}\sum_{k=1}^{2}h_{kk},
\end{align*}
where \((\cdot)^{\prime}\) denotes \(\frac{d}{dr}\) and $\partial_3=\frac{\partial}{\partial r}$. 
This can be checked by a straightforward calculation. Note that
\(
	|h|^2=(h_{33})^2+2\sum_{i=1}^{2}(e^rh_{i3})^2+\sum_{i,j=1}^{2}(e^{2r}h_{ij})^2.
\)
Thus we are interested in equations for \(h_{33},e^rh_{i3},e^{2r}h_{ij}\), rather than for \(h_{33},h_{i3},h_{ij}\). Using the above, 
it is straightforward to check that if $h,f$ only depend on $r$, then 
the equation \(\mathcal{L}h=f\) is equivalent to
\begin{equation}\label{Lh=f in a cusp in coordinates}
	\begin{cases}
	(h_{33})^{\prime \prime}\, \, \, \, \, \, \,-2(h_{33})^{\prime}\,\,\,\,\,-4h_{33}&=-2f_{33}  \\
	(e^{r}h_{i3})^{\prime \prime}\,\,-2(e^{r}h_{i3})^{\prime}-3e^{r}h_{i3}&=-2e^{r}f_{i3} \\
	(e^{2r}h_{ij})^{\prime \prime}-2(e^{2r}h_{ij})^{\prime}&=-2e^{2r}f_{ij}+2\delta_{ij}({\rm tr}(h)-h_{33}).
	\end{cases}
\end{equation}
The set of roots of the polynomials associated to these ODEs are \(\{1-\sqrt{5} ,1+\sqrt{5} \}\), \(\{-1 ,  3 \}\) and \(\{ 0 ,  2 \}\). The exact form of this linear system of ODEs is \textit{not} important. All what matters is that it is \textit{some} system of ODEs whose fundamental solutions we can write down explicitly. 
Moreover, tracing the equation \(\mathcal{L}h=f\) yields \(\frac{1}{2}\Delta {\rm tr}(h)+2 {\rm tr}(h)={\rm tr}(f)\). For a function \(u\) that only depends on \(r\) it holds \(-\Delta u=u^{\prime \prime}-2u^\prime\). Thus
\begin{equation}\label{PDE for trace as ODE}
	{\rm tr}(h)^{\prime \prime}-2{\rm tr}(h)^\prime-4{\rm tr}(h)=-2{\rm tr}(f).
\end{equation}
The roots of the polynomial \(Q(X)=X^2-2X-4\) associated to this ODE are \(1 \pm \sqrt{5}\).

At this point we make another important comment as to why we work with the decomposition norm \(||\cdot||_{C_\lambda^{2,\alpha};\ast}\) instead of just the exponential norm \(||\cdot||_{C_\lambda^{2,\alpha}}\) (see \Cref{Subsec: various norms} for the definition of these norms). 

\begin{rem}\label{Use of *-norm}\normalfont
As mentioned previously, the counterexamples of \Cref{Sec: Counterexamples} show that working with the hybrid norms of \Cref{subsection - hybrid norm} will no longer be sufficient in the absence of a positive lower bound on the injectivity radius. A natural condition to rule out the counterexamples of \Cref{Sec: Counterexamples} is to require that the sectional curvatures converge to \(-1\) exponentially fast inside \(M_{\rm small}\),
and thus it is natural to work with weighted Hölder norms. However, if we used \(||\cdot||_{C_\lambda^{2,\alpha}}\) instead of \(||\cdot||_{C_\lambda^{2,\alpha};\ast}\) for definining the hybrid norm in the source space \(C_\lambda^{2,\alpha}\big({\rm Sym}^2(T^*M)\big)\), then the operator \(\mathcal{L}:C_\lambda^{2,\alpha}\big({\rm Sym}^2(T^*M)\big) \to C_\lambda^{0,\alpha}\big({\rm Sym}^2(T^*M)\big)\) would \textit{not} be invertible (with universal constants).
This is because the system (\ref{Lh=f in a cusp in coordinates})
has some constant fundamental solution (coming from the root \(0\)). Therefore, in \(M_{\rm small}\) there are bounded solutions of \(\mathcal{L}h=0\) that are \textit{not} decaying.
 The fundamental solutions corresponding to the root \(0\)
are the tensors with \(e^{2r}h_{ij}=\mathrm{const.}\), but these are exactly the trivial Einstein variations (see \Cref{def:trivialeinstein}). This explains why in the source space we have to work with a weighted norm that isolates trivial Einstein variations and only considers their unweighted \(C^0\)-norms.
\end{rem}

Now we explain in detail how to use the averaging operator and the linear system of ODEs in (\ref{Lh=f in a cusp in coordinates}) to obtain growth estimates in \(M_{\rm small}\). Let \(f \in C_\lambda^{0,\alpha}\big( {\rm Sym}^2(T^*M)\big)\) be arbitrary, and let \(h\) be a solution of 
\[
	\mathcal{L}h=f,
\]
where \(\mathcal{L}\) is the elliptic operator given by \(\mathcal{L}h= \frac{1}{2}\Delta_Lh+2h\) with respect to the given metric \(g\) of \(M\). We start with considering a rank \(2\) cusp \(C\) of \(M\). Note that \(C_{\rm small}\cong T^2 \times [0,\infty)\). Let \(g_{cusp}\) be the model metric on \(C_{\rm small}\) given by \Cref{Existence of approximate cusp metric}. This satisfies
\[
	|g-g_{cusp}|_{C^2}(x)=O(\varepsilon_0e^{-\eta r(x)}),
\]
where \(r(x)=d(x,\partial C_{\rm small})\), and \(\varepsilon_0\), \(\eta\) are the constants appearing in the curvature decay condition (\ref{curvature decay}). Let \(\mathcal{L}_{cusp}\)
be the elliptic operator \(\mathcal{L}_{cusp}h=\frac{1}{2}\Delta_Lh+2h\) with respect to the model metric \(g_{cusp}\). Set \(f_c:=\mathcal{L}_{cusp}h\). Then \(|f-f_c|(x)=O(\varepsilon_0|h|_{C^2}(x)e^{-\eta r(x)})\) by the above estimate on \(|g-g_{cusp}|_{C^2}\). Thus
\begin{align*}
	|f_c|(x) \leq & |f|(x)+O\big(\varepsilon_0|h|_{C^2}(x)e^{-\eta r(x)}\big) \\
	=& O\big( ||f||_{0,\lambda}e^{-\lambda r(x)}+\varepsilon_0|h|_{C^2}(x)e^{-\eta r(x)}\big).
\end{align*}

Let \(\hat{\cdot}\) be the averaging operator with respect to the model metric \(g_{cusp}\). Using \(ii)\) of \Cref{Properties of averaging operator} we get
\begin{align*}
	|\widehat{f_c}|(r) =O\left(||f||_{0,\lambda}e^{-\lambda r}+\varepsilon_0 e^{-(\eta-2) r}\int_{T(r)}|h|_{C^2}(y) \, d\vol_2(y) \right),
\end{align*}
where we used that 
\[\vol_2(T(r))=e^{-2r}\vol_2(\partial C_{\rm small})=O(e^{-2r})\] since by definition \(\diam (\partial C_{\rm small})\) is bounded by a universal constant. Define the function \(\psi:\bbR_{\geq 0} \to \bbR\) by
\begin{equation}\label{def of psi}
	 \psi(r):=\int_{T(r)}|h|_{C^2}(y) \, d\vol_2(y).
\end{equation}
Hence
\begin{equation}\label{growth of hat(f_c)}
	|\widehat{f_c}|(r)=O\left(||f||_{0,\lambda}e^{-\lambda r}+\varepsilon_0 \psi(r)e^{-\lambda r} \right)
\end{equation}
since \(\eta \geq 2+\lambda\). 

By $iii)$ and $iv)$ of \Cref{Properties of averaging operator} and the identity
\(f_c=\mathcal{L}_{cusp}h\),  it holds 
\[
	\mathcal{L}_{cusp}\hat{h}=\widehat{f_c}.
\]
As \(\hat{h}\) and \(\widehat{f_c}\) only depend on \(r\), \(\mathcal{L}_{cusp}\hat{h}=\widehat{f_c}\) is the linear system of ODEs given by (\ref{Lh=f in a cusp in coordinates}). Due to the growth estimate (\ref{growth of hat(f_c)}), we can invoke the following two basic ODE results to obtain a growth estimate for \(\hat{h}\). 

In the formulation of these ODE results, we denote by \(I\) either  \(\bbR_{\geq 0}\) or an interval of the form \([0,R-1]\) for some \(R \geq 2\). Moreover, for a polynomial \(Q=\sum_ma_mX^m\) we write \(Q(\frac{d}{dr})\) for the differential operator \(\sum_ma_m\frac{d^m}{dr^m}\).

\begin{lem}\label{Transfer of exponential rates} Let \(Q \in \bbR[X]\) be a quadratic polynomial with two distinct real roots \(\lambda_1\), \(\lambda_2\). Let \(y:I \to \bbR\) be a solution of the ODE
\[
	Q\left(\frac{d}{dr} \right)(y)=u,
\]
where \(u:I \to \bbR\) is a function satisyfing \(u(r)=\sum_{k=1}^m O\big( \beta_k e^{\mu_k r}\big)\) for some \(\beta_k \in \bbR_{\geq 0}\), and \(\mu_k \in \bbR\setminus \{\lambda_1,\lambda_2\}\). Then
\[
	y(r)=A_1e^{\lambda_1 r}+A_2e^{\lambda_2 r}+\sum_{k=1}^mO\big(\beta_k e^{\mu_k r}\big)
\]
for some constants \(A_1,A_2 \in \bbR\). 
\end{lem}

\begin{lem}\label{Transfer of exponential rates - L^1 condition} Let \(Q \in \bbR[X]\) be a quadratic polynomial with two distinct real roots \(\lambda_1\), \(\lambda_2\). Let \(y:I \to \bbR\) be a solution of 
\[
	Q\left(\frac{d}{dr}\right)(y)=u,
\]
where \(u\) satisfies \(|u(r)|\leq e^{a r}\psi(r)\) for some \(a \in \bbR\) and \(\psi \in L^1(\bbR_{\geq 0})\). Then
\[
	y(r)=A_1e^{\lambda_1 r}+A_2e^{\lambda_2 r}+O\big(||\psi||_{L^1(\bbR_{\geq 0})}e^{ar} \big)
\]
for some \(A_1,A_2 \in \bbR\).
\end{lem}

In \Cref{Transfer of exponential rates} and \Cref{Transfer of exponential rates - L^1 condition}, the universal constant absorbed by \(O(...)\) is allowed to depend on \(\lambda_1\), \(\lambda_2\), and \(a\), but \textit{not} on \(R\) (in case \(I=[0,R-1]\)). We again refer to \Cref{big O notation} for our convention of the \(O\)-notation.

\begin{proof}[Proof of \Cref{Transfer of exponential rates} and \Cref{Transfer of exponential rates - L^1 condition}] Both of these lemmas follow easily from the explicit integral formulas for solutions of linear ODEs.
\end{proof}

In order to successfully apply \Cref{Transfer of exponential rates - L^1 condition}, we need to control the \(L^1\)-norm of the function \(\psi\) defined in (\ref{def of psi}). This is expressed in the following lemma. 

\begin{lem}\label{L^1-estimate for psi} Let the constants and the manifold \(M\) be as in \Cref{trivial Einstein variation and a priori estimate in a cusp}. Let \(f \in C^{0,\alpha}\big({\rm Sym}^2(T^*M)\big)\) with \(||f||_{C^{0,\alpha}(M)} < \infty\), and let \(h \in C^2({\rm Sym}^2(T^*M))\cap H_0^1(M)\) be a solution of
\[
	\mathcal{L}h=f.
\] 
Fix a rank 2 cusp \(C\) of \(M\), and let \(\psi:\bbR_{\geq 0} \to \bbR\) be given by
\[
	\psi(r)=\int_{T(r)}|h|_{C^2}(y) \, d\vol_2(y),
\]
where \(T(r) \subseteq C_{\rm small}\) is the torus
of all points of distance \(r\) to \(\partial C_{\rm small}\). Then \(\psi \in L^1(\bbR_{\geq 0})\), and
\[
	||\psi||_{L^1(\bbR_{\geq 0})}=O(||f||_{0,\lambda}),
\]
where \(||\cdot||_{0,\lambda}\) is the norm defined in (\ref{Hybrid 0-norm - exponential}).
\end{lem}

Not all components of the hybrid norm \(||f||_{0,\lambda}\) are needed for this estimate. In fact, we only need \(||f||_{C^0}\) and \(e^{\frac{b}{2}d(x,M_{\rm thick})}||f||_{L^2(M;\omega_x)}\) (\(x \in M_{\rm thin} \setminus M_{\rm small}\)). Moreover, the curvature decay condition (\ref{curvature decay})
is not needed for this estimate.

\begin{proof}By the co-area formula it holds
\begin{equation}\label{L^1 norm of psi=weighted L^1-norm of h}
	\int_0^\infty \psi(r) \, dr=\int_{C_{\rm small}}|h|_{C^2}(x) \, d\vol(x).
\end{equation}
Since by definition \(\diam(\partial C_{\rm small})\) is bounded by a universal constant, it is easy to see that \(\vol(C_{\rm small})\) is also bounded by a universal constant. Combining this with (\ref{L^1 norm of psi=weighted L^1-norm of h}), and using the Cauchy-Schwarz inequality yields
\begin{equation}\label{L^1 norm of psi=weighted L^2-norm of h 2}
	\int_0^\infty \psi(r) \, dr \leq C \left(\int_{C_{\rm small}}|h|_{C^2}^2(x) \, d\vol(x)\right)^{\frac{1}{2}}
\end{equation}
for some universal constant \(C\). Therefore, it suffices to obtain \(H^2\)-estimates of \(h\) in \(C_{\rm small}\).

For \(\varepsilon_0>0\) small enough it follows from \Cref{A-priori estimate with finite distance away from the small part} and Schauder estimates that
\begin{equation}\label{C^2 control at the beginning of C_small}
	|h|_{C^2}(x)=O(||f||_{0,\lambda}) \, \text{ for all }\, x \, \text{ with }\, r(x) \leq 1.
\end{equation}
Choose a smooth bump function \(\varphi:M \to [0,1]\) satisfying \(\varphi=0\) on \(M \setminus C_{\rm small}\), \(\varphi(x)=1\) for \(x \in C_{\rm small}\) with \(r(x) \geq 1\), and \(||\nabla \varphi||_{C^0(M)} \leq 2\). Applying \Cref{weighted integral estimate - smooth} with this \(\varphi\), and using (\ref{C^2 control at the beginning of C_small}) we obtain (for \(\varepsilon_0>0\) small enough)
\[
	\int_{C_{\rm small}}|h|^2 \, d\vol \leq C\int_{C_{\rm small}}|f|^2 \, d\vol + O(||f||_{0,\lambda}^2).
\]
Analogous to Step 1 in the proof of \Cref{A-priori estimate for L} we get \(H^2\)-estimates for the solutions \(h\) of \(\mathcal{L}h=f\) in terms of \(||h||_{L^2}\) and \(||f||_{L^2}\) by standard computations using integration by parts. Because of (\ref{C^2 control at the beginning of C_small}) we can control the boundary terms when invoking integration by parts on \(C_{\rm small}\). Therefore, for some universal constant \(C>0\) it holds
\begin{equation}\label{L^2 estimate inside C small}
	\int_{C_{\rm small}}|h|_{C^2}^2 \, d\vol \leq C\int_{C_{\rm small}}|f|^2 \, d\vol+O(||f||_{0,\lambda}^2).
\end{equation}
Note \(||f||_{L^2(C_{\rm small})}=O(||f||_{C^0})\) since \(\vol(C_{\rm small})\) is bounded by a universal constant. Thus the desired inequality follows from (\ref{L^1 norm of psi=weighted L^2-norm of h 2}) and (\ref{L^2 estimate inside C small}).
\end{proof}

We are now finally in a position to prove the growth estimate in a cusp.

\begin{proof}[Proof of \Cref{trivial Einstein variation and a priori estimate in a cusp}]Later when we prove that \(\mathcal{L}\) is surjective, it will be important to have an argument that only assumes \(h \in C^2\big({\rm Sym}^2(T^*M)\big)\cap H_0^1(M)\), but that does \textit{not} assume \(||h||_{C^0(M)}<\infty\). For this reason we will as long as possible only assume \(h \in C^2\big({\rm Sym}^2(T^*M)\big)\cap H_0^1(M)\), and point out from which point on we really need the assumption \(||h||_{C^0}<\infty\).

Let \(g_{cusp}\) be the model metric in the rank 2 cusp \(C\) given by \Cref{Existence of approximate cusp metric}. Because of the estimate on \(|g-g_{cusp}|_{C^2}\) in \Cref{Existence of approximate cusp metric}, it is irrelevant in \(C_{\rm small}\) whether quantities such as the \(L^2\)-norm are taken w.r.t. \(g\) or \(g_{cusp}\), and we will compute all quantities w.r.t. \(g_{cusp}\). Denote by \(\mathcal{L}_{cusp}\) the elliptic operator \(\frac{1}{2}\Delta_L+2\mathrm{id}\) with respect to the model metric \(g_{cusp}\). We define \(f_c:=\mathcal{L}_{cusp}h\). Moreover, let \(\hat{\cdot}\) denote the averaging operator with respect to the model metric \(g_{cusp}\) (see \Cref{Properties of averaging operator}). In the discussion after \Cref{Properties of averaging operator} we showed (see (\ref{growth of hat(f_c)}))
\[
	\mathcal{L}_{cusp}\hat{h}=\widehat{f_c} \quad \text{and} \quad  |\widehat{f_c}|(r)=O(||f||_{0,\lambda}e^{-\lambda r}+\varepsilon_0 \psi(r)e^{-\lambda r}),
\]
where \(\psi\) was defined in (\ref{def of psi}). Moreover, as \(\hat{h}\), \(\widehat{f_c}\) only depend on \(r\), \(\mathcal{L}_{cusp}\hat{h}=\widehat{f_c}\) is the linear system of ODEs given
by (\ref{Lh=f in a cusp in coordinates}). Namely, by (\ref{PDE for trace as ODE}), and the first two equations in (\ref{Lh=f in a cusp in coordinates}) we have
\[
	\begin{cases}
	Q_1(\frac{d}{dr})({\rm tr}(\hat{h}))&=-2{\rm tr}(\widehat{f_c}) \\
	Q_1(\frac{d}{dr})(\hat{h}_{33})&=-2(\widehat{f_c})_{33} \\
	Q_2(\frac{d}{dr})(e^r \hat{h}_{i3})&=-2e^r (\widehat{f_c})_{i3} 
	\end{cases}
\]
for some quadratic polynomials \(Q_1\) and \(Q_2\) with roots \(\{1-\sqrt{5},1+\sqrt{5}\}\) and \(\{-1,3\}\). As \(|\widehat{f_c}|\) satisfies the above growth estimate, and since \(-\lambda \nin \{1\pm \sqrt{5},-1,3\}\) we can apply \Cref{Transfer of exponential rates} and
\Cref{Transfer of exponential rates - L^1 condition}.

We know \(||\psi||_{L^1(\bbR_{\geq 0})}=O(||f||_{0,\lambda})\) from \Cref{L^1-estimate for psi} (note that \Cref{L^1-estimate for psi} does \textit{not} assume \(||h||_{C^0}<\infty\)). Thus we get from \Cref{Transfer of exponential rates} and \Cref{Transfer of exponential rates - L^1 condition} 
\begin{equation}\label{growth in cusp - growth of coefficients 1}
\begin{cases}
	{\rm tr}(\hat{h})(r)&=a_1 e^{(1-\sqrt{5})r}+a_2e^{(1+\sqrt{5})r}+O\big(||f||_{0,\lambda}e^{-\lambda r} \big); \\
	\hat{h}_{33}(r)&=b_1 e^{(1-\sqrt{5})r}+b_2e^{(1+\sqrt{5})r}+O\big(||f||_{0,\lambda}e^{-\lambda r} \big); \\
	e^r\hat{h}_{i3}(r)&=c_1^{(i)}e^{-r}\, \, \, \, \, \, \, \, \, \, \, +c_2^{(i)}e^{3r}\,\,\,\,\,\,\,\,\,\,\,+O\big(||f||_{0,\lambda}e^{-\lambda r} \big);
	\end{cases}
\end{equation}
for some constants \(a_1,a_2,b_1,b_2,c_1^{(i)},c_2^{(i)} \in \bbR\). Note that \(h \in L^2(C_{\rm small}) \subseteq L^1(C_{\rm small})\) since \(C_{\rm small}\) has finite volume, and \(\vol(T(r))=O(e^{-2r})\), where \(T(r)\subseteq C_{\rm small}\) is the torus all whose points have distance \(r\) to \(\partial C_{\rm small}\). 
Hence \(e^{-2r}|\hat{h}|(r) \in L^1(\bbR_{\geq 0})\).
In particular, \(e^{-2r} {\rm tr}(\hat{h}) (r), e^{-2r}\hat{h}_{33}(r), e^{-2r}\big(e^{r}\hat{h}_{i3}(r)\big) \in L^1(\bbR_{\geq 0})\)
$(i=1,2)$, 
and thus \(a_2=b_2=c_2^{(i)}=0\). 

We know that \(\max_{\partial C_{\rm small}}|h|=O(||f||_{0,\lambda})\) due to \Cref{A-priori estimate away from the small part}, and 
we have \(|\hat{h}|(0)=O\big( \max_{\partial C_{\rm small}}|h|\big)\) by \(ii)\) of \Cref{Properties of averaging operator}. Hence evaluating at \(r=0\) yields
\[
	a_1,b_1,c_1^{(i)}=O\big(||f||_{0,\lambda} \big).
\]
As \(\lambda < 1 < \sqrt{5}-1\) this implies
\begin{equation}\label{growth in cusp - bound on coefficients 1}
	|{\rm tr}(\hat{h})(r)|, |\hat{h}_{33}(r)|, |e^r\hat{h}_{i3}(r)|=O\big(||f||_{0,\lambda} e^{-\lambda r}\big).
\end{equation}

By the last equation in (\ref{Lh=f in a cusp in coordinates}) there is a quadratic polynomial \(Q_3\) with roots \(0\) and \(2\) so that
\begin{align*}
	Q_3\left(\frac{d}{dr}\right)(e^{2r}\hat{h}_{ij})=& 2\delta_{ij}({\rm tr}(\hat{h})-\hat{h}_{33})-2e^{2r}(\widehat{f_c})_{ij} \\
	=& O\big(||f||_{0,\lambda}e^{-\lambda r}+\varepsilon_0\psi(r)e^{-\lambda r} \big),
\end{align*}
where we used the growth rate of \(|\hat{f_c}|\) in (\ref{growth of hat(f_c)}), and the one for \({\rm tr}(\hat{h}), \hat{h}_{33}\) in (\ref{growth in cusp - bound on coefficients 1}). Again invoking \Cref{Transfer of exponential rates} and \Cref{Transfer of exponential rates - L^1 condition}, and using \Cref{L^1-estimate for psi} to estimate \(||\psi||_{L^1(\bbR_{\geq 0})}\), we conclude
\begin{equation}\label{growth in cusp - growth of coefficients 2}
e^{2r}\hat{h}_{ij}=d_1^{(i,j)}+d_2^{(i,j)}e^{2r}+O\big(||f||_{0,\lambda}e^{-\lambda r} \big)
\end{equation}
for some constants \(d_1^{(i,j)},d_2^{(i,j)} \in \bbR\). As before, \(e^{-2r}|\hat{h}|(r) \in L^1(\bbR_{\geq 0})\) implies \(d_2^{(i,j)}=0\). Exactly as before, evaluating at \(r=0\) we obtain
\begin{equation}\label{growth in cusp - bound on coefficients 2}
	d_1^{(i,j)}=O\big(||f||_{0,\lambda} \big).
\end{equation}

Define an Einstein variation \(v\) in \(C_{\rm small}\) by \(v_{ij}(r)=d_1^{(i,j)}e^{-2r}\). Note that \({\rm tr}(\hat{h})=\hat{h}_{33}+{\rm tr}(v)\), and that \({\rm tr}(v)\) is constant. Hence \(\hat{h}_{33}(r),{\rm tr}(\hat{h})(r) \xrightarrow{r \to \infty}0\) implies that \({\rm tr}(v)=0\). Therefore, \(v\) is indeed a trivial Einstein variation. From (\ref{growth in cusp - bound on coefficients 2}) we know 
\[
	|v|=O(||f||_{0,\lambda}).
\]
Moreover, (\ref{growth in cusp - bound on coefficients 1}), (\ref{growth in cusp - growth of coefficients 2}), the fact that \(d_2^{(i,j)}=0\), and the definition of \(v\) imply
\begin{equation}\label{hat(h)-v exponential estimate}
	||\hat{h}-v||_{C_\lambda^0}=O\big(||f||_{0,\lambda}\big).
\end{equation}
In particular, \(|v|=O(||f||_{0,\lambda})\) and (\ref{hat(h)-v exponential estimate}) yield
\begin{equation}\label{bound on hat(h) in C_small}
	\sup_{C_{\rm small}}|\hat{h}|=O(||f||_{0,\lambda}).
\end{equation}
We will need this fact in the proof of the surjectivity of \(\mathcal{L}\). Until this point we did \textit{not} need \(||h||_{C^0}<\infty\), but only \(h \in C^2\big({\rm Sym}^2(T^*M)\big) \cap H_0^1(M)\).

From now on we use the assumption \(||h||_{C^0}<\infty\), which implies \(||h||_{C^1}<\infty\) due to Schauder estimates. 
Since $\lambda <1$, we deduce from \(v)\) of \Cref{Properties of averaging operator} that 
 \(||h-\hat{h}||_{C_\lambda^0}<\infty\) and  hence \(||h-v||_{C_\lambda^0}<\infty\). This proves the existence of a trivial Einstein variation as stated in \Cref{trivial Einstein variation and a priori estimate in a cusp}. Uniqueness of such a trivial Einstein variation is clear because trivial Einstein variations have constant \(C^0\)-norm (with respect to \(g_{cusp}\)).

If we assume \(||h||_{C^0}, ||f||_{C^{0,\alpha}} \leq 1\), then these last considerations can be made more quantitative. Indeed, under this assumption we have \(||h||_{C^1}=O(1)\) by Schauder estimates. Recall that \(\diam (\partial C_{\rm small})\) is bounded by a universal constant due to the definition of the small part. Thus \(|h-\hat{h}|(x)=O\big( e^{-r(x)}\big)\) by \(v)\) of \Cref{Properties of averaging operator}. Together with (\ref{hat(h)-v exponential estimate}) this implies
\[
	e^{\lambda r(x)}|h-v|(x)=O(||f||_{0,\lambda}+e^{-(1-\lambda)r(x)}).
\]
This finishes the proof of (\ref{exponential growth estimate for h-v in cup}). As we already showed \(|v|=O(||f||_{0,\lambda})\), this also yields (\ref{growth estimate in a cusp - eq}). This completes the proof.
\end{proof}

The proof \Cref{Growth estimate in a tube - pinched curvature} is very similar to that of \Cref{trivial Einstein variation and a priori estimate in a cusp}, so we only highlight the differences. 

Let \(T\) be a Margulis tube of \(M\) with core geodesic \(\gamma\). Denote by \(R\) the radius of \(T_{\rm small}\), that is,
the distance from \(\gamma\) to \(\partial T_{\rm small}\). The model metric \(g_{cusp}\) given by \Cref{tubes are almost cusps - pinched curvature} is only defined on \(T_{\rm small} \setminus N_1(\gamma)\), and it satisfies \(|g-g_{cusp}|_{C^2}(x)=O\big(\varepsilon_0 e^{-\eta r(x)}+e^{2(r(x)-R) }\big)\), where \(r(x)=d(x, \partial T_{\rm small})\). Define \(\mathcal{L}_{cusp}\), \(f_c\), and \(\psi\) exactly as in the case of a cusp (the only difference being that now \(\psi\) is only defined on \([0,R-1]\)). As \(||f||_{0,\lambda},||h||_{C^0}\leq 1\) by assumption, it holds \(||h||_{C^2}=O(1)\) due to Schauder estimates. Thus we can estimate \(|\widehat{f_c}-\hat{f}|(x)=O\big(\varepsilon_0\psi(r)e^{-\lambda r}+e^{2(r-R)}\big)\). So \(|\widehat{f_c}|\) satisfies the growth estimate
\begin{equation}\label{growth of hat(f_c) - tube preversion}
	|\widehat{f_c}|(r)=O\big(||f||_{0,\lambda}W_{\lambda}(r)+\varepsilon_0\psi(r)e^{-\lambda r}+e^{2(r-R)}\big),
\end{equation}
where \(W_\lambda\) is the inverse weight entering the definition of \(||\cdot||_{C_\lambda^0}\). In \(T_{\rm small}\) it is given by \(W_\lambda(r)=e^{-\lambda r}+e^{\lambda (r-R)}\). In contrast to the definition of the cusp, \(r\) now only takes values in a bounded interval. For this reason we can \textit{not} get rid of the exponentially growing fundamental solutions as easily as in the case of a cusp. To remedy this, we evoke the following basic ODE lemma.

\begin{lem}\label{Growth estimate for ODEs} Let \(Q \in \bbR[X]\) be a quadratic polynomial with distinct real roots \(\lambda_2 \leq 0 < \lambda_1\), \(R \geq 2\), \(\psi \in L^1([0,R-1])\) so that \(||\psi||_{L^1([0,R-1])}\) is bounded by a universal constant, and let \(\mathcal{W}:[0,R-1] \to \bbR\) be a function of the form 
\[
	\mathcal{W}(r)=\sum_{k=1}^m\beta_k e^{\mu_k r}
\]
for some \(\beta_k \in \bbR_{\geq 0}\), \(\mu_k \nin \{\lambda_1,\lambda_2\}\), so that \(\mathcal{W}(r)\) is bounded by a universal constant for \(r \in [R-2,R-1]\). Let \(y:[0,R-1] \to \bbR\) be a \(C^2\) function with \(|y| \leq 1\) and
\[
	Q\left(\frac{d}{dr}\right)(y)=O\big(\mathcal{W}(r)+\psi(r)e^{-ar}\big)
\]
for some \(a \geq 0\). Then there is a universal constant \(C\) so that
\begin{equation}\label{Growth estimate for ODEs - eq1}
	|y|(r) \leq C \Big( e^{\lambda_2 r} \big(|y|(0)+\mathcal{W}(0)+||\psi||_{L^1([0,R-1])}\big)+e^{\lambda_1(r-R)}+\mathcal{W}(r)+||\psi||_{L^1}e^{-ar}\Big)
\end{equation}
for all \(r \in [0,R-1]\). In particular, 
\begin{equation}\label{Growth estimate for ODEs - eq2}
	|y|(r) \leq C\Big( |y|(0) +\mathcal{W}(0)+||\psi||_{L^1([0,R-1])} + e^{\lambda_1(r-R)}+\mathcal{W}(r)   \Big).
\end{equation}
\end{lem}

As in \Cref{Transfer of exponential rates} and \Cref{Transfer of exponential rates - L^1 condition} the universal constant \(C\) is allowed to depend on \(\lambda_1\), \(\lambda_2\), and \(a\), but \textit{not} on \(R\).

\begin{proof}As \(\mu_k \nin \{\lambda_1,\lambda_2\}\), and \(\psi \in L^1([0,R-1])\), \Cref{Transfer of exponential rates} and \ref{Transfer of exponential rates - L^1 condition} imply
\begin{equation}\label{Growth estimate for ODEs - eq3}
	y(r)=A_1e^{\lambda_1 r}+A_2e^{\lambda_2 r}+O\big( \mathcal{W}(r)+||\psi||_{L^1}e^{-ar} \big)
\end{equation}
for some constants \(A_1,A_2 \in \bbR\). For ease of notation we abbreviate \(R^\prime:=R-1\). Since \(|y|,\mathcal{W},||\psi||_{L^1}e^{-ar}=O(1)\) on \([R^\prime-1,R^\prime]\) (this uses \(a \geq 0\)), also \(A_1e^{\lambda_1 r}+A_2e^{\lambda_2 r}=O(1)\) on \([R^\prime-1,R^\prime]\). Using this at \(r=R^\prime\) and \(r=R^\prime-1\) gives the linear system of equations
\[
	\begin{pmatrix}
1 & 1 \\
e^{-\lambda_1} & e^{-\lambda_2}
\end{pmatrix}
\begin{pmatrix}
A_1e^{\lambda_1 R^\prime} \\
A_2 e^{\lambda_2 R^\prime}
\end{pmatrix}		
=
\begin{pmatrix}
O(1) \\
O(1)
\end{pmatrix}	.
\]
The operator norm \(||L^{-1}||_{\rm op}\) of the inverse of the linear operator \(L=\begin{pmatrix}
1 & 1 \\
e^{-\lambda_1} & e^{-\lambda_2}
\end{pmatrix}\) only depends on \(\lambda_1,\lambda_2\). Hence \(A_1e^{\lambda_1 R^\prime}\) and \(A_2e^{\lambda_2 R^\prime}\) are bounded by a universal constant. We know \(y(0)=A_1+A_2+O\big( \mathcal{W}(0)+||\psi||_{L^1} \big)\) due to (\ref{Growth estimate for ODEs - eq3}). Invoking (\ref{Growth estimate for ODEs - eq3}) once more we obtain
\begin{align*}
	y(r)=&A_1e^{\lambda_1 r}+A_2e^{\lambda_2 r} \\
	&+O\big(\mathcal{W}(r)+||\psi||_{L^1}e^{-ar} \big) \\
	=& e^{\lambda_1 (r-R^\prime)} \big(A_1e^{\lambda_1 R^\prime} \big) + e^{\lambda_2 r} \Big(y(0)-e^{-\lambda_1 R^\prime}\big(A_1e^{\lambda_1 R^\prime} \big)+O\big(\mathcal{W}(0)+||\psi||_{L^1}\big) \Big) \\
	 &+ O\big( \mathcal{W}(r)+||\psi||_{L^1}e^{-ar} \big) \\
	=&  e^{\lambda_1 (r-R^\prime)} O(1) + e^{\lambda_2 r} \Big(y(0)+e^{-\lambda_1 R^\prime}O(1)+O\big(\mathcal{W}(0)+||\psi||_{L^1}\big) \Big) \\
	&+ O\big( \mathcal{W}(r)+||\psi||_{L^1}e^{-ar} \big),
\end{align*}
that is, there is a universal constant \(C\) so that
\[
	|y|(r) \leq C \Big(e^{\lambda_1 (r-R^\prime)}+e^{\lambda_2 r} \big(|y|(0)+e^{-\lambda_1 R^\prime}+\mathcal{W}(0)+||\psi||_{L^1} \big)+\mathcal{W}(r)+||\psi||_{L^1}e^{-ar}\Big)
\]
for all \(0 \leq r \leq R^\prime\). As \(\lambda_2 < \lambda_1\), so that \(e^{\lambda_2 r}e^{-\lambda_1 R^\prime}\leq e^{\lambda_1(r-R^\prime)}\), and \(R^\prime=R-1\), this completes the proof of (\ref{Growth estimate for ODEs - eq1}).
\end{proof}

To apply \Cref{Growth estimate for ODEs} we need to make sure that the growth rates of \(|\widehat{f_c}|\) are different from the fundamental solutions of the system of ODEs in (\ref{Lh=f in a cusp in coordinates}). For this reason, we weaken the growth control on \(|\widehat{f_c}|\) in (\ref{growth of hat(f_c) - tube preversion}) to
\begin{equation}\label{growth of hat(f_c) - tube}
	|\widehat{f_c}|(r)=O\big(||f||_{0,\lambda}W_{\lambda}(r)+\varepsilon_0\psi(r)e^{-\lambda r}+e^{\frac{3}{2}(r-R)}\big).
\end{equation}
Analogous to the case of a cusp, it still holds \(||\psi||_{L^1([0,R-1])}=O(||f||_{0,\lambda})\). Indeed, the proof of \Cref{L^1-estimate for psi} goes through without modification. Now the proof of \Cref{Growth estimate in a tube - pinched curvature} follows by applying \Cref{Growth estimate for ODEs} with 
\[
	\mathcal{W}(r)=||f||_{0,\lambda}W_\lambda(r)+e^{\frac{3}{2}(r-R)}=||f||_{0,\lambda}\big(e^{-\lambda r}+e^{\lambda (r-R)} \big)+e^{\frac{3}{2}(r-R)}
\]
and \(\psi\) defined in (\ref{def of psi}) componentwise to the linear system of ODEs \(\mathcal{L}_{cusp}\hat{h}=\hat{f_c}\) given in (\ref{Lh=f in a cusp in coordinates}). (To be precise: Similar to the proof of \Cref{Growth estimate in a tube - pinched curvature}, one first applies \Cref{Growth estimate for ODEs} to the equations for \({\rm tr}(\hat{h}),\hat{h}_{33},e^r\hat{h}_{i3}\), and then for \(e^{2r}\hat{h}_{ij}\) one applies \Cref{Growth estimate for ODEs} to the above \(\mathcal{W}\) \(+\) the growth of \({\rm tr}(\hat{h}),\hat{h}_{33}\) obtained from (\ref{Growth estimate for ODEs - eq1}).) This gives growth estimates for \(\hat{h}\). By \(v)\) of \Cref{Properties of averaging operator} we again know \(|h-\hat{h}|(x)=O(e^{-r(x)})\). Remembering that \(r_\gamma=R-r\) will yield the estimate (\ref{Growth estimate in small part}), thus finishing the proof of \Cref{Growth estimate in a tube - pinched curvature}.


\subsection{A priori estimates}\label{Subsec: a priori estimates in non-compact case}

The goal of this section is to prove the a priori estimate of \Cref{Invertibility of L - non-compact case}, that is, that there exists a universal constant \(C\) so that
\[
	||h||_{2,\lambda;\ast} \leq C||\mathcal{L}h||_{0,\lambda} 
\]
for all \(h \in C_\lambda^{2,\alpha}\big({\rm Sym}^2(T^*M)\big)\). The norm \(||\cdot||_{2,\lambda;\ast}\) is a mixture of weighted Sobolev norms and the decomposition norm \(||\cdot||_{C_\lambda^{2,\alpha};\ast}\). By \Cref{integration by parts works in finite volume case}, the weighted Sobolev-estimate \(||h||_{H^2(M;\omega_x)} \leq C||f||_{L^2(M;\omega_x)}\) 
still holds if \(M\) is non-compact but has finite volume. Therefore, it suffices to prove an a priori estimate 
\(
	||h||_{C_\lambda^{2,\alpha}(M);\ast} \leq C||\mathcal{L}h||_{0,\lambda}.
\)
We do this in two steps. First we establish global \(C^{2,\alpha}\)-estimates, and then prove the estimate for the \(\ast\)-norm.

For the proof of a global \(C^{2,\alpha}\)-estimate, 
we adapt the arguments in \cite[Lemma 6.1]{Bamler2012}. 

\begin{prop}\label{C^0 estimate from exponential hybrid norm}For all \(\alpha \in (0,1)\), \(\Lambda \geq 0\), \(\lambda \in (0,1)\), \(b>1\), \(\delta \in (0,2)\), \(r_0 \geq 1\), and \(\eta \geq 2+\lambda\) there 
exist constants \(\varepsilon_0\), \(\bar{\epsilon}_0\) and \(C >0\) with the following property. Let \(M\) be a Riemannian \(3\)-manifold of finite volume that satisfies 
\[
	|\sec+1|\leq \varepsilon_0, \quad ||\nabla \Ric||_{C^0(M)} \leq \Lambda ,
\]
and
\[
	\max_{\pi \subseteq T_xM}|\mathrm{sec}(\pi)+1|, \, |\nabla R|(x), \, |\nabla^2R|(x) \leq \varepsilon_0 e^{-\eta d(x,\partial M_{\rm small})} \quad \text{for all } \, x \in M_{\rm small}.
\]
Then for all \(h \in C_\lambda^{2,\alpha}\big( Sym^2(TM)\big)\) it holds
\[
	||h||_{C^{2,\alpha}(M)} \leq C ||\mathcal{L}h||_{0,\lambda},
\]
where \(||\cdot||_{0,\lambda}\) is the norm defined in (\ref{Hybrid 0-norm - exponential}) with respect to any \(\bar{\epsilon}\leq \bar{\epsilon}_0\).
\end{prop}

\begin{proof}\textbf{Step 1 (Reducing the problem):} Let \(\varepsilon_0>0\) be small enough so that one can apply \Cref{trivial Einstein variation and a priori estimate in a cusp}, \Cref{Growth estimate in a tube - pinched curvature}, and \Cref{Einstein variations with decay are trivial} below, and choose \(\bar{\epsilon}_0 > 0\) so that \Cref{pointwise C^0 estimate given inj rad bound} applies. Due to Schauder estimates (\Cref{Schauder for tensors}) it suffices to prove a \(C^0\)-estimate. 

Arguing by contradiction, assume that such a $C^0$-estimate does not hold.  
Then there exist a sequence of counterexamples, given by a sequence of finite volume Riemannian \(3\)-manifolds \((M^i,g^i)\) satisfying the curvature assumptions stated in \Cref{C^0 estimate from exponential hybrid norm}, and tensors \(h^i \in C_\lambda^{2,\alpha}\big({\rm Sym}^2(T^*M)\big)\) so that
\[
	||h^i||_{C^0(M^i)}=1 \quad \text{and} \quad ||\mathcal{L}_{g^i}h^i||_{0,\lambda} \xrightarrow{i \to \infty}0 ,
\]
where \(||\cdot||_{0,\lambda}\) is the norm defined in (\ref{Hybrid 0-norm - exponential}) with respect to some \(\bar{\epsilon}^i \leq \bar{\epsilon}_0\). Abbreviate \(f^i:=\mathcal{L}_{g^i}h^i\). Choose points \(x^i \in M^i\) with \(|h^i|(x^i) \geq \frac{1}{2}\).  

If \(x^i \in M^i_{\rm thick}\), then \Cref{pointwise C^0 estimate given inj rad bound} shows \(||h^i||_{C^0}\leq C||f^i||_{0} \to 0\), where \(||\cdot||_0\) is the norm defined in (\ref{Hybrid 0-norm}) with respect to \(\bar{\epsilon}^i \leq \bar{\epsilon}_0\). This is a contradiction.

Hence \(x^i \in M_{\rm thin}^i\) for all \(i\) large enough. By \Cref{A-priori estimate away from the small part}, and the definition (\ref{Hybrid 0-norm - exponential}) 
of \(||\cdot||_{0,\lambda}\), it holds \(|h^i|(x)\leq C||f^i||_{0,\lambda}\) for all \(x \in M_{\rm thin}^i\setminus M_{\rm small}^i\). Thus \(x^i \in M^i_{\rm small}\) for all \(i\) large enough. After passing to a subsequence, either \(x^i\) is contained in a cusp \(C_{l^i} \subseteq M^i\) for all \(i\), or \(x^i\) is contained in a Margulis tube \(T_{k^i} \subseteq M^i\) for all \(i\). We distinguish these two cases.

\textbf{Step 2 (Cusps):} First we consider the case that \(x^i\) is contained in a cusp \(C_{l^i}\). Abbreviate \(r^i:=d(x^i,\partial(C_{l^i})_{\rm small})\). By the growth estimate (\ref{growth estimate in a cusp - eq}) in \Cref{trivial Einstein variation and a priori estimate in a cusp} we know that
\[
	|h^i|(x^i) \leq C\big(||f^i||_{0,\lambda}+e^{-r^i} \big)
\]
for a universal constant \(C\). As \(||f^i||_{0,\lambda} \to 0\), this shows that for \(i\) large enough it holds
\[
	\frac{1}{2}\leq |h^i|(x^i)\leq \frac{1}{4}+Ce^{-r^i},
\]
and thus \(r^i \leq \log(4C)\), that is, \(r^i\) is bounded by the universal constant \(\bar{R}=\log(4C)\). By \Cref{A-priori estimate with finite distance away from the small part} there is some \(C(\bar{R})\) so that
\[
	|h^i|(x) \leq C(\bar{R})||f^i||_{0,\lambda}
\]
for all \(x \in N_{\bar{R}}(C_{l^i}\setminus (C_{l^i})_{\rm small})\). In particular, \(|h^i|(x^i)\leq C(\bar{R})||f^i||_{0,\lambda}\). However, this contradicts \(|h^i|(x^i) \geq \frac{1}{2}\) and \(||f^i||_{0,\lambda} \to 0\).

\textbf{Step 3.1 (Radius of \((T_{k^i})_{\rm small}\)):} The goal of this step is to show 
\[
	R_{k^i} \to \infty \quad \text{as }i \to \infty,
\]
where \(R_{k^i}\) is the radius of the small part \((T_{k^i})_{\rm small}\), that is, the distance of the core geodesic \(\gamma_{k^i}\) to \(\partial (T_{k^i})_{\rm small}\). Assume that this is wrong. Then, after passing to a subsequence, it holds \(R_{k^i} \leq \bar{R}\) for all \(i\) for some \(\bar{R} > 0\). Choose a finite cover \(\hat{T}_{k^i} \to T_{k^i}\) so that \(\ell(\hat{\gamma}_{k^i}) \in [1,10]\). Note that the \(\hat{T}_{k^i}\) have uniform lower bounds on the injectivity radius, uniform bounds on the sectional curvature and the covariant derivative of the Ricci tensor. Moreover, \(d(\hat{x}^i,\hat{\gamma}_{k^i})\leq R_{k^i} \leq \bar{R}\) where $\hat x^i$ is a preimage of $x^i$ in $\hat T_{k^i}$. 
Therefore, after passing to a subsequence, it holds
\[
	(\hat{T}_{k^i},\hat{x}^i) \xrightarrow{\text{pointed }\, C^{2,\beta}} (T^\infty,x^\infty),
\]
where the convergence is pointed \(C^{2,\beta}\)-convergence for some \(\beta \in (0,\alpha)\). Here \(T^\infty\) is a negatively curved tube of possibly finite length (see the discussion before \Cref{Einstein variations with decay are trivial} for our definition of a negatively curved tube). After passing to a subsequence we may assume that the limit \(R_{\infty}=\lim_i R_{k^i} \in [0,\bar{R}]\) exists. Since the distance of \((T_{k^i})_{\rm small}\) to \(M_{\rm thick}\) is at least \(\mu/2\) by \Cref{containedinthin}, the length of \(T^\infty\), that is \(\sup_{x \in T^\infty}d(x,\gamma_{\infty})\), is at least \(R_{\infty}+\mu/2\).

Denote the lifts of \(h^i\) and \(f^i\) to \(\hat{T}_{k^i}\) by  \(\hat{h}^i\) and \(\hat{f}^i\). Since \(||\hat{h}^i||_{C^{2,\alpha}(M^i)}\) is uniformly bounded (due to Schauder estimates),
after passing to a subsequence we may assume that \(\hat{h}^i \to h^\infty\) in the pointed \(C^{2,\beta}\)-sense. From \(|h^i(x^i)|\geq \frac{1}{2}\) it follows \(|h^\infty(x^\infty)|\geq \frac{1}{2}\). Similarly, \(\hat{f}^i \to 0\) in the pointed \(C^{0,\beta}\)-sense. Due to the stability of elliptic PDEs, we get \(\mathcal{L}^\infty h^\infty=0\). 

\Cref{A-priori estimate away from the small part} shows that 
\[
	\max_{T_{k^i}\setminus (T_{k^i})_{\rm small}}|h^i|\leq C||f^i||_{0,\lambda}.
\]
The same still holds for \(\hat{h}^i\). As \(R_{k^i} \to R_{\infty}\), and \(||f^i||_{0,\lambda} \to 0\), this implies \(h^\infty=0\) outside of \(N_{R_{\infty}}(\gamma^\infty)\). In particular, this shows that \(h^\infty\) has compact support because the length of \(T^\infty\) is at least \(R_{\infty}+\mu/2\). Note that \(|\sec(M^i)+1|\leq \varepsilon_0\) implies \(|\sec(T^\infty)+1|\leq \varepsilon_0\). By our choice of \(\varepsilon_0\) we can apply \Cref{Einstein variations with decay are trivial} (or rather the comment following it) and conclude \(h^\infty=0\). This contradicts \(|h^\infty(x^\infty)|\geq \frac{1}{2}\). Therefore, \(R_{k^i} \to \infty\). 

\textbf{Step 3.2 (Estimate in a tube):} 
By the growth estimate (\ref{Growth estimate in small part}), we know that for all \(x \in (T_{k^i})_{\rm small}\setminus N_1(\gamma_{k^i})\) it holds
\begin{equation}\label{growth estimate in small part 2}
	|h^i|(x) \leq C \Big(||f^i||_{0,\lambda}+e^{- r_{\partial T^i}(x)}+e^{-\frac{3}{2} r_{\gamma_{k^i}}(x) }\Big),
\end{equation}
where \(r_{\partial T^i}(x)=d\big(x,\partial (T_{k^i})_{\rm small}\big)\) and \(r_{\gamma_{k^i}}(x)=d(x,\gamma_{k^i})\). Recall that \(|h^i(x^i)|\geq \frac{1}{2}\), and \(||f^i||_{0,\lambda} \to 0\). Thus there is a subsequence so that either \(r_{\gamma_{k^i}}(x^i)\) or \(r_{\partial T^i}(x^i)\) stays bounded. We show that both these cases lead to a contradiction, thus completing the proof. Note that these two cases exclude each other because of Step 3.1.

\textit{Case 1: \(r_{\partial T^i}(x^i)\) is bounded}

Let \(\bar{R} \in \bbR\) be so that \(r_{\partial T^i}(x^i) \leq \bar{R}\) for all \(i \in \bbN\). In Step 3.1 we showed that the radius \(R_{k^i}\) of \((T_{k^i})_{\rm small}\) goes to infinity. In particular, \(N_{\bar{R}}(M^i\setminus M^i_{\rm small})\) is disjoint from the \(1\)-neighbourhood of the core geodesic \(\gamma_{k^i}\) for all \(i\) large enough. 
Thus by applying \Cref{A-priori estimate with finite distance away from the small part} to the Margulis tube \(T_{k^i}\), we obtain that there is some constant \(C(\bar{R})\) so that
\[
	|h^i|(x) \leq C(\bar{R})||f^i||_{0,\lambda}
\]
for all \(x \in N_{\bar{R}}(T_{k^i}\setminus (T_{k^i})_{\rm small})\). In particular, \(|h^i|(x^i) \leq C(\bar{R})||f^i||_{0,\lambda} \to 0\). But this contradicts \(|h^i|(x^i) \geq \frac{1}{2}\).

\textit{Case 2: \(r_{\gamma_i}(x^i)\) is bounded}

Analogous to Step 3.1 (that is, going to appropriate covers and taking a convergent subsequence) we can construct a tensor \(h^\infty\) on a tube \(T^\infty\) with the following properties: it solves \(\mathcal{L}^\infty h^\infty=0\), and there is a point \(x^\infty \in T^\infty\) so that \(|h^\infty(x^\infty)|\geq \frac{1}{2}\). Moreover, since \(R_{k^i} \to \infty\) the tube \(T^\infty\) is complete and of infinite length. As in Step 3.1 it holds \(|\sec(T^\infty)+1|\leq \varepsilon_0\). Moreover, (\ref{growth estimate in small part 2}), \(||f^i||_{0,\lambda}\to 0\), and \(r_{\partial T^i}(x^i) \to\infty\) imply 
\[
	|h^\infty|(x) \leq Ce^{-\frac{3}{2} r_{\gamma^\infty}(x)}
\]
for all \(x \in T^\infty \setminus N_1(\gamma_\infty)\). So we can apply \Cref{Einstein variations with decay are trivial} and conclude \(h^\infty=0\). However, this contradicts \(|h^\infty(x^\infty)|\geq \frac{1}{2}\).
\end{proof}

In the proof of \Cref{C^0 estimate from exponential hybrid norm} we used the next lemma several times. It is the analog of Proposition 8.3 in \cite{Bamler2012}. However, since we are in the presence of pinched negative curvature there is a much shorter proof than that in \cite{Bamler2012}. In its formulation, a complete negatively curved solid $3$-torus 
is the quotient $T=\Gamma\backslash \tilde{M}$ of a simply connected negatively curved \(3\)-manifold \(\tilde{M}\) by an infinite cyclic group $\Gamma$ of loxodromic isometries. 
The core geodesic of such a solid torus is the projection of the axis \(\tilde{\gamma}\) of the elements of $\Gamma$.

\begin{lem}\label{Einstein variations with decay are trivial}There exists \(\varepsilon_0 > 0\) with the following property. Let \(T\) be a complete solid 
\(3\)-torus with \(|\sec(T)+1| \leq \varepsilon_0\), and let \(h\) be a symmetric \(C^2\)-tensor that solves \(\mathcal{L}h=0\). Assume that for  some constant \(C>0\)
\[
	|h|(x) \leq C e^{-\frac{3}{2} r_{\gamma}(x)}
\]
for all \(x \in T\setminus N_1(\gamma)\), where \(\gamma\) is the core geodesic. Then \(h\) vanishes identically. 
\end{lem}

The statement also holds if \(T\) is of finite length (that is, \(T=\Gamma \backslash N_{R}(\tilde{\gamma})\) for some \(R>0\)) if we assume that \(h\) has compact support. Indeed, the same proof will go through since the the argument mostly relies on integration by parts.

\begin{proof}We first argue that \({\rm tr}(h)=0\). Taking the trace in \(\mathcal{L}h=0\) yields 
\[\Delta \big({\rm tr}(h) \big)+2{\rm tr}(h)=0.\] 
By assumption, \(|{\rm tr}(h)|(x) \to 0\) as \(r_{\gamma}(x) \to \infty\). 
So there is some \(x_0 \in T\) with \(|{\rm tr}(h)|(x_0)=||{\rm tr}(h)||_{C^0(T)}\). By possibly replacing \(h\) with \(-h\) we may assume that \({\rm tr}(h)(x_0) \geq 0\). 
Recall that by our sign convention \(-\Delta u={\rm tr}(\nabla^2 u)\). As \({\rm tr}(h)\) assumes its maximum at \(x_0\), we have \(-\Delta \big({\rm tr}(h) \big)(x_0)\leq 0\). Hence
\[
	||{\rm tr}(h)||_{C^0(T)}={\rm tr}(h)(x_0)=-\frac{1}{2}\Delta \big({\rm tr}(h) \big)(x_0) \leq 0
\]
and consequently \({\rm tr}(h)=0\).

By regularity of solutions of elliptic equations and Schauder theory, 
 \(h\) is smooth and 
 all covariant derivatives \(\nabla^kh\) satisfy the same kind of estimate, that is, \(|\nabla^kh|(x)\leq C_k e^{-\frac{3}{2} r_\gamma(x)}\). Jacobi field comparison shows that \(\mathrm{area}(\partial N_R(\gamma))=O(e^{\frac{5}{2}R})\) if \(\varepsilon_0>0\) is small enough. Hence \(\nabla^kh \in L^2(T)\) for all \(k \in \bbN\). Therefore,  \cite{Gaffney1954} shows that one can apply integration by parts to \(h\) and all its derivatives.

The proof of the Poincaré inequality (\Cref{Poincare inequality}) only needed integration by parts and some tensor calculus. Hence the Poincaré inequality also holds in the present situation, and we may apply it to \(h\) since \(h\) has vanishing trace. Thus if \(\varepsilon_0>0\) is small enough it holds
\[
	||h||_{L^2(T)}\leq \frac{1}{3-c\,\varepsilon_0}||\nabla h||_{L^2(T)},
\]
where \(c=c(3)\) is the constant from \Cref{Poincare inequality}. Recall that \(\mathcal{L}h=\frac{1}{2}\Delta h +\frac{1}{2}\Ric(h)+2h\). Since \({\rm tr}(h)=0\), it follows from Lemma \ref{Estimate for (Ric(h),h)} that \(\frac{1}{2}\langle \Ric(h),h\rangle \geq -\big( 3+c^\prime \varepsilon_0\big)|h|^2\)  for a constant \(c^\prime > 0\). Therefore, applying \((\cdot,h)_{L^2(T)}\) to the equation \(\mathcal{L}h=0\) yields (remembering that we are allowed to integrate by parts)
\[
	0=\big(\mathcal{L}h,h \big)_{L^2(T)} \geq \frac{1}{2}||\nabla h||_{L^2(T)}-(1+c^\prime \varepsilon_0)||h||_{L^2(T)} \geq \left(\frac{3-c\,\varepsilon_0}{2}-(1+c^\prime \varepsilon_0) \right)||h||_{L^2(T)}.
\]
This implies \(h=0\) if we choose \(\varepsilon_0>0\) small enough so that \(\left(\frac{3-c\, \varepsilon_0}{2}-(1+c^\prime \varepsilon_0) \right)>0\).
\end{proof}

Our proof of the a priori estimate in \Cref{Invertibility of L - non-compact case} is a variation 
of that for \cite[Proposition 5.1]{Bamler2012}. We explain the central ideas first before presenting the complete proof. It suffices to prove an a priori estimate for the \(||\cdot||_{C_\lambda^0;\ast}\)-norm. Similar to the proof of \Cref{C^0 estimate from exponential hybrid norm} we assume that such an estimate does not hold. So there is a sequence \((M^i,g^i)\) of Riemannian manifolds satisyfing the assumptions of \Cref{Invertibility of L - non-compact case}, and \(h^i \in C_{\lambda}^0\big({\rm Sym}^2(T^*M)\big)\) so that
\[
	||h^i||_{C_\lambda^0;\ast}=1 \quad \text{and} \quad ||f^i||_{0,\lambda} \to 0,
\]
where \(f^i:=\mathcal{L}_{g^i}h^i\). Consider the canonical decompositions
\[
	h^i=\bar{h}^i+\sum_k \rho_k u_k^i+\sum_l \varrho_lv_l^i,
\]
given by \Cref{Canonical choice of trivial Einstein variation} resp. \Cref{trivial Einstein variation and a priori estimate in a cusp}, so that (up to a multiplicative universal constant) it holds
\[
	||h^i||_{C_\lambda^0;\ast}=||\bar{h}^i||_{C_\lambda^0}+\max_{k}|u_k^i|+\max_{l}|v_l^i|.
\]
In this outline we only consider tubes. From \Cref{Canonical choice of trivial Einstein variation} (or rather its proof) we know that \(|u_k^i|=O(||h^i||_{C^0})\), and hence \(|u_k^i| \to 0\) by \Cref{C^0 estimate from exponential hybrid norm}. For simplicity assume that \(||\bar{h}^i||_{C_\lambda^0}=1\), and that there exists a point \(x^i\) in some Margulis tube \(T_{k^i}\) so that \(|\bar{h}^i|_{C_\lambda^0}(x^i)=\frac{1}{W_\lambda(x^i)}|\bar{h}^i|_{C^0}(x^i)=1\). Note that \(x^i \in (T_{k^i})_{\rm small}\),
\[
	r_{\partial T}(x^i) \to \infty \quad \text{and} \quad r_{\gamma_{k^i}}(x^i)\to \infty,
\]
where \(r_{\partial}(x)=d(x,\partial (T_{k^i})_{\rm small})\) and \(r_{\gamma_{k^i}}(x)=d(x,\gamma_{k^i})\) for the core geodesic \(\gamma_{k^i}\) of \(T_{k^i}\). Indeed, otherwise the weight \(\frac{1}{W_\lambda(x^i)}\) is bounded from above by some constant, and so \Cref{C^0 estimate from exponential hybrid norm} would yield a contradiction. 

Abbreviate \(r:=r_{\partial T}\). For simplicity assume that \(r(x^i)=\frac{R_{k^i}}{2}\), 
where \(R_{k^i}\) is the radius of \((T_{k^i})_{\rm small}\), so that the weight \(\frac{1}{W_\lambda}\) is maximal at \(x^i\). 
Let 
\(s(x)=r(x)-\frac{R_{k^i}}{2}\) be the radial function centered at \(x^i\). 
By construction, 
the rescaled tensors \(\underline{h}^i:=\frac{1}{W_\lambda(x^i)}\bar{h}^i\) satisfy a growth estimate
\begin{equation}\label{growthofh}
	|\underline{h}^i|(x) \leq C \big(e^{-\lambda s(x)}+e^{\lambda s(x)} \big)
\end{equation}
for some constant \(C\) (independent of \(i\)). 
As in the proof \Cref{C^0 estimate from exponential hybrid norm}, after passing to suitable covers and taking a convergent subsequence, we obtain a two sided infinite hyperbolic cusp \((T^2 \times \bbR,g_{cusp})\) and 
a tensor field \(\underline{h}^\infty\) only depending on the radial coordinate \(s\) so that \(|\underline{h}^\infty|(x^\infty)=1\), \(\mathcal{L}^\infty\underline{h}^\infty=0\), and 
\[
	|\underline{h}^\infty|(x)\leq C\big(e^{-\lambda s(x)}+e^{\lambda s(x)} \big).
\]
As \(\underline{h}^\infty\) only depends on \(s\), \(\mathcal{L}^\infty\underline{h}^\infty=0\) is the linear system of ODEs in (\ref{Lh=f in a cusp in coordinates}). If \(|\lambda|\) is smaller than the absolute value of the non-zero exponential growth rates of the fundamental solutions of (\ref{Lh=f in a cusp in coordinates}), the growth condition (\ref{growthofh}) 
implies that \(\underline{h}^\infty\) is a trivial Einstein variation (see \Cref{Einstein variations with growth estimate}). 

On the other hand, by definition, \(\bar{h}^i\) satisfies
\[
	|\bar{h}^i|(x^i) \leq |\bar{h}^i-u|(x^i)
\]
for any trivial Einstein variation \(u\) on \((T_{k^i})_{\rm small}\) (see the proof of \Cref{Canonical choice of trivial Einstein variation}). This implies 
\[
	|\underline{h}^\infty|(x^\infty) \leq |\underline{h}^\infty-u|(x^\infty)
\]
for any trivial Einstein variation \(u\) on \((T^2\times \bbR,g_{cusp})\). But as \(\underline{h}^\infty\) is itself a trivial Einstein variation, choosing \(u=\underline{h}^\infty\) implies \(|\underline{h}^\infty|(x^\infty)=0\). But this contradicts \(|\underline{h}^\infty|(x^\infty)=1\) (which we know from \(|\underline{h}^i|(x^i)=1\)).

From this outline the following is apparant. 
Inside the small part of a tube, the key requirement on the weight \(\frac{1}{W_\lambda}\) is that away from its maximum, the decay rate 
of the weight is strictly smaller than the absolute value of the non-zero exponential growth rates of the fundamental solutions of the linear system of ODEs in (\ref{Lh=f in a cusp in coordinates}).

Now we present the complete proof.

\begin{proof}[Proof of the a priori estimate in \Cref{Invertibility of L - non-compact case}]

\textbf{Step 1 (Set-Up):}  As mentioned in the beginning of this section, we know that the weighted integral estimates hold, 
and so we only have to show \(||h||_{C_\lambda^{2,\alpha};\ast}\leq C||\mathcal{L}h||_{0,\lambda}\) for a universal constant \(C\). 
Because of the Schauder estimates for the \(\ast\)-norm (\Cref{Schauder estimates for decomposition norm}), it suffices to prove 
\begin{equation}\label{starestimate}
||h||_{C_{\lambda}^{0}(M);\ast} \leq C ||\mathcal{L}h||_{0,\lambda}.\end{equation}
Denote by \(\varepsilon_0,\bar{\epsilon}_0\) the constants obtained in \Cref{C^0 estimate from exponential hybrid norm}.

Arguing by contradiction, we assume that a constant $C$ as in (\ref{starestimate}) does not exist.
Then there exists a sequence of finite volume Riemannian \(3\)-manifolds \((M^i,g^i)\) satisfying the curvature assumptions in \Cref{Invertibility of L - non-compact case}, and tensor fields \(h^i\) such that
\[
	||h^i||_{C_{\lambda}^{0}(M^i);\ast}=1 \, \, \text{for all } \, i \in \bbN \quad \text{and} \quad ||f^i||_{0,\lambda} \to 0, \, \text{as} \,\, i \to \infty,
\]
where as before we abbreviate \(f^i:=\mathcal{L}_{g^i}h^i\).

Consider the canonical decomposition 
\[
	h^i=\bar{h}^i+\sum_{k}\rho_k^i u_k^i+\sum_{l}\varrho_{l}^iv_{l}^i,
\]
where \(u_k^i\) and \(v_l^i\) are the canonical choices of trivial Einstein variations on a Margulis tube and a cusp, respectively,  
given by \Cref{Canonical choice of trivial Einstein variation} and \Cref{trivial Einstein variation and a priori estimate in a cusp}. So it holds
\[
	||h^i||_{C_{\lambda}^{0}(M^i); \ast} \leq ||\bar{h}^i||_{C_{\lambda}^{0}(M)}+\max_{k}|u_k^i| +\max_{l}|v_l^i| \leq C \big(||h^i||_{C_{\lambda}^{0}(M^i); \ast}+||f^i||_{0,\lambda}\big)
\]
for a universal constant \(C\). We know \(|u_k^i|\leq C||h^i||_{C^{0}(M)}\) from \Cref{Canonical choice of trivial Einstein variation} (or rather its proof) and \(||h^i||_{C^0}\leq C||f^i||_{0,\lambda}\) by \Cref{C^0 estimate from exponential hybrid norm}, so that \(\max_{k}|u_k^i| \to 0\). Also \(\max_{l}|v_l^i| \to 0\) as we have \(|v_l^i|=O(||f^i||_{0,\lambda})\) by \Cref{trivial Einstein variation and a priori estimate in a cusp}. 
  Together with \Cref{C^0 estimate from exponential hybrid norm}, this shows that for all \(i\) large enough it holds
\[
	\frac{3}{4}\leq ||\bar{h}^i||_{C_{\lambda}^{0}} \leq C \quad \text{and} \quad ||\bar{h}^i||_{C^{0}(M)} \to 0, \, \text{as} \,\, i \to \infty.
\]

Choose \(x^i \in M^i\) with \(|\bar{h}^i|_{C_{\lambda}^0}(x^i)=\frac{1}{W_\lambda(x^i)}|\bar{h}^i|(x^i)\geq \frac{1}{2}\). Since the inverse weight function \(W_\lambda\) 
is constant to \(1\) outside \(M_{\rm small}\), it holds \(|\bar{h}^i|_{C_{\lambda}^0}(x)=|\bar{h}^i|_{C^0}(x)\) for \(x \nin M_{\rm small}^i\) and hence 
\(x^i \in M_{\rm small}^i\). 
After passing to a subsequence, either \(x^i\) is contained 
in a cusp \(C_{l^i}\) for all \(i\), or \(x^i\) is contained in a Margulis tube \(T_{k^i}\) for all \(i\). We distinguish these two cases.

\textbf{Step 2 (Cusps):} We start by considering the case that \(x^i\) is contained in a cusp \(C_{l^i}\). Abbreviate \(r^i=d(x^i,\partial (C_{l^i})_{\rm small})\). The growth estimate (\ref{exponential growth estimate for h-v in cup}) of \Cref{trivial Einstein variation and a priori estimate in a cusp} shows that for some universal constant \(C\), we have 
\[
	\frac{1}{2}\leq |\bar{h}^i|_{C_{\lambda}^0}(x^i)\leq C \big(||f^i||_{0,\lambda}+e^{-(1-\lambda)r^i} \big),
\]
and thus 
\[
	\frac{1}{2}\leq \frac{1}{4}+Ce^{-(1-\lambda)r^i}
\]
for large enough \(i\). Hence \(r^i \leq (1-\lambda)^{-1}\log(4C)\), that is, \(r^i\) is bounded by the universal constant \(\bar{R}:=(1-\lambda)^{-1}\log(4C)\). But then \(\frac{1}{2} \leq |\bar{h}^i|_{C_\lambda^0}(x^i) \leq e^{\lambda \bar{R}}||\bar{h}^i||_{C^0} \to 0\), which is a contradiction.

\textbf{Step 3 (Tubes):} It remains to consider the case that \(x^i\) is contained in a Margulis tube \(T_{k^i}\). Denote by \(\gamma_{k^i}\) the core geodesic of \(T_{k^i}\). Set \(\bar{f}^i:=\mathcal{L}_{g^i}\bar{h}^i\). Since \(||\rho_k^i||_{C^{2,\alpha}}\) is uniformly bounded, it follows from the proof of \Cref{Schauder estimates for decomposition norm} that \(||\mathcal{L}_{g^i}(\sum_k \rho_k^iu_k^i)||_{C_{\lambda}^0}\leq C\max_{k}|u_k^i| \to 0\), and the analogous statement holds for \(v_l^i\). Thus
\[
	||\bar{f}^i||_{C_{\lambda}^0} \to 0, \,  \text{as} \,\, i \to \infty.
\]
Recall that \(||\bar{h}^i||_{C^{0}(M)} \to 0\), and \(\frac{1}{W_\lambda(x^i)}|\bar{h}^i|_{C^0}(x^i)\geq \frac{1}{2}\) from the definition of \(x^i\). Thus \(W_{\lambda}(x^i) \to 0\). Hence, for \(i\) large enough, \(x^i \in M_{\rm small}^i\) and 
\[
	r_{\partial T^i}(x^i),r_{\gamma_{k^i}}(x^i) \to \infty,
\] 
where \(r_{\partial T^i}(x)=d\big(x,\partial (T_{k^i})_{\rm small}\big)\) and \(r_{\gamma_{k^i}}(x)=d(x,\gamma_{k^i})\). Abbreviate \(r^i:=r_{\partial T^i}(x^i)\).

Consider the torus \(T^i(r^i):=\{x \in T_{k^i} \, | \, r_{\partial T^i}(x)=r^i\}\) in \(T_{k^i}\) containing \(x^i\). Take a covering \(\hat{T}^i \to (T_{k^i})_{\rm small}\) such that
\[
	\mathrm{diam}\big( \hat{T}_{k^i}(r^i)\big) \leq 10  \quad \text{and} \quad \mathrm{inj}(\hat{x}^i) \geq 1,
\]
where \(\mathrm{diam}\big( \hat{T}^i(r^i)\big)\) is the diameter with respect to the intrinsic metric of \(\hat{T}_{k^i}(r^i)\).
Therefore, after passing to a subsequence, 
\[
	\big( \hat{T}^i,\hat{x}^i\big) \xrightarrow{\text{pointed }C^{2,\beta}} \big( T^\infty,x^\infty\big),
\]
where $\hat x^ i$ is a preimage of $x^i$ in $\hat T^i$ and 
where the convergence is pointed \(C^{2,\beta}\)-convergence for some \(\beta \in (0,\alpha)\). From the curvature decay condition (\ref{curvature decay}), and the fact that \(r_{\partial T^i}(x^i), r_{\gamma_{k^i}}(x^i) \to \infty\), it follows that the limit manifold \(T^\infty\) is a hyperbolic cusp which is two-sided infinite, that is, \(T^\infty=T^2 \times \bbR\) and 
\[
	g^\infty=e^{-2r}g_{Flat}+dr^2
\]
for some flat metric \(g_{Flat}\) on the torus \(T^2\). Denote by \(\hat{h}^i\) the pullback of \(\bar{h}^i\) to \(\hat{T}^i\), and analogously \(\hat{f}^i\) shall denote the pullback of \(\bar{f}^i\). Since going to covers does not change Hölder or \(C^k\)-norms, all the estimates on \(\bar{h}^i\) and \(\bar{f}^i\) still hold for \(\hat{h}^i\) and \(\hat{f}^i\). 

Set \(s^i(x):=r_{\partial T^i}(x)-r^i\). 
Then $s^i(x^i)=0$ and 
\(s^i \to s^\infty\), where \(s^\infty\) is the \(\bbR\)-coordinate in \(T^\infty=T^2 \times \bbR\) with \(s^\infty(x^\infty)=0\).

Abbreviate \(r(\cdot):=r_{\partial T^i}(\cdot)\), and write \(R^i:=R_{k^i}\) for the radius of \((T_{k^i})_{\rm small}\), that is, the distance of the core geodesic to \(\partial (T_{k^i})_{\rm small}\). After passing to a subsequence, we may distinguish between the following three cases. We will show that each case leads to a contradiction, thus completing the proof.

\textit{Case 1: \(R^i-2r^i \to -\infty\)}

Note
\[
	e^{\lambda(R^i-r^i)}W_{\lambda}(x)=e^{\lambda(R^i-r^i)} \left(e^{-\lambda r(x)}+e^{\lambda (r(x)-R^i)} \right)=e^{\lambda (R^i-2r^i)}e^{-\lambda s^i(x)}+e^{\lambda s^i(x)}.
\]
Define \(\underline{h}^i\) to be the rescaled tensor \(e^{\lambda (R^i-r^i)}\hat{h}^i\). As \(||\hat{h}^i||_{C_{\lambda}^{0}} \leq C\)
\[
	|\underline{h}^i|(x)=e^{\lambda(R^i-r^i)}W_{\lambda}(x) \left|\frac{1}{W_{\lambda}(x)}|\hat{h}^i|(x) \right| \leq C \big( e^{\lambda (R^i-2r^i)}e^{-\lambda s^i(x)}+e^{\lambda s^i(x)}\big).
\]
Thus \(\underline{h}^i\) is locally uniformly bounded near \(x^i\). By Schauder-estimates, the same holds true for its derivatives. 
As a consequence, after passing to a subsequence we obtain that 
\(\underline{h}^i \to \underline{h}^\infty\) in the pointed \(C^{2,\beta^\prime}\)-sense for some \(\beta^\prime \in (0,\beta)\) 
where the symmetric \((0,2)\)-tensor \(\underline{h}^\infty\) on \(T^\infty\) satisfies
\[
	|\underline{h}^\infty|(x) \leq Ce^{\lambda s^\infty(x)}.
\]
Note \(|\underline{h}^i|(\hat{x}^i)=\frac{1}{W_{\lambda}(\hat{x}^i)}|\hat{h}^i|(\hat{x}^i) \geq \frac{1}{2}\), and hence \(|\underline{h}^\infty|(x^\infty) \geq \frac{1}{2}\). In particular, \(\underline{h}^\infty\) is non-zero.

The same calculation as above shows that \(\underline{f}^i=e^{\lambda (R^i-r^i)}\hat{f}^i\) satisfies
\[
	|\underline{f}^i|(x) \leq ||f^i||_{C_{\lambda}^{0}}\big( e^{\lambda (R^i-2r^i)}e^{-\lambda s^i(x)}+e^{\lambda s^i(x)}\big).
\]
As \(||\hat{f}^i||_{C_{\lambda}^{0}} \to 0\), this implies \(\underline{f}^i \to 0\). By stability of elliptic equations the limit tensor \(\underline{h}^\infty\) therefore solves \(\mathcal{L}^\infty \underline{h}^\infty=0\). 

Note that the tensor \(\underline{h}^\infty\) is obtained as a limit of tensors that are lifts of tensors on \((T_{k^i})_{\rm small}\), and since the tubes \((T_{k^i})_{\rm small}\) converge to a line in the pointed Gromov-Hausdorff topology, the limit tensor \(\underline{h}^\infty\) only depends on \(s^\infty\) (we refer to \Cref{Subsec: Growth estimates} for the definition of what it means for a tensor to only depend on the \(\bbR\)-coordinate \(s\)). \Cref{Einstein variations with growth estimate} below then shows 
that \(\underline{h}^\infty=0\), which  contradicts
\(|\underline{h}^\infty|(x^\infty) \geq \frac{1}{2}\).

\textit{Case 2: \(R^i-2r^i \to \infty\)}

Compute
\[
	e^{\lambda r^i}W_{\lambda}(x)=e^{\lambda r^i} \left(e^{-\lambda r(x)}+e^{\lambda (r(x)-R^i)} \right)=e^{-\lambda s^i(x)}+e^{-\lambda (R^i-2r^i)}e^{\lambda s^i(x)}.
\]
By the same arguments as in \textit{Case 1}, after passing to a subsequence, the rescaled 
tensors \(\underline{h}^i:=e^{\lambda r^i}\hat{h}^i\) converge to some \textit{non-zero} \(\underline{h}^\infty\) which only depends on \(s^\infty\) and that satisfies
\[
	\mathcal{L}^\infty \underline{h}^\infty = 0 \quad \text{and} \quad |\underline{h}^\infty|(x) \leq Ce^{-\lambda s^\infty(x)} 
\]
for some \(C>0\). Again by
\Cref{Einstein variations with growth estimate} below, this implies \(\underline{h}^\infty=0\), a contradiction.

\textit{Case 3: \(R^i-2r^i \to A \in \bbR\)}

The same calculation as in \textit{Case 1} shows that the rescaled tensors \(\underline{h}^i:=e^{\lambda (R^i-r^i)}\hat{h}^i\) converge to some \textit{non-zero} \(\underline{h}^\infty\) that only depends on \(s^\infty\), satisfies the growth estimate
\[
	|\underline{h}^\infty|(x)\leq C \big( e^{A}e^{-\lambda s^\infty(x)}+e^{\lambda s^\infty (x)}\big),
\]
and solves \(\mathcal{L}^\infty\underline{h}^\infty=0\). \Cref{Einstein variations with growth estimate} shows that \(\underline{h}^\infty\) is a trivial Einstein variation. 

By the proof of \Cref{Canonical choice of trivial Einstein variation}, we know that \(\bar{h}^i\) satisfies \(|\bar{h}^i|(c_{k^i}^i) \leq |\bar{h}^i-u|(c_{k^i}^i)\) for any trivial Einstein variation \(u\) on \((T_{k^i})_{\rm small}\). Here \(c_{k^i}^i \in (T_{k^i})_{\rm small})\) is a chosen point with \(r(c_{k^i}^i)=\frac{R^i}{2}\). Since \(C^0\)-norms and the space of trivial Einstein variations are 
invariant under passing to covers and scaling we also have 
\[
	|\underline{h}^i|(\hat{c}_k^i) \leq |\underline{h}^i-u|(\hat{c}_k^i),
\]
for every trivial Einstein variation \(u\) on \(\hat{T}^i\) (notations are as above). 
By assumption, \(r^i-\frac{R_k^i}{2}\) is bounded. Hence \(d(\hat{x}^i,\hat{c}_{k^i}^i)\) is also bounded. Therefore, there is a limit point \(c^\infty \in T^\infty\). Note that any trivial Einstein variation on \(T^\infty\) is the limit of trivial Einstein variations on \(\hat{T}^i\). Thus
\[
	|\underline{h}^\infty|(c^\infty) \leq |\underline{h}^\infty-u|(c^\infty)
\]
for every trivial Einstein variation \(u\) on \(T^\infty\). Because \(\underline{h}^\infty\) is itself a trivial Einstein variation, choosing \(u=\underline{h}^\infty\) shows \(\underline{h}^\infty(c^\infty)=0\). As trivial Einstein variations have constant norm (with respect to a cusp metric), this implies \(\underline{h}^\infty=0\) everywhere. This is a contradiction.
\end{proof}

The following lemma was used at the end of the above proof to show that each of the three cases leads to a contradiction. The lemma plays the role of Proposition 7.1 in \cite{Bamler2012}. Our proof is the same as that in \cite{Bamler2012}, but adapted to our context. Concerning terminology, we refer to \Cref{Subsec: Growth estimates} for the notion of a cusp metric, and what it means for a tensor to only depend on \(r\).

\begin{lem}\label{Einstein variations with growth estimate} Assume \(T^2\times \bbR\) is equipped with a hyperbolic cusp metric. Let \(h\) be a tensor that only depends on the \(\bbR\)-coordinate \(r\), and that solves \(\mathcal{L}h=0\). If \(h\) satisfies 
\[
	|h|(r) \leq C \left( e^{-\lambda r}+e^{\lambda r}\right)
\]
for some \(C>0\) and \(\lambda \in (0,1)\), then \(h\) is a trivial Einstein variation. If moreover \(h\) satisfies
\[
	|h|(r) \leq C e^{-\lambda r} \quad \text{for all } \, \, r \in \bbR \quad \text{or} \quad |h|(r) \leq C e^{\lambda r} \quad \text{for all } \, \, r \in \bbR,
\]
then \(h=0\).
\end{lem}

\begin{proof}We first show \({\rm tr}(h)=0\). From (\ref{PDE for trace as ODE}) we have
\(Q_1(\frac{\partial}{\partial r})\big({\rm tr}(h)\big)(r)=0\) for a quadratic polynomial \(Q_1\) with roots \(1\pm \sqrt{5}\). Thus by applying \Cref{Transfer of exponential rates} we get
\[
	{\rm tr}(h)(r)=A_{+}e^{(1+\sqrt{5})r}+A_{-}e^{(1-\sqrt{5})r}
\]
for some \(A_{+},A_{-} \in \bbR\). By assumption \(|{\rm tr}(h)|(r) \leq C(e^{\lambda r}+e^{-\lambda r})\). Since \(\lambda < 1+\sqrt{5}\), taking \(r \to \infty\) implies \(A_{+}=0\). Similarly, since \(\lambda < \sqrt{5}-1\), taking \(r \to -\infty\) shows \(A_{-}=0\). Thus \({\rm tr}(h)=0\) everywhere.

Using (\ref{Lh=f in a cusp in coordinates}), the same argument shows \(h_{33}=0\) and \(h_{i3}=0\) everywhere. For \(h_{i3}\) this uses the assumption \(\lambda < 1\).

By (\ref{Lh=f in a cusp in coordinates})
we have \(Q_3(\frac{\partial}{\partial r})\big(e^{2r}h_{ij}\big)=0\) for a quadratic polynomial \(Q_3\) with roots \(0\) and \(2\). So invoking \Cref{Transfer of exponential rates} yields
\[
	e^{2r}h_{ij}=A+Be^{2r}
\]
for some \(A, B \in \bbR\). As \(|e^{2r}h_{ij}| \leq |h|(r) \leq C(e^{\lambda r}+e^{-\lambda r})\) and \(\lambda < 2\), taking \( r \to \infty\) implies \(B=0\). So \(h_{ij}=Ae^{-2r}\) for some \(A \in \bbR\). 

With everything up to now we have shown that \(h\) is a trivial Einstein variation. Now assume that \(h\) either satisfies \(|h|(r) \leq C^{\lambda r}\) or  \(|h|(r) \leq C^{-\lambda r}\). Note that trivial Einstein variations have constant norm (with respect to a cusp metric). Thus taking \(r \to -\infty\) or \(r \to \infty\) implies \(|h|=0\).
\end{proof}


\subsection{Surjectivity of \(\mathcal{L}\)}\label{Subsec: Surjectivity of L}

In order to establish \Cref{Invertibility of L - non-compact case} we have to show that
\[
	\mathcal{L}: \Big( C_{\lambda}^{2,\alpha}\big({\rm Sym}^2(TM)\big), ||\cdot||_{2 ,\lambda;\ast}\Big) \longrightarrow \Big( C_{\lambda}^{0,\alpha}\big({\rm Sym}^2(TM)\big), ||\cdot||_{0,\lambda} \Big)
\]
is an invertible operator, and that \(||\mathcal{L}^{-1}||_{\rm op}\) is bounded by a universal constant. In \Cref{Subsec: a priori estimates in non-compact case} we proved that \(\mathcal{L}\) satisfies an a priori estimate \(||h||_{2,\lambda;\ast} \leq C||\mathcal{L}h||_{0,\lambda}\). Therefore, to complete the proof of \Cref{Invertibility of L - non-compact case}, we have to show that \(\mathcal{L}\) is surjective. This will be done with the strategy used in the proof of \Cref{Invertibility of L}, which had two main ingredients:
\begin{itemize}
\item (Weak solutions exist)
  For any $f$ in the target space, the equation \({\mathcal L}(h)=f\) has a weak
solution. 
\item (Regularity) If \(f\) is a smooth tensor in the target space and \(h\) is a solution of \(\mathcal{L}h=f\), then \(h\) is contained in the source space.
\item (Approximation) For any \(f\) in the target space, there is a sequence of smooth tensors \((f_i)_{i \in \bbN}\) converging to \(f\);
\end{itemize}

In the setting of \Cref{Section - Invertibiliy of L}, the regularity part is immediate due to local (euclidean) regularity theory of elliptic PDEs. In our present non-compact setting this is no longer the case. Local regularity theory only shows that \(h\) is of a certain regularity in local coordiantes, but it does \textit{not} necessarily mean that the globally defined norm \(||h||_{2,\lambda;\ast}\) is finite. 
That this is indeed the case is shown in the following lemma. Here we assume that \(M\) satisfies the assumptions stated in \Cref{Invertibility of L - non-compact case}. Also, we refer to \Cref{big O notation} for our convention of the \(O\)-notation.

\begin{lem}\label{Regularity of L}Let \(f \in C_{\lambda}^{0,\alpha}\big({\rm Sym}^2(T^*M)\big) \) and \(h \in C^2\big({\rm Sym}^2(T^*M)\big) \cap H_0^1(M)\) be a solution of
\[
	\mathcal{L} h=f.
\]
Let \(C\) be a cusp of \(M\), and \(i_C \in (0,1)\) so that \(i_C \leq \inj(x)\) for all \(x \in \partial C_{\rm small}\). Then it holds
\[
	 \sup_{C_{\rm small}}|h|=O\Big(\frac{1}{i_C}||f||_{0,\lambda}\Big).
\]
In particular, \(\sup_M|h| < \infty\) and \(h \in C_{\lambda}^{2,\alpha}\big({\rm Sym}^2(T^*M)\big)\).
\end{lem}

Because the boundary tori \(\partial C\) can have arbitrary large diameter, there can be no universal lower bound on \(i_C\). The point of \Cref{Regularity of L} is \textit{not} to obtain a universal bound for \(||h||_{C^0}\), but just that \(h \in  C_{\lambda}^{2,\alpha}\big({\rm Sym}^2(T^*M)\big)\). Once this is known, one obtains universal estimates by the results of \Cref{Subsec: a priori estimates in non-compact case}.

\begin{proof}Fix a cusp \(C\), and let \(g_{cusp}\) be the hyperbolic cusp metric on \(C_{\rm small}\) given by \Cref{Existence of approximate cusp metric}. This satisfies \(|g-g_{cusp}|_{C^2}(x)=O\big(\varepsilon_0 e^{-\eta r(x)} \big)\), where \(r(x)=d(x,\partial C_{\rm small})\), and \(\eta \) is the decay rate in the curvature decay condition (\ref{curvature decay}). 
In particular, up to universal constants, the \(C^{k}\)-norms (\(k \leq 2\))  with respect to \(g\) and \(g_{cusp}\) agree.  Let \(\hat{\cdot}\) be the averaging operation of the cusp metric \(g_{cusp}\) (see \Cref{Properties of averaging operator}). First, we recall an estimate for \(|\hat{h}|\) that was established in the proof of \Cref{trivial Einstein variation and a priori estimate in a cusp}. In a second step we obtain estimates on \(|h-\hat{h}|\). We do this by using De Giorgi–Nash–Moser estimates in the universal cover to reduce this problem to bounding \(L^2\)-norms in the universal cover \(\tilde{M}\), which we obtain from weighted \(L^2\)-estimates in \(M\). Throughout this proof, for \(r\geq 0\) we denote by \(T(r)\) the torus in \(C_{\rm small}\) all of  whose points have distance \(r\) to \(\partial C_{\rm small}\).

\textbf{Step 1 (Estimating \(|\hat{h}|\)):} In the proof of \Cref{trivial Einstein variation and a priori estimate in a cusp} we established (\ref{bound on hat(h) in C_small}), which states
\begin{equation}\label{bound on hat(h) in C_small - v2}
	\sup_{C_{\rm small}}|\hat{h}| =O(||f||_{0,\lambda}).
\end{equation}
Even though in the formulation of \Cref{trivial Einstein variation and a priori estimate in a cusp} we assume \(||h||_{C^0}<\infty\), we pointed out that for (\ref{bound on hat(h) in C_small}) only the assumption \(h \in C^2\big({\rm Sym}^2(T^*M)\big)\cap H_0^1(M)\) is needed.

We also mention the following estimates that will be needed in Step 2.
Denote by \(\mathcal{L}_{cusp}\) the operator \(\frac{1}{2}\Delta_L +2\mathrm{id}\) with respect to the metric \(g_{cusp}\). From \(|g-g_{cusp}|_{C^2}(x)=O\big(\varepsilon_0 e^{-\eta r(x)} \big)\) it follows that
\begin{equation}\label{L surjective - eq0}
	|\mathcal{L}_{cusp}h-f|_{C^0}(x) \leq O\big( \varepsilon_0 e^{-\eta r(x)}|h|_{C^2}(x)\big).
\end{equation}
Now \(ii)\) of \Cref{Properties of averaging operator} and (\ref{L surjective - eq0}) together yield
\begin{equation}\label{L surjective - eq1}
	\left|\widehat{\mathcal{L}_{cusp}h}-\hat{f}\right|_{C^0}(x) =O\big( \varepsilon_0 \psi(r)e^{-(\eta -2)r}\big)
\end{equation}
where \(r=r(x)\), and \(\psi \in L^1(\bbR_{\geq 0})\) is the function defined in (\ref{def of psi}).

\textbf{Step 2 (Estimating \(|h-\hat{h}|\)):} Since \(\mathrm{dim}(M)=3\), 
we can apply the estimates from \Cref{Nash-Moser} with \(q=4\). 
Hence the same argument that led to (\ref{C^0 from L^2}) shows that for all \(x \in C_{\rm small}\), and any lift \(\tilde{x} \in \tilde{M}\) of \(x\), it holds
\begin{equation}\label{Nash-Moser for h-hat(h)}
	|\tilde{h}-\tilde{\hat{h}}|_{C^0}(\tilde{x}) \leq C\Big(||\tilde{h}-\tilde{\hat{h}}||_{L^2(B(\tilde{x},\rho))}+||\mathcal{L}_{cusp}(\tilde{h}-\tilde{\hat{h}})||_{L^2(B(\tilde{x},\rho))} \Big)
\end{equation}
for some universal constant \(C\). Here \(\rho>0\) is the universal radius appearing in the definition of the Hölder norms. 

We want to bound these \(L^2\)-norms in \(\tilde{M}\) by weighted \(L^2\)-norms in \(M\). Towards this goal, note that for any \(y \in M\) and lift \(\tilde{y} \in \tilde{M}\) of \(y\) it holds \(\#\big(\pi^{-1}(y) \cap B(\tilde{y},2\rho) \big) \leq C\frac{1}{\inj(y)^2}\), where \(C\) is a universal constant, and \(\pi:\tilde{M}\to M\) is the universal covering projection. Indeed, this follows by a simple area counting argument.
Note that (up to universal constant) it is irrelevant whether 
the injectivity radius is taken with respect to \(g\) or \(g_{cusp}\). Choose \(i_C \in (0,1)\) so that \( \inj(x) \geq i_C\) for all \(x \in \partial C_{\rm small}\). Since \(g_{cusp}\) is hyperbolic, the argument from the proof of \Cref{counting preimages II} shows 
that \(\inj(y) \geq e^{-d(y,\partial C_{\rm small})}i_C\) for all \(y \in C_{\rm small}\). Therefore, there is a universal constant \(C\) so that for any function \(u:M \to \bbR_{\geq 0}\), \(x \in C_{\rm small}\), and lift \(\tilde{x} \in \tilde{M}\) of \(x\) it holds
\begin{equation}\label{L surjective - eq4}
	\int_{B(\tilde{x},\rho)}\tilde{u} \, d\vol_{\tilde{g}_{cusp}} \leq C \frac{1}{i_C^2}\int_{B(x,\rho)}e^{2r(y)}u \, d\vol_{g_{cusp}},
\end{equation}
where \(\tilde{u}=u \circ \pi\) and \(r(y)=d(y,\partial C_{\rm small})\). See the claim in the proof of \Cref{A-priori estimate away from the small part} for more details as to why this integral estimate follows from the preimage counting.

We now show how (\ref{L surjective - eq4}) can be used to estimate the \(L^2\)-norms in (\ref{Nash-Moser for h-hat(h)}). We start with the first term.
Inequality (\(\ast \ast\)) on p. 520 of \cite{GromovMetricStructures07}
shows that for any flat 2-torus $T^2$ of diameter $1$, we have
$\lambda_1(T^2)\geq e^{-2}$. Together with a scaling argument, this implies that
if $T^2$ is a flat $2$-torus of 
\(\diam(T^2) \leq 1\),
then \(\lambda_1(T^2) \geq \frac{1}{\diam(T^2)^2}e^{-2}\). Therefore, for any function \(u\) it holds
\[
	\int_{T(r)}|u-\hat{u}|^2 \, d\vol_{cusp} \leq e^2 \diam(T(r),g_{cusp})^2\int_{T(r)}|\nabla u|^2 \, d\vol_{cusp},
\]
where \(\hat{u}\) is the average of \(u\) over \(T(r)\) (see the discussion before \Cref{Properties of averaging operator}). Applying this componentwise and multiplying by \(e^{2r}\) implies
\[
	e^{2r}\int_{T(r)}|h-\hat{h}|_{C^0}^2 \, d\vol_{cusp} \leq e^2 D^2\int_{T(r)}|h|_{C^1}^2 \, d\vol_{cusp},
\]
where \(D\) is the universal constant appearing in the definition of the small part. Thus
\begin{equation}\label{L surjective - eq5}
	\int_{C_{\rm small}}e^{2r}|h-\hat{h}|_{C^0}^2 \, d\vol_{g_{cusp}} \leq e^2 D^2\int_{C_{\rm small}}|h|_{C^1}^2 \, d\vol_{cusp} =O(||h||_{H^2(C_{\rm small})}^2).
\end{equation}

We know from the proof of \Cref{L^1-estimate for psi} that \(||h||_{H^1(C_{\rm small})}=O(||f||_{0,\lambda})\) (in fact, we even showed \(||h||_{H^2(C_{\rm small})}=O(||f||_{0,\lambda})\)). Thus (\ref{L surjective - eq4}) and (\ref{L surjective - eq5}) yield
\begin{equation}\label{L^2 estimate in universal cover of h-hat(h)}
	\int_{B(\tilde{x},\rho)}|\tilde{h}-\tilde{\hat{h}}|^2 \, d\vol_{\tilde{g}_{cusp}} =O\Big(\frac{1}{i_C^2}||f||_{0,\lambda}^2\Big).
\end{equation}
This completes the bound of the first \(L^2\)-norm in (\ref{Nash-Moser for h-hat(h)}).

Towards bounding the second \(L^2\)-norm in (\ref{Nash-Moser for h-hat(h)}), we use the triangle inequality, (\ref{L surjective - eq0}), \(ii)\) of \Cref{Properties of averaging operator}, and (\ref{L surjective - eq1}) to estimate
\begin{equation}\label{upper bound for L(h-hat(h)) in universal cover}
\begin{split}
	|\mathcal{L}_{cusp}\tilde{h}-\mathcal{L}_{cusp}\tilde{\hat{h}}|(\tilde{x}) \leq & |\mathcal{L}_{cusp}\tilde{h}-\tilde{f}|(\tilde{x}) +|\tilde{f}|(\tilde{x}) +|\tilde{\hat{f}}|(\tilde{x}) +|\mathcal{L}_{cusp}\tilde{\hat{h}}-\tilde{\hat{f}}|(\tilde{x})  \\
	= & O\Big(\varepsilon_0e^{-\eta \tilde{r}}|\tilde{h}|_{C^2}(\tilde{x}) +||f||_{0,\lambda}e^{-\lambda \tilde{r}}+\varepsilon_0 \psi(r)e^{-(\eta -2)\tilde{r}}\Big),
\end{split}
\end{equation}
where \(\tilde{r}=r \circ \pi\), and \(\psi \in L^1(\bbR_{\geq 0})\) was defined in (\ref{def of psi}).
Invoking (\ref{L surjective - eq4}), and using \(\eta > 1\), we obtain
\begin{align}\label{L^2 bound for L(h-hat(h)) - part 1}
	\int_{B(\tilde{x},\rho)}\big(e^{-\eta \tilde r}|\tilde{h}|_{C^2}\big)^2 d\vol_{\tilde{g}_{cusp}}\leq & C\frac{1}{i_C^2}\int_{B(x,\rho)}e^{-2(\eta-1) r}|h|_{C^2}^2 d\vol_{g_{cusp}} \notag \\
	\leq & C\frac{1}{i_C^2}\int_{C_{\rm small}}|h|_{C^2}^2 d\vol_{g_{cusp}} \notag \\
	= & O\Big(\frac{1}{i_C^2}||f||_{0,\lambda}^2 \Big),
\end{align}
where for the last inequality we used \(||h||_{H^2(C_{\rm small})}=O(||f||_{0,\lambda})\) (this was established in the proof of \Cref{L^1-estimate for psi}).
Also
\begin{equation}\label{L^2 bound for L(h-hat(h)) - part 2}
	\int_{B(\tilde{x},\rho)}e^{-2\lambda \tilde{r}} \, d\vol_{\tilde{g}_{cusp}} 
	\leq \mathrm{vol}(B(\tilde{x},\rho)) 
	=O\big(1\big).
\end{equation}

Recall that the function \(\psi\) is defined by \(\psi(r):=\int_{T(r)} |h|_{C^2} \). Denote by \(\strokedint_{T(r)}\) the average integral \(\frac{1}{\vol_2(T(r))}\int_{T(r)}\). By the Jensen inequality
\(
	\left( \strokedint_{T(r)}|h|_{C^2} \right)^2 \leq \strokedint_{T(r)}|h|_{C^2}^2,
\)
which implies \(\big(e^r\psi(r)\big)^2 \leq \int_{T(r)}|h|_{C^2}^2\). So \(e^r\psi(r) \in L^2(\bbR_{\geq 0})\), and \(||e^r\psi||_{L^2(\bbR_{\geq 0})} \leq ||h||_{H^2(C_{\rm small})}=O(||f||_{0,\lambda})\). Because \(\rho\) is universal, \(\mathrm{area}(B(\tilde{x},\rho)\cap\{\tilde{y} \, | \, \tilde r(\tilde{y})=s\})\) is bounded by a universal constant \(C\) for all \(s\). Thus, as \(\eta > 1\),
\begin{align}\label{L^2 bound for L(h-hat(h)) - part 3}
	\int_{B(\tilde{x},\rho)}\left(e^{-(\eta -2)\tilde r}\psi\right)^2 \leq &  
	 C \int_{r(x)-\rho}^{r(x)+\rho}e^{-2(\eta -1)r}(e^r\psi)^2 \, dr \notag \\
	\leq & C\int_{r(x)-\rho}^{r(x)+\rho}(e^r\psi)^2 \, dr \notag \\
	= & O\big(||f||_{0,\lambda}^2 \big).
\end{align} 
Combining (\ref{Nash-Moser for h-hat(h)}), (\ref{L^2 estimate in universal cover of h-hat(h)})-(\ref{L^2 bound for L(h-hat(h)) - part 3}) yields 
\[
	\sup_{C_{\rm small}}|h-\hat{h}|=O\Big(\frac{1}{i_C}||f||_{0,\lambda}\Big).
\]
Together with (\ref{bound on hat(h) in C_small - v2}) this implies \(\sup_{C_{\rm small}}|h|=O\left(\frac{1}{i_C}||f||_{0,\lambda} \right)\).

For the last assertions, note that \(M\) has only finitely many cusps because \(M\) has finite volume. The proven estimate and the compactness of \(M \setminus \bigcup_{C}C_{\rm small}\) immediately imply that \(\sup_{M}|h| < \infty\). To conclude \(h \in C_\lambda^{2,\alpha}\big( {\rm Sym}^2(T^*M)\big)\) we have to show that \(||h||_{C_\lambda^{2,\alpha};\ast}< \infty\). As \(||\mathcal{L}h||_{C_\lambda^{0,\alpha}}<\infty\) by assumption, the Schauder estimates for the \(\ast\)-norm (\Cref{Schauder estimates for decomposition norm}) reduce the problem to showing \(||h||_{C_\lambda^0;\ast}<\infty\). Due to the compactness of \(M \setminus \bigcup_{C}C_{\rm small}\), it suffices to show \(||h||_{C_\lambda^0;\ast}<\infty\) in each rank 2 cusp of \(M\). But this follows from \Cref{trivial Einstein variation and a priori estimate in a cusp} since \(||h||_{C^0}<\infty\).
\end{proof}

Recall from the introduction of this section that establishing surjectivity of \(\mathcal{L}\) requires an approximation and a regularity result. 
\Cref{Regularity of L} is the regularity statement. The approximation result is given by the next lemma.

\begin{lem}\label{approximation of Hölder tensors by smooth ones - exponential}Let \(f \in C_{\lambda}^{0,\alpha}\big({\rm Sym}^2(T^*M) \big)\) and \(\beta \in (0,\alpha)\). Then there is a sequence \((f_{\varepsilon})_{\varepsilon>0} \subseteq C^{\infty}\big({\rm Sym}^2(T^*M) \big)\) so that \(\lim_{\varepsilon \to 0}||f-f_{\varepsilon}||_{C_{\lambda}^{0,\beta}(M)}=0\).
\end{lem}

\begin{proof}It is well known that for \(u \in C^{0,\alpha}(\bbR^n)\) there is a sequence \((u_\varepsilon)_{\varepsilon>0} \subseteq C^\infty(\bbR^n)\) so that \(||u-u_\varepsilon||_{C^{0,\beta}} \to 0\) as \(\varepsilon \to 0\). Moreover, if \(u\) has compact support inside some open set \(\Omega \subseteq \bbR^n\), the \(u_\varepsilon\) can be chosen to have compact support in \(\Omega\) too. 

Denote by \(r\) the distance function to \(M \setminus M_{\rm small}\). For \(k \geq 0\) define \(U_k:=r^{-1}\big((k-1,k+1)\big)\). Choose a partition of unity \((\eta_k)_{k \geq 0}\) subordinate to the cover \(\{U_k\}_{k \geq 0}\). By applying the above approximation result locally, we see that for each \(k \) there is \(f_\varepsilon^{(k)}\) so that \(\mathrm{supp}(f_\varepsilon^{(k)}) \subseteq U_k\) and \(||(\eta_kf)-f_\varepsilon^{(k)}||_{C^{0,\beta}} \leq \frac{\varepsilon}{2} e^{-\lambda (k+1)}\). Then \(f_\varepsilon:=\sum_{k=0}^\infty f_\varepsilon^{(k)}\) has the desired property. Indeed, let \(x \in M\) be arbitrary and and choose \(k_0 \in \bbN\) so that \(k_0 \leq r(x) < k_0+1\). Then \(U_{k_0}\) and \(U_{k_0+1}\) are the only sets of the cover \(\{U_k\}_{k \geq 0}\) which may  contain \(x\). Thus
\begin{align*}
	|f-f_{\varepsilon}|_{C^{0,\beta}}(x) & \leq |(\eta_{k_0}f)-f_\varepsilon^{(k_0)}|_{C^{0,\beta}}(x)+ |(\eta_{k_0+1}f)-f_\varepsilon^{(k_0+1)}|_{C^{0,\beta}}(x) \\
	&\leq \frac{\varepsilon}{2} e^{-\lambda (k_0+1)}+\frac{\varepsilon}{2} e^{-\lambda (k_0+2)} \\
	& \leq \varepsilon e^{-\lambda r(x)} \\
	&\leq \varepsilon W_\lambda(x),
\end{align*}
and hence \(||f-f_{\varepsilon}||_{C_{\lambda}^{0,\beta}(M)}=\sup_{x \in M}\frac{1}{W_\lambda (x)}|f-f_{\varepsilon}|_{C^{0,\beta}}(x) \leq \varepsilon\).
\end{proof}

Now we are ready to present the proof of \Cref{Invertibility of L - non-compact case}.

\begin{proof}[Proof of \Cref{Invertibility of L - non-compact case}] As the a priori estimate was established in \Cref{Subsec: a priori estimates in non-compact case}, it remains to show that \(\mathcal{L}\) is surjective. The same argument as in \Cref{Invertibility of L} shows that for any \(f \in L^2\big({\rm Sym}^2(T^*M) \big)\), there is a weak solution \(h \in H_0^1\big({\rm Sym}^2(T^*M) \big)\) of \(\mathcal{L}f=h\). Fix \(f \in C_{\lambda}^{0,\alpha}\big({\rm Sym}^2(T^*M) \big)\). Invoking Lemma \ref{approximation of Hölder tensors by smooth ones - exponential} we obtain a sequence \((f_i)_{i \in \bbN} \subseteq C^\infty\big({\rm Sym}^2(T^*M) \big)\) so that \(||f-f_i||_{C_{\lambda}^{0,\frac{\alpha}{2}}} \to 0\) as \(i \to \infty\). Choose weak solutions \(h_i \in H_0^1\big({\rm Sym}^2(T^*M) \big)\) of \(\mathcal{L}h_i=f_i\). Then \(h_i \in C^{\infty}\big({\rm Sym}^2(T^*M) \big)\) and \(\mathcal{L}h_i=f_i\) holds in the classical sense. Moreover, \Cref{Regularity of L} implies \(h_i \in C_{\lambda}^{0,\frac{\alpha}{2}}\big({\rm Sym}^2(T^*M) \big)\). Note that the norms \(||\cdot||_{C_{\lambda}^{0,\frac{\alpha}{2}}}\) and \(||\cdot||_{0,\lambda}\) are equivalent on \(C_{\lambda}^{0,\frac{\alpha}{2}}\big({\rm Sym}^2(T^*M) \big)\) (but with a non-universal constant). Therefore, the a priori estimate from \Cref{C^0 estimate from exponential hybrid norm} gives
\[
	||h_i-h_j||_{C^{2}} \leq C ||f_i-f_j||_{C_{\lambda}^{0,\frac{\alpha}{2}}} \to 0 \quad \text{ as }\, i,j \to \infty
\]
for some (non-universal) constant \(C\). So \((h_i)_{i \in \bbN} \subseteq C^2\big({\rm Sym}^2(T^*M) \big)\) is a Cauchy sequence. Let \(h \in C^2\big({\rm Sym}^2(T^*M) \big)\) 
be the limit tensor field. As \(M\) has finite volume, \(C^2\)-convergence implies \(H^1\)-convergence. Thus \(h \in H_0^1\big({\rm Sym}^2(T^*M) \big)\) and \(\mathcal{L}h=f \in C_{\lambda}^{0,\alpha}\big({\rm Sym}^2(T^*M) \big)\). Invoking \Cref{Regularity of L} once more yields \(h \in C_{\lambda}^{2,\alpha}\big({\rm Sym}^2(T^*M) \big)\). Therefore, \(\mathcal{L}\) is a surjective mapping from \(C_{\lambda}^{2,\alpha}\big({\rm Sym}^2(T^*M) \big)\) to \(C_{\lambda}^{0,\alpha}\big({\rm Sym}^2(T^*M) \big)\).
\end{proof}


\section{Proof of the pinching theorem without lower injectvity radius bound}\label{Sec: Proof of pinching without inj radius bound}

We can now finally state and prove the full version of \Cref{Pinching without inj radius bound - introduction}.

\begin{thm}\label{Pinching without inj radius bound - full version}For all \(\alpha \in (0,1)\), \(\Lambda \geq 0\),  \(\lambda \in (0,1)\), \(\delta \in (0,2)\), \(r_0 \geq 1\), \(b > 1\), and \(\eta \geq 2+\lambda\) there exist constants 
\(\varepsilon_0=\varepsilon(\alpha,\Lambda,\lambda,\delta,r_0,b,\eta)>0\) and 
\(C=C(\alpha,\Lambda,\lambda,\delta,r_0,b,\eta)>0 \) 
with the following property. Let \(M\) be a \(3\)-manifold that admits a complete Riemannian metric \(\bar{g}\) satisfying the following conditions for some \(\varepsilon \leq \varepsilon_0\):
\begin{enumerate}[i)]
\item \(\vol(M,\bar{g}) < \infty\);
\item \(-1-\varepsilon \leq \mathrm{sec}_{(M,\bar{g})} \leq -1+\varepsilon\);
\item It holds
\[
	\max_{\pi \subseteq T_xM}|\mathrm{sec}(\pi)+1|, \, |\nabla R|(x), \, |\nabla^2R|(x) \leq \varepsilon e^{-\eta d(x,\partial M_{\rm small})}
\]
for all \( x \in M_{\rm small}\);
\item \(|| {\nabla} \Ric(\bar{g})||_{C^0(M,\bar{g})} \leq \Lambda\);
\item It holds
\[
	\int_Me^{-(2-\delta)r_x(y)}|\Ric(\bar{g})+2\bar{g}|_{\bar{g}}^2 \, d\vol_{\bar{g}}(y) \leq \varepsilon^2
\]
for all \(x \in M\) with 
\[\int_{B(x,2r_0)\setminus B(x,r_0)}e^{-(2-\delta)r_x(y)} \, d\vol_{\bar{g}}(y)>\varepsilon_0,\] where \(r_x(y)=d_{\bar{g}}(x,y)\);
\item It holds
\[
	e^{bd(x,M_{\rm thick})}\int_Me^{-(2-\delta)r_x(y)}|\Ric(\bar{g})+2\bar{g}|_{\bar{g}}^2 \, d\vol_{\bar{g}}(y) \leq \varepsilon^2
\]
for all \(x \in M_{\rm thin} \setminus M_{\rm small}\).
\end{enumerate}
Then there exists a hyperbolic metric \(g_{\rm hyp}\) on \(M\) so that
\[
	||g_{\rm hyp}-\bar{g}||_{2,\lambda;\ast} \leq C \varepsilon^{1-\alpha},
\]
where \(||\cdot||_{2,\lambda;\ast}\) is the norm defined in (\ref{Hybrid 2-norm - exponential}) with respect to 
the metric \(\bar{g}\) and the constants \(\alpha,\lambda,b,\varepsilon_0,\delta,r_0\).

Moreover, if for some \(\beta \leq 1-\frac{1}{2}\delta\) and \(U \subseteq M\) it holds
\[
	\int_Me^{-(2-\delta)r_x(y)}|\Ric(\bar{g})+2\bar{g}|_{\bar{g}}^2 \, d\vol_{\bar{g}}(y) \leq \varepsilon^{2(1-\alpha)} e^{-2 \beta \dist_{\bar{g}}(x,U \cup \partial M_{\rm rhick})} \quad \text{for all }x \in M_{\rm thick},
\]
then
\[
	|g_{\rm hyp}-\bar{g}|_{C^{2,\alpha}}(x) \leq C\varepsilon^{1-\alpha}e^{-\beta\mathrm{dist}_{\bar{g}}(x,\,U \, \cup \, \partial M_{\rm thick})} \quad \text{for all }x \in M_{\rm thick}.
\]
In particular, if \(\bar{g}\) is already hyperbolic outside a region \(U \subseteq M\), and if 
\[
	\int_{U}|\Ric(\bar{g})+2\bar{g}|_{\bar{g}}^2 \, d\vol_{\bar{g}} \leq \varepsilon^2,
\]
then it holds
\[
	|g_{\rm hyp}-\bar{g}|_{C^{2,\alpha}}(x) \leq C\varepsilon^{1-\alpha}e^{-(1-\frac{1}{2}\delta)\mathrm{dist}_{\bar{g}}(x,\,U \, \cup \, \partial M_{\rm thick})} \quad \text{for all }x \in M_{\rm thick}.
\]
\end{thm}

The following slight generalisation follows from \Cref{Invertbility of L - eta > 1}.

\begin{rem}\normalfont
\Cref{Pinching without inj radius bound - full version} holds for all \(\lambda \in (0,1)\) and \(\eta > 1\).
\end{rem}

\begin{proof}Since the proof is basically identical to that of \Cref{Pinching with inj radius bound - full version} we only sketch it and highlight differences. For \(R>0\) we denote by \(\bar{B}(0,R) \subseteq C_\lambda^{2,\alpha}\big({\rm Sym}^2(T^*M)\big)\) the closed ball of radius \(R\) around the \(0\)-section with respect to 
the norm \(||\cdot||_{2,\lambda;\ast}\). Analogous to the proof of \Cref{Pinching with inj radius bound - full version} we consider 
\[
	\Psi: \bar{B}(0,R) \to C_\lambda^{2,\alpha}\big({\rm Sym}^2(T^*M)\big), h \mapsto h -\mathcal{L}^{-1}\Phi(\bar{g}+h),
\]
where \(\Phi\) is the Einstein operator defined in (\ref{Def of Phi}), and \(\mathcal{L}=(D\Phi)_{\bar{g}}\). We want to show that for \(R>0\) small enough, 
\(\Psi\) is \(\frac{1}{2}\)-Lipschitz. Since \(||\mathcal{L}^{-1}||_{\rm op}\) is bounded by a universal constant (\Cref{Invertibility of L - non-compact case}), it suffices to show that for any \(h \in B(0,R)\) and any \(h^\prime \in C_\lambda^{2,\alpha}\big({\rm Sym}^2(T^*M)\big)\) it holds
\[
	||(D\Phi)_{\bar{g}+h}(h^\prime)-(D\Phi)_{\bar{g}}(h^\prime)||_{0,\lambda} \leq c(R)||h^\prime||_{0,\lambda;\ast}
\]
for some universal \(c(R)\) with \(c(R) \to 0\) as \(R \to 0\). Indeed, this can be shown by the exact same argument as in \cite[pages 898 and 899]{Bamler2012}.

Fix \(R>0\) such that \(\Psi\) is \(\frac{1}{2}\)-Lipschitz. Then it follows from the definition (\ref{Hybrid 2-norm - exponential}) of \(||\cdot||_{2,\lambda;\ast}\), the conditions \(ii)-v)\) and \Cref{Invertibility of L - non-compact case} that \(||\Psi(0)||_{2,\lambda;\ast} \leq C\varepsilon^{1-\alpha}\), which is at most \(\frac{R}{2}\) if  \(\varepsilon_0>0\) is small enough. 
Remembering that in dimension three Einstein metrics have constant sectional curvature, the existence of a hyperbolic metric \(g_{\rm hyp}\) now follows exactly as in the proof of \Cref{Pinching with inj radius bound - full version}.

To show the improved estimate, note that (\ref{improved estimate - eq 1}) still holds for points in the thick part of $M$, 
that is, for all \(h \in C_\lambda^{2,\alpha}\big({\rm Sym}^2(T^*M)\big)\) and all \(x \in M_{\rm thick}\) it holds
\begin{equation*}\label{improved estimate v2 - eq1}
	|h|_{C^{2,\alpha}}(x)\leq C_0 \Big(|\mathcal{L}h|_{C^{0,\alpha}}(x)+||\mathcal{L}h||_{L^2(M;\omega_x)} \Big)
\end{equation*}
for a universal constant \(C_0\). Also choose \(C_0\) large enough so that \(||\Psi(0)||_{2,\lambda;\ast}\leq \frac{1}{2}C_0 \varepsilon^{1-\alpha}\). Abbreviate \(R_{\varepsilon}=C_0\varepsilon^{1-\alpha}\), so that \(\Psi\) restricts to a \(\frac{1}{2}\)-Lipschitz endomorphism of \(\bar{B}(0,R_\varepsilon)\). For an appropriatly chosen \(C_1 \gg C_0\) define
\begin{equation*}\label{improved estimate v2 - eq2}
	\mathcal{U}:=\left\{h \in \bar{B}(0,R_\varepsilon) \, | \, h \text{ satisfies (\ref{improved estimate v2 - eq3}) for all }x \in M_{\rm thick}  \text{ and (\ref{improved estimate v2 - eq4}) for all }x \in M \right\},
\end{equation*}
where the inequalities (\ref{improved estimate v2 - eq3}) and (\ref{improved estimate v2 - eq4}) appearing in the definition of \(\mathcal{U}\) are
\begin{equation}\label{improved estimate v2 - eq3}
	|h|_{C^{2,\alpha}}(x)\leq C_1 \varepsilon^{1-\alpha}e^{-\beta \dist_{\bar{g}}(x, \, U \, \cup \, \partial M_{\rm thick} )}
\end{equation}
and
\begin{equation}\label{improved estimate v2 - eq4}
		||h||_{H^2(M;\omega_x)}\leq C_1 \varepsilon^{1-\alpha}e^{-\beta \dist_{\bar{g}}(x, \, U \, \cup \, \partial M_{\rm thick} )}
\end{equation}
The rest of the proof of \Cref{Pinching with inj radius bound - full version} carries over with only one small additional observation. For \(h \in \mathcal{U}\) the estimate (\ref{improved estimate v2 - eq3}) even holds for all \(y \in N_\rho(M_{\rm thick})\). Here \(\rho > 0\) is the radius appearing in the definition of the Hölder norms. Indeed, if \(y \in N_{\rho}(M_{\rm thick})\setminus M_{\rm thick}\), then \(d(y,\partial M_{\rm thick}) \leq \rho\), and thus
\[
	|h|_{C^{2,\alpha}}(y)\leq ||h||_{2,\lambda;\ast} \leq C_0\varepsilon^{1-\alpha} \leq \big(C_0e^{\rho}\big)\varepsilon^{1-\alpha}e^{-(1-\frac{1}{2}\delta)d(y, \, U \, \cup \, \partial M_{\rm thick} )}
\]
So \(h \in \mathcal{U}\) satisfies (\ref{improved estimate v2 - eq3}) for all \(y \in N_\rho(M_{\rm thick})\) if we choose \(C_1 > C_0e^\rho\). We really need the estimate (\ref{improved estimate v2 - eq3}) in an enlarged region because in order to check that \(\mathcal{U}\) is \(\Psi\)-invariant, it is necessary to control \(\max_{y \in B(x,\rho)}|h|_{C^{2,\alpha}}^2(y)\) for \(x \in M_{\rm thick}\) (see (\ref{improved estimate - eq 4})).
\end{proof}

Analogous to \Cref{Pinching with inj rad bound - non-orientable}, \Cref{Invertibility of L - non-compact case - non-orientable} implies the following.

\begin{rem}\normalfont
\Cref{Pinching without inj radius bound - full version} also holds when \(M\) is non-orientable.
\end{rem}


\section{Drilling and filling}\label{Sec: drillingand} 

In the first part of this section we consider 
a hyperbolic 3-manifold $M$ of finite volume and
a collection $T_1,\dots,T_k$ of Margulis tubes in $M$.
Each $T_i$ is a solid 3-torus whose boundary 
$\partial T_i$ is a flat
2-torus which is  locally isometrically 
embedded in $M$. The \emph{meridian} of the tube $T_i$ is an 
essential simple closed curve $\alpha_i$ on $\partial T_i$ 
which is homotopically trivial in $M$. It is unique up to free homotopy. 
Up to homotopy equivalence
and hence isometry (by the Mostow rigidity theorem), the 
manifold $M$ is uniquely determined by the homotopy type of 
$M-\cup_iT_i$ and the choice of these meridians. The core curve of the tube
$T_i$ is a primitive closed geodesic $\beta_i\subset T_i$. 

\emph{Drilling} of the geodesics $\beta_i$, 
that is, removal of the geodesic $\beta_i$ 
for each $i$,  defines a new manifold $\hat M$. Brock and Bromberg
(Theorem 6.1 of \cite{BB04}) showed that if the sum of the
lengths of the geodesics $\beta_i$ is sufficiently small,
then the manifold $\hat M$ admits a complete hyperbolic
metric of finite volume for which 
each Margulis tube $T_i$ about one of the geodesics
$\beta_i$ has been replaced by a rank two cusp $C_i$. Furthermore,
the hyperbolic metric on 
$\hat M-\cup_i C_i$ is $L$-bilipschitz to the 
hyperbolic metric on $M-\cup_iT_i$ for a number
$L$ which tends to one as the sum of the lengths of the
geodesics $\beta_i$ tends to zero.

The work \cite{BB04} does not give an effective upper bound for the total
length of the geodesics $\beta_i$ for which the drilling result holds true, 
nor is the dependence of the bilipschitz constant $L$ 
on this total length explicit. Such effective bounds were recently
obtained by Futer, Purcell and Schleimer
\cite{FPS21}. 

The first main goal of this section is to establish a version of
the drilling result of Brock and Bromberg \cite{BB04} as an application
of our main theorem. As in \cite{BB04}, our result is not effective, 
but it allows the drilling
of an arbitrary number of geodesics, with only a universal
upper length bound for each of them, 
provided that these geodesics
are sufficiently sparsely distributed in $M$. 

For an application of our 
methods, it is more convenient to control a Margulis tube via the
length of its meridian on the boundary torus and not via
the length of the core geodesic. Thus we begin with 
comparing the information on meridional length with the information
on the length of the core geodesic. 

Let us consider for the moment
an arbitrary Margulis tube $T$ with core geodesic $\beta$ of length $\ell >0$
and boundary $\partial T$ 
in some hyperbolic 3-manifold $M$. If $R>0$ is the radius of the tube, 
that is, the distance of the core geodesic $\beta$ to the boundary torus $\partial T$, 
then
the meridian of $T$ is a simple closed geodesic on the flat torus $\partial T$ of length
$2\pi \sinh R$.
In particular, since the
injectivity radius of $\partial T$ roughly equals the Margulis constant for hyperbolic 
3-manifolds, 
the radius $R$ is bounded from below by a universal positive constant.
Cutting $\partial T$ open along a meridian
yields a flat cylinder with boundary length $2\pi \sinh R$ and height
$\ell \cosh R$. The area of $\partial T$ equals
$2\pi \ell \sinh R \cosh R$. 
 
In general, the relation between the length $\ell$ of the
core geodesic of a Margulis tube and the radius $R$ of the tube
is delicate. The following effective bound is a special case of
Theorem 1.1 of \cite{FPS19}.

\begin{thm}[Futer, Purcell and Schleimer]\label{futer}
  Let $\epsilon \leq 0.3$ be a Margulis constant for hyperbolic
  3-manifolds. 
  Let $M$ be a hyperbolic 3-manifold  and let
  $N\subset M$ be a Margulis tube whose core geodesic has
  length $\ell< 8\epsilon^2$.  Then the radius of the tube is at least
 $\operatorname{arcosh} \frac{\epsilon}{\sqrt{8\ell}}$. 
\end{thm}

Theorem \ref{futer} makes the idea effective
that a very short core geodesic is contained in
a Margulis tube of very large radius. As a consequence,
finding non-effective upper
length bounds on closed geodesics which we can drill from a hyperbolic
manifold, with effective control on the geometry of the resulting
manifold with cusps, is equivalent to finding non-effective
lower bounds on tube radii about closed geodesics
which allow for geometrically controlled drilling.

We show


\begin{thm}[The drilling theorem]\label{drilling}
For any \(\varepsilon >0\), \(\kappa \in (0,1)\) and \(m>0\) there exists
a number $R=R(\varepsilon,\kappa,m) >0$
with the following property. 
Let $M$ be a finite volume hyperbolic 3-mani\-fold, and let 
$T_1,\dots,T_k$ be a family of Margulis tubes in $M$. Let $R_i>0$ be the radius of the 
tube $T_i$, and let $\beta_i$ be its core geodesic. 
If for each $r>0$ and each $x\in M$ we have
$\# \{i\mid {\rm dist}(x,T_i)\leq r\}\leq me^{\kappa r }$ and if $R_i\geq R$ for all $i$, 
then the manifold $\hat M$ 
obtained from $M$ by drilling each of the 
geodesics $\beta_i$ admits a complete hyperbolic metric of 
finite volume, and the restriction of this hyperbolic metric to the complement
of the cusps $C_i$ obtained from the drilling is $\varepsilon$-close 
in the $C^2$-topology to the metric on $M-\cup_{i\leq k}T_i$.
\end{thm}

\begin{rem}\normalfont
Our drilling theorem is weaker than Theorem 6.1 of \cite{BB04}
as we require that the 
manifold $M$ is of finite volume rather than just geometrically finite.
Furthermore, in contrast to the work \cite{FPS21}, our estimates
are not effective. But it is also stronger than the results obtained
in \cite{BB04,FPS21} 
as it allows for drilling of an arbitrary number of closed geodesics,
contained in Margulis tubes of tube radius  
larger than a fixed constant, provided that the tubes are 
sufficiently sparsely distributed in the manifold $M$, and it gives better
geometric control on the drilled manifold.

In fact, the improved estimate in \Cref{Pinching without inj radius bound - full version} and the estimate (\ref{bound sum}), which will be obtained during the proof of \Cref{drilling}, immediately imply the following. For any \(\delta >0 \) there exists \(R=R(\delta,\varepsilon,\kappa,m) >0\) so that on the thick part of the drilled manifold, the hyperbolic metric is \(\varepsilon e^{-(1-\frac{1}{2}\kappa-\delta)\dist(\cdot,M_{\rm thin})}\)-close to the original metric. In particular, by choosing \(\delta \leq \frac{1}{2}(1-\kappa)\) one can always arrange that on the thick part of the drilled manifold, the hyperbolic metric is \(\varepsilon e^{-\frac{1}{2}\dist(\cdot,M_{\rm thin})}\)-close to the original metric.

\end{rem}

\begin{proof}[Proof of \Cref{drilling}]We split the proof into several steps.

\textbf{Step 1 (Construction of an approximating metric):}
  Let us consider the boundary $\partial T$ of a Margulis tube $T$ 
  in a hyperbolic $3$-manifold of finite volume. 
This is a flat torus, and the meridian of the tube $T$ 
defines a foliation of $\partial T$ by closed geodesics. 
In polar coordinates $(r,\theta,y)$ about the
core geodesic $\beta$ of the Margulis tube, 
the hyperbolic metric can be written as 
\[g=dr^2 +(\sinh r)^2d\theta^2+(\cosh r)^2dy^2\]
where $dy$ is a one-form on $T$ which vanishes on the immersed 
totally geodesics hyperbolic planes which intersect the core geodesic 
of the tube orthogonally and where 
$r>0$ is the radial distance from the core curve of the tube $T$.  
The bilinear form
$(\sinh r)^2 d\theta^2 +(\cosh r)^2dy^2$ defines the flat metric on the level
tori $\{r={\rm const}\}$ 
of the distance function from 
the core geodesic.

Choose a smooth function $\sigma:\mathbb{R}\to [0,1]$
such that $\sigma\equiv 1$ in a neighbourhood of $(-\infty,-1]$
and $\sigma\equiv 0$ in a neighbourhood of $[0,\infty)$. Let $\hat R$
be the radius of the Margulis 
tube $T$. In the coordinates
$(r,\theta,y)$, define a new metric $\hat g$
on a tubular neighborhood $C\subset T$
about $\partial T$ of radius $1$ by 
\begin{align*}\hat g= & dr^2+\left(\big(1-\sigma(r-\hat R)\big)\sinh(r)+\sigma(r-\hat R)\frac{1}{2}e^{r} \right)^2d\theta^2  \\
& +\left(\big(1-\sigma(r-\hat R)\big)\cosh(r)+\sigma(r-\hat R)\frac{1}{2}e^{r} \right)^2dy^2.
\end{align*}
Since the first and second derivatives of $\sigma$ are uniformly bounded and
the derivatives of the function $e^r/2-\sinh(r)=e^{-r}/2$ equal $\pm e^{-r}/2$ and similarly for
$e^{r}/2-\cosh(r)$, the metric $\hat g$ is $a_1e^{-2\hat R}$-close to
$g$ in the $C^2$-topology 
on $C$ where $a_1>0$ is a universal constant.
Furthermore, as the metric $\hat g$ coincides with the hyperbolic metric 
$g$ of $M$ in the complement of the
Margulis tube $T$, it can be viewed as a metric on $M-\hat T$ which is 
$a_1 e^{-2\hat R}$-close to $g$ in the $C^2$-topology, where $\hat T=T-C$.
The sectional curvature of $\hat g$ equals $-1$ outside of the collar $C$, and it is 
contained in the interval $[-1-a_2 e^{-2\hat R},-1+a_2 e^{-2\hat R}]$ for a universal 
constant $a_2>0$ not depending on $T$ or $M$. 
The same argument shows that \(||{\nabla}\Ric(\hat{g})||_{C^0} \leq \Lambda\) for a universal constant \(\Lambda > 0\).

Near the boundary component of $C$ which is contained in the
interior of the tube $T$, the 
metric $\hat g$ is of the form $dr^2+ e^{2r} g_0$ where $g_0$
is a fixed flat metric on the 
distance tori to $\partial T$ 
for the metric $g$. Such a warped product metric is the local model for 
a hyperbolic metric on a rank two cusp. Thus we can glue a hyperbolic 
rank two cusp to the interior boundary component
of the collar $C$ in such a way that the resulting manifold $\hat M$ 
is obtained from $M$ by drilling the closed geodesic $\beta$, and 
$\hat M$ is equipped with a complete Riemannian metric $\bar{g}$ whose restriction to 
$\hat M-\hat T$ coincides with $\hat g$ and is hyperbolic in the complement of the
collar $C$. In particular, $\hat{g}$ coincides with $g$ on 
$M-T$ (using the natural embedding of $M-T$ into $\hat M$).

The same construction can be done for all the Margulis tubes \(T_1,...,T_k\) simultaneously. 
That is, the manifold \(\hat{M}\) which is obtained from \(M\) by drilling the closed geodesics \(\beta_1,...,\beta_k\) admits a complete finite volume Riemannian metric \(\bar{g}\) such that
\begin{itemize}
\item \(\bar{g}=g\) on \(M \setminus \bigcup_{i=1}^kT_i\);
\item \(|\bar{g}-g|_{C^2} \leq ae^{-2R}\) in \(M \setminus \bigcup_{i=1}^k\hat{T_i}\);
\item \(|\sec(\bar{g})+1| \leq ae^{-2R_i}\) inside the cusp that is obtained by drilling \(\beta_i\);
\item \(\bar{g}\) is hyperbolic outside the union of the collars \(\bigcup_{i=1}^kC_i\);
\item \(||{\nabla}\Ric(\bar{g})||_{C^0(M,\bar{g})}\leq \Lambda\).
\end{itemize}
Here \(a\) and \(\Lambda\) are both universal constants.

\textbf{Step 2 (Reducing the problem):} For sufficiently large $R_i$, 
all the conditions in \Cref{Pinching without inj radius bound - introduction}
besides the integral estimate are clearly satisfied. Thus it remains to prove the integral estimate.
Since 
$\mathrm{area}(\partial T_i)=2\pi \ell(\beta_i)\cosh(R_i)\sinh(R_i)$ 
where \(R_i\) is the radius of the tube \(T_i\), we have
\[\vol_{\bar{g}}(C_i) \leq ce^{R_i}\] for a universal constant \(c\), where \(C_i\) is the collar around \(\partial T_i\). 
Namely, if $R$ is large, then the same holds true for $\sinh(R)$. Thus a shortest closed geodesic on $T_i$, whose length equals twice the 
Margulis constant $\mu$ and hence roughly equals one, is different from a meridian and 
has to cross through the cylinder of height $\ell(\beta_i)\cosh(R_i)$ obtained by
cutting $T_i$ open along a meridian. As a consequence, its length is at least $\ell(\beta_i)\cosh(R_i)$, that is, we have $\ell(\beta_i)\cosh(R_i)\leq 2\mu$. 
Together with the curvature estimate, this implies
\begin{equation}\label{L^2 estimate in single collar}
	\int_{C_i}|\Ric(\bar{g})+2\bar{g}|_{\bar{g}}^2 \, d\vol_{\bar{g}}\leq C\vol(C_i)\left(e^{-2R_i}\right)^2 \leq Ce^{-3R_i} \leq Ce^{-3R},
\end{equation}
where in the last inequality we used the assumption that the radius \(R_i\) of \(T_i\) is at least \(R\). 

Write \(\delta_i(x):=\dist(x,\partial T_i)\) and $r_x(y)=d(x,y)$. 
As \(C_i\) has width \(1\), 
for all $x\in M$, all $i\leq k$ and all \(y \in C_i\) it holds \(\delta_i(x) \leq d(x,y)+1=
r_x(y)+1\), whence \(\lceil \delta_i(x) \rceil \leq r_x(y)+2 \). Therefore, we deduce from (\ref{L^2 estimate in single collar})
\begin{equation}\label{reduction to estimating a sum}
	\int_{\hat{M}}e^{-(2-\delta)r_x(y)}|\Ric(\bar{g})+2\bar{g}|_{\bar{g}}^2 \, d\vol_{\bar{g}}(y) \leq C \left(\sum_{i=1}^ke^{-(2-\delta)\lceil \delta_i(x) \rceil} \right)e^{-3R}.
\end{equation}
Here \(\delta \in (0,2)\) is a small constant that will be determined later.

\textbf{Step 3 (Estimating the sum)}: By assumption, we have \(\#\{i \, | \, \delta_i(x)  \leq r\} \leq me^{\kappa r}\).
Observe that by choosing $\kappa^\prime>\kappa$, if $\bar R=\bar{R}(\kappa,\kappa^\prime)$ is big enough, then 
by the second bullet point at the end of Step 1, after replacing $\kappa$ by $\kappa^\prime$, this estimate also holds true in $\hat M$. 
Choose once and for all such a number $\kappa^\prime\in (\kappa,1)$.
From now on all distances are taken in \(\hat{M}\) and with respect to  \(\bar{g}\). 

To estimate the sum in (\ref{reduction to estimating a sum}) we need to analyze the functions \(\lceil \delta_i(x) \rceil\). 
Observe that \(\lceil \delta_i(x) \rceil\) also satisfies the growth condition. Indeed, since for all \(u \in \bbR\) and \(r \in \bbN\) it holds \(u \leq r\) if and only if \(\lceil u \rceil \leq r\), we have
\[
	\#\{i \, | \, \lceil \delta_i(x) \rceil \leq r\}=\#\{i \, | \, \delta_i(x)  \leq r\} \leq me^{\kappa^\prime r}
\]
for all \(r \in \bbN\). Fix \(x \in M\). Let \(0 < r_0 < ... < r_l\) be an enumeration of \(\{\lceil \delta_i(x) \rceil  \, | \, i=1,...,k\}\). So
\begin{equation}\label{sum replaced by other sum}
	\sum_{i=1}^ke^{-(2-\delta)\lceil \delta_i(x) \rceil} =\sum_{j=0}^l \#\{i \, | \, \lceil \delta_i(x) \rceil =r_j\}e^{-(2-\delta)r_j} \leq m\sum_{j=0}^le^{-(2-\delta - \kappa^\prime)r_j}.
\end{equation}
As \((r_j)_{j=0,...,l}\) is an increasing sequence of natural numbers and since \(r \mapsto e^{-(2-\delta-\kappa^\prime)r}\) is decreasing (assuming \(\delta+\kappa^\prime <2\)), we get
\begin{equation}\label{bound sum}
	\sum_{j=0}^le^{-(2-\delta - \kappa^\prime)r_j} \leq \int_{r_0-1}^\infty e^{-(2-\delta-\kappa^\prime)} \, dr=\frac{1}{2-\delta-\kappa^\prime}e^{-(2-\delta-\kappa^\prime)(r_0-1)}.
\end{equation}
Finally note that \(d(x,\hat{M}_{\rm thick}) \leq d(x,\partial T_i)\) since \(\partial T_i \subseteq \hat{M}_{\rm thick}\), and hence
 \(d(x,\hat{M}_{\rm thick}) \leq r_0\). Therefore, if we choose \(\delta=\delta(\kappa) \in (0,2)\) and \(b=b(\kappa)>1\) small enough so that \(\delta+\kappa^\prime+b < 2\), then
\[
	e^{bd(x,\hat{M}_{\rm thick})}\int_{\hat{M}}e^{-(2-\delta)r_x(y)}|\Ric(\bar{g})+2\bar{g}|_{\bar{g}}^2 \, d\vol_{\bar{g}}(y) \leq Ce^{-3R}
\]
by combining (\ref{reduction to estimating a sum}), (\ref{sum replaced by other sum}) and (\ref{bound sum}). For \(R\) large enough, this will be at most \(\varepsilon^2\). 
Therefore, we can evoke \Cref{Pinching without inj radius bound - introduction} to complete the proof.
\end{proof}

Controlled Dehn filling can be done with precisely the same argument. The following is 
a version of a theorem of Hodgson and Kerckhoff \cite{HK08}. 
In its formulation, we start with a finite volume hyperbolic $3$-manifold $M$ and 
a collection of cusps $C_1,\dots,C_k$ whose boundaries $\partial C_i$  are 2-tori.
Our goal is to fill these cusps by replacing them by solid tori, and show that if 
the lengths on $\partial C_i$ of the meridians of the solid tori are sufficiently
large, then the filled in manifold admits a hyperbolic metric which is close to the 
metric of $M$ on the complement of the cusps. Our result 
is weaker than the result in \cite{HK08} as the filling
constants we find are not effective. 
On the other hand, similarly to Theorem \ref{drilling} and 
unlike the results of \cite{HK08}, the filling constants do not depend on 
the number of cusps to be filled but only on sparsity of these cusps in $M$.
Furthermore, the lower length bound on the meridional 
geodesic $\gamma$ we use for the 
filling  is the actual length of $\gamma$ on the flat torus $T$ and not its 
\emph{normalized length}, defined to be the length
of $\gamma$ on the flat torus $T^\prime$ which is 
obtained from $T$ by rescaling the metric so
that the volume of $T^\prime$ 
equals one. 

\begin{thm}[The filling theorem]\label{filling}
For any \(\varepsilon >0\), \(\kappa \in (0,1)\) and \(m>0\) there exists
a number $L=L(\varepsilon,\kappa,m) >0$
with the following property. Let $M$ be a finite volume hyperbolic 3-mani\-fold, $C_1,\dots,C_k\subseteq M$ be a finite collection of torus 
cusps, and assume that  
for each $r>0$ and each $x\in M$ we have
$\# \{i\mid {\rm dist}(x,C_i)\leq r\}\leq me^{\kappa r}$.
For each $i\leq k$ let $\alpha_i$ be a flat simple closed geodesic in  
$\partial C_i$ of length $L_i\geq L$. 
Then the manifold obtained from $M$ by filling the cusps $C_i$, 
with meridian $\alpha_i$,  
is hyperbolic, and the restriction of its metric to the complement
of the Margulis tubes obtained from the filling is $\varepsilon$-close to the metric on 
$M-\cup_iC_i$. 
\end{thm}

\begin{proof}
  The proof is analogous to the proof of Theorem \ref{drilling}.
Namely, let $\alpha_i\subset \partial C_i$
be a closed geodesic. Then $\alpha_i$ defines a foliation of $\partial C_i$ by closed 
geodesics, and there is a dual orthogonal foliation by geodesics (which are not necessarily closed).
Let us assume that the length of $\alpha_i$ equals $e^{R_i}/2$ for a number $R_i>0$. 
Reversing the argument in the proof of Theorem \ref{drilling},
define a new metric on the tubular
neighborhood $N(\partial C_i,1)$ 
of radius one about $\partial C_i$ in the cusp $C_i$
as follows. Write the hyperbolic metric 
in horospherical coordinates in the form 
$dt^2 + e^{2t} dx^2 +e^{2t}dy^2$ where the euclidean coordinates $(x,y)$ on the flat torus
are such that the horizontal lines are the geodesics parallel to $\alpha_i$ and the vertical 
lines define the orthogonal foliation. 
 
Let $\sigma:\mathbb{R}\to [0,1]$ be a smooth function 
which vanishes on $[0,\infty)$ and equals $1$ on $(-\infty,-1]$.
Define a metric $\hat g_i$ on 
$C_i$ by 
\[\hat g_i= g -(\sigma(t)e^{-R_i-t}/2)dx^2+(\sigma(t)e^{-R_i-t}/2) dy^2.\]
Then the metric $\hat g_i$ on the distance tori of distance
bigger than one can be written in orthogonal coordinates as $(\sinh (R_i-t))^2 dx^2+(\cosh (R_i-t))^2dy^2$. 
In particular, if $\ell_i >0$ is such that 
the height of the cylinder obtained by cutting the boundary component of $N(\partial C_i,1)$ distinct from 
$\partial C_i$ open along $\alpha_i$
equals $\ell_i \cosh (R_i-1)$, then we can glue a hyperbolic tube
to this boundary with core 
curve of radius $\ell_i$ and meridian $\alpha_i$. 

The resulting manifold $\hat M$ is obtained from $M$ by
Dehn filling of $\alpha_i$. Furthermore, it 
is equipped with a Riemannian metric of curvature in the interval $[-1-\epsilon,-1+\epsilon]$ which 
is of constant curvature $-1$ outside of the sets $N(\partial C_i,1)$,  and 
the integral of the traceless Ricci curvature fulfills the assumptions in Theorem \ref{drilling}. 

Thus the arguments in the proof of Theorem \ref{drilling} apply without change 
and show the theorem.
\end{proof}


\section{Effective hyperbolization I}\label{effective}

A \emph{handlebody of genus $g\geq 1$} is a
compact 3-manifold $H$
which is diffeomorphic to the connected sum of $g$ solid tori and 
whose boundary
is a closed oriented surface $\partial H=\Sigma$ of genus $g$. 
It is characterized up to marked homotopy equivalence by 
its \emph{disk set}, which can be thought of as 
the collection of all essential simple closed
curves on $\partial H$ which bound embedded disks in $H$.
Equivalently, the disk set is the set of all essential
simple closed curves
in $\partial H$ which are
homotopic to zero in $H$.

If we glue two
handlebodies $H_1,H_2$ along their boundaries with 
an orientation reversing diffeomorphism $f:\partial H_1\to \partial H_2$,
then the resulting 3-manifold is closed and oriented. 
Up to homotopy and
hence diffeomorphism, it only depends on the isotopy class of $f$, 
in fact, only on the double coset of this isotopy class in the mapping 
class group of $\partial H$ which allows for 
precomposition of $f$ with an element of the \emph{handlebody group} of 
$H_1$ and postcomposition of $f$ with an element of the handlebody group of $H_2$.

The \emph{curve graph} ${\cal C\cal G}(\Sigma)$ 
of $\Sigma=\partial H$ is the graph whose
vertices are isotopy classes of simple closed curves on $\Sigma$
and where two such curves can be connected by an edge of length
one if they can be realized disjointly. The curve graph of
$\Sigma$ is known to be a hyperbolic geodesic metric graph.
The disk set of $H$ determines a full subgraph 
${\cal D}$ of ${\cal C\cal G}(\Sigma)$ whose vertex set is the disk set of $H$. 
This subgraph is
uniformly \emph{quasi-convex} \cite{MM04} in ${\cal C\cal G}(\Sigma)$.
This means that there
exists a number $k>0$ only depending on the genus of the
surface $\Sigma$ 
such that for any two disks $a,b\in {\cal D}$,
any geodesic in ${\cal C\cal G}(\Sigma)$ connecting $a$ to $b$ is contained
in the $k$-neighborhood of ${\cal D}$.

The \emph{Hempel distance} of the Heegaard splitting $f$ is defined
to be the distance in ${\cal C\cal G}(\Sigma)$ between
the disk set 
${\cal D}_2$ of $H_2$ and the image
${\cal D}_1$ under the gluing map $f$ of the disk set
of $H_1$. If the Hempel distance 
is at least three,  then the manifold $M_f$ 
is known to be aspherical and atoroidal \cite{He01}
and hence by the geometrization
theorem, it admits a hyperbolic metric. The goal of this section is to give
an effective construction of such a metric not depending on any 
earlier hyperbolization result provided that the gluing map $f$ fulfills some
combinatorial requirement which for example is satisfied for random 
3-manifolds. We refer to \cite{HV22} for a detailed account on the geometry
of random 3-manifolds.

To introduce the combinatorial condition, note that 
since ${\cal D}_i$ is a quasi-convex subset of the curve graph
and the curve graph is hyperbolic \cite{MM99}, if the distance between
${\cal D}_1$ and ${\cal D}_2$ is larger than a constant $b>0$ only
depending on the hyperbolicity constant of ${\cal C\cal G}(\Sigma)$
(which does not depend on $\Sigma$) and 
of the quasi-convexity constant for the embedding
${\cal D}_i\to {\cal C\cal G}(\Sigma)$ (which only
depends on $\Sigma$), then there is a coarsely well defined
shortest geodesic $\zeta$ in ${\cal C\cal G}(\Sigma)$ connecting
${\cal D}_1$ to ${\cal D}_2$. This means that $\zeta$ is a
geodesic in ${\cal C\cal G}(\Sigma)$ which connects 
a point $c_1\in {\cal D}_1$ to a point $c_2\in {\cal D}_2$, and
if $\nu$ is any geodesic in ${\cal C\cal G}(\Sigma)$ 
connecting a point in ${\cal D}_1$ to a
point in ${\cal D}_2$, then $\zeta$ is entirely contained in the
$r$-neighborhood of $\nu$ where $r>0$ is a constant only
depending on $\Sigma$. For the remainder of this section, we always assume
that $d_{\cal C\cal G}({\cal D}_1,{\cal D}_2)\geq b$.

For a proper essential connected subsurface $S$ of $\Sigma$
different from a pair of pants and an annulus,
define the \emph{arc and curve graph} ${\cal C\cal G}(S)$ of 
$S$ to be the graph whose vertices are essential simple closed curves in 
$S$ or essential arcs with endpoints in the boundary 
$\partial S$ of $S$. Two such arcs or curves are connected by 
an edge of length 1 if they can be realized disjointly.
If $S$ is an annulus then this construction has to be modified.
As we do not need more precise information here,
we omit a more detailed discussion which can be found in \cite{MM00}. 
For a simple closed curve $c\in {\cal C\cal G}(\Sigma)$
with an essential intersection
with $S$, the \emph{subsurface projection} of $c$ into 
${\cal C\cal G}(S)$ is the union of all intersection components 
of $c$ with $S$ (properly interpreted if $S$ is an annulus).

By the above discussion and Theorem 3.1 of \cite{MM00}, there exists 
a number $p>0$ with the following property. 
Let as before $\zeta$ be a shortest
geodesic in ${\cal C\cal G}(\Sigma)$ connecting ${\cal D}_1$ to ${\cal D}_2$. 
Let us assume that  
there exists a proper essential connected 
subsurface $S\subset \Sigma$
whose boundary $\partial S$ consists of a collection of
simple closed curves whose distance to each of the endpoints of $\zeta$ is
at least $p$. Let us also assume that 
the diameter of the subsurface projection of 
the endpoints of $\zeta$ into $S$ equals $k\geq 2p$. Then for any 
pair $(a_1,a_2)\in {\cal D}_1\times {\cal D}_2$, the diameter of the 
subsurface projection of $a_1,a_2$ into $S$ is at least $k-p\geq p$. 
Furthermore, any geodesic in ${\cal C\cal G}(\Sigma)$ 
connecting $a_1$ to 
$a_2$ passes through a simple closed curve which is 
disjoint from $S$.

The following is the main result of this section. For its
formulation, we define an  \emph{$\varepsilon$-model metric} on 
a closed ashperical atoroidal 3-manifold $M$ to be a metric which
fulfills the assumptions in \Cref{Pinching without inj radius bound - full version} 
for the control constant $\varepsilon$.

\begin{thm}\label{relative}
For every $k\geq 2p,\varepsilon >0$ 
there exists a number $b=b(\Sigma,k,\varepsilon)>0$ with the following property.
Let ${\cal D}_1,{\cal D}_2$ be the disk sets of the manifold $M_f$.  
Assume that a minimal geodesic $\zeta$ 
in ${\cal C\cal G}(\Sigma)$  
connecting ${\cal D}_1$ to ${\cal D}_2$ contains a subsegment $\hat \zeta$ 
of length at least $b$ whose endpoints do not have
any subsurface projection of diameter at least $k$ into any
subsurface of $\Sigma$.
Then $M_f$ admits an explicit $\varepsilon$-model metric which is
$\varepsilon$-close in the $C^2$-topology 
to a  hyperbolic metric.
\end{thm}

We begin the proof of Theorem \ref{relative}
by recalling some results from \cite{HV22}. 

The geometry of the curve graph of the surface $\Sigma$ 
is coarsely tied to the geometry
of the \emph{Teichm\"uller space} ${\cal T}(\Sigma)$ of
$\Sigma$. Namely, there is a (coarsely well-defined)
${\rm Mod}(\Sigma)$-equivariant Lipschitz map 
$\Upsilon:{\cal T}(\Sigma)\to {\cal C\cal G}(\Sigma)$, called the
{\em systole map}, that associates to every marked
hyperbolic structure $X\in{\cal T}(\Sigma)$ a shortest geodesic
$\Upsilon(X)$ on it. It follows from
Masur-Minsky \cite{MM99} (see Lemma 3.3 of \cite{MM00} for a precise
account, and note that small extremal length of a closed curve on a
Riemann surface is equivalent to small hyperbolic length) 
that there exists a
constant $L>1$ only depending on $\Sigma$ such that
for every Teichm\"uller geodesic $\gamma:I\to{\cal T}(\Sigma)$
(here $I$ is a connected subset of $\mathbb{R}$),
the composition $\Upsilon \circ \gamma:I\to{\cal C\cal G}(\Sigma)$ is an
{\em unparameterized} $L$-{\em quasi-geodesic}.
This means that there exists a homeomorphism
$\rho:J\to I$ such that the composition
$\Upsilon\circ \gamma\circ \rho$ is an $L$-quasi-geodesic in
${\cal C\cal G}(\Sigma)$.
Moreover, if we restrict our attention to the
\emph{$\delta$-thick part} ${\cal T}_\delta(\Sigma)$ of Teichm\"uller space
of all hyperbolic metrics whose \emph{systole}, that is, the length of a shortest
closed geodesic, is at least $\delta$, 
then the situation improves: In \cite{H10} it is shown that
for every $\delta>0$ there exist $L_\delta>1$
such that if $\gamma$ is parameterized by arc length on an
interval $I$ of length at least $L_\delta$ and if
$\gamma(I)\subset {\cal T}_\delta(\Sigma)$, then
$\Upsilon \circ \gamma$ is
a {\em parameterized} $L_\delta$-quasi-geodesic.

Using the notations from Theorem \ref{relative}, 
let us now assume that for some $k\geq 2p$, 
a shortest 
geodesic $\zeta$ 
connecting ${\cal D}_1$ to ${\cal D}_2$  
contains a subsegment $\hat \zeta$ with 
the  property that
there does not exist \emph{any} proper essential subsurface $S$ of $\Sigma$
for which the diameter of the subsurface projection of the endpoints of
$\hat \zeta$ into $S$ is larger than $k$. If the length of $\hat \zeta$ is 
sufficiently large, then we shall construct from this segment  
two convex cocompact hyperbolic
metrics on a handlebody which contain large almost isometric 
regions whose injectivity radius is bounded from 
below by a universal constant. 
We then glue the handlebodies along these regions with a map which
is close to an isometry for these metrics and construct
a closed 3-manifold diffeomorphic to $M_f$ with an $\epsilon$-model metric
which can be deformed to a hyperbolic metric using
\Cref{Pinching without inj radius bound - introduction}.
 
To implement this program, we follow  
\cite{HV22} and introduce a notion of 
{\em relative bounded combinatorics} and {\em height}.
Fix a sufficiently small threshold $\delta>0$. 
Denote by $d_{\cal T}$ the distance on
${\cal T}(\Sigma)$ for the Teichm\"uller metric.

\begin{definition}[Relative Bounded Combinatorics]\label{relativedef}
  Consider $Y,X\in{\cal T}(\Sigma)$. We say that $(Y,X)$ has
  {\em relative $\delta$-bounded combinatorics} with respect to the
  handlebody $H$ with disk set ${\cal D}$ if
  the Teichm\"uller geodesic $[Y,X]$ connecting $Y$ to $X$ is contained in
  ${\cal T}_\delta(\Sigma)$ and if
\[
d_{\cal C\cal G}({\cal D},\Upsilon(Y))+d_{\cal C\cal G}(\Upsilon(Y),\Upsilon(X))\le d_{\cal C\cal G}({\cal D},\Upsilon(X))+\frac{1}{\delta}.
\]
The {\em height} of the pair $(Y,X)$ is $d_{\cal T}(Y,X)$. 
\end{definition}

A \emph{convex cocompact} metric on a handlebody $H$ is a complete
hyperbolic metric on the interior of $H$ with the
following property. The hyperbolic metric determines
up to conjugacy an embedding
of the fundamental group of $H$ (which is the free group with $g$ generators)
into $PSL(2,\mathbb{C})$. The image group $\Gamma$ acts on the boundary
$\partial \mathbb{H}^3$ of hyperbolic 3-space, preserving a decomposition
of $\partial \mathbb{H}^3$ into the \emph{limit set} $\Lambda(\Gamma)$
and the \emph{domain of discontinuity} $\Omega(\Gamma)$.

The quotient $\mathbb{H}^3\cup \Omega(\Gamma)/\Gamma$ is compact
and homeomorphic to the handlebody $H$. Moreover,
as the action of $\Gamma$ on $\Omega(\Gamma)$ preserves the conformal
structure, the quotient $\Omega(\Gamma)/\Gamma$ is the surface
$\Sigma$ equipped with a conformal structure $X\in {\cal T}(\Sigma)$.
Up to isometry, the convex cocompact metric on $H$ is determined by
$X$, and the corresponding hyperbolic handlebody will be denoted by 
$H(X)$.
The \emph{convex core} ${\cal C\cal C}(H(X))$ of $H(X)$ 
is the
quotient of the convex hull of $\Lambda(\Gamma)$ in $\mathbb{H}^3$ 
by the action of
$\Gamma$, with boundary $\partial {\cal C\cal C}(H(X))$.
The convex core ${\cal C\cal C}(H(X))$ is homeomorphic to the handlebody $H$.

A \emph{product region} in 
a convex cocompact hyperbolic handlebody $H(X)$ with boundary 
surface $\Sigma$ is 
a codimension $0$ submanifold $U\subset H(X)$ contained in the 
convex core ${\cal C\cal C}(H(X))$ of $H(X)$ which is homeomorphic
to $\Sigma\times [0,1]$ with a homeomorphism whose restriction to 
each surface $\Sigma\times \{s\}$  is homotopic to the inclusion
$\Sigma\to \partial {\cal C\cal C}(H(X))\subset H(X)$. 
If $U$ is such a product region then we can
define the \emph{width} ${\rm width}(U)=\inf\{d(x,y)\mid x\in \partial ^+U,y\in \partial U^-\}$
where $\partial^{\pm}U$ are the two boundary components of $U$. 
If the width of the product region is at
least $D$ and the diameter is at most $2D$ then
we say that the product region has \emph{size} $D$ (see Section 5 of \cite{HV22}).

A product region $U\subset H(X)$ can be used to 
decompose the handlebody $H(X)$ into two connected components. 
The first component is the closed subset
of $H(X)-U$ containing 
$H(X)-{\cal C\cal C}(H(X))$ 
(it is straightforward that there are no choices made in this
construction).   
The second component is its complement, which is an open subset
of $H$ containing $U$. 
We define the \emph{gluing block} $H_U$ of $H(X)$ to be the
component containing $U$.


Fix a number $\alpha\in (0,1)$. 
For a number $\xi>0$, define a \emph{$\xi$-almost isometry} 
between two Riemannian manifolds $(M_1,\rho_1),(M_2,\rho_2)$ to be a smooth map
$\Phi:M_1\to M_2$ such that $\Vert \rho_1-\Phi^*\rho_2\Vert_{C^{2}}<\xi$.
Our main technical tool is the following
Theorem 5.12 of \cite{HV22}. In its formulation,
$d_{\cal C\cal G}$ denotes as before the distance in the curve graph of 
$\Sigma=\partial H_i$. 

\begin{thm}[The gluing theorem]\label{gluinghv}
For $\delta >0$ there exists $\iota=\iota(\delta)>0, D=D(\delta)>0$, 
 and for $\xi>0$ there exists 
$h_{\rm gluing}(\delta,\xi)>0$ such that the following holds true.
Let $H_1,H_2$ be two handlebodies of genus $g$ and let 
$f:\partial H_1\to \partial H_2$ be a gluing map.
Let $[Y,X]\subset {\cal T}_\delta(\Sigma)$ be a geodesic segment 
satisfying the following relative bounded combinatorics and large heights properties:
\begin{itemize}
\item{
$d_{\cal T}(Y,X)
\in[h,2h]$ for some $h>h_{\rm gluing}(\delta,\xi)$.}
\item{ If ${\cal D}_1$ denotes the disk set of the handlebody $H_1$, then the pair $(Y,X)$ satisfies 
\[
  d_{\cal C\cal G}(\Upsilon(X),\Upsilon(Y)) +
  d_{\cal C\cal G}(\Upsilon(Y),{\cal D}_1)\leq 
d_{\cal C\cal G}(\Upsilon(X),{\cal D}_1)+\frac{1}{\delta}. 
\]
The same holds true for the pair $(f^{-1}X,f^{-1}Y)$ and the disk set ${\cal D}_2$ of $H_2$.}
\end{itemize}
Consider $N_1=H(Y),N_2=H(f^{-1}X)$.
Then there exist: 
\begin{itemize}
\item{Product regions $U_j\subset{\cal C\cal C}(N_j)$ of size $D$ for $j=1,2$. 
We denote by $N_j^0\subset{\cal C\cal C}(N_j)$ the gluing blocks they define.}
\item{An orientation reversing $\xi$-almost isometric diffeomorphism
$\Phi:U_1\to U_2$ for $j=1,2$ in the homotopy class of $f$.}
\end{itemize}
In particular, we can form the 3-manifold
\[
X_f=N_1^0\cup_{\Phi:U_1\to U_2}N_2^0
\] 
obtained from the disjoint union of $N_1^0,N_2^0$ by identifying a point $x\in N_1^0$ 
with its image under $\Phi$ in $N_2^0$. 
This manifold is diffeomorphic to $M_f=H_1\cup_f H_2$. 
Denote by $\Omega$ the image in $X_f$ of  $U_1\cup U_2$.  
The manifold $X_f$
comes equipped with a Riemannian metric $\rho$ with the following properties:
\begin{enumerate}[i)]
\item{The sectional curvature of $\rho$ is contained in the interval $(-1-\xi,-1+\xi)$, and it is constant $-1$ on $X_f-\Omega$.}
\item{The diameter of $\Omega$ is at most $2D$, and the injectivity radius on $\Omega$ is at least $\iota$.} 
\item{The two components of $X_f-\Omega$ are isometric to the complement in ${\cal C\cal C}(N_j)$ 
of collar neighborhoods about the boundary  of ${\cal C\cal C}(N_j)$ 
of uniformly bounded diameter
(depending on $h$ and hence on $\xi,\delta$).}
 \end{enumerate}
\end{thm}

We call the metric constructed in Theorem \ref{gluinghv} from the convex cocompact 
handlebodies $H(Y),H(f^{-1}X)$
and the gluing map $f$ 
a \emph{$\xi$-model metric with $\delta$-bounded combinatorics}. 
Theorem \ref{gluinghv} 
then can be restated as saying that if
$M_f$ fulfills the assumption stated in Theorem \ref{gluinghv},
then it admits a $\xi$-model metric with $\delta$-bounded combinatorics. 
Note that the lower injectivity radius bound on $\Omega$ is not explicitly stated in 
Theorem 5.12 of \cite{HV22}, but is discussed in Section 5.3 of \cite{HV22}. 
The metric has constant curvature $-1$ outside of the open subset $\Omega$, called 
the \emph{gluing region} in the sequel. Furthermore, since it is constructed by gluing 
two almost isometric hyperbolic metrics with a gluing function all of whose derivatives
are uniformly bounded (and in fact small depending on the geometric data which enter
in the construction), the covariant derivative of ${\rm Ric}$ is pointwise uniformly bounded
by a constant only depending on $\delta$.

By construction, the gluing region 
contains an open subset diffeomorphic to $\Sigma\times (0,1)$ whose diameter is 
bounded from above by a constant $2D>0$ only depending on $\delta$,
whose injectivity radius is bounded from below by a constant $\iota$ 
only depending on $\delta$ and is such that the distance between the two boundary 
surfaces of this set is at least $D$. This implies that its 
volume is contained in the interval $[v^{-1},v]$ for a number 
$v>0$ only depending on the genus $g$ of the handlebody and on $\delta$. 



Theorem \ref{Pinching without inj radius bound - introduction}
from the introduction
can be used to promote a model metric to a 
hyperbolic metric which is close to the model metric in the
$C^2$-topology. We summarize this as follows. 

\begin{prop}\label{hyperbolization}
For all  $\varepsilon >0,k >0$ there
exist numbers $b=b(\varepsilon,k)>0$,
$\xi=\xi(\varepsilon,k)>0$, $\delta=\delta(\varepsilon,k)>0$ and $v=v(\varepsilon,k)>0$ 
with the following properties.
Let $f:\partial H_1\to \partial H_2$ be a gluing map und use this to define
the disk sets ${\cal D}_1,{\cal D}_2$. 
Assume that a shortest geodesic $\zeta$ connecting ${\cal D}_1$ to ${\cal D}_2$
contains a subsegment $\hat \zeta$ of length at least $b$ such
that for any proper essential subsurface $S$ of 
$\Sigma$, the diameter of the subsurface projection of the endpoints 
of $\hat \zeta$ into $S$ is at most $k$. 
Then the manifold 
$M_f$ admits a $\xi$-model metric with $\delta$-bounded combinatorics
which is $\varepsilon$-close in the $C^2$-topology to a hyperbolic metric.
If $\tilde \zeta$ is another subsegment of $\zeta$ of length at least $b$ 
which has the same properties as $\hat \zeta$ and is disjoint from $\hat \zeta$,
then
these two segments determine a submanifold of $M_f$ diffeomorphic to 
$\Sigma\times [0,1]$ whose 
volume for the hyperbolic metric on 
$M_f$ is at least $v(d_{\cal C\cal G}(\hat \zeta,\tilde \zeta))$.
\end{prop}



\begin{proof} The distance formula Theorem 6.12 of \cite{MM00} 
and its variation for the Teichm\"uller metric together with 
the main result of \cite{H10} and Lemma 6.7 and Lemma 6.8 of \cite{HV22} shows that 
for every $k>0$ 
there are numbers $m_0=m_0(k)$,
$\sigma=\sigma(k)>0$ and $L=L(k)>1$ 
with the following property.

Let $m\geq 3m_0$ and let 
$\eta:[0,m]\to {\cal C\cal G}(\Sigma)$ be a geodesic 
with the property that there exists no proper essential 
subsurface $S$ of $\Sigma$ such that the diameter of 
the projection of the endpoints $\eta(0),\eta(m)$ of $\eta$
into the arc and  curve graph of $S$ is 
larger than $k$.  Let $X,Y\in {\cal T}(\Sigma)$ be such that the $Y$-length of the 
curve $\eta(0)$ is not larger than a Bers constant for $\Sigma$, and that the same holds
true for the $X$-length of the 
curve $\eta(m)$.  Let
$[Y,X]$ be the Teichm\"uller geodesic connecting $Y$ to $X$. 
Then there exists a subsegment 
$[Y_0,X_0]\subset [Y,X]$ entirely contained in ${\cal T}_\sigma(\Sigma)$ 
with the property that $\Upsilon\vert [Y_0,X_0]$ is an $L$-quasi-geodesic 
connecting a point in the $L$-neighborhood of $\eta(m_0)$ to a point in the 
$L$-neighborhood of $\eta(m-m_0)$.
In particular, we have 
\[d_{\cal C\cal G}(\eta(m_0),\Upsilon(Y_0))+d_{\cal C\cal G}(\Upsilon(Y_0),\Upsilon(X_0))\leq 
d_{\cal C\cal G}(\eta(m_0),\Upsilon(X_0))+L\]
and similarly for $\eta(m-m_0)$ and $\Upsilon(X_0)$.

It follows from the construction in the previous paragraph that 
up to a uniform adjustment of constants, 
if $\eta$ is a subsegment of a minimal geodesic connecting 
${\cal D}_1$ to ${\cal D}_2$, then the pair $(Y_0,X_0)$ has relative 
$\sigma$-bounded combinatorics with respect to the handlebody $H_1$ with disk
set ${\cal D}_1$, and $(X_0,Y_0)$  has relative $\sigma$-bounded combinatorics with respect to 
the handlebody $H_2$ with disk set ${\cal D}_2$ in the sense of Definition \ref{relativedef}. 
The height $d_{\cal T}(Y_0,X_0)$
is bounded from below by $(m-2m_0-2L)/c-c$ for a universal constant $c>0$ 
by the fact that the image of 
$[Y_0,X_0]$ under $\Upsilon$ is an $L$-quasi-geodesic connecting two points in 
${\cal C\cal G}(\Sigma)$ of distance at least $m-2m_0-2L$ and the fact that $\Upsilon$ is coarsely
$c$-Lipschitz.
As a consequence, for any 
$\xi >0$, if $h=h_{\rm gluing}(\sigma,\xi)>0$ is as in Theorem \ref{gluinghv} and if 
$m>ch+2m_0+2L+c^2$, then this height is at least  $h$.

Recall that the diameter $D$ and hence the volume
of the gluing region $\Omega$ in the statement of 
Theorem \ref{gluinghv} for 
$\delta=\sigma$ and $\xi$ is bounded from above by a constant which 
only depends on $\sigma$ but not on $\xi$. 
Since 
the sectional curvature of the model metric is contained in 
the interval $[-1-C\xi,-1+C\xi]$, we know that for a given number $\varepsilon >0$
and the fixed number $\sigma$ which only depends on $k$, 
there exists a number $\xi_0=\xi_0(\epsilon,\sigma)>0$ such that if $\xi<\xi_0$, then the 
$\xi$-model metric with relative $\sigma$-bounded combinatorics on $M_f$ 
fulfills the assumptions in \Cref{Pinching without inj radius bound - introduction}
for this number $\varepsilon$.
An application of Theorem \ref{Pinching without inj radius bound - introduction}
then shows that there is a hyperbolic metric
on $M_f$ in the $\varepsilon$-neighborhood of the model metric in the 
$C^2$-topology.

We are left with the volume estimate. To this end note that 
the construction of the Teichm\"uller segment $[Y_0,X_0]$ which gave
rise to the gluing region $\Omega$ only used a sufficiently long
subsegment $\hat \zeta$ of a minimal geodesic in ${\cal C\cal G}(\Sigma)$ connecting
${\cal D}_1$ to ${\cal D}_2$. Let us now assume that $\tilde \zeta$ is a
second such subsegment which is disjoint from 
$\hat \zeta$, and let us assume that
it is contained in the subsegment of $\zeta$ connecting ${\cal D}_1$ to 
$\hat \zeta$. Let $[W,V]$ be a Teichm\"uller geodesic segment
constructed from $\tilde \zeta$ as in the first paragraph of this proof, and let
$[W_0,V_0]\subset [W,V]$ be the subsegment in ${\mathcal T}_\sigma(\Sigma)$ 
found with the argument in the second paragraph of this proof. 
By Proposition 4.1 of \cite{HV22}, the convex cocompact handlebody
$H(Y)$ which entered the above construction 
contains a submanifold $N$ which is $\xi$-almost isometric to 
a submanifold $N_0$ of the quasi-fuchsian manifold $Q(V_0,Y_0)$ defined by the marked hyperbolic
surfaces $V_0,Y_0$,  and this
submanifold contains the complement 
of a collar of uniformly bounded height about the boundary 
in its convex core ${\cal C\cal C}(Q(V_0,Y_0))$. 
By Theorem \ref{gluinghv}, 
the submanifold $N$ of $H(Y)$ 
is isometrically embedded in $M_f$, equipped with the model metric 
constructed as above from the segment $\hat \zeta$.

The model manifold theorem \cite{M10} or earlier work of Brock \cite{Br03} shows that there exists
a constant $\rho >0$ such that 
the volume of the manifold $N_0$ and hence the volume of $N$ is 
bounded from below by 
$\rho d_{WP}(V_0,Y_0)$,
where $d_{WP}$ is the distance in ${\cal T}(\Sigma)$ induced by the Weil Petersson metric.
Namely, by \cite{Br03}, the volume of the quasi-fuchsian manifold $Q(V_0,Y_0)$ is 
bounded from below by $\rho^\prime d_{WP}(V_0,Y_0)$ for a constant $\rho^\prime=
\rho^\prime(\Sigma)$, and ${\rm vol}(Q(V_0,Y_0)-N_0)$ is uniformly bounded. 
As by bounded combinatorics and the explicit construction, 
the Weil Petersson distance $d_{WP}(V_0,W_0)$ is large, the volume estimate for 
$N$ follows from an adjustment of constants. 

Now there exist a number $C>0$ such that  
$d_{WP}(V_0,Y_0)\geq  Cd_{\cal C\cal G}(\Upsilon(V_0),\Upsilon(Y_0))$ \cite{Br03} and hence
the volume of the submanifold $N$ of $M_f$ with respect to the model
metric is bounded from below by 
$\rho Cd_{\cal C\cal G}(\Upsilon(V_0),\Upsilon(Y_0))\geq \rho Cd_{\cal C\cal G}(\hat \zeta, \tilde \zeta)$.
%
Thus \Cref{Pinching without inj radius bound - introduction} 
 implies that the same holds true for $M_f$, equipped with the hyperbolic metric
 (recall the convention of adjusting constants). 
This completes the proof of the proposition.
\end{proof}

\begin{rem}\label{volumesharp}\normalfont
The volume estimate in Proposition \ref{hyperbolization} is far from being sharp. Namely,
in the proof, we used the fact that under suitable assumptions on the gluing map $f$, 
the hyperbolic manifold $M_f$ contains an embedded subset which is almost 
isometric to the complement of a collar of uniformly bounded height 
about the boundary in ${\cal C\cal C}(Q(V_0,Y_0))$ where
$Q(V_0,Y_0)$ is a quasi-fuchsian manifold whose conformal boundaries $V_0,Y_0$ are contained in the 
thick part of Teichm\"uller space. By a result of Brock \cite{Br03}, the volume of 
$Q(V_0,Y_0)$ is proportional to the Weil-Petersson distance $d_{WP}(V_0,Y_0)$ 
between $V_0,Y_0$ in ${\cal T}(\Sigma)$, and the ratio 
$d_{\cal C\cal G}(\Upsilon(V_0),\Upsilon(Y_0))/d_{WP}(V_0,Y_0)$
can be arbitrarily small.

Since the Weil-Petersson distance between $V_0,Y_0\in {\cal T}(\Sigma)$ is proportional to the distance
in the pants graph between shortest pants decompositions for $V_0,Y_0$ \cite{Br03},
this leads us to conjecture that the volume of a hyperbolic 3-manifold $M_f$ with Heegaard surface $\Sigma$
of minimal genus is proportional to the minimal distance in the pants graph 
between two pants decompositions $P_1\subset {\cal D}_1$ and $P_2\subset {\cal D}_2$, with constants
only depending on $\Sigma$.
\end{rem}

\section{A priori geometric bounds for closed
  hyperbolic manifolds}
\label{lengthbounds}

The goal of this section is to obtain some geometric control on a closed hyperbolic 
$3$-manifold $M_f$ 
constructed by gluing two handlebodies $H_1,H_2$ with
boundary $\partial H_1=\partial H_2=\Sigma$ with 
a gluing map $f$ which does not fulfill the combinatorial condition
in Theorem \ref{relative}. This leads to the proof of Theorem \ref{lengthbound - introduction}
from the introduction. We always assume that the Hempel distance of 
the Heegaard splitting is at least $3$. This rules out the existence of trivial handles in the 
Heegaard surface.

As before, denote by ${\cal D}_1,{\cal D}_2$
the disk sets of $M_f$, viewed as subsets of the curve graph ${\cal C\cal G}(\Sigma)$
of $\Sigma$.
Call a proper essential subsurface $Y$ of 
$\Sigma$ 
\emph{strongly incompressible}
in $M_f$ if the distance in ${\cal C\cal G}(\Sigma)$ between 
$\partial Y$ and ${\cal D}_1\cup {\cal D}_2$ is at least three.
This implies 
that the boundary $\partial Y$ of $Y$ consists of simple closed curves
in $\Sigma$ which are not homotopic to zero in $M_f$. 
More concretely, we have

\begin{lem}\label{incompressiblesub}
Let $Y\subset \Sigma$ be a strongly incompressible subsurface.
\begin{enumerate}[i)]
\item 
For any boundary component $\gamma$ of $Y$,
the inclusion $\Sigma\setminus \gamma\to M_f\setminus \gamma$ is $\pi_1$-injective. 
\item 
If $\alpha\subset Y$ is an embedded essential 
arc with endpoints on $\partial Y$, then
$\alpha$ is not homotopic in $M_f$ into $\partial Y$ keeping the endpoints in 
$\partial Y$.
\item If $\alpha,\beta$ are two disjoint non-homotopic essential arcs
in $Y$, then $\alpha,\beta$ are not homotopic in 
$M_f$ keeping the endpoints in 
$\partial Y$. 
\end{enumerate}
\end{lem}
\begin{proof}
The first statement of the lemma follows from Dehn's lemma \cite{He76},
applied to the complement
of a small open tubular neighborhood $N$ of $\gamma$ in $M_f$.  
Namely, $\Sigma-N$ is a properly embedded bordered surface in $M_f-N$,  and hence if 
 there is an essential closed curve in
$\Sigma-\gamma$ which is contractible in $M_f-\gamma$, then 
there is an essential 
simple closed curve in $\Sigma-\gamma$ which bounds a disk in $M_f$.
 But this contradicts the fact that
any diskbounding simple closed 
curve in $\Sigma$ has essential intersections with $\gamma$.

Now let $\alpha\subset Y$ be an embedded essential arc with endpoints on the same component $\zeta$ of $\partial Y$.
Then each component of the boundary of a small neighborhood of $\alpha\cup \zeta$ in $Y$ 
is an essential simple 
closed curve in $Y$. This curve is not homotopic to zero in $M_f$ by the discussion in the previous 
paragraph. But this means that $\alpha$ is not homotopic in $M_f$  
into $\partial Y$ keeping the endpoints in $\partial Y$.

Similarly, if $\alpha_1,\alpha_2$ are two disjoint non-homotopic arcs in $Y$ with endpoints on the 
same (not necessarily distinct) boundary components $\zeta_1,\zeta_2$ of $Y$,
then their union with suitable
chosen subarcs of $\zeta_1\cup \zeta_2$ defines an essential simple closed curve in $Y$ to which the above 
discussion applies. Thus such arcs can not be homotopic in $M_f$ keeping the endpoints
in $\partial Y$.
\end{proof}

Denote by $d_{\cal C\cal G}$ the distance in the curve graph of
$\Sigma$.
Lemma \ref{incompressiblesub} does not state that
for a subsurface $Y\subset \Sigma$ with $d_{\cal C\cal G}(\partial Y,
{\cal D}_1\cup {\cal D}_2)\geq 3$, the inclusion
$Y\to M_f$ is $\pi_1$-injective. However we have the following
weaker statement. 

For its formulation, 
for a proper essential subsurface $Y$ of $\Sigma$ denote by 
$d_Y$ the distance in the arc and curve graph of
$Y$, and  
${\rm diam}_Y$ denotes the diameter of subsets of this graph.
Furthermore, if $\alpha_1,\alpha_2$ are simple closed curves in $\Sigma$ which have
an essential intersection with $Y$,  then we write 
$d_Y(\alpha_1,\alpha_2)$ to denote the distance in the arc and curve graph of 
$Y$ between the \emph{subsurface projections} of $\alpha_1,\alpha_2$, that is, 
the components of $\alpha_i\cap Y$.

\begin{lem}\label{homotopyidentity}
  There exists a number $p=p(\Sigma)>4$
  with the following property.
Let $Y\subset \Sigma$ be a strongly incompressible subsurface whose 
boundary $\partial Y$, as a geodesic multicurve in $M_f$, 
 fulfills
 $d_{\cal C\cal G}(\partial Y,{\cal D}_1\cup {\cal D}_2)\geq p$.
 If $h:\Sigma\times [0,1]\to M_f$ is any homotopy of the inclusion
 which preserves $\partial Y$ and if
 $h_1:\Sigma\to \Sigma$ is a homotopy equivalence, then $h_1$ 
induces the identity on $\pi_1(\Sigma)$.
\end{lem}

\begin{proof} By \cite{MM00}, there exists a
number $p=p(\Sigma)>4$ with the following 
property. Let $\alpha,\beta$ be simple closed curves on $\Sigma$ 
and let $Y\subset \Sigma$ be a subsurface
which has an essential intersection with $\alpha,\beta$ and  
such that
$d_Y(\alpha,\beta)\geq p-1$; then any geodesic in ${\cal C\cal G}(\Sigma)$ connecting
$\alpha$ to $\beta$ has to pass through a curve disjoint from $Y$.

Since by \cite{MM04}, the disk sets ${\cal D}_1,{\cal D}_2$ are 
uniformly quasi-convex subsets of ${\cal C\cal G}(\Sigma)$, this implies that
up to increasing $p$, the following holds true.
Let $Y\subset \Sigma$ be 
any proper essential subsurface such that 
$d_{\cal C\cal G}(\partial Y,{\cal D}_i)\geq p$ $(i=1,2)$;
then ${\rm diam}_Y({\cal D}_i)\leq p$.

Let $Y\subset \Sigma$ be such a subsurface and let
$h:\Sigma\times [0,1]\to M$ be a homotopy of the inclusion
$h_0:\Sigma \to M$ which fixes $\partial Y$. Assume that
$h_1$ is a homotopy equivalence of $\Sigma$ 
onto $h_0(\Sigma)$. Then
$h_1$ defines a mapping class
$\varphi\in {\rm Mod}(W)$, the mapping class group of the component $W$.
We claim that $\varphi$ induces the identity on $\pi_1(W)$.

Namely, as $W$ is a surface with non-empty boundary, the group
${\rm Mod}(W)$ does not have elements of finite order. Thus
if $\varphi$ is not trivial, then either
$\varphi$ is a pseudo-Anosov mapping class of $W$, or
$\varphi$ preserves a non-trivial multicurve
$\beta\subset W$. Furthermore, there exists 
a subsurface $Z$ of $W$ which is preserved by
$\varphi$, and if $Z$ is not an annulus, then the restriction of
$\varphi$ to $Z$ is a pseudo-Anosov mapping class, and if
$Z$ is an annulus, then the restriction of $\varphi$ to $Z$ is a
Dehn twist. The mapping class $\varphi$ induces the identity on
$\pi_1(W)$ if and only if it is a composition of Dehn twists
about the boundary curves of $W$ or is trivial.

To see that this is the case, note that by composition, 
for each $k\geq 1$, the
mapping class $\varphi^k$ can also be represented by a homotopy of
$\Sigma$ which preserves $\partial Y$.
But $d_{\cal C\cal G}(\partial Y,{\cal D}_1\cup {\cal D}_2)\geq p\geq 4$
and hence since $\partial W$ is disjoint fromm $\partial Y$,
we have $d_{\cal C\cal G}(\partial W,{\cal D}_1\cup {\cal D}_2)\geq p-1$.
In particular, 
any diskbounding simple closed curve in $\Sigma$ has an essential
intersection with $Z$. Furthermore, by the choice of
$p$, we have ${\rm diam}_Y({\cal D}_i)\leq p$ $(i=1,2)$.

Since $\varphi$ is induced by a homotopy of $\Sigma$ in $M$,
it has to preserve the diskbounding curves in $\Sigma$ as this set
is determined by the topology of $M$.
Now if $\varphi$ preserves the non-peripheral subsurface $Z\subset W$ and acts
on $Z$ as a pseudo-Anosov mapping class (here we include the case
that $Z$ is an annulus and $\varphi$ is a Dehn twist) then
\[{\rm diam}_Z({\cal D}_1\cup {\cal D}_2,
 \varphi^k({\cal D}_1\cup {\cal D}_2))\to \infty\quad (k\to \infty).\]
As $\phi^k({\cal D}_1\cup {\cal D}_2)={\cal D}_1\cup {\cal D}_2$, 
and as ${\rm diam}_Z({\cal D}_1\cup {\cal D}_2)<\infty$, this 
is a contradiction which shows the lemma.
\end{proof}

Using  an idea of Minsky \cite{M00}, we establish  
an a priori upper bound for the total length of the boundary $\partial Y$ 
of $Y$ for the hyperbolic metric on $M_f$ 
in terms of 
the diameter of the subsurface projection of ${\cal D}_1,{\cal D}_2$ into $Y$.

More precisely, for a simple closed multi-curve $\gamma$ 
on $\Sigma$, 
denote by $\ell_f(\gamma)$ the sum of the 
minimal lengths
of representatives of the free homotopy classes of the components of $\gamma$
in the hyperbolic manifold $M_f$.
By convention, we have $\ell_f(c)=0$ for any curve on $\Sigma$ which is 
homotopically trivial in $M_f$.
In the statement of the following result and later on, $p\geq 4$ is the
constant from Lemma \ref{homotopyidentity}.

\begin{thm}[A priori length bounds]\label{minsky}
There exists a number $p=p(\Sigma)\geq 3$, and for every $\epsilon >0$ there exists a number 
$k=k(\Sigma,\epsilon)>0$ with the following property. Let $M_f$ be a hyperbolic 3-manifold with 
Heegaard surface $\Sigma$, Hempel distance at least $4$ and disk sets ${\cal D}_1,{\cal D}_2$.
If $Y\subset \Sigma$ is a proper
essential subsurface of $\Sigma$, with $d_{\cal C\cal G}(\partial Y,{\cal D}_1\cup {\cal D}_2)\geq p$ 
and ${\rm diam}_Y({\cal D}_1\cup {\cal D}_2)\geq k$, then $\ell_f(\partial Y)\leq \epsilon$.
\end{thm}

Theorem \ref{minsky} can be thought of as a version of
Theorem B of \cite{M00} in a different setting.
The fact that Heegaard surfaces in $M_f$ are compressible requires
however a substantial modification of the proof. 
 
Following \cite{M00},  
the main tool for the proof of Theorem \ref{minsky} 
are \emph{pleated surfaces}. The pleated surfaces we are interested in 
are maps $g:\Sigma\to M_f$ in the homotopy class of the inclusion
$\Sigma\to M_f$ 
together with a hyperbolic metric $\sigma$ on $\Sigma$
satisfying the following two conditions.
\begin{itemize}
\item $g$ is \emph{path-isometric} with respect to $\sigma$. 
\item There exists a $\sigma$-geodesic lamination $\lambda$ on $\Sigma$
whose leaves are mapped to geodesics by $g$. In the complement of $\lambda$,
$g$ is totally geodesic.
\end{itemize}
The geodesic lamination $\lambda$ is called
the \emph{pleating lamination} of $g$. 
We refer to \cite{M00} for more details on
pleated surfaces as used in our context. 
We call the hyperbolic metric $\sigma$ on $\Sigma$ which has the above
properties the metric \emph{induced} by the map $g$. Note that this makes
sense since $\sigma$ is indeed the pull-back of the hyperbolic metric on $M_f$ 
by $g$ (properly
interpreted on the pleating locus). 

We fix now once and for all a constant $\kappa_0$ 
which is smaller than a Margulis constant for hyperbolic surfaces and  
with the following properties (see 
p.139 of \cite{M00} for details).  
\begin{enumerate}
\item[(P1)] 
For any hyperbolic surface $S$ with geodesic boundary
$\partial S$, 
any two essential properly embedded arcs $\tau,\tau^\prime$ 
in $S$ with endpoints on $\partial S$ 
whose lengths are at most $\kappa_0$
are either homotopic keeping endpoints in $\partial S$, or they 
are disjoint. 
\item[(P2)] 
If $\alpha$ is a simple closed geodesic
on a hyperbolic surface $S$ and
if $\alpha$ contains a point $x\in S$ of injectivity
radius smaller than $\kappa_0$, then $x$ is contained in a Margulis
tube $A$ of $S$, and 
either $\alpha$ equals 
the core curve of $A$, or the subarc of $A\cap \alpha$ containing $x$ 
crosses through $A$, that is, it connects the two distinct boundary components of $A$.
\end{enumerate}

The second property follows from the fact that
a closed geodesic in a hyperbolic surface which enters
sufficiently deeply into a
Margulis tube but is not entirely contained in the tube 
either crosses through the tube, 
or it has self-intersections.

A \emph{bridge arc} for an essential proper non-annular 
subsurface $Y\subset \Sigma$ is
an embedded arc $\alpha\subset Y$ 
with both endpoints on $\partial Y$ which is not homotopic in $M_f$ into 
$\partial Y$ keeping the endpoints in $\partial Y$. 
%
For a hyperbolic metric $\sigma$ on $\Sigma$, define a \emph{minimal proper arc}
to be a bridge arc $\tau$ for $Y$  which is minimal
in $\sigma$-length among all such arcs. 
The following is a version of Lemma 4.1 of 
\cite{M00}. 

\begin{lem}[Lemma 4.1 of \cite{M00}]\label{projectionbound1}
There exists a number $D_1=D_1(\Sigma)>0$ with the following property. 
Let $Y\subset \Sigma$ be a proper essential non-annular subsurface which is strongly
incompressible 
for the hyperbolic 3-manifold $M_f$. Then for every $\gamma\in {\cal D}_1\cup
{\cal D}_2$ there exists a pleated surface
$g_\gamma$ in the homotopy class of the inclusion $\Sigma\to M_f$ mapping
$\partial Y$ geodesically, with induced metric $\sigma(g_\gamma)$, such that for 
any minimal proper arc $\tau$ in $(Y,\sigma(g_\gamma))$ we have 
\[d_Y(\gamma,\tau)\leq D_1.\]
\end{lem}

\begin{proof}
The proof of Lemma 4.1 of \cite{M00} carries over with no essential modification.
Namely, let $\gamma\subset \Sigma$ be
  a simple closed curve which defines an element of ${\cal D}_1\cup {\cal D}_2$. 
 By assumption on $Y$, the curve $\gamma$ has an
  essential intersection with $\partial Y$. 
  
Modify $\gamma$ by
spinning it about $\partial Y$.
That is, let ${\cal T}_{\partial Y}$ be the mapping class that performs
  one positive Dehn twist about each component of $\partial Y$. The sequence of curves
  ${\cal T}_{\partial Y}^n(\gamma)$ converge,
  as $n\to \infty$, to a finite-leaved lamination
 $\lambda$ whose non-compact leaves spiral about $\partial Y$  and whose closed 
 leaves are precisely $\partial Y$.
 Since the distance in ${\cal C\cal G}(\Sigma)$ between $\gamma$ and
  $\partial Y$ is at least three, the complement $\Sigma-(\partial Y\cup \gamma)$ 
  is a union of simply connected components and hence
  the complementary regions of $\lambda$ are 
  simply connected as well. 
  Add finitely many leaves to $\lambda$ so that the resulting lamination
  $\lambda^\prime$ is maximal. To simplify notations, we identify $\lambda$ with 
  $\lambda^\prime$. 
  
The lamination $\lambda$ is the pleating lamination of a
  pleated surface $g_\lambda$ in $M_f$ mapping $\lambda$ geodesically, with 
  induced metric $\sigma_\lambda$ (compare \cite{M00}).
 Namely, let $\eta$ be a component of $\partial Y$. Then $\eta$ is not homotopic 
 to zero in $M_f$ and hence it can be represented by a unique closed geodesic $\hat \eta$.
 Let $\hat M_f$ be the covering of $M_f$ whose fundamental group is infinitely cyclic and
 generated by the loop $\hat \eta$. Then $\hat M_f$ is a solid torus with core curve the geodesic $\hat \eta$, 
 and $\eta$ lifts to a closed curve in $\hat M_f$. Mapping a point $p$ on $\eta\subset \hat M_f$ 
 to its shortest distance
 projection to $\hat \eta$ and connecting $p$ to its image by a geodesic arc determines 
 a canonical homotopy of $\eta$ to $\hat \eta$ which projects to a homotopy in $M_f$.
 Using this homotopy, any essential arc in $Y$ with endpoints on $\eta$ can be extended to an arc 
 with endpoints on $\hat \eta$.  
 
 Thus the intersection arcs of the simple closed curve 
 $\gamma$ with $\partial Y$, assumed without loss of generality to be essential,  
 define a collection of arcs in $M_f$ with boundary on the 
 collection $\hat \partial Y$ of geodesics in $M_f$ 
 homotopic to $\partial Y$. By the 
 second part of Lemma \ref{incompressiblesub}, such an extended arc is not 
 homotopic into $\hat \partial Y$ keeping the endpoints in $\hat \partial Y$, and by the 
 third part of 
 Lemma \ref{incompressiblesub}, two such extended arcs are homotopic 
 in $M_f$ keeping the endpoints in $\hat \partial Y$ only if the corresponding arcs in $Y$ are homotopic
 keeping the endpoints in $\partial Y$. Namely, two distinct such arcs are disjoint up to homotopy. 
 Using the above homotopy which deforms $\eta$ to $\hat \eta$, this yields that 
 each such arc can be represented by a unique nontrivial geodesic arc in $M_f$ 
 with endpoints in $\hat \partial Y$ which meets $\hat \partial Y$ orthogonally at the endpoints 
 
 Now spinning $\gamma$ about the boundary components of $\partial Y$ corresponds to 
 turning the endpoints of the geodesic arcs with boundary on $\hat \partial Y$ about the components of 
 $\hat \partial Y$. 
Taking a limit as the number of turns goes to infinity results in replacing the arcs by
infinite geodesics which spiral about the components of $\partial \hat Y$. These geodesics define the geometric 
realization of the lamination $\lambda$ in $M_f$. After adding finitely many leaves, the lamination 
$\lambda$ decomposes $\Sigma$ into finitely many ideal triangles. These triangles bound totally geodesic 
immersed ideal triangles in $M_f$ whose union defines the pleated surface $g_\lambda$.

  Denote by $R_\lambda$ the complement in 
  $(\Sigma,\sigma_\lambda)$ of the $\kappa_0$-Margulis tubes whose 
  cores are components of $\lambda$ (and hence of $\partial Y$), where $\kappa_0$ is the 
  constant chosen above with properties (P1) and (P2).  
  Realize the diskbounding curve $\gamma\subset \Sigma$ 
  by its geodesic representative for $\sigma_\lambda$. Denoting by 
  $\ell_{\sigma_\lambda}(\alpha)$
  the length of a geodesic arc $\alpha$ for the hyperbolic metric $\sigma_\lambda$, Theorem 3.5 
  and formula (4.3) of \cite{M00} show that 
  \[\ell_{\sigma_\lambda}(\gamma\cap R_\lambda)\leq 2C \iota(\gamma,\partial Y)\]
  for a universal constant $C=C(\Sigma,\kappa_0)>0$ where $\iota(\gamma,\partial Y)$ is the geometric
  intersection number. 
Note that Theorem 3.5 of \cite{M00} holds true as stated for homotopically trivial curves in 
$M_f$, that is, for curves of vanishing length. 

As on p.139 of \cite{M00}, it now follows that there exists at least one 
component arc of $\gamma\cap  Y\cap R_\lambda$ of length at most $4C$. 
Given a minimal proper arc $\tau$ for $(Y,\sigma_\lambda)$, 
Lemma 2.1 of \cite{M00} 
then bounds $d_Y(\gamma,\tau)$ from above by a universal constant.
This is what we wanted to show.
\end{proof}



Following once more \cite{M00}, 
we next turn to the proof of an analogue of Lemma \ref{projectionbound1}
for essential incompressible
annuli. For its formulation, define a \emph{bridge arc} for
a simple closed curve $\alpha$ in $\Sigma$ to be an embedded arc in $\Sigma$
with both endpoints in $\alpha$ which meets $\alpha$ only at its endpoints and 
is not homotopic into $\alpha$ keeping the endpoints in $\alpha$.
Let $\sigma$ be
a hyperbolic metric on $\Sigma$ and let $\alpha$
be a simple closed geodesic for this metric. Define a
\emph{mininmal curve crossing $\alpha$} to be a simple closed curve constructed
in the following way. Pick one side of $\alpha$ in $\Sigma$ and let 
$\tau$ be a minimal length primitive bridge arc for $\alpha$, that is, a bridge arc
whose interior is disjoint from $\alpha$, that is incident to $\alpha$ on this side.
Let $\tau^\prime$ be a minimal length primitive bridge arc for $\alpha$ that is 
incident to $\alpha$ on the other side. If one of these arcs meets 
$\alpha$ on both sides then put $\tau=\tau^\prime$. Choose $\beta$ to 
be a minimal length shortest simple closed curve that can be represented
as a concatenation of $\tau,\tau^\prime$ (if they are different) and arcs on $\alpha$.
Thus $\beta$ crosses through $\alpha$ once or twice.
%

\begin{lem}[Lemma 4.3 of \cite{M00}]\label{projectionbound3}
  Given $\epsilon >0$ there exists $D_2=D_2(\Sigma,\epsilon)>0$ such
that for a closed hyperbolic 3-manifold $M_f$ the following holds.
  Let $Y\subset \Sigma$ be a proper essential strongly incompressible 
  annulus with core curve $\alpha$ so that $\ell_f(\alpha)\geq \epsilon$.
  Then for every $\gamma\in {\cal D}_1\cup {\cal D}_2$, there exists
  a pleated surface $g_\gamma$, with induced metric $\sigma(g_\gamma)$,
  mapping $\alpha$ geodesically 
  and the property that a minimal curve $\beta$ for $\sigma(g_\gamma)$ crossing $\alpha$
  satisfies
\[d_Y(\gamma,\beta)\leq D_2.\]  
\end{lem} 
 
\begin{proof}
  The proof of Lemma 4.3 of \cite{M00} is valid without any change.
Construct a lamination $\lambda$ from $\alpha$ and 
  $\gamma\in {\cal D}_1\cup {\cal D}_2$ by spinning
  $\gamma$ about $\alpha$. As in the proof of Lemma \ref{projectionbound1}, 
  by the assumption on the distance in
  ${\cal C\cal G}(\Sigma)$ between $\alpha$ and $\gamma$, 
 all complementary components of $\lambda$
  are simply connected. Adding finitely many leaves to $\lambda$
  yields a maximal lamination and hence a pleated surface $g_\lambda$,
  with metric $\sigma_\lambda$. 
 As in Lemma \ref{projectionbound1}, it follows from  
 Theorem 3.5 of \cite{M00} that 
  \[\ell_{\sigma_\lambda}(\gamma)\leq
    2C\iota(\alpha,\gamma)\]
  for a universal constant $C=C(\tau,\epsilon)>0$.

  The remainder of the argument in the proof of Lemma 4.3 of \cite{M00}
  only uses the geometry of $\sigma_\lambda$ on $\Sigma$ and does not
  use any information on the hyperbolic 3-manifold containing
  the pleated surface $g_\lambda$. It is thus valid without any adjustment.
\end{proof}  

In \cite{M00}, the proof of a version of 
Theorem \ref{minsky}
is completed by proving a universal upper length
bound for minimal bridge arcs for 
pleated surfaces constructed from 
proper essential strongly incompressible
subsurfaces $Y\subset \Sigma$ with big boundary length
(this is the core of the proof of Lemma 4.2 and Lemma 4.4 of \cite{M00}). 
This step requires a substantial modification for Heegaard surfaces.
We formulate what we need in the following proposition. For the remainder of this
section, $M_f$ always denotes a hyperbolic 3-manifold with Heegaard surface $\Sigma$,
Hempel distance at least $4$ and disk sets ${\cal D}_1,{\cal D}_2$. Recall that for any 
proper incompressible subsurface $Y\subset \Sigma$, any simple closed curve 
$c$ on $\Sigma$ with $d_{\cal C\cal G}(c,\partial Y)\geq 3$ gives rise to a pleated surface
containing $\partial Y$ in its pleating lamination. The number $p>4$ is as in 
Lemma \ref{homotopyidentity}.

\begin{prop}
\label{projectionbound2}
For any $\epsilon >0$ there exists a number
$D_3=D_3(\Sigma,\epsilon)>0$ with the following property. 
Let $Y\subset \Sigma$ be a proper essential subsurface with 
$d_{\cal C\cal G}(\partial Y,{\cal D}_1\cup {\cal D}_2)\geq p$, 
and let $g_1,g_2$ be a pair of pleated surfaces in the homotopy class of the 
inclusion $\Sigma\to M_f$ mapping $\partial Y$ geodesically which 
are constructed from diskbounding simple closed curves $\gamma_1\in {\cal D}_1,
\gamma_2\in {\cal D}_2$.
Let $\sigma(g_1),\sigma(g_2)$
be the induced hyperbolic metrics on $\Sigma$ and let 
$\tau_1,\tau_2$ be minimal proper arcs for $\sigma(g_1),\sigma(g_2)$ if $Y$ is not an annulus,
or minimal curves crossing $\alpha$ for $\sigma(g_1),\sigma(g_2)$ if $Y$ is an annulus.
If $\ell_f(\partial Y)\geq
\epsilon$ then 
\[d_Y(\tau_1,\tau_2)\leq D_3.\] 
\end{prop}

We are now ready to 
deduce Theorem \ref{minsky}  from Lemma \ref{projectionbound1}, 
Lemma \ref{projectionbound3} and Proposition \ref{projectionbound2}. 

\begin{proof}[Proof of Theorem \ref{minsky}]
  Let $Y\subset \Sigma$ be a proper essential 
  subsurface with $d_{\cal C\cal G}(\partial Y,{\cal D}_1\cup {\cal D}_2)\geq p$. 
  Let $\epsilon >0$ and assume that $\ell_f(\partial Y)\geq \epsilon$.
 Let $\gamma_1\in {\cal D}_1,\gamma_2\in {\cal D}_2$ be two diskbounding
  simple closed curves in $\Sigma$. 
  
  If $Y$ is non-annular then  
  apply Lemma \ref{projectionbound1}
  to obtain two pleated surfaces mapping $\partial Y$ geodesically, and
  minimal proper arcs $\tau_1,\tau_2$ in $Y$ with respect to the
  two induced metrics on $\Sigma$,
  with $d_Y(\gamma_i,\tau_i)\leq D_1$ $(i=1,2)$. 
  Proposition \ref{projectionbound2}
  then implies that $d_Y(\gamma_1,\gamma_2)\leq 2D_1+D_3(\Sigma,\epsilon)$.
  
  If $Y$ is an annulus, then apply Lemma \ref{projectionbound3} to obtain
  two pleated surfaces mapping the core curve $\alpha$ of 
  $Y$ to a geodesic,  and minimal curves $\beta_1,\beta_2$ crossing $\alpha$
  with respect to the two induced metrics on $\Sigma$, with 
  $d_Y(\beta_i,\gamma_i)\leq D_2=D_2(\Sigma,\epsilon)$ for $i=1,2$. 
  An application of Proposition \ref{projectionbound2} shows as before that
  $d_Y(\gamma_1,\gamma_2)\leq 2D_2+D_3(\Sigma,\epsilon)$. 
 This completes the proof of Theorem \ref{minsky}.
\end{proof}

We are left with proving Proposition \ref{projectionbound2} which
is the main technical result of this section.

To facilitate notations, call a system $X\subset \Sigma$
of nontrivial homotopically distinct disjoint simple
closed curves in $\Sigma$ 
\emph{stongly incompressible in $M_f$} if
$d_{\cal C\cal G}(X,{\cal D}_1\cup {\cal D}_2)\geq 3$. This notion is
compatible with the notion of a strongly incompressible subsurface of
$\Sigma$. Note that $X$ is strongly incompressible if and only if
the same holds true for each of its components.

If $X\subset \Sigma$ is a strongly incompressible curve system, then we
denote by ${\cal P}(X)$ the collection of all pleated surfaces
in $M_f$ in the homotopy class of the inclusion $\Sigma\to M_f$, 
with pleating lamination a complete (that is, maximal and approximable
in the Hausdorff topology by simple closed geodesics) 
finite geodesic lamination whose
minimal components are precisely the components of $X$.

An important fact is that for any strongly incompressible 
curve system $X\subset \Sigma$, any two
pleated surfaces $g,h\in {\cal P}(X)$ can be deformed into each other 
with a homotopy consisting of surfaces with controlled geometry. To make this precise, 
we define ${\cal L}(X)$ to be the collection of all maps
$g:\Sigma\to M_f$ 
in the homotopy class of the inclusion
with the following additional property. There exists a
hyperbolic metric $\sigma(g)$ on $\Sigma$ such that
for this metric, the map $g$ is one-Lipschitz and maps each component of 
$X$ isometrically onto its geodesic representative in $M_f$. 
Note that the metric $\sigma(g)$ on $\Sigma$ is part of 
the data which define a point in ${\cal L}(X)$ although it may not be unique.
Note also that we have ${\cal P}(X)\subset {\cal L}(X)$.
If $g$ is a pleated surface then $\sigma(g)$ is assumed to be 
the hyperbolic metric on $\Sigma$ defined by $g$.

A \emph{path} in ${\cal L}(X)$ is a continuous map 
$h:\Sigma\times [a,b]\to M_f$ for some interval 
$[a,b]\subset \mathbb{R}$ 
such that for each $s\in [a,b]$, there is a marked hyperbolic metric 
$\sigma(s)$ on $\Sigma$ depending continuously on $s$ and such that 
the map $h_s:x\in \Sigma\to h_s(x)=h(x,s)\in M$ is a point in ${\cal L}(X)$
for the metric $\sigma(s)$. If a point of the path,
say the point $h_a$, is a pleated
surface, then we require that $\sigma(a)=\sigma(h_a)$. 

Following Section 3 of \cite{M00}, define two pleated surfaces 
$f,g:\Sigma\to M_f$ to be \emph{homotopic relative to a common pleating
lamination} $\mu$ if $\mu$ is a sublamination of the pleating lamination
of $f,g$ and if $f$ and $g$ are homotopic by a family of maps which fixes
$\mu$ pointwise. 

The following is a slight strengthening of a well know construction which goes
back to Thurston (see
p.140 of \cite{M00} and \cite{C93} for more on earlier accounts). Recall from 
Lemma \ref{incompressiblesub} 
that
for an incompressible curve system $X\subset \Sigma$, an essential arc $\alpha$ 
in $\Sigma$ with endpoints on 
$\partial X$ lifts to an arc in the universal covering $\mathbb{H}^3$ of $M_f$ which 
connects two distinct lifts of the components of $ X$ containing the endpoints
of $\alpha$. The second part of the lemma extends Lemma 3.3 of \cite{M00}. 

\begin{lem}\label{homotopy}
Let $X\subset \Sigma$ be a 
strongly incompressible curve system. Then 
any $g,h\in {\cal P}(X)$ can be connected by a path in ${\cal L}(X)$. Furthermore,
$g,h$ are homotopic relative to any  common pleating lamination $\mu\subset X$.
\end{lem} 

\begin{proof}
Let us first consider two pleated surfaces $g_0,g_1\in {\cal P}(X)$ which are related by 
a \emph{diagonal move}. By this we mean the following. The pleating lamination 
$\lambda$ of $g_0$ is an extension of $X$. It 
decomposes $\Sigma$ into ideal triangles whose sides spiral 
about $X$. Isolated leaves of $\lambda$ do not belong to $X$. 
Removal of such an isolated leaf $\alpha$ 
results in a geodesic lamination $\lambda^\prime$ whose complementary components
are ideal triangles and one ideal quadrangle $Q$.  The leaf $\alpha$ connects 
two opposite vertices of $Q$ and subdivides $Q$ 
into two ideal triangles. 
We assume that the pleating lamination for $g_1$ is obtained from 
$\lambda$ by replacing $\alpha$ by the diagonal $\beta$ 
of $Q$ connecting the other two opposite
vertices. 

Our goal is to construct a path in ${\cal L}(X)$
connecting $g_0$ to $g_1$. To this end  
let $\tilde g_0:\mathbb{H}^2\to \mathbb{H}^3$
be a lift of $g_0$ to the universal coverings  
$\mathbb{H}^2$ of $\Sigma$ and $\mathbb{H}^3$ of $M_f$.
Let $\tilde Q\subset \mathbb{H}^2$ be a lift of the  
ideal quadrangle $Q$.  The image $\tilde g_0(\tilde Q)$ of 
$\tilde Q$ under the map $\tilde g_0$ is the union of two ideal
triangles which are 
glued along a common side. Let 
$(a_1,a_2,a_3,a_4)\subset \partial \mathbb{H}^3$ be the ordered collection 
of points in the ideal boundary $\partial \mathbb{H}^3$ of 
$\mathbb{H}^3$ which are the images of the ordered vertices of $\tilde Q$.
This ordered quadruple of points spans an
ideal tetrahedron $T\subset \mathbb{H}^3$.  Note that by 
the third part of Lemma \ref{incompressiblesub}, this 
tetrahedron is non-degenerate. 
The map $\tilde g_0$ maps
 $\tilde Q$ onto 
the union $\tilde Q_0$ of two adjacent sides of $T$.
The image of $\tilde Q$ under a suitably chosen lift $\tilde g_1$ of $g_1$ equals 
the union $\tilde Q_1$ of the remaining two adjacent sides of $T$. 
Four of the six edges of $T$ are the sides of $\tilde g_0(\tilde Q)$, and the remaining two 
edges are the images under $\tilde g_0,\tilde g_1$ 
of the lifts $\tilde \alpha,\tilde \beta$ of the diagonals $\alpha,\beta$ of $Q$ to $\tilde Q$. 
The restriction of $\tilde g_0,\tilde g_1$ to $\tilde Q$ is a path isometry onto $\tilde Q_0,\tilde Q_1$, 
respectively. 

The piecewise totally geodesic quadrangle $\tilde Q_0$ is equipped with an intrinsic 
hyperbolic metric. Let $\tilde \beta_0\subset \tilde Q_0$ be the intrinsic geodesic 
which connects the 2 ideal vertices of $\tilde Q_0$ which
are different from the 
endpoints of $\tilde g_0(\tilde \alpha)$. Then $\tilde \beta_0$ is a  piecewise geodesic line 
in $\mathbb{H}^3$ which 
 intersects the geodesic 
$\tilde g_0(\tilde \alpha)$ in a single point $x_0$.
The point $x_0$ is the finite vertex of a partition of $\tilde Q_0$ into
4 totally geodesic triangles with one vertex at 
$x_0$ and two ideal vertices. 
The total cone angle, that is, the sum of the angles at $x_0$ of these
triangles, equals $2\pi$.
Construct in the same way a point $x_1\in \tilde Q_1$ as the intersection point
between the two 
intrinsic geodesics connecting the two pairs of opposite ideal vertices of $\tilde Q_1$. 
As before, $x_1$ is the finite vertex of a partition of 
$\tilde Q_1$ into 4 totally geodesic
triangles with total cone angle $2\pi$ at $x_1$. 

 Connect $x_0$ to $x_1$ by a geodesic arc $\gamma:[0,1]\to T\subset \mathbb{H}^3$ parameterized
proportional to arc length on $[0,1]$. For each $t\in [0,1]$ consider the union $\tilde Q_t$ of the 
4 totally geodesic triangles $A_i(t)$ $(i=1,\dots,4)$ 
with one vertex at $\gamma(t)$
which have the same ideal vertices as the 
triangles which subdivide $\tilde Q_0$. Note that this notation is consistent with the 
above definition of $\tilde Q_0,\tilde Q_1$. 
Each of the 4 
boundary geodesics 
of $\tilde Q_0$ is contained in precisely one of the triangles from
the collection $\tilde Q_t$.  
If we choose the labels of the triangles 
$A_i(t)$ in such a way that for each $i$, the triangles $A_i(t)$ contain the 
same boundary geodesic of $\tilde Q_0$ for all $t$,  then these triangles depend
continuously on $t$. Since $T$ is the convex hull of its ideal vertices and 
$\gamma\subset T$, the total cone angle at $\gamma(t)$ 
of the union of these triangles is
at least $2\pi$, and it is $2\pi$ at the endpoints $x_0=\gamma(0)$, 
$x_1=\gamma(1)$ of $\gamma$.

For $t\in [0,1]$ let 
$q(t)\geq 0$ be such that the total cone angle of $\tilde Q_t$ 
at $\gamma(t)$ equals $2\pi(1+q(t))$.
Denote by $\nu_i(t)$ the angle of the triangle $A_i(t)\subset \tilde Q_t$ 
at $\gamma(t)$. 
Let $\hat \nu_i(t)=\nu_i(t)/(1 +q(t))\leq \nu_i(t)$ $(i=1,2,3,4)$; we have
$\sum_{i}\hat \nu_i(t)=2\pi$ for all $t$. 
Let $B_i(t)$ be the hyperbolic triangle with 
two ideal vertices and one vertex of angle $\hat \nu_i(t)$. 
Note that there exists a natural isometric embedding of $A_i(t)$ into 
$B_i(t)$ so that the image  contains the 
biinfinite side of $B_i(t)$. This embedding is unique if 
we require that
the finite vertex of $A_i(t)$ is contained in the minimal 
geodesic $\xi_i(t)$ of $B_i(t)$ which connects the finite vertex of $B_i(t)$ 
to the opposite side.
Denote the image of $A_i(t)$ under this embedding again by $A_i(t)$.

By the choice of the angles $\hat \nu_i(t)$, the triangles $B_i(t)$ can be 
glued along their sides which are adjacent to the finite vertex cyclically in the order
prescribed by the order of the triangles $A_i(t)$ in the polygon $\tilde Q_t$ 
to a hyperbolic ideal quadrangle $B(t)$ with a distinguished vertex
$q(t)$. The ideal quadrangle $B(t)$ contains the 
union $A(t)$ of the triangles $A_i(t)$. 
This construction does not depend on choices and hence depends continuously  
on $t$. Moreover, $B(t)-A(t)$ is a region which is star shaped with 
respect to the point $q(t)$. This region consists of the interior of 
an embedded relatively compact quadrangle $C(t)$, with 
an ideal triangle attached to each of its sides. 

By invariance under the action of $\pi_1(\Sigma)$, the hyperbolic quadrangle $B(t)$ determines
a hyperbolic metric $\sigma(t)$ on $\Sigma$ depending continuously on $t$, and
$\sigma(0)=\sigma(g_0),\sigma(1)=\sigma(g_1)$. 
Thus we are left with constructing a continuous map 
$h:\Sigma\times [0,1]\to M_f$ such that for each $t$, the restriction of 
$h$ to $\Sigma\times \{t\}$ is a one-Lipschitz map $(\Sigma,\sigma(t))\to M_f$ mapping 
$\lambda^\prime$ geodesically.

There is a natural 1-Lipschitz map $B(t)\to \tilde Q_t$ which maps 
each of the triangles $A_i(t)$ isometrically, collapses the complementary
quadrangle $C(t)$ to a point and collapses the ideal triangles 
attached to the sides of $C(t)$ to one of its infinite length sides by collapsing a 
geodesic arc contained in one of these triangles with endpoints on the two distinct
infinite sides of the triangle 
to a point if its endpoints are 
identified in $\tilde Q_t$. This construction 
defines a one-Lipschitz map $(\Sigma,\sigma(t))\to 
M_f$ depending continuously on $t$ and mapping $X$ isometrically. 
As for $t=0$ and $t=1$ the collapsing map
equals the identity, we obtain a 
path in ${\cal L}(X)$ connecting $g_0$ to 
$g_1$ provided that $g_0$ and $g_1$ are related by a diagonal move. 

Note that this construction does not use any information on the pleating
locus of $g_0,g_1$ beyond the information that 
these pleating loci differ by a diagonal
move. Moreover, it yields a path $g_s\in {\cal L}(X)$ whose restriction to the
intersection of the pleating loci of $g_0,g_1$ is the identity. 
In particular,
the pleated surfaces $g_0,g_1$ are homotopic relative to the pleating 
lamination $X$.

To complete the proof of the lemma we are left with showing that 
any two pleated surfaces $g_0,g_1\in {\cal P}(X)$ 
can be connected by a finite
chain of pleated surfaces in ${\cal P}(X)$ 
so that any two consecutive pleated surfaces in
the chain are related by a diagonal move. 
That this is possible is an immediate consequence of a result
of Hatcher \cite{H91} (compare \cite{C93} and \cite{M00} for more details
about this fact). By concatenation, this shows that 
any two pleated surfaces in ${\cal P}(X)$ 
can be connected by a path in ${\cal L}(X)$, and $g_0,g_1$ are homotopic
relative to the pleating lamination $X$.
\end{proof}

\begin{rem}\label{pleatedhomotopy}\normalfont
  The proof of Lemma \ref{projectionbound1}
  together with Lemma \ref{homotopy}
and Lemma \ref{homotopyidentity} yield
additional information on $M_f$. 
Namely, let as before $Y\subset \Sigma$ be a strongly incompressible subsurface
and let $\alpha\subset \Sigma-\partial Y$ be any system of pairwise 
disjoint non-homotopic arcs 
with endpoints on $\partial Y$ which 
decompose $Y$ into the maximal possible number of 
simply connected regions. 
Then $\alpha$ determines a pleated surface
$g$ in $M_f$ in the homotopy class of the inclusion 
$\Sigma\to M_f$ 
whose pleating lamination contains $\partial Y$ as the union of its minimal components. 
This pleated surface only depends on $\partial Y$ and the homotopy classes of 
the components of $\alpha$ as arcs in $M_f$ with boundary on $\partial Y$. 
If $\alpha_1,\alpha_2$ are two such arc systems, and if 
$\alpha_1$ contains an arc $\zeta_1$ which is homotopic in $M_f$ 
relative to $\partial Y$
to an arc $\zeta_2$ from $\alpha_2$, 
then $\zeta_1$ and $\zeta_2$ determine the same isolated leaf  
of the pleating lamination of the pleated surface in $M_f$ 
constructed from  $\alpha_1,\alpha_2$. Since by the proof of
Lemma \ref{projectionbound1}
the pleated surfaces constructed in this way 
are naturally homotopic to the inclusion
$\Sigma\to M_f$, Lemma \ref{homotopyidentity} yields that
the arcs $\zeta_1,\zeta_2$ are in fact homotopic in $\Sigma$.
\end{rem}


The strategy is now to obtain geometric information on minimal proper arcs 
or minimal curves for a pleated surface $g\in {\cal P}(\partial Y)$ from information on 
the geodesic representative $\partial Y\subset M_f$. In the following 
elementary observation, $\ell_f(c)$ denotes as before the length in $M_f$ of a geodesic 
representative of a multicurve $c$ in $\Sigma$.

\begin{lem}\label{lengthbound}
For every $\epsilon >0$ there exists a number $L=L(\epsilon)>0$ with the following 
property. 
Let $Y\subset \Sigma$ be a proper essential strongly incompressible 
non-annular subsurface 
with $\ell_f(\partial Y)\geq \epsilon$; then for any $g\in {\cal L}(\partial Y)$, the
$\sigma(g)$-length of a minimal proper arc for $Y$ is at most $L$.
Moreover, for any $\kappa_1>0$ there exists a number $R_1=R_1(\kappa_1)>0$ with 
the following property. If $\beta\subset \partial Y$ is a component with $\ell_f(\beta)\geq R_1$
then there exists a bridge arc for $Y$ with one endpoint on $\beta$ and of 
$\sigma(g)$-length at most $\kappa_1$.
\end{lem} 
\begin{proof} 
Let $m\in [1,3g-3]$ be the number of components of $\partial Y$. 
By the collar theorem for hyperbolic metrics on $\Sigma$, 
if $\beta$ is a component of $\partial Y$ of length 
  at least $\epsilon/m$,  
then the supremum $\rho$ 
of the $\sigma(g)$-heights of a half-collar about $\beta$
 is bounded from above by a number $L/2$   
 only depending on $\epsilon/m$.
 By the choice of $\rho$, the boundary of a half-collar of radius $\rho$ 
 about $\beta$ can not be retracted into $\beta$. Hence there exists a radial geodesic
segment of length $\rho$ 
emanating from the side of $\beta$ determined by the half-collar
whose endpoint either is the endpoint of another such segment or
is contained in $\beta$. In both cases, we find an  
essential arc $\tau$ with endpoints on $\beta$ whose length does not exceed 
$2\rho$.

Using the half-collar
at the side of $\beta$ contained in $Y$,  
we can assume that a neighborhood of at least one
endpoint of the essential arc $\tau$ is contained in $Y$.
Since the second endpoint of $\tau$ is contained in $\beta$, we conclude
that there is a (possibly proper) subarc of $\tau$ which
is a bridge arc for $Y$, and the length of this arc is at most $2\rho\leq L$
as claimed. 

Using again the collar lemma for hyperbolic metrics on $\Sigma$ (or a standard area
estimate), for a given number 
$\kappa_1>0$, 
if $R_1>0$ is sufficiently large then the height 
of a half-collar about a simple closed geodesic of length at least $R_1$ for any 
hyperbolic metric on $\Sigma$ is smaller than $\kappa_1/2$. By the above discussion,
this implies that if $\partial Y$ contains a component $\beta$ of length at least $R_1$, then 
there exists a bridge arc for $Y$ of length at most $\kappa_1$ with one 
endpoint on $\beta$. This completes the proof of the lemma.
\end{proof}

We use Lemma \ref{homotopy} and Lemma \ref{lengthbound} to establish the
following version of Lemma 4.2 of \cite{M00}.
In its formulation, 
$\ell_f(\partial Y)$ denotes as before the length of $\partial Y$ 
with respect to the hyperbolic metric on $M_f$.

\begin{lem}\label{length}
For all $\epsilon >0,R>0$ there exists a number $k_0=k_0(\epsilon,R)>0$ with the following property.
Let $Y\subset \Sigma$ be a proper essential strongly incompressible 
subsurface and assume that $\ell_f(\partial Y)\leq R$. If 
${\rm diam}_Y({\cal D}_1\cup {\cal D}_2)\geq k_0$, then
\[\ell_f(\partial Y)<\epsilon.\]
\end{lem}

\begin{proof} 
  Let $R>0,\epsilon <R$ be fixed and assume that $\ell_f(\partial Y)\in  
[\epsilon,R]$.  
If $Y$ is not an annulus, then 
by Lemma \ref{lengthbound}, there exists a number 
$L=L(\epsilon)>0$  such that for every $g\in {\cal L}(\partial Y)$, the 
$\sigma(g)$-length of a minimal proper arc for $Y$ is at most $L$.

We claim that there also is a uniform upper length bound 
$L^\prime=L^\prime(\epsilon,R)$ 
for a minimal curve for $Y$ if $Y$ is an annulus.
Namely, let $Y$ be an annulus with core curve $c$.
If $\sigma$ is a hyperbolic metric on $\Sigma$ such that 
the $\sigma$-length of $c$ is contained in the interval 
$[\epsilon,R]$, then 
the $\sigma$-height of a half-collar about $c$ is 
bounded from above by a constant $\rho >0$ only depending on 
$\epsilon$. Thus there exists a proper arc $\tau$ 
for $\sigma$ of length at most $2\rho$ with both endpoints on $c$.
If $\tau$ 
leaves and returns at the two different sides of $c$, then
the endpoints of $\tau$ can be connected by a subarc of $c$ 
of length at most $R/2$ to yield
a simple closed curve crossing through $c$ of length at most
$2\rho +R/2$. If all proper arcs $\tau$ for $c$ of length at most 
$2\rho$ leave and return to the 
same side of $c$, then we can find such a proper arc $\tau$
leaving and returning to a fixed side of $c$, and a second arc 
$\tau^\prime$ leaving and returning to the other side of $c$.
Furthermore, we may assume that $\tau$ and $\tau^\prime$ are disjoint. 
The endpoints of
$\tau$ and $\tau^\prime$ can be connected by disjoint 
subarcs of $c$
to yield a simple closed curve crossing through $c$ of length
at most $4\rho+R$. 

We first show the lemma in the case that $Y$ is not an annulus. 
Following the proof of Lemma 4.2 of \cite{M00},
let $g_0,g_1\in {\cal P}(\partial Y)$
be two pleated surfaces and assume that $g_0$ is constructed
from $Y$ and a maximal system $A$ of arcs with endpoints in $\partial Y$ 
which contain the 
intersection arcs with $Y$ of a disk from the disk set ${\cal D}_2$, and 
that $g_1$ is constructed from $Y$ and a maximal system $B$ of arcs
which contain the intersection arcs with $Y$ of a disk from the disk
set ${\cal D}_1$. By Lemma \ref{homotopy}, these pleated surfaces
can be connected in ${\cal L}(\partial Y)$ by a path
$g_t$ $(t\in [0,1])$. Let $\sigma(g_t)$ be the corresponding path in the 
Teichm\"uller space 
${\cal T}(\Sigma)$ 
of $\Sigma$ connecting $\sigma(g_0)$ to $\sigma(g_1)$. 
Given any bridge arc $\tau$ for $Y$, 
let $E_\tau\subset [0,1]$ denote the set of $t$-values for
which $\tau$ is 
homotopic rel $\partial Y$ to a minimal proper arc
with respect to $\sigma(g_t)$. 
Continuity of the metrics $\sigma(g_t)$ in $t$ 
implies that $E_\tau$ is closed, and the
family $\{E_\tau\}$ covers $[0,1]$. 
The path $g_t$ determines a coarsely well defined 
map $\Psi$ from the interval $[0,1]$ into the arc and curve graph of $Y$.
This map associates to $t\in [0,1]$ the set of all $\tau$ with $t\in E_\tau$.

Following p.141 of \cite{M00}, we
observe that if $E_\tau\cap E_{\tau^\prime}\not=\emptyset$
then up to homotopy, $\tau$ intersects
$\tau^\prime$ in at most one point and hence the distance
in the arc and curve graph of $Y$ between $\tau,\tau^\prime$ is at most $2$.
Thus the coarsely well defined map $\Psi$ has 
the following property. 
If $\tau\in \Psi[0,1]$, and if $\Psi[0,1]$ consists of  
more than one point, then any
$\tau^\prime\in \Psi[0,1]-\{\tau\}$ fulfills $d_Y(\tau,\tau^\prime)\leq 2$.
Together with 
Lemma \ref{projectionbound1}, this shows that 
$\Psi[0,1]$ contains a sequence of 
arcs $\tau_0,\tau_1,\cdots, \tau_n$ so that 
$d_Y(\tau_0,A)\leq D_1$, $d_Y(\tau_n,B)\leq D_1$ and 
$1\leq d_Y(\tau_i,\tau_{i+1})\leq 2$ for all $i$.
As a consequence, it holds $n\geq d_Y({\cal D}_1,{\cal D}_2)/2-2D_1$.  
We also may assume that the arcs $\tau_i$ are pairwise non-homotopic as
arcs in $Y$ with endpoints in $\partial Y$.

By Remark \ref{pleatedhomotopy}, if two such arcs $\tau_i,\tau_j$ are
homotpic in $M_f$ keeping the endpoints in $\partial Y$, then
$\tau_i,\tau_j$ are homotopic in $\Sigma$ keeping the endpoints in
$\partial Y$.
Together this implies the following. Among 
the bridge arcs $\tau_i$ of $Y$,  
there are at least $d_Y({\cal D}_1,{\cal D}_2)/2-2D_1=q$ arcs which are 
pairwise non-homotopic in $M_f$ keeping the endpoints in $\partial Y$.

Since the maps $g_t\in {\cal L}(\partial Y)$ are one-Lipschitz, 
the union of $\partial Y$ with the homotopy classes 
of minimal proper arcs for the metrics $\sigma(g_t)$, 
viewed as arcs in $M_f$ with boundary in $\partial Y$ via the 1-Lipschitz maps
$g_t:\Sigma\to M_f$, 
can be represented in $M_f$ by a 1-complex $V$ with at most $m$ components
where $m\leq 3\vert \chi(\Sigma)\vert/2$ is
the number of components of $\partial Y$. 
The diameter in $M_f$ of each component of $V$ is at most 
$R+mL$. Each such minimal proper arc 
$\tau$ together with one or two segments of $\partial Y-\tau$ gives rise to 
a loop in this one-complex $V$ of length at most $2R+2L$. 
Up to homotopy rel $\partial Y$,  
these based loops are images by the inclusion $\Sigma\hookrightarrow M_f$ of  
\emph{simple} closed curves contained in $Y$. Such a simple closed 
curve is a component
of the boundary of a small neighborhood of the union of $\tau$
with the components of $\partial Y$ containing the endpoints of $\tau$.

Given a component $V_0$ of $V$, choose a basepoint $x$ for $V_0$ 
in a component of $\partial Y$ contained in $V_0$. Connecting each of the loops in $V_0$ 
constructed in the previous paragraph to $x$ determines a collection of 
based loops in $M_f$ which up to homotopy are images of
based simple loops contained
in $Y$. The length of each such loop is at most $3R+(2m+2)L$. 
By Lemma \ref{incompressiblesub}, 
no two distinct of these loops are homotopic in $M_f$.

As bridge arcs 
$\tau,\tau^\prime$ for $Y$ which are homotopically distinct in $M_f$ 
give rise to homotopy classes which do not have a common power and hence
which do not commute,  
 a standard application  of the Margulis lemma gives an upper bound 
 $M=M(3R+(2m+2)L)$ for the number of such elements of $\pi_1(M_f)$ which 
 can translate any point a distance $3R+(2m+2)L$ or less (see p.141 of \cite{M00} for more
 details). As a consequence, the number $q$ of homotopy classes of arcs obtained
 from the above construction is at most $M$. Together we conclude that 
 \[d_Y({\cal D}_1,{\cal D}_2)\leq 2q+4D_1\leq 2M+4D_1.\]
 This complete the proof of 
the lemma in the case that $Y\subset \Sigma$ is not an annulus. 

If $Y$ is an annulus then the above argument carries over in
the same way, where the arc and curve graph is now the arc graph of the 
annular cover of $\Sigma$ whose fundamental group equals the 
fundamental group of $Y$. We refer to Lemma 4.4 of \cite{M00} for 
more details of this argument which is valid in our context with as only
addition the above counting estimates for homotopy classes of arcs in 
$M_f$ with endpoints in the core curve of $Y$. 
An application of Lemma \ref{projectionbound3} then completes
the proof of the lemma. 
\end{proof}

Let $\kappa_0>0$ be a constant which has properties (P1) and (P2) from the beginning
of this section.
By possibly decreasing $\kappa_0$, we may assume that it is a Margulis constant
for hyperbolic surfaces and hyperbolic 3-manifolds. 
In the sequel we always assume that the thin part of 
a hyperbolic 3-manifold is determined by such a constant $\kappa_0$.

We next investigate strongly incompressible surfaces $Y\subset \Sigma$ with the 
property that the geodesic representatives of their 
boundaries $\partial Y$ 
enter deeply into a Margulis tube of $M_f$ away from the core
curve of the tube.
For a number $\nu <\kappa_0$ and a Margulis tube $T$ for $M_f$, we call the 
set $T^{<\nu}$ of all points in $T$ of injectivity radius smaller than $\nu$ the 
\emph{$\nu$-thin part of $T$}.

\begin{lem}\label{pleatedmargulis}
  There exists a number $\nu_0 <\kappa_0$ with the following property.
 Let $T\subset M_f$ be a Margulis tube and 
let $Y\subset \Sigma$ be a strongly incompressible subsurface
  whose boundary $\partial Y$, as a geodesic multicurve in $M_f$,  
  intersects $T^{<\nu_0}$ in the complement of the core geodesic of $T$.
Then for every map $g\in {\cal L}(\partial Y)$, and any choice $\sigma(g)$ of a 
corresponding hyperbolic 
metric, there exists
a Margulis tube for $\sigma(g)$ whose core curve 
$\alpha\subset \Sigma$  intersects $Y$ in 
a bridge arc $\tau$ of length smaller than $\kappa_0$ if $Y$ is not 
an annulus, or which is a simple closed curve crossing through $\partial Y$ if 
$Y$ is an annulus. 
Moreover, 
one of the following not mutually exclusive 
possibilities is
satisfied.
\begin{enumerate}[i)]
\item Up to homotopy, $g(\alpha)$ bounds a disk $D\subset T\subset M_f$. 
\item $g(\alpha)$ is homotopic to a nontrivial multiple of the core curve of $T$.
  Furthermore, there exists a diskbounding 
  simple closed curve $\beta$ on $\Sigma$ which is 
  disjoint from $\alpha$.
\item Up to homotopy, 
any component of the intersection of $g(\Sigma)$ with the 
1-neigh\-bor\-hood of 
$T^{<\nu_0}$ is an annulus, 
and this annulus is the image under $g$ of a Margulis tube for $\sigma(g)$. 
The core curve of each such tube is mapped by $g$ to a curve homotopic 
to a nontrivial multiple of the core curve of $T$.
\end{enumerate}
\end{lem}

\begin{proof} For the fixed choice of a Margulis constant
 $\kappa_0>0$ for  hyperbolic surfaces, 
there exists a number $\ell>0$ only depending on $\Sigma$ and
$\kappa_0$ such that for any hyperbolic metric on $\Sigma$, 
the diameter of any component of the $\kappa_0$-thick part of $\Sigma$ 
is at most $\ell$.
 Let $\nu_0>0$ be sufficiently small that the $2\ell$-neighborhood of 
the $\nu_0$-thin part $T^{<\nu_0}$ 
of a Margulis tube $T$ for $M_f$ is entirely 
contained in $T$. 

Let $Y\subset \Sigma$ be a strongly incompressible surface with
boundary $\partial Y$ and let 
$g\in {\cal L}(\partial Y)$.
There exists a decomposition of $\Sigma$
into thick and thin components for the metric $\sigma(g)$.
The thin components are Margulis tubes about closed
$\sigma(g)$-geodesics of length
smaller than $2\kappa_0$. By the choice of $\ell$ and $\nu_0$, for any 
Margulis tube $T\subset M_f$, 
any component of the thick part of $\Sigma$ for the metric $\sigma(g)$ 
whose image under $g$ intersects 
$T^{<\nu_0}$ is mapped by $g$ into $T$.

Assume that the geodesic multicurve $\partial Y\subset M_f$ 
intersects the $\nu_0$-thin part $T^{<\nu_0}$ 
  of a Margulis tube $T\subset M_f$ in the complement of the core 
  curve of the tube. 
 Since $\partial Y$ is a union of closed geodesics in $M_f$ and the only
  closed geodesic in $M_f$ which is entirely contained in the tube $T$ is the core curve of the tube,
$\partial Y$ intersects the boundary 
$\partial T$ of $T$, which is a torus smoothly embedded in $M_f$. 


Let $W_0\subset \Sigma$ be the union of all components
of the $\sigma(g)$-thick part of $\Sigma$
whose images under $g$ do \emph{not} intersect
$T^{<\nu_0}$, and let $W\subset \Sigma$ be the union of $W_0$ with all
Margulis tubes for $\sigma(g)$ whose images under $g$ 
do not intersect $T^{<\nu_0}$. Then the closure $Z$ of 
$\Sigma-W$ is a closed nonempty essential subsurface of $\Sigma$. The surface 
$Z$ is a union of components of the thick part of $\Sigma$ and some Margulis tubes.
Each component of the thick part of $Z$ is mapped by $g$ into the tube $T$. 
Since the fundamental group of $T$ is cyclic, and the map 
$g$ induces a surjection $g_*:\pi_1(\Sigma)\to \pi_1(M_f)$, 
this implies that the subsurface  
$Z$ of $\Sigma$ is proper. Note that by assumption, the surface $Z$ is intersected
by $\partial Y$.

Now $g(W)$ is disjoint from the core curve of $T$, and the complement of 
the core curve of $T$ deformation retracts onto the boundary 
$\partial T$ of $T$. Thus up to 
modifying $g$ with a homotopy and replacing $T$ by 
the complement in $T$ of a suitably chosen collar about 
$\partial T$, we may assume that
$g(W)\cap T=\emptyset$.

There are now two possibilities. In the first case,
$Z$ contains a component $Z_0$ 
of the $\sigma(g)$-thick part of $\Sigma$. Then we have $g(Z_0)\subset T$.


Let $x_0\in \partial Z_0$ and 
consider the map $g_*^{Z_0}:\pi_1(Z_0,x_0)\to \pi_1(M_f,g(x_0))$. 
Since $Z_0\subset \Sigma$ is a properly embeded
connected surface different from
an annulus, its fundamental group 
$\pi_1(Z_0)$ is a 
free group with at least two generators.  As $\pi_1(T)$ is infinite cyclic, 
the kernel of $g_*^{Z_0}$ is nontrival.
Now $Z_0\subset \Sigma$ is a proper subsurface
of the embedded Heegaard surface $\Sigma$. Therefore the loop theorem 
Theorem 4.2 of \cite{AR04} (see also Theorem 4.2 of \cite{He76}) 
shows that there exists a 
simple closed curve $c\subset Z_0$ such that $g(c)$ is homotopically 
trivial in $M_f$ and hence in $T$. 

To be more precise, 
since $g(Z_0)\subset T$, if $\alpha\subset Z_0$ is 
any closed curve such that $g(\alpha)$ is contractible in $M_f$,
then $g(\alpha)$  is contractible in $T$. Thus there exists a homotopy of 
$g(\alpha)$ to the trivial curve which does not intersect $g(W)\subset M_f-T$, and,
consequently, there exists a homotopy of $\alpha\subset 
Z_0\subset \Sigma\subset M_f$ to 
the trivial curve which does not intersect $W\subset M_f$.

Cutting $M_f$ open along $W\subset \Sigma$ yields a manifold 
$N$ whose boundary $\partial N$ consists of two
copies of $W$, glued along the boundary. Up 
to homotopy, each component of
$\Sigma -W=Z$ is a properly embedded surface in $N$.
In particular, this holds true for the component $\hat Z_0$ of
$\Sigma -W$ containing $Z_0$ (a priori,
$Z_0$ may be a proper subsurface of $\hat Z_0$). Note that
$\hat Z_0$ is an  oriented, two-sided properly embedded 
subsurface of $N$ which is   
different from a disk and a 2-sphere.  

Since a loop in $Z_0\subset \hat Z_0$ which is contractible in
$M_f$ is contractible in $M_f-W$, it is contractible in $N$. 
Therefore the homomorphism $\pi_1(\hat Z_0)\to \pi_1(N)$ induced by 
the inclusion $\hat Z_0\to N$ is not injective. The loop theorem
Theorem 4.2 of \cite{AR04} then shows that there is a simple closed curve
$c\subset \hat Z_0$ which bounds an embedded disk in $N$ and
hence in $M_f$.

If $c$ is either peripheral in $\hat Z_0$ or the core curve of a Margulis
tube for $\sigma(g)$ contained in $\hat Z_0$,
then $c$ is 
a core curve of a Margulis tube for $\sigma(g)$ 
which is diskbounding in $M_f$. Identify $c$ with its geodesic representative 
for $\sigma(g)$. 
As the subsurface 
$Y\subset \Sigma$  is strongly incompressible by assumption, its boundary
$\partial Y$ has to cross
through the diskbounding simple closed curve $c$. 
If $Y$ is not an annulus and if 
$\xi\subset \partial Y$ is an embedded arc crossing through
$c$, then a subarc of $c$ connecting $\xi\cap c$ with
the point in $c\cap \partial Y$ which is closest along $c$ and leaves
$\xi$ at the side of $\xi$ contained in $Y$ 
is a bridge arc for $Y$ of $\sigma(g)$-length at most $\kappa_0$.
If $Y$ is an annulus, $c$ is a simple closed curve of 
$\sigma(g)$-length less than $\kappa_0$ which crosses through 
$\partial Y$. This shows that the first 
possibility stated in the lemma is fulfilled. 
Note that in the case that $\hat Z_0$ is a 3-holed sphere, 
the only simple closed curves in $\hat Z_0$ are the boundary curves and hence
the simple closed curve $c$ in $\hat Z_0$ is automatically peripheral.  

On the other hand, if no diskbounding simple closed curve
$c\subset \hat Z_0$ is the core curve of a Margulis tube for
$\sigma(g)$, then the second case in statement of
the lemma holds true. Namely, in this case a peripheral curve
$d\subset \hat Z_0$ is the core curve of a Margulis tube and
disjoint from $c$. Furthermore, the curve $d$ is mapped by $g$ into $T$ and hence
it is homotopic to a nontrivial multiple of the core curve of $T$.
Since 
$d_{\cal C\cal G}(\partial Y,{\cal D}_1\cup
{\cal D}_2)\geq 3$ by assumption, the multicurve
$\partial Y$ has to cross through $d$. We then find a bridge
arc for $Y$ 
of $\sigma(g)$-length smaller than $\kappa_0$ which is a subarc of $d$, or,
if $Y$ is an annulus, choose $d$ as a simple curve crossing through
$\partial Y$ of length smaller than $\kappa_0$. Thus the second possibility in the 
statement of the lemma is fulfilled. 
This completes the analysis of the case when the image under $g$ of the 
$\sigma(g)$-thick part of $\Sigma$ intersects $T^{<\nu_0}$.

If the image under $g$ of the $\sigma(g)$-thick part of $\Sigma$ does not
intersect $T^{<\nu_0}$, then each intersection 
point of $g(\Sigma)$ with $T^{<\nu_0}$ 
is contained in the image of a Margulis tube for $\sigma(g)$.
Since $\partial Y$ intersects
$T^{<\nu_0}$ in the complement of the core curve of $T$, there is  simple closed 
curve $\alpha\subset \Sigma$ which is freely homotopic to the core curve of 
one of these Margulis tubes, of $\sigma(g)$-length at most $\kappa_0$, which intersects
$\partial Y$ and which is mapped by $g$ into $T$. As before, 
if $g(\alpha)$ is contractible in $M_f$ then we are in the situation described in the 
first case of the lemma.
Otherwise $g(\alpha)$ is homotopic in $T$ to a multiple of the core curve of $T$, and 
we conclude that 
the third case described in the lemma is fulfilled. 
This completes the proof of the lemma.
\end{proof}

In the following lemma, the constant $p>4$ is as in Lemma \ref{homotopyidentity}.

\begin{lem}\label{nothinpart}
There exist numbers $k_1=k_1(\Sigma)>0$, 
and $\nu_1<\nu_0$ 
with the following properties. Assume that the Hempel distance
$d_{\cal C\cal G}({\cal D}_1,{\cal D}_2)$  for 
$M_f$ is at least $4$. 
Let $Y\subset \Sigma$ be a strongly incompressible subsurface whose 
boundary $\partial Y$, as a geodesic multicurve in $M_f$, 
 intersects the $\nu_1$-thin part $T^{<\nu_1}$ of a Margulis tube $T\subset 
M_f$ in the complement of the core geodesic of $T$ and fulfills
$d_{\cal C\cal G}(\partial Y,{\cal D}_1\cup {\cal D}_2)\geq p$.  
Then we have 
\[{\rm diam}_Y({\cal D}_1\cup {\cal D}_2)\leq  k_1.\]
\end{lem}

\begin{proof} The number $p>4$ was chosen so that the following holds true
 \cite{MM00}.
 Let $\alpha,\beta$ be simple closed curves on $\Sigma$ 
and let $Y\subset \Sigma$ be a subsurface
which has an essential intersection with $\alpha,\beta$ and    0
such that
$d_Y(\alpha,\beta)\geq p$; then any geodesic in ${\cal C\cal G}(\Sigma)$ connecting
$\alpha$ to $\beta$ has to pass through a curve disjoint from $Y$.
Furthermore, 
if $Y\subset \Sigma$ is
any proper essential subsurface such that 
$d_{\cal C\cal G}(\partial Y,{\cal D}_i)\geq p$ $(i=1,2)$;
then ${\rm diam}_Y({\cal D}_i)\leq p$.

Let $\ell>0$ be an upper bound for the diameter of a component of the thick
part of a hyperbolic metric on $\Sigma$ for a Margulis constant $\kappa_0>0$ as in 
(P1),(P2).  
For $\nu_0>0$ as in Lemma \ref{pleatedmargulis}, 
let $\nu_1<\nu_0$ be sufficiently small that 
the neighborhood of radius $\ell$ of the $\nu_1$-thin part $M_f^{<\nu_1}$ of $M_f$ is contained in the 
$\nu_0$-thin part $M_f^{<\nu_0}$ of $M_f$. Let us assume that 
$Y\subset \Sigma$ is a proper essential subsurface with $d_{\cal C\cal G}(\partial Y,{\cal D}_i)\geq p$ for 
$i=1,2$ and such that 
$\partial Y$ intersects the
$\nu_1$-thin part $T^{<\nu_1}$ of a Margulis tube $T$ 
in the complement of the core curve of $T$.

Consider first the case that $Y$ is not an annulus. 
Let $\gamma_i\in {\cal D}_i$ for $i=1,2$ and 
$g,h\in {\cal P}(\partial Y)$ be pleated surfaces for $\gamma_1,\gamma_2$
as in Lemma \ref{projectionbound1} . 
By Lemma \ref{homotopy}, we can 
connect $g,h$ by a path $h_s$ $(s\in [0,1], h_0=g,h_1=h)$ 
in ${\cal L}(\partial Y)$. Let $\sigma(h_s)$ be a corresponding 
path in Teichm\"uller space connecting $\sigma(g)$ to $\sigma(h)$.
We showed in Lemma \ref{pleatedmargulis}  
that for each $s$, there exists a bridge arc $\tau_s$ for $Y$ 
of $\sigma(h_s)$-length smaller than $\kappa_0$ 
which is contained in a simple closed curve
$\alpha_s$ on $\Sigma$ of $\sigma(h_s)$-length at most 
$\kappa_0$, and $\alpha_s$ is homotopic 
to the core curve of a Margulis tube for $\sigma(h_s)$ and crossed through by $\partial Y$.
Furthermore, up to homotopy, 
we may assume that $h_s(\alpha_s)\subset T^{<\nu_0}$. By the choice of $\kappa_0$, 
the bridge arc $\tau_s$ is disjoint from a minimal proper arc for $Y$ and 
the metric $\sigma(h_s)$ (see p. 139 of \cite{M00} for more details). 

For each $s$ there exists a connected open neighborhood $V_s$ 
of $s$ in $[0,1]$ so that the curve $\alpha_s$ has the properties stated
in the previous paragraph for each 
$t\in V_s$. By compactness, the interval $[0,1]$ can be covered by finitely many
of the sets $V_s$. Thus we may assume that there is a partition 
$0=s_0< \dots <s_n=1$ such that for each $i<n$, the interval
$[s_i,s_{i+1}]$ is contained in $V_i=V_{s_i}$. Then for each $s\in [s_i,s_{i+1}]$
the curve $\alpha_{s_i}$ is of length smaller than $\kappa_0$
for the metric $\sigma(h_s)$. In particular, by the choice of 
$\kappa_0$, the curves $\alpha_{s_i}$ and $\alpha_{s_{i+1}}$ are disjoint. 

Assume that the number $n$ of partition points of $[0,1]$ is minimal with the above 
property. This then implies that for all $i$, the curve $\alpha_i$ is not 
homotopic to $\alpha_{i+1}$.
If $n=1$, or, equivalently, if 
$\alpha_{s_i}=\alpha_{s_j}$ for all $i,j$, then there exists a
bridge arc $\tau$ for $Y$ 
which up to homotopy is of length at most $\kappa_0$
for each of the metrics $h_s$. Since this bridge arc is of distance at most 1 
in the arc and curve graph of $Y$ 
to a $\sigma(h_s)$-minimal proper arc for $Y$,  it then follows from 
the choice of $g,h$ and 
Lemma \ref{projectionbound1} that 
the diameter of the subsurface projection of $\gamma_1\cup \gamma_2$
into $Y$ is at most $2D_1+2$.

If $n\geq 2$ then by minimality, 
we can not find a simple closed curve in $\Sigma$ which 
is the core curve of a Margulis tube for $\sigma(h_s)$ with the properties
stated above for all $s\in [s_i,s_{i+2}]$ and all $i$. In particular, we have 
$\alpha_{s_i}\not=\alpha_{s_{i+1}}$ for all $i$. Furthermore, there is at least
one $s\in [s_{i+1},s_{i+2}]$ such that for the metric $\sigma(h_s)$, the curve 
$\alpha_{s_i}$ is not the core curve of a Margulis tube with the properties stated
above. 
Now $\alpha_{s_{i}}$ is crossed through by $\partial Y$ and is 
mapped by $h_{s_{i}}$ into $T^{<\nu_0}$, furthermore we may assume that 
the restrictions to $\partial Y$ of the maps $h_s$ coincide. 
As a consequence, there is some $s\in [s_{i+1},s_{i+2}]$ and 
a component 
of the (closure of the)
thick part of $\sigma(h_{s})$ whose image under the map
$h_s$ intersects $T^{\leq \nu_0}$. Let $s\geq s_{i+1}$ be the smallest number 
with this property. By continuity, the curve $\alpha_{s_i}$ is the core curve
of a Margulis tube for $\sigma(h_s)$. 
 
By the choice of $\nu_0$, by Lemma \ref{pleatedmargulis} and the 
definition of the set $V_i$, there exists a 
diskbounding simple closed curve $c_i$ on $\Sigma$ 
which is disjoint from the core curve of any Margulis tube for $\sigma(h_s)$ and hence which is 
disjoint from $\alpha_{s_i}$. Note that we may have $c_i=\alpha_{s_i}$. 
This curve then belongs to one of the disk sets
${\cal D}_1,{\cal D}_2$. Moreover, 
property (1) or (2) in Lemma \ref{pleatedmargulis} holds true for $\sigma(h_s)$.

Using this argument inductively, we conclude that either 
$n=1$ and ${\rm diam}_Y({\cal D}_1\cup {\cal D}_2)\leq 2D_1+2$ 
by the beginning of the proof, or 
for each $i\geq 1$, the curve
$\alpha_{s_i}$ and hence the bridge arc $\tau_{s_i}$ for $Y$
is disjoint from a diskbounding simple closed curve
$c_i$ on $\Sigma$. 

Since $\alpha_{s_i},\alpha_{s_{i+1}}$ are disjoint for all $i$, we have
$d_{\cal C\cal G}(c_i,c_{i+1})\leq 3$. 
But $d_{\cal C\cal G}({\cal D}_1,{\cal D}_2)\geq 4$ 
by assumption, 
and therefore if $c_{i}\in {\cal D}_j$ for $j=1$ or $j=2$ 
then the same holds true for $c_{i+1}$. 
By induction on $i$, we deduce that up to renaming, we have 
$c_i\in {\cal D}_1$ for all $i$.

As a consequence, 
by the choice of $p$, we have ${\rm diam}_Y(\cup_ic_i)\leq p$.
Lemma \ref{projectionbound1} 
then shows
that ${\rm diam}_Y(\gamma_1 \cup \gamma_2)\leq 
p+2(D_1+2)$. Since $\gamma_i\in {\cal D}_i$ for $i=1,2$ were
arbitrarily chosen and ${\rm diam}_Y({\cal D}_j)\leq p$, 
this yields that 
${\rm diam}_Y({\cal D}_1\cup {\cal D}_2)\leq p+2(D_1+2)$ which is
what we wanted to show.

The argument in the case that $Y$ is an annulus is identical to the 
above discussion, with the only difference that in each step, the 
bridge arc $\tau$ for $Y$ is replaced by the simple closed curve 
$\alpha$ crossing through $\partial Y$. Additional details will be
left to the reader.  
\end{proof}

From now on we always assume that the Hempel distance of the manifold 
$M_f$ is at least $4$, and we let $p>0$ be the number from
Lemma \ref{homotopyidentity}.
We use Lemma \ref{nothinpart} to control the diameters of the 
subsurface projections
of ${\cal D}_1\cup {\cal D}_2$ into subsurfaces $Y$ whose boundaries
have large diameter in $M_f$. 

\begin{lem}\label{largediameter}
There exist numbers $R_2=R_2(\Sigma)>0,k_2=k_2(\Sigma)>0$ with the following property.
Let $Y\subset \Sigma$ be a strongly incompressible subsurface with 
$d_{\cal C\cal G}(\partial Y,{\cal D}_1\cup {\cal D}_2)\geq p$. 
If  
$\partial Y$ contains a component $\beta$ whose diameter 
in $M_f$ is at least $R_2$, then 
\[{\rm diam}_Y({\cal D}_1\cup {\cal D}_2)\leq k_2.\]
\end{lem}

\begin{proof} Let $\nu_1<\kappa_0$ be as in Lemma \ref{nothinpart}. 
Choose $R_2>0$ sufficiently large that
the following holds true. Consider 
a hyperbolic metric $\sigma$ on 
$\Sigma$, and let $x,y\in \Sigma$ be two points of distance
at least $R_2$; then any path in $\Sigma$
connecting $x$ to $y$ crosses through a Margulis tube whose
core curve has length smaller than $\nu_1$.

Assume 
that $\partial Y$ has a component
$\beta$ whose diameter in $M_f$ is at least $R_2$.
By Lemma  \ref{nothinpart}, it suffices to consider the case that 
on $\beta$, 
the injectivity radius of $M_f$ is bounded from below
by $\nu_1/2$.

By the choice of $R_2$,
for every $g\in {\cal L}(\partial Y)$ and corresponding
hyperbolic metric $\sigma(g)$, there exists 
a simple closed curve $\alpha$ on $\Sigma$ of
$\sigma(g)$-length smaller than
$\nu_1$ which is crossed through by $\beta$.
Since the injectivity radius of $M_f$ on $\beta$ 
is at least $\nu_1/2$, the curve $g(\alpha)$
bounds a disk in $M_f$. In other words, $\alpha$ is contained 
in one of the disk sets for $M_f$, say the disk set ${\cal D}_1$. 

As in the proof of Lemma \ref{nothinpart}, we find that 
if $Y$ is not an annulus, then a subarc of
$\alpha$ is a bridge arc for $Y$. 
In other words, 
there exists a bridge arc for $Y$ of $\sigma(g)$-length
smaller than $\nu_1<\kappa_0$
which is a subarc of the diskbounding
simple closed curve $\alpha$. If $Y$ is an annulus then we choose 
$\alpha$ as a simple closed curve of length smaller than $\kappa_0$ 
which 
crosses through $\partial Y$.

We argue now as in the proof of Lemma \ref{nothinpart}. 
Let $\gamma_i\in {\cal D}_i$ $(i=1,2)$ and let
$g,h\in {\cal P}(\partial Y)$ be pleated surfaces for $\gamma_1,\gamma_2$ 
as in Lemma \ref{projectionbound1} or Lemma \ref{projectionbound3}. 
Connect $g,h$ by a path $h_s$ $(0\leq s\leq 1)$ in 
${\cal L}(\partial Y)$.
For each $s$ choose a diskbounding simple closed curve
$\alpha_s$ for $h_s$ of $\sigma(h_s)$-length smaller than $\nu_1$ 
which is crossed through
by $\beta$. The curve $\alpha_s$ contains a 
bridge arc for $Y$ of $\sigma(h_s)$-length smaller than $\kappa_0$.
By continuity, 
there exists an open neighborhood $V_s$ of $s$ in $[0,1]$ so that
for every $t\in V_s$, the $\sigma(h_t)$-length of $\alpha_s$ is smaller
than $\kappa_0$.

Cover the interval $[0,1]$ by finitely many of the sets $V_s$.
This covering can be used to find a partition
$0=s_0<\dots < s_n=1$ such that for all $i$, we have
$[s_i,s_{i+1}]\subset V_{s_i}$.
Now for each $i$, if $\alpha_{s_i}$ is different from $\alpha_{s_{i+1}}$,
then the curves $\alpha_{s_i},\alpha_{s_{i+1}}$ are core curves of
Margulis tubes for the same metric $h_{s_{i+1}}$ and hence they
are disjoint. This implies that if $\alpha_{s_i}\in {\cal D}_j$
$(j\in \{1,2\})$ then
the same holds true for $\alpha_{s_{i+1}}$. 
As a consequence, if $\alpha_0\in {\cal D}_j$ then so is
$\alpha_1$.

By assumption on $Y$, we have ${\rm diam}_Y({\cal D}_j)\leq p$ 
(see the proof of Lemma \ref{nothinpart}).
Using once more Lemma \ref{projectionbound1}
and Lemma \ref{projectionbound3}, this implies as in the proof of Lemma \ref{nothinpart}
that 
$d_Y(\gamma_1,\gamma_2)\leq p+2(D_1+1)$ if $Y$ is not an annulus, and
$d_Y(\gamma_1,\gamma_2)\leq p+2(D_2+1)$ otherwise. 
This is what we wanted to show.
\end{proof} 


 The proof of Lemma \ref{largediameter} uses the assumption that 
the diameter of a component $\beta$ of $\partial Y$ in $M_f$ 
is large to conclude that
for any $g\in {\cal L}(\partial Y)$, a component of
$\partial Y$ crosses through a Margulis tube for
$\sigma(g)$ whose core curve is diskbounding.
It is not used elsewhere in the proof. 
Thus the statement of the lemma can be extended in the following 
way. For a multicurve $\partial Y\subset \Sigma$, define
a \emph{simplicial path} $g_s\subset {\cal L}(\partial Y)$ 
between two pleated surfaces $(\Sigma,g_0),(\Sigma,g_1)$ to be a path 
which consists of pleated surfaces connected by a diagonal exchange path as
in Lemma \ref{homotopy}. We assume for convenience that such a path 
is parameterized on the interval $[0,1]$, but there are no other requirements
for the parameterization.
Let as before $\kappa_0>0$ be a constant with 
properties (P1),(P2). 
We say that such a simplicial path $g_s\subset {\cal L}(\partial Y)$  is \emph{thick-thin incompatible} if 
for every $s$ there exists a simple closed curve 
in $(\Sigma,\sigma(g_s))$ of length smaller than $\kappa_0/10$ 
whose image under $g_s$ is contractible in $M$. 
The constants $R_2>0,p>0$ in the
following lemma are as in Lemma \ref{largediameter}.

\begin{lem}\label{thin}
  Let $Y\subset \Sigma$ be an essential subsurface with
  $d_{\cal C\cal G}(\partial Y,
{\cal D}_1\cup {\cal D}_2)\geq p$ 
and let 
$s\in [0,1]\to
g_s\in {\cal L}(\partial Y)$ be a thick-thin incompatible path.
Then for $\ell=1$ or $\ell=2$ and each $s$, there exists
a bridge arc $\tau_s$ for $Y$ of $\sigma(g_s)$-length at most
$\kappa_0$ which is the subsurface projection of a diskbounding
curve in ${\cal D}_\ell$, or there is a diskbounding curve in ${\cal D}_\ell$
of $\sigma(g_s)$-length at most $\kappa_0$
crossing through $\partial Y$ if $Y$ is an annulus. 
In particular, the 
distance in the arc and curve graph of $Y$
between $\tau_0,\tau_1$ 
is at most $p$.
\end{lem}
\begin{proof}
By uniform quasi-convexity of the
  disk set ${\cal D}_\ell$ in ${\cal C\cal G}(\Sigma)$ $(\ell=1,2)$ and
  the lower bound $p$ on the distance $d_{\cal C\cal G}(\partial Y,
  {\cal D}_1\cup {\cal D}_2)$, the diameter of the subsurface projection
  of ${\cal D}_i$ into $Y$ is bounded from above by $p$
  (see the proof of Lemma \ref{largediameter}).

  Now if $Y$ is not an annulus, if 
  $g\in {\cal L}(\partial Y)$ and if there exists a simple
  closed curve $c$ in $\Sigma$ of $\sigma(g)$-length at most $\kappa_0$
  whose image under the map $g$ is contractible in $M$, then
  $\partial Y$ crosses through $c$ and hence there exists a bridge
  arc for $Y$ of $\sigma(g)$-length
  at most $\kappa_0$ which is a subarc of
  a curve in ${\cal D}_\ell$ for $\ell=1$ or $\ell=2$. The path $g_s$ then
  coarsely determines in this way a sequence of elements of
  ${\cal D}_1\cup {\cal D}_2$ containing such short bridge arcs 
  where we may assume that two consecutive
  of these elements are disjoint. Thus this path is entirely contained
  in ${\cal D}_\ell$ for $\ell=1$ or $\ell=2$.

As a consequence, 
  there are bridge arcs for $\sigma(g_0),\sigma(g_1)$ of length at most
  $\kappa_0$ which are projections of curves in ${\cal D}_\ell$.
Then their distance in the arc and curve graph of $Y$ is at most $p$.
\end{proof}

Using what we established so far, we are left with analyzing
subsurfaces $Y$ whose boundary $\partial Y$ have
a component $\beta$ with $\ell_f(\beta)\geq R_2$
and diameter in $M$ which is smaller than $R_2$ together
with maps $g\in {\cal L}(\partial Y)$
whose thin part is mapped to the thin part of $M$.
Compressibility of the Heegaard surface $\Sigma\subset M$
causes considerable difficulties, and the remainder of this section
is devoted to overcoming this problem.
The strategy was laid out by Thurston \cite{T86} and is based on an 
argument by contradiction.

To set up the proof, 
consider a sequence
$(M_n,x_n)$ $(n\geq 1)$ of pointed hyperbolic 
$3$-manifolds so that the injectivity radius
${\rm inj}(x_n)$ of $M_n$ at $x_n$ is bounded 
from below by a positive constant not depending on $n$.
We say that the sequence
\emph{converges} 
to a pointed hyperbolic $3$-manifold $(M,x)$ \emph{in
the pointed geometric topology}  
if for every $R>0,\xi>0$ there is a number 
$n(R,\xi)>0$, and for every $n\geq n(R,\xi)$ there exists a smooth embedding,
the \emph{approximation map} $k_n:U_n\subset M\to M_n$, such that $k_n$
is defined on the ball $B_M(x,R)\subset U_n$ of
radius $R$ centered at $x\in M$,
it sends $k_n(x)=x_n$,
and the restriction of $k_n$ to $B_M(x,R)$ satisfies
$\Vert \rho_M-k_n^*\rho_{M_n}\Vert_{C^2(B_M(x,R))}\leq \xi$
where $\rho_M,\rho_{M_n}$ are 
the metric tensors on $M,M_n$. 
We then say 
that the restriction of $k_n$ to
$B(x,R)$ is \emph{$\xi$-almost isometric}. 

The following is well known (see e.g. Chapter E of \cite{BP92} for details).

\begin{prop}\label{convergence}
If $(M_n,x_n)$ is a sequence of pointed hyperbolic $3$-manifolds
such that ${\rm inj}(x_n)\geq \chi>0$ for all $n$, then up to passing to a
subsequence, the sequence $(M_n,x_n)$ converges in the pointed geometric
topology to a pointed hyperbolic $3$-manifold $(M,x)$.
\end{prop}

We can also consider convergence of pleated surfaces in the pointed 
geometric topology. The following lemma establishes
a first control on such pleated surfaces.

\begin{lem}\label{convergence2}
Let $M_n$ be a sequence of closed hyperbolic $3$-manifolds
with Heegaard surface $\Sigma$ of Hempel distance at least $4$. 
Assume that 
there exists a number $\chi>0$ and 
for each $n$ a pleated surface
$g_n:(\Sigma,\sigma(g_n))\to M_n$ homotopic to the inclusion 
with the following properties.
\begin{enumerate}
\item There exists a point $x_n\in \Sigma$ with 
${\rm inj}(x_n)\geq \chi$ and ${\rm inj}(g_n(x_n))\geq \chi$.
\item The thin part of $(\Sigma,\sigma(g_n))$ consist of annuli
whose core curves are of distance at least 
$3$ to ${\cal D}_1\cup {\cal D}_2$.
\end{enumerate}
Then up to passing to a subsequence, the pointed manifolds 
$(M_n,g_n(x_n))$ converge in the pointed geometric topology to a 
pointed hyperbolic $3$-manifold $(M,\hat x)$, and the pointed 
pleated surfaces
$g_n:(\Sigma,x_n)\to (M_n,g_n(x_n))$ converge to a pleated 
surface $g:W\to M$ where 
$W$ is a finite volume hyperbolic surface 
homeomorphic to an essential subsurface of 
$\Sigma$ of negative Euler characteristic, 
and $g$ maps cusps of $W$ to cusps of $M$. Furthermore,
if $W\not=\Sigma$ then $g$ is $\pi_1$-injecctive.
\end{lem}

\begin{proof} Fix as before a 
  Margulis constant $\kappa_0<\chi$ for hyperbolic surfaces.
Since there are only finitely many topological types of subsurfaces of 
$\Sigma$, after passing to a subsequence we may assume that there
exists a connected subsurface $W_n\subset \Sigma$ containing $x_n$
with geodesic boundary such 
that the following holds true.
\begin{itemize}
\item The homeomorphism type of the surfaces $W_n$ does not depend on $n$.
\item If $\hat W_n$ denotes the component of the $\kappa_0$-thick part of 
$(\Sigma,\sigma(g_n))$ containing $x_n$, then the inclusion 
$\hat W_n\to W_n$ is a homotopy equivalence.
\item The $\sigma(g_n)$-length of each boundary component of $W_n$ 
tends to $0$ as $n\to \infty$.
\end{itemize}

Using the first property above, we may identify each $W_n$ with a fixed 
subsurface $W$ of $\Sigma$. The surface $W$ has negative Euler characteristic
and may coincide with $\Sigma$.

By assumption, the image under $g_n$ of no boundary component 
$\alpha$ of $g_n(W)$
is contractible in $M_n$. Since $g_n$ is $1$-Lipschitz, this implies that 
for each $n$, the curve $g_n(\alpha)$ is homotopic to a closed
geodesic $\hat \alpha_n$ in $M$. 
As $n\to \infty$, the lengths of the curves $g_n(\alpha)$ and hence of
$\hat \alpha_n$ 
tend to zero.

Since the $\sigma(g_n)$-diameter of any component of 
the $\kappa_0$-thick part of 
$W_n$ is bounded from above by a constant only depending
on $\Sigma$, it follows from 
Proposition \ref{convergence} and the Arzela Ascoli theorem 
that up to passing to another subsequence, we may assume that 
the pointed manifolds $(M_n,g_n(x_n))$ converge in the pointed geometric topology 
to a pointed hyperbolic $3$-manifold $(M,\hat x)$,
and the pleated surfaces 
$g_n\vert W_n:(W_n,x_n)\to M_n$ converge in the pointed geometric
topology to a pleated surface
$g:(W,x)\to (M,\hat x)$ where $W$ is obtained from $W_n$ by
replacing each boundary component 
by a puncture. The surface $W$ is equipped
with a complete finite volume hyperbolic metric. 
Furthermore, since the length of $g_n(\alpha)$ tends to $0$ as
$n\to \infty$, the map $g$ maps cusps of $W$ to cusps of $M$.

We are left with showing that if $W\not=\Sigma$, then 
$g$ is $\pi_1$-injective.
Thus assume that $W$ has at least one cups. We know that
$g$ maps cusps in $W$ to cusps in $M$ and hence a closed curve
$\alpha\subset W$ whose image under $g$ is contractible
in $M$ is essential in $W$. Assume to the contrary that
$\alpha\subset W$ is such a curve.

Denote by $D\subset \mathbb{C}$ the closed unit disk. 
There exists a continuous map $\psi:D\to M$ with $\psi(\partial D)=g(\alpha)$.  
By compactness of $D$ and hence of $\psi(D)\subset M$, 
for large enough  $n$ the set $\psi(D)\subset M$ is contained in the 
domain $U_n$ of the almost isometric map $k_n$ which determine the 
geometric convergence $(M_n,g_n(x_n))\to (M,\hat x)$, 
and hence 
$k_n(g(\alpha))\subset M_n$ is contractible as well. 

But $\alpha$ is contained in $W$, and a simple closed
curve $\xi$ going around a cups of $W$ is a geometric 
limit of boundary curves  of $W_n$ which are of distance 
at least $3$ from ${\cal D}_1\cup {\cal D}_1$.   
Convergence
of the maps $g_n:W_n\subset \Sigma\to M_n$ to the map 
$g:W\to M$ yields that for large enough $n$, 
the image $k_ng(\xi)$ of $g(\xi)$
under $k_n$ is homotopic to 
the image under the map $g_n\vert W_n: W_n\to M_n$ of a boundary component 
of $W_n$, and $k_ng(\alpha)$ is contractible in 
$M_n-g_n(\partial W_n)$. 

On the other hand, 
since $d_{\cal C\cal G}(\partial W_n,{\cal D}_1\cup {\cal D}_2)\geq 3$
by assumption, Lemma \ref{incompressiblesub} shows that the
surface $g_n(W_n)$ is incompressible in $M-g_n(\partial W_n)$. 
This is a contradiction
which shows that indeed, $g$ is $\pi_1$-injective and 
completes the proof of the lemma.
\end{proof}

We use Lemma \ref{convergence2} to establish the following
relative version of Thurston's uniform injectivity result.
Let ${\bf P}(M)$ be the projectivized tangent bundle of the hyperbolic 3-manifold $M$. 
If $g:(\Sigma,\sigma(g))\to M$  is a pleated surface, with pleating lamination 
$\lambda$, then there is a map $p:\lambda\to {\bf P}(M)$ which 
associates to $x\in \lambda$ the tangent line of $g(\lambda)$ at $g(x)$. The number
$R_2>0$ in the statement of the proposition is the number from Lemma \ref{largediameter}.

\begin{prop}\label{thurstonrelative2}
  For $b>0$ there exists a number $R_3=R_3(b)>R_2$, and for all
  $\epsilon >0$ there exists a number $\delta=\delta(b,\epsilon)>0$
  with the following property. Let $M$ be a closed hyperbolic 3-manifold
  with Heegaard surface $\Sigma$ and Hempel distance at least $4$.
  Let $X\subset \Sigma$ by a multicurve
with $d_{\cal C\cal G}(X,{\cal D}_1\cup {\cal D}_2)\geq p$.  
Assume that the diameter in $M$ of the geodesic representative of each 
component of $X$ is at most $R_2$ and that 
$X$ has a component $\beta$ whose length in $M$ 
at least $R_3$.
Let $g\in {\cal P}(X)$ be a pleated surface with all core curves
of Margulis tubes incompressible in $M$ and assume that
there exists a simple closed curve $\alpha$ on $\Sigma$ disjoint from
$X$ of $\sigma(g)$-length at most $b$. Then  
\[d_{\sigma(g)}(x,y)\leq \epsilon  
\text{ for all } x,y\in \beta \text{ with } d_{{\bf P}(M)}(p(x),p(y))\leq \delta
\]
where $d_{\sigma(g)}$ denotes the distance function of the hyperbolic
  metric $\sigma(g)$ on $\Sigma\supset \beta$. 
\end{prop}

\begin{proof} Following the strategy of \cite{T86}, assume to the
  contrary that the proposition does not hold true.
  Then there are numbers $b>0,\epsilon >0$ for which no
$R_3(b)>0,\delta(b,\epsilon) >0$ as in the statement of the proposition exists.
  This means that there exists a sequence of counterexamples,
  consisting of a sequence of hyperbolic 3-manifolds $M_n$ with Heegaard surface 
  $\Sigma$, multicurves $X_n\subset \Sigma$ with 
  $d_{\cal C\cal G}(X_n,{\cal D}_1\cup {\cal D}_2)\geq p$, 
  pleated surfaces $g_n:\Sigma\to M_n$ whose 
  pleating lamination contains $X_n\supset \beta_n$ 
  and the following properties. 
  \begin{itemize}
  \item The length of $\beta_n$ is at least $n$.
  \item 
  The diameter in $M_n$ of each component of $X_n$ is at most $R_2$. 
 \item The core curve of each Margulis tube of $\sigma(g_n)$ 
  is not homotopic to zero.
  \item There exists a simple closed curve $\alpha_n\subset \Sigma$ 
  disjoint from $X_n$ of $\sigma(g_n)$-length at most $b$.
  \item 
  There are points $x_n,y_n\in \beta_n$ with $d_{\sigma(g_n)}(x_n,y_n)\geq \epsilon$
  and $d_{{\bf P}(M)}(p(x_n),p(y_n))\leq 1/n$.
  \end{itemize}

Since the diameter in $M_n$ of the curve $\beta_n$ is at most $R_2$ and the length
of $\beta_n$ tends to infinity with $n$, for sufficiently large $n$ the curve $\beta_n$ is not the core 
curve of a Margulis tube and 
$\beta_n$ is contained in 
the $\chi$-thick part of $M_n$ for a constant $\chi>0$ only depending on $R_2$. 
Since core curves of Margulis tubes in $\Sigma$ for $\sigma(g_n)$ are not null-homotopic in $M_n$, 
we can apply Lemma \ref{convergence2}. It yields that 
by passing to a subsequence, we 
may assume that
the pointed pleated surfaces $g_n:(\Sigma,x_n)\to (M_n,g_n(x_n))$ converge in the geometric
topology to a pointed pleated surface $g:(W,x)\to (M,g(x))$ where $W$ is a finite
volume hyperbolic surface homeomorphic to the interior of an essential 
subsurface of $\Sigma$
and $g$ maps cusps to cusps.
The pleating lamination of $g$ contains the geometric 
limit $\beta$ of the simple closed curves $\beta_n$.

By Lemma \ref{convergence2}, if $W\not=\Sigma$ then the map 
$g:W\to M$ is $\pi_1$-injective. 
If $W=\Sigma$ then there exists a simple closed curve $\alpha$ on 
$\Sigma$ of 
$\sigma(g)$-length at most $b$ which is disjoint from a geometric
limit $\beta$ of the curves
$\beta_n$. Furthermore, the 
argument in the proof of Lemma \ref{convergence2} yields that the map
$g:\Sigma-\alpha\to M-g(\alpha)$ is $\pi_1$-injective. No modification of the 
argument is required. 
Let $Z=W$ if $W\not=\Sigma$ or $Z=\Sigma-\alpha$ otherwise where
we identify $\alpha$ with its geodesic representative for
the metric $\sigma(g)$.

Assume for the moment that $W=\Sigma$. 
By the collar lemma for hyperbolic surfaces, the distance in $(Z,\sigma(g))$ 
between $\alpha$ and the geodesic lamination $\beta$ is bounded from below by 
a positive constant $\rho=\rho(b) <\chi/2$ only depending on $b$.
Modify $g$ with a homotopy which equals the identity on the complement
of the $\rho/2$-neighborhood of $\alpha$ to a map $\hat g$ which maps 
$\alpha$ to the complement of the $\rho/2$-neighborhood of 
$g(\beta)$. This can be done in such a way that the map 
$\hat g$ is $L$-Lipschitz for some $L>0$ and that it maps
$W$ into $M^{>\chi/2}$. 

Resume the general case that includes $W\not=\Sigma$.
If $W\not=\Sigma$ then put $L=1$ and $\alpha=g(\alpha)=\emptyset$.
Consider the diagonal lamination
$\beta\times \beta \subset Z\times Z$.
If $x_1,x_2\in \beta$ are two points
which have the same image under $p$, then the leaves
$\ell_1,\ell_2$ of $\beta$ through $x_1,x_2$ are mapped to the same
geodesics in $\hat g(Z)\subset M$. Since the lamination $\beta$ and hence 
$\beta\times \beta$ is compact,
the leaf $\ell_1\times \ell_2$ of $\beta\times \beta$
enters a small neighborhood $U\subset Z\times  Z$ 
infinitely often. We may assume that the diameter of $U$
for the product metric on $Z\times Z$ is smaller than 
$\rho/2L$.

From each return of $\ell_1\times \ell_2$ to $U$
one can construct closed loops in $Z$ based at $x_1,x_2$ 
by connecting the endpoints 
of the subarcs of $\ell_1,\ell_2$ determined by these return times by
an arc of length at most $\rho/2L$.
The two resulting closed curves are not homotopic in $Z$. 
As the map $\hat g$ is $L$-Lipschitz, 
their images under $\hat g$ are obtained from each other
by concatenation with a loop of length smaller than
$\rho/2<\chi/4$, based at a point in $g(\beta)$.
But the injectivity radius of $M$ on $g(\beta)$ 
is a least $\chi/2$. This implies that 
the images under $\hat g$ of
these loops are homotopic with a homotopy entirely contained in
the $\rho/2$-neighborhood of $g(\beta)$. Now the
$\rho/2$-neighborhood of $g(\beta)$ is contained 
in $M-\hat g(\alpha)$.
 Since the map $\hat g:Z\setminus \alpha \to M\setminus \hat g(\alpha)$ is $\pi_1$-injective,
we deduce as on p.232 of \cite{T86}
that the leaves $\ell_1,\ell_2$ of
$\beta$ are identical. 

That this leads to a contradiction to the assumption 
$d_{\sigma(g_n)}(x_n,y_n)\geq \epsilon$ for all $n$ follows from 
the arguments on p.232-233 of \cite{T86} which work directly
  with compact subsurfaces filled by limit laminations and uses nowhere that
  the underlying surface is closed or of finite volume.
This completes the proof of 
 the proposition.
\end{proof}

\begin{rem}\label{relincom}\normalfont
It follows from the proof of \Cref{thurstonrelative2} that 
under the assumption in the proposition, there exists a constant
$\rho=\rho(b)>0$ 
such that the restriction of a map $g\in {\cal L}(X)$ to the 
$\rho$-neighborhood of $\beta$ is incompressible within the 
$\rho$-neighborhood of $g(\beta)\subset M$. The point here is 
that $\rho$ only depends on $b$. Furthermore, the conclusion of the 
proposition also holds true for any element $g\in {\cal L}(X)$ 
which is contained in some simplicial path in ${\cal L}(X)$. Namely,
the argument only used that the maps considered are one-Lipschitz and
map the boundary $\partial Y$ of the subsurface $Y$ isometrically.
\end{rem}

Recall that if $d_{\cal C\cal G}(\partial Y,{\cal D}_1\cup {\cal D}_2)\geq p$ then a simple closed 
curve on $\Sigma$ which is disjoint from $\partial Y$ is not homotopic to 
zero in $M_f$ and hence it has a geodesic representative in $M_f$. 

\begin{cor}\label{shortdisjoint}
For every $b>0$ 
there exists a number $R_4=R_4(b)>R_3(b)$ with the following property.
  Let $\partial Y\subset \Sigma$ be a subsurface 
with $d_{\cal C\cal G}(\partial Y,{\cal D}_1\cup {\cal D}_2)\geq p$   
  and assume that there exists
  a component $\beta$ of $\partial Y$ 
  whose geodesic representative in $M$ is 
  contained in a subset of
  $M$ of diameter at most $R_2$ and has length at least $R_4$. 
  Let $g_s\subset {\cal L}(\partial Y)$
  be a simplicial path and assume that for each $s$, the core
  curve of any Margulis tube for $\sigma(g_s)$
  is incompressible in $M$. Assume furthermore that for each $s$ there
  exists a simple closed curve on $(\Sigma,\sigma(g_s))$ 
  disjoint from
  $\partial Y$ of $\sigma(g_s)$-length
  at most $b$. If $Y$ is not an annulus then there exists a bridge
  arc $\tau$ for $Y$ whose length is smaller than $\kappa_0$ for
 each of the metrics $\sigma(g_s)$. If $Y$ is an annulus then there exists
 a simple closed curve crossing through $Y$
 which intersects a minimal curve crossing through $Y$ for each
 of the metrics $\sigma(g_s)$ in at most 2 points.
\end{cor}
\begin{proof} As in the proof of \Cref{thurstonrelative2}, 
there exists a universal constant $\chi <\kappa_0$ 
only depending on $R_2$ with the following property. 
Let $g\in {\cal L}(\partial Y)$ be such that the core curves of
Margulis tubes for $\sigma(g)$ 
are incompressible in $M$. Then for  
any point $x\in \beta$, the injectivity radius of 
$\sigma(g)$ at $x$ is at least $\chi$.

Let 
$\delta=\delta(b,\chi/4)>0$ be as in \Cref{thurstonrelative2}.
By standard hyperbolic 
geometry (see \cite{M00} for details), there exists a number 
$R_4=R_4(b,\chi/4)>0$ with the property that for any
closed geodesic $\zeta$ in a hyperbolic 3-manifold 
$M$ whose diameter in $M$ is at most $R_2$ and whose length 
is at least $R_4$, there are three points $z_1,z_2,z_3\in \zeta$ 
with $d_{{\bf P}(M)}((pT\zeta)(z_i),(pT\zeta)(z_j))< \delta$ 
and such that the distance along $\zeta$ between 
$z_i,z_j$ is 
larger than $2\chi$ $(i=1,2,3)$. Here $pT\zeta$ denotes the projectivized
tangent line of $\zeta$.

Let $g_s\subset {\cal L}(\partial Y)$ be a simplicial path
as in the statement of the corollary for this number $R_4$. 
For each $s$
the restriction of the map $g_s:(\Sigma,\sigma(g_s))\to M$
to the $\sigma(g_s)$-geodesic $\beta\subset \partial Y$ is an isometry onto
a geodesic $\hat \beta\subset M$. If the diameter of 
$\hat \beta$ in $M$  is at most $R_2$ and its length is at least $R_4$ then 
by the previous paragraph, 
there are points $z_1,z_2,z_3\in \hat \beta$ with the following property.
Let $x_i\in \beta$ be the preimage of $z_i$ under $g_s$; then 
$d_{{\bf P}(M)}(px_i,px_j)< \delta$ and the distance along $\beta$ between
$x_i,x_j$ is larger than $2\chi$.
It then follows from \Cref{thurstonrelative2} and \Cref{relincom} 
that 
$d_{\sigma(g_s)}(x_i,x_j)< \chi/4$ for all $s$ $(i,j=1,2,3)$.

Since the injectivity radius of $\sigma(g_s)$ at $x_i,x_j$ 
is at least $\chi$ for all $s$, the points $x_i,x_j$ can be connected
in $(\Sigma,\sigma(g_s))$ by a unique minimal geodesic arc $\alpha_s$ 
of length at most $\chi/4$. Since the metrics $\sigma(g_s)$ 
depend continuously on $s$, the arc depends continuously on $s$ and hence
its homotopy class with fixed endpoints is independent on $s$.
See also Lemma \ref{homotopy} and Remark \ref{pleatedhomotopy} for 
a similar statement.

Consider first the case that $Y$ is not an annulus. 
If $\alpha_s$ is contained in $Y$, then $\alpha_s$ is a bridge arc 
for $Y$ with 
the required properties. If $\alpha_s$ has a proper subsegment contained 
in $Y$ with one endpoint an endpoint of $\alpha_s$, then 
the same argument applies to this subsegment. 
Now the points $x_2,x_3$ are contained in the $\chi/4$-neighborhood of
$x_1$, and the simple closed geodesic $\beta$ crosses through these points.
Since $\beta$ does not have self-intersections, locally the subarcs
of $\beta$ through the points $x_i$ 
decompose a suitably chosen disk neighborhood of $x_1$ into two strips
with boundary in $\beta$ which are separated by a subarc of $\beta$, say
the subarc through $x_j$, and two half-disks. 
Then the subarc of $\beta$ through $x_j$ divides a small  disk about
$x_j$ into two half-disks, at least one of which is contained in
$Y$. Thus for either $\ell=j+1$ or $\ell=j-1$ 
(indices are taken modulo 3), the initial segment
of the minimal geodesic connecting $x_j$ to $x_\ell$ is contained in
$Y$. Its first intersection with $\partial Y$ defines a bridge
arc for $Y$ of length smaller than $\chi/2<\kappa_0/2$, and up to homotopy 
keeping the endpoints in $\partial Y$,
the length of this arc 
is smaller than $\kappa_0/2$ for all $s$. This shows the corollary
in the case that $Y$ is not an annulus.

Following once more \cite{M00}, the argument in the case that
$Y$ is an annulus is analogous. Namely, the above argument 
shows that for each of the two sides of the geodesic representative
$\beta $ of the core curve of $Y$ 
we can find a homotopy class of an arc  $\tau$ 
with endpoints in $\beta$ which leaves $\beta$ from the chosen side 
and whose 
length is smaller than
$\kappa_0/2$ for each of the metrics $g_s$. 
If the arc $\tau$ 
leaves and returns to different sides of $\beta$ then its 
concatenation with a subarc of $\beta$ determines a simple
closed curve crossing through $Y$ with the properties we are looking for. 
Otherwise there are two such arcs leaving and returning to distinct sides
of $Y$, and the union of these arcs with subarcs of $\partial Y$ define
a simple closed curve with the desired properties. This follows once
more from the fact that for a given hyperbolic metric, any two bridge arcs
for a subsurface $Y$ of length at most $\kappa_0$ are disjoint up to 
homotopy keeping the endpoints on 
$\partial Y$.
\end{proof}  

We are left with analyzing pleated surfaces $(\Sigma,g)$ whose pleating 
locus contains the boundary $\partial Y$ of a 
subsurface $Y\subset \Sigma$ with the following properties.
\begin{enumerate}
\item $d_{\cal C\cal G}(\partial Y,{\cal D}_1\cup {\cal D}_2)\geq p$.
\item The diameter of $(\Sigma,\sigma(g))$ is uniformly bounded.
\item The length in $(\Sigma,\sigma(g))$ 
of any simple closed curve disjoint from
  $\partial Y$ is large.
\end{enumerate}

A geodesic lamination $\beta$ on a surface $\Sigma$ is said to 
\emph{fill} $\Sigma$ if all complementary components of $\beta$ are
simply connected. We have

\begin{lem}\label{surfaces}
  Let $(\Sigma,\sigma_n)$ be a sequence of hyperbolic surfaces of uniformly
  bounded diameter. Let $\beta_n\subset (\Sigma,\sigma_n)$ 
  be a simple closed multicurve
  whose length tends to infinity with $n$ and assume that
  the same holds true for the length of any simple closed curve
  disjoint from $\beta_n$.
  Then up to passing to a subsequence and
  the action of the mapping class group, 
  the triple $(\Sigma,\sigma_n,\beta_n)$ converges
in the geometric topology to a triple
  $(\Sigma,\sigma,\beta)$ where $\sigma$ is a hyperbolic metric
  on $\Sigma$ of uniformly bounded diameter and $\beta$ is a
  geodesic lamination which fills $\Sigma$.
\end{lem}  
\begin{proof}
Since by assumption the diameters of the hyperbolic metrics $\sigma_n$ 
are bounded from above by a universal constant
and since the mapping class group acts cocompactly on the thick part of
Teichm\"uller space, up to the action of 
the mapping class group we may assume that the hyperbolic metrics 
$\sigma_n$ converge in Teichm\"uller space to a hyperbolic metric 
$\sigma$. 

The space of geodesic laminations on 
$(\Sigma,\sigma)$
equipped with the Hausdorff topology on compact subsets of $\Sigma$ 
is compact. Thus up to passing to 
a subsequence, the geodesic laminations $\beta_n$ converge as $n\to \infty$ 
in the Hausdorff topology to a 
geodesic lamination $\beta$. We have to show that $\beta$ fills $\Sigma$.

Namely, otherwise there is a simple closed curve $\alpha\subset \Sigma$ disjoint 
from $\beta$. As $\beta_n\to \beta$ in the Hausdorff topology, either the curve 
$\alpha$ is disjoint from $\beta_n$ for all sufficiently large $n$, or $\alpha$ is a component of
$\beta$. Now the length of $\alpha$ for the metric $\sigma$ is close to the length of 
$\alpha$ for $\sigma_n$ and hence by the assumption that the 
$\sigma_n$-length of 
any closed curve disjoint from $\beta_n$ tends to infinity with $n$, the curve $\alpha$ can 
not be disjoint from $\beta_n$ for large $n$. 

We conclude that any simple closed curve $\alpha$ which is disjoint from $\beta$ is a component of 
$\beta$. Moreover, as $\beta$ is a limit in the Hausdorff topology of 
simple closed curves with an essential intersection with $\alpha$, if $A\subset \Sigma$ is any
annular neighborhood of $\alpha$ then $\beta$ intersects both components of 
$A-\alpha$. But this just means that $\beta$ fills $\Sigma$ and completes the proof of 
the lemma.
\end{proof}  


\begin{lem}\label{allconverge}
For $D>0, b>D$ there exists a number $R_5=R_5(D,b)>0$ with the following property.
Let $Y\subset \Sigma$ be an essential subsurface 
with $d_{\cal C\cal G}(\partial Y,{\cal D}_1\cup {\cal D}_2)\geq p$ 
and 
let $(\Sigma,g_s)\subset {\cal L}(\partial Y)$ be a
simplicial path with the property that the diameter
of $\sigma(g_s)$ is at most $D$ for all $s$. Assume that 
the $\sigma(g_0)$-length of every simple closed curve on $\Sigma$ 
disjoint from $\partial Y$ is at least $R_5$.
Then the following holds true.
\begin{enumerate}
\item The $\sigma(g_1)$-length of every simple closed curve disjoint from 
$\partial Y$ is at least $b$.
\item If $Y$ is not an annulus then 
there exists a bridge arc for $Y$ of length at most
$\kappa_0$ for each of the metrics $\sigma(g_s)$. If $Y$ is an annulus
then there exists a simple closed curve crossing through $Y$ which 
intersects a minimal curve crossing through 
$\sigma(g_s)$ in at most 2 points for all $s$.
\end{enumerate}
\end{lem}
\begin{proof}
  Again we argue by contradiction and we assume that there are $D>0,b>0$ 
  such that a number $R_5=R_5(D,b)>0$ with property (1) in the lemma does not exist. 
 We then obtain a sequence 
  of counter examples, consisting of the following data.
  \begin{enumerate}
  \item[(a)] A hyperbolic $3$-manifold $M_n$ with Heegaard surface 
  $\Sigma$ and an essential subsurface $Y_n\subset \Sigma$ 
  with $d_{\cal C\cal G}(\partial Y_n,{\cal D}_1\cup {\cal D}_2)\geq p$ 
  whose boundary
  has diameter at most $D$ in $M_n$. 
  \item[(b)] A simplicial path $(\Sigma,g_n^s)\subset {\cal L}(\partial Y_n)$ 
$(s\in [0,1])$ such that the $\sigma(g_n^s)$-diameter of $\Sigma$ is bounded from 
above by $D$ for all $n,s$ and that the $\sigma(g_n^0)$-length 
of every simple closed curve on $\Sigma$ disjoint from $\partial Y_n$ 
tends to infinity with $n$.
\item[(c)] A simple closed curve $\alpha_n$ disjoint from $\partial Y_n$ whose
$\sigma(g_n^1)$-length is at most $b$.
\end{enumerate} 

Choosing a point $z_n\in M_n$ on the geodesic representative 
of $\partial Y_n$,
by passing to a subsequence we may assume
 that the pointed hyperbolic manifolds $(M_n,z_n)$ converge in the pointed 
 geometric topology to a pointed 
 hyperbolic $3$-manifold $(M,z)$. Note to this
 end as before that the injectivity radius of $M_n$ at $z_n$ is bounded from below 
 by a universal constant.

 By assumption, the diameters of the hyperbolic  
 metrics $\sigma(g_n^s)$ are uniformly
 bounded. Thus for each $n$ 
 the images $g_n^s(\Sigma)\subset M_n$ 
 are contained in a 
 compact subset of $M_n$ of diameter bounded from above by a constant independent 
 of $n$. Namely, the maps $g_n^s$ are one-Lipschitz and their images 
 $g_n^s(\Sigma)$ 
 contain the geodesic representatives of $\partial Y_n$.

 Since for all $n$ the maps $g_n^s$ are homotopic within a fixed compact subset of 
 $M_n$, and by Lemma \ref{homotopy} and Remark \ref{pleatedhomotopy} they
 all define the same marked homotopy class of maps $\Sigma\to M_n$, 
 up to passing to a subsequence and the action of the mapping class group,
 the almost isometric maps $k_n:U_n\subset M\to M_n$ whose 
 existence follows from the assumption on geometric convergence $M_n\to M$
 define a fixed marked homotopy class of maps 
 $h:\Sigma\to M$.
  If $(\Sigma,g)$ is any geometric limit of 
 one of the maps $(\Sigma,g_n^s)$ as $n\to \infty$, then this limit is contained
 in this fixed marked homotopy class.  

For large enough $n$, the geodesic multicurve $\partial Y_n$ 
of uniformly bounded diameter in $M_n$ defines 
via the almost isometric map $k_n$ 
a geodesic multicurve $\partial \hat Y_n$ 
in $M$ which is the image of $\partial Y_n$ under the preferred homotopy class
of maps $\Sigma\to M$. The length in $M$ of these multicurves tends to 
infinity as $n\to \infty$. Using the preferred homotopy class of maps 
$\Sigma\to M$, for large enough $n$ each pleated surface 
$(\Sigma,g_n^s)$ in the simplicial path $g_n^s\subset {\cal L}(\partial Y_n)$  
determines via the map $k_n$ 
a pleated surface $(\Sigma,\hat g_n^s)$ in $M$. In particular, this holds true for 
$s=0,1$. Furthermore, by the explicit construction of simplicial paths in 
${\cal L}(\partial Y_n)$, we also obtain a corresponding simplicial path $\hat g_n^s
\subset {\cal L}(\partial \hat Y_n)$. In particular, the maps 
$\hat g_n^0, \hat g_n^1:\Sigma\to M$ are homotopic relative to the common
pleating lamination $\partial \hat Y_n$.

Using once more the almost isometric maps $k_n$, 
the length of any simple closed curve on $(\Sigma,\sigma(\hat g_n^0))$ disjoint from
$\partial \hat Y_n$ tends to infinity with $n$. Thus by Lemma \ref{surfaces},
up to passing to another subsequence we may assume that the 
multicurves $\partial \hat Y_n$ converge as $n\to \infty$ in the geometric topology 
to a geodesic lamination 
$\hat \mu^0$ on $\Sigma$ which fills $\Sigma$. 

On the other hand, by assumption, for each $n$ there exists a simple  
closed curve on 
$(\Sigma,\sigma(g_n^1))$ of length at most $b$ which is disjoint from $\partial Y_n$.
Via the almost isometric maps $k_n$, for large enough $n$ the same holds true
(for a perhaps slightly larger constant) for $(\Sigma,\sigma(\hat g_n^1))$. 
By passing once more to a subsequence, this property passes on to a
limiting pleated surface $(\Sigma,\hat g^1)$. In other words, there exists 
a simple closed curve $c$ on $\Sigma$ which is disjoint from a limit
$\hat \mu^1$ of the geodesic laminations $\partial \hat Y_n$ 
on $\Sigma$. Since $\hat \mu^1$ is a limit of 
$\partial \hat Y_n$ in the Hausdorff topology, we deduce that 
$c$ is disjoint from $\partial \hat Y_n$ for all large enough $n$.

But a limiting lamination $\hat \mu^0$ for 
$\sigma(\hat g_n^0)$ fills $\Sigma$. Since by Remark \ref{pleatedhomotopy}
the pleated surfaces $\hat g_n^0$ and $\hat g_n^ 1$ induce the
same maps $\pi_1(\Sigma)\to \pi_1(M)$ as marked homomorphisms, 
this implies that $\partial \hat Y_n$ has 
an essential intersection with $c$ for all large enough $n$. 
Recall that this is a topological property. This is a contradiction
and yields that there exists a number $R_5=R_5(D,b)>0$ for which property (1) 
stated in the proposition holds true.

We are left with showing that up to perhaps enlarging $R_5$, property (2) is satisfied as well. 
Again we argue by contradiction and we assume that the statement does not hold 
true. Then there is a sequence of counter examples with properties (a) and (b) from the beginning
of this proof and that moreover no such bridge arcs or simple closed curves exist.
By what we established so far, we may assume that as $n\to \infty$ 
and up to the action of the mapping class group, the hyperbolic 
metrics $\sigma(g_n^0)$ and $\sigma(g_n^1)$ on $\Sigma$ 
converge in Teichm\"uller space 
to metrics $\sigma^0,\sigma^1$, and the geodesic laminations 
$\beta_n\subset (\Sigma,\sigma(g_n^i))$ $(i=0,1)$ converge
to laminations $\hat \mu^0,\hat \mu^1$ which fill $\Sigma$. 
The pleated surfaces $g_n^0,g_n^1$ converge in the pointed
geometric topology to pleated surfaces 
$g^0,g^1:\Sigma\to M$ in the same homotopy class which 
map $\mu^0,\mu^1$ leafwise isometrically to the same geodesic lamination
$\mu\subset M$. This lamination is a geometric limit of the geodesic representatives
of the preimages $\hat \beta_n$ of the geodesics $\beta_n$ under the almost
isometric metric maps $k_n$ which define the geometric convergence. 

It follows from the above discussion that the laminations
$\mu^ 0,\mu^1$ coincide as marked geodesic lamination on $\Sigma$.
Denote this lamination by $\mu$ for convenience.
Let $\hat g^0,\hat g^1$ be two limiting pleated surfaces which are limits of 
the sequence $g_n^0,g_n^1$, respectively. 
Let $Q\subset \Sigma$ be a complementary component of the filling geodesic 
lamination $\mu$ which is contained in the pleating lamination for $\hat g^0,\hat g^1$
and let 
$\ell_1,\ell_2$ be two oriented boundary leaves of $Q$ which are backward asymptotic. 
Then for a given $\epsilon >0$, there are points $x\in \ell_1,y\in \ell_2$ of 
distance at most $\epsilon/2$ for both $\hat g^0,\hat g^1$. For large enough $n$
such that $\hat g^i(\Sigma)$ is contained in the domain of the map $k_n$, 
there are points on $\partial Y_n$ close to the images of
$x,y$ under $k_n$ of distance at most $\epsilon$ for both $g_n^0,g_n^1$ and 
such that there exists a bridge arc with endpoints on $\partial Y_n$ which is of length 
at most $\epsilon$ for both $g_n^0,g_n^1$. If this bridge arc is contained in 
$Y_n$ then this is a contradiction. This is in particular the case if 
$\partial Y_n$ is a non-separating simple closed curve.

We  are left with showing that the latter property can always be assumed. Namely,
if $Y_n$ is not an annulus then 
$Y_n\subset \Sigma$ is a subsurface of negative Euler characteristic. 
The pleating lamination $\lambda_n$ for $(\Sigma,g_n^0)$ 
decomposes $Y_n$ into ideal triangles. A limiting
ideal triangle is contained in the limit of the subsurfaces $Y_n$ and hence the above construction
applied to a complementary polygon which is contained in the limit of the surfaces
$Y_n$ yields the desired property. 
This completes the proof of the lemma in the case that $Y$ is not an annulus. 

The argument for the case
that $Y$ is an annulus is completely analogous and will be omitted. 
\end{proof}

The next proposition combines what we established so far 
to a subsurface projection bound for subsurfaces with 
large length geodesic realization. 
In its formulation,
the constants $R_2>0,k_2>0,p>0$ are as in Lemma \ref{largediameter}. 

\begin{prop}\label{diameter}
There exists a
number $R_6=R_6(\Sigma)>R_2$ with the
following property. Let $Y\subset \Sigma$ be a proper essential
subsurface. Assume that
$d_{\cal C\cal G}(\partial Y,{\cal D}_1\cup
{\cal D}_2)\geq p$, that the diameter in $M_f$ of each 
component of $\partial Y$ is at most $R_2$ and that 
$\partial Y$ contains a component
$\beta$ of length at least $R_6$; then 
\[{\rm diam}_Y({\cal D}_1\cup {\cal D}_2)\leq k_2.\]
\end{prop}

\begin{proof} We only show the proposition for non-annular subsurfaces,
the claim for annuli follows from exactly the same argument.

Thus let $Y$ be a proper essential 
subsurface of $\Sigma$ with $d_{\cal C\cal G}(\partial Y,
{\cal D}_1\cup {\cal D}_2)\geq p$. Assume that the diameter in 
$M_f$ of each component of $\partial Y$ is at most $R_2$ where
$R_2>0$ is as in Lemma \ref{largediameter}. 

Choose a number $b_0>0$ which is larger than $10$ times the Bers constant for 
$\Sigma$. For this number $b_0$ let $b_1=R_5(R_2,b)$ be as in Lemma \ref{allconverge}.
Define $b_2=R_5(R_2,b_1)$. Note that by Lemma \ref{allconverge}, if
$(\Sigma,g)$ is a pleated surface in the homotopy class of the inclusion of 
a Heegaard surface whose pleating lamination contains $\partial Y$ for an 
essential subsurface $Y$ of $\Sigma$, if
the diameter of $\sigma(g)$ is at most $R_2$, the length of $\partial Y$ is at least
$b_2$ and if there exists a simple closed curve disjoint from $\partial Y$ of length 
at most $b_1$, then all pleated surfaces in a path consisting of surfaces of diameter
at most $R_2$ contain a simple closed curve disjoint from $\partial Y$ of length at most
$b_2$. Let $R_6=R_4(b_2)$ be as in Proposition \ref{thurstonrelative2}.

Assume that the length of some component of $\partial Y$ is at least $R_6$.
Choose diskbounding simple closed curves $c_i\in {\cal D}_{i+1}$
$(i=0,1)$ and
use these curves to construct pleated surfaces
$g_0,g_1:\Sigma\to M$ with pleating lamimation defined
by spinning $c_i$ about $\partial Y$. Connect
$g_0$ to $g_1$ by a simplicial path $g_s\subset {\cal L}(\partial Y)$.
It follows from the assumption on
$Y$ that there exists a partition $0=t_0<\cdots <t_n=1$ of 
$[0,1]$ such that the path
$g_i=g\vert [t_{i-1},t_i]$ has one of the following properties.
\begin{enumerate}
\item If $i$ is even then for each $s\in [t_{i-1},t_i]$ 
there exists a Margulis tube for $\sigma(g_s)$ with core
  curve of length at most $\kappa_0/10$, and this core curve
  is diskbounding.
\item If $i$ is odd then 
either for each $s\in [t_{i-1},t_i]$ the diameter of $\sigma(g_s)$ is at most $R_2$ 
and for all $s$ there exists a simple closed curve disjoint from
$\partial Y$ whose $g_s$-length is at most $b_2$, or 
for all $s$ the shortest $\sigma(g_s)$-length 
of a simple closed curve disjoint from $Y$ is at least $b_1$.  
\end{enumerate}

Now by Corollary \ref{shortdisjoint} and 
Lemma \ref{allconverge}, for each odd $i$ there exists a bridge 
arc for $Y$ of $\sigma(g_s)$-length smaller than $\kappa_0$
for each $s\in [t_{i-1},t_i]$. 
Furthermore, 
for each even $i$ there exists a bridge arc of length at most $\kappa_0$ 
contained in a diskbounding curve, and each of these
bridge arcs is a subarc of a curve from ${\cal D}_j$ for $j=1$ or $j=2$.
But as the distance in ${\cal C\cal G}(\Sigma)$ between
${\cal D}_1$ and ${\cal D}_2$ is at least 4 by assumption,
for consecutive even $i$ the disk sets which give rise to the 
short bridge arcs coincide. 
As a consequence,
either there is a short bridge arc persisting along the path,
or all the short bridge arcs are uniformly close
in the arc and curve graph of $Y$ 
to a bridge arc contained in a diskbounding curve from a fixed
disk set, say the set ${\cal D}_1$. This shows the proposition.
\end{proof}

\begin{proof}[Proof of \Cref{minsky}]
  Let $R=R_6>R_2$ where $R_2,R_6$ are as in Lemma \ref{largediameter}
  and Proposition \ref{diameter}.
For this number $R$ and the given number $\epsilon >0$ let $k=k(R,\epsilon)>0$ be as in Lemma
\ref{length}. Assume that the diameter of the subsurface projection of
${\cal D}_1\cup {\cal D}_2$ into $Y$ is at least $k$. By Lemma \ref{largediameter} and 
Proposition \ref{diameter}, we know that the length of $\partial Y$ is at most $R$. 
But then an application of Lemma \ref{length} shows that this length is in fact smaller than $\epsilon$.
This is what we wanted to show. 
\end{proof}

\section{Effective hyperbolization II}\label{convexcocompact}

The goal of this section is to complete the proof of
effective hyperbolization for closed 3-manifolds $M_f$ 
with large Hempel distance. Theorem \ref{relative} takes care of the case 
that a minimal geodesic in the curve graph of the Heegaard surface $\Sigma$
connecting the two disk sets ${\cal D}_1,{\cal D}_2$ for $M_f$ 
has a sufficiently long subsegment with bounded 
combinatorics (no large subsurface projections). Thus we are left with  
considering manifolds for which there are
proper strongly incompressible subsurfaces
$Y\subset \Sigma$ so that the diameter of the subsurface projection of
${\cal D}_1\cup {\cal D}_2$ into $Y$ is larger than some a
priori fixed constant.

Our strategy is to choose two of these subsurfaces of $\Sigma$, 
say the surfaces $Y_1,Y_2$, whose boundaries
are sufficiently far away in the curve graph of $\Sigma$ 
from both ${\cal D}_1\cup {\cal D}_2$ and from each other, and to   
choose a boundary curve $\alpha_1$ of $Y_1$.
Drilling the curve $\alpha_1$ from $M_f$ 
results in a non-compact
manifold $M_1$ with one end $C_1$ homeomorphic to $T^2\times (0,\infty)$
where $T^2$ denotes a 2-torus. Proposition 3.1 of \cite{FSV19} shows that 
the manifold $M_1$ is irreducible, atoroidal and Haken and hence it 
admits a complete finite volume hyperbolic metric
by Thurston's geometrization theorem for Haken manifolds \cite{T86}, \cite{Th86b}. 
The end $C_1$ of $M_1$ is a cusp for this hyperbolic 
metric. The torus $T^2=\partial C_1$ inherits a flat metric from the hyperbolic
metric on $M_1$.

By the Dehn filling theorem Theorem \ref{filling}, removal of 
$C_1$ and gluing a solid torus to the boundary $\partial C_1$ of $M_1-C_1$ 
in such a way that the meridian for the gluing is sufficiently long 
for the flat metric on $\partial C_1$ 
yields a closed hyperbolic manifold $N_1$. We show that 
for a suitable choice of the meridian for the gluing, the Heegaard surface 
$\Sigma$ for $M_f$ also is a Heegaard surface for $N_1$, and we can control
distances in the curve graph of $\Sigma$ and sizes of subsurface projections
for the disk sets of both $M_f$ and $N_1$. 
In particular, we observe that the Heegaard distance of $N_1$ coincides with the
Heegaard distance of $M_f$, and 
the diameter of the subsurface projection of the disk sets 
${\cal D}_1\cup {\cal D}_2$ of $M_f$ 
into the subsurface $Y_2$ of $\Sigma$ 
essentially coincides with the diameter of the subsurface projections of the disk
sets of $N_1$. 

By Theorem \ref{minsky}, the length of the boundary  
$\partial Y_2$ of $Y_2$ for the hyperbolic metric on $N_1$
is bounded from above by a constant only depending
on the size of the subsurface projection
of ${\cal D}_1\cup {\cal D}_2$ into $Y_2$,  
but not on the filling meridian defining $N_1$. Thus if the diameter of this
subsurface projection is large, then this length is smaller than
any a priori chosen constant.

Do the above construction with the manifold $M_f$ and a 
boundary curve $\alpha_2$ of the subsurface $Y_2$ of $\Sigma$. 
Drilling $\alpha_2$ from $M_f$ yields an irreducible atoroidal Haken manifold
$M_2$ with one end $C_2$ which admits a complete finite volume hyperbolic 
metric. Filling the cusp using a suitably chosen long meridian on the 
boundary of $C_2$ results in a hyperbolic manifold $N_2$. 
Using again Theorem \ref{minsky}, the length of the boundary
$\partial Y_1$ of $Y_1$ in $N_2$ is smaller than any a priori chosen constant 
provided that the diameter of the subsurface
projection of ${\cal D}_1\cup {\cal D}_2$ into $Y_1$ is sufficiently large,
independent of the choice of the filling of 
the cusp of $M_2$.

But this means the following.
Let $\hat M_f$ be the manifold obtained from $M_f$ by drilling both
curves $\alpha_1,\alpha_2$. This manifold admits a 
finite volume hyperbolic metric
with two rank two cusps $\hat C_1,\hat C_2$. There exists a uniform lower bound
for the lengths of the curves on the boundaries $\partial \hat C_1,
\partial \hat C_2$ of $\hat C_1,\hat C_2$ 
corresponding to the meridians for the filling of $\hat C_1,\hat C_2$ which
gives rise to $M_f$, and this lower bound only depends on the diameters of the subsurface 
projections of the disk sets of $M_f$ into $Y_1,Y_2$. respectively. 
As a consequence, we can use Theorem \ref{filling} to fill both cusps
and construct a hyperbolic metric on $M_f$.

To implement this strategy we have to assure that
suitably chosen Dehn surgeries about
a boundary curve 
of a strongly incompressible subsurface $Y\subset \Sigma$
yield manifolds $N$ with the same Heegaard surface $\Sigma$ as $M_f$, and we have to 
control the disk sets of the surgered manifold as well as their 
distances in the curve graph of 
$\Sigma$.

To set up this control we use Theorem 3.1 of 
\cite{MM00}. For a proper essential subsurface $Y\subset \Sigma$, 
we denote as before by $d_Y$ the distance in the arc and curve graph of $Y$.

\begin{thm}[Masur-Minsky]\label{subsurfaceprojection}
There exist constants $m=m(\Sigma)>0,p=p(\Sigma)<m$ with 
the following properties. Let $\alpha,\beta\in {\cal C\cal G}(\Sigma)$ be 
two simple closed curves and let $Y\subset \Sigma$ be a proper essential 
subsurface. 
If $d_Y(\alpha,\beta)\geq m$, then  any geodesic $\zeta:[0,n]\to  
{\cal C\cal G}(\Sigma)$ connecting $\alpha=\zeta(0)$ to 
$\beta=\zeta(n)$ has to pass through
a curve $\zeta(j)$ (for some $j\in [1,n-1]$)
which is disjoint from $Y$.
Furthermore, if $j\in [3,n-3]$ and if $a\in [0,j-3]$, $b\in [j+3,n]$ 
then 
\[ \vert d_Y(\zeta(a),\zeta(b))- d_Y(\alpha,\beta)\vert \leq p.\]
\end{thm}

Here the last part of Theorem \ref{subsurfaceprojection} follows from the 
fact that a geodesic in ${\cal C\cal G}(\Sigma)$ can contain at most three
simple closed curves disjoint from a subsurface $Y$, and their mutual distance
is at most 2. Thus if $j\in [3,n-3]$ and if $a\in [0,j-3]$, $b\in [j+3,n]$ 
then up to adjusting constants,
Theorem 3.1 of \cite{MM00} shows that
$d_Y(\zeta(0),\zeta(a))\leq  p/2$, $d_Y(\zeta(b,\zeta(n))\leq p/2$ and hence
the statement follows from the triangle inequality.

Using the constants $m,p$ from Theorem \ref{subsurfaceprojection}, 
let us assume that $Y\subset \Sigma$ is a proper essential subsurface which is 
strongly incompressible for $M_f$ and such that the diameter of the 
subsurface projection of ${\cal D}_1\cup {\cal D}_2$ into $Y$ is at least
$2m$. Assume also that $d_{\cal C\cal G}(\partial Y,{\cal D}_i)\geq 3$. 
Let $\zeta:[0,n]\to {\cal C\cal G}(\Sigma)$ be a 
minimal geodesic in ${\cal C\cal G}(\Sigma)$ connecting
${\cal D}_1$ to ${\cal D}_2$. Choose 
\emph{markings} $\mu_1,\mu_2$ for $\Sigma$ whose pants decompositions
are composed of curves in ${\cal D}_1,{\cal D}_2$ and which contain 
$\zeta(0),\zeta(n)$. Here a marking  consists of a pants decomposition 
of $\Sigma$ together with a system of spanning curves, one for each component 
$\alpha$ 
of the pants decomposition $P$, which is disjoint from $P-\alpha$ and intersects
$\alpha$ in the minimal number of points (one or two, depending on the topological 
type of the component of $\Sigma-(P-\alpha)$ containing $\alpha$).

The marking graph is the graph whose vertices are markings and whose
edges are given by so-called \emph{elementary moves},
consisting of removal of one of 
the marking curves and replacing it by another curve while keeping all the 
remaining curves from the marking (see \cite{MM00} for a detailed discussion). 
Choose a simplicial path $\mu_s$ $(s\in [0,u])$
in the marking graph of $\Sigma$ so that each point of $\mu$
contains a point of $\zeta$. We require moreover that
the pants decomposition of the 
endpoints $\mu_0$ of $\mu$ consists of curves in the disk set
${\cal D}_1$ and contains $\zeta(0)$, and that the pants decomposition
of the endpoint $\mu_u$ of $\mu$ consists of curves in the disk set
${\cal D}_2$ and contains the endpoints of $\zeta(n)$ of $\zeta$. 
Since $\zeta$ passes through a curve disjoint from $Y$, we may assume that 
there exists a point in $\mu$ which contains the boundary of 
$Y$ as part of the pants decomposition.  
View the 3-manifold $M_f$ as being glued from two handlebodies 
$H_1,H_2$ of genus $g$ and the manifold $\Sigma\times [1,2]$, where 
$\Sigma\times \{1\}$ is equipped with the marking $\mu_0$, and 
$\Sigma\times \{2\}$ is equipped with the marking $\mu_u$. In this way 
the manifold $M_f$ is completely determined by the pair of 
 markings $(\mu_0,\mu_u)$ of $\Sigma$.

Let $T$ be a solid torus. Its boundary $\partial T$ 
contains the meridian as a distinguished
homotopy class of a simple closed curve $c$, characterized by
being  contractible in $T$. A \emph{longitude} for $T$ is a simple closed curve
on $\partial T$ which intersects $c$ in a unique point and is isotopic to 
the core curve of $T$, that is, to a generator of the fundamental group of $T$.

Let $\alpha$ be a boundary curve of the strongly incompressible subsurface 
$Y$ of $\Sigma$. Then $\alpha$ is not contractible in $M_f$ and hence it is the 
core curve of a solid torus $T\subset M_f$. 
Choose as a longitude on the boundary of $T$ 
the curve $\alpha
\subset \Sigma$. 
This construction associates to the torus $T\subset M_f$ with core curve $\alpha$ a 
preferred meridian-longitude pair $(c,\alpha)$ on $\partial T$.  
Theorem 6.2 of \cite{C96} now states the following.

\begin{thm}[Comar]\label{comar}
Let 
$(c,\alpha)$ be a preferred meridian-longitude pair in the boundary
of a tube in $M_f$ with core curve $\alpha\subset Y$. Let ${\cal T}_\alpha$ be the 
left Dehn twist about $\alpha\subset \Sigma$. Then the manifold defined by the 
pair of markings $(\mu_0, {\cal T}_\alpha^q \mu_u)$ is 
obtained from $M_f$ by $(1,q)$-Dehn surgery along $\alpha$ for all $q$.
\end{thm} 

To control the diameter of subsurface projections of the disk sets of the 
Dehn surgered manifolds we begin 
with some preliminary discussion about Dehn twists.

\begin{lem}\label{modify}
Let $k>0$, let $A\subset \Sigma$ be an annulus and let
${\cal T}_A$ be the left Dehn twist about the core curve of $A$. Let 
$c,d\in {\cal C\cal G}(\Sigma)$ be simple closed curves which have an
essential intersection with $A$.
Then there exists a number $q\in \mathbb{Z}$ such that 
the diameter of the subsurface projection of $c,{\cal T}_A^qd$ into $A$ 
is contained in 
$[k,k+2]$. Up to perhaps replacing $q$ by $q\pm 1$, the number
$q$ is unique if we require in addition that its absolute value is
minimal with this property.
\end{lem}

The ambiguity in the choice of $q$ in the last statement of the lemma reflects the
fact that the subsurface projection into an annulus is only well defined up to the 
ambiguity of possibly adding a single positive or negative twist. 

\begin{proof}
The subsurface projection into $A$ of two simple closed curves $c,d$ on $\Sigma$
is defined as follows. Equip $\Sigma$ with an auxiliary hyperbolic metric.
Consider the covering $V$ of $\Sigma$ with fundamental
group $\pi_1(A)$. This covering is an annulus, equipped with a complete 
hyperbolic metric. Both ends of $V$ have infinite volume
and hence the ideal boundary of $V$ consists of two disjoint circles. 
Since $c,d$ intersect $A$ essentially, there are components $\tilde c,\tilde d$ of 
a lift of $c,d$ to $V$ 
which are arcs abutting on the two distinct components of 
the ideal boundary of $V$. Up to an additive constant of $\pm 1$, the diameter of the 
subsurface projection of $c,d$ into 
$A$ then equals the number of essential intersections of $\tilde c,\tilde d$. 

Let $\hat d$ be an essential arc in $V$ with the same endpoints as $\tilde d$ which
is disjoint from $\tilde c$ except  perhaps at its endpoints. 
Up to homotopy with fixed endpoints, we have $\hat d={\cal T}_A^n\tilde d$ for some
$n\in \mathbb{Z}$; note that this 
makes sense since the Dehn twist ${\cal T}_A$ lifts to $V$. 
Write $q=n-k$ if $n\geq 0$, and write $q=n+k$ if $n<0$. Then
${\cal T}_A^{q}(\tilde d)$ has $k\pm 1$ intersections with
$\tilde c$. Thus the diameter of the subsurface projections into $A$ of the curves
$c,{\cal T}_A^{q}d$ is contained in the interval $[k-1,k+1]$.
Furthermore, up to perhaps replacing $q$ by $q\pm 1$, the number $q$ is
the unique number of minimal absolute value with this property. 
This shows the lemma.
\end{proof}

Let $\zeta:[0,n]\to {\cal C\cal G}(\Sigma)$ be a minimal geodesic
connecting ${\cal D}_1$ to ${\cal D}_2$ where as before, ${\cal D}_1\cup
{\cal D}_2$ are the disk sets of $M_f$.
Let $Y\subset \Sigma$ be a 
proper essential strongly incompressible subsurface
such that $d_Y({\cal D}_1,{\cal D}_2)\geq m+2p$ where $m,p>0$ are as in Theorem \ref{subsurfaceprojection}.
Let $\alpha\subset \partial Y$
be a boundary curve and let ${\cal T}_\alpha$ be the left Dehn twist about $\alpha$. 
By the choice of $m$, there exists $j>0$ be such that $\zeta(j)$ is disjoint from $Y$. 
Let $\ell >2m$ and let $q\in \mathbb{Z}$ be as in Lemma \ref{modify} 
of minimal absolute value such that
the diameter of the subsurface projection of $\zeta(0), {\cal T}_\alpha^q\zeta(n)$ into 
the annulus $A\subset \Sigma$ with core curve $\alpha$ is contained in the interval 
$[\ell,\ell+2]$.

Since the Dehn twist
${\cal T}_\alpha^q$ fixes the curve $\zeta(j)$, we can define a modified 
path $\zeta_\alpha:[0,n]\to {\cal C\cal G}(\Sigma)$ by 
\[ \zeta_\alpha(u) =\begin{cases} \zeta(u)  &\text{ for } u\leq j\\
{\cal T}_\alpha^q(\zeta(u)) & \text{ for } u\geq j\end{cases}.\]
Note that by Theorem \ref{comar}, the curve $\zeta_\alpha$ connects the 
disk sets
$\hat {\cal D}_1,\hat {\cal D}_2$ of the manifold $\hat M$ obtained from 
$M$ by $(1,q)$-Dehn surgery along $\alpha$. 
We have

\begin{lem}\label{deletesub}
\begin{enumerate}[i)]
\item 
The path $\zeta_\alpha$ is a geodesic in ${\cal C\cal G}(\Sigma)$. 
\item If $Y\subset \Sigma$ is not an annulus then 
$d_Y(\zeta_\alpha(0),\zeta_\alpha(n))=
d_Y(\zeta(0),\zeta(n))$.  
\item Let  $Z\subset \Sigma$ be a proper
  incompressible subsurface such that
$d_{\cal C\cal G}(\partial Y,\partial Z)\geq 5$.
Assume that $d_Z(\zeta(0),\zeta(n))\geq m+2p$ and that $\ell\in (0,n)$ is such that
$\zeta(\ell)$ is disjoint from 
$Z$. Then $\vert \ell-j \vert \geq 3$, and
\begin{align} 
  d_Z(\zeta(0), {\cal T}_\alpha^q \zeta(n))\geq d_Z(\zeta(0),\zeta(n))-2p
 &\text{ if } \ell <j,  \notag\\ 
d_{T^q_\alpha Z}(\zeta(0),{\cal T}^q_\alpha \zeta(n))\geq 
d_Z(\zeta(0),\zeta(n))-2p & \text{ if } \ell>j.\notag\end{align}
\end{enumerate}
\end{lem}

\begin{proof}
  The path $\zeta_\alpha$ has the same length as $\zeta$.
We claim that it is a geodesic in ${\cal C\cal G}(\Sigma)$.

To show the claim 
let $\beta:[0,b]\to {\cal C\cal G}(\Sigma)$ be a geodesic connecting
$\beta(0)=\zeta_\alpha(0)$ to $\beta(b)=\zeta_\alpha(n)$. 
Its length $b$ is at most the length
$n$ of the path $\zeta_\alpha$. Now note that 
if $Y$ is not an annulus, then 
$d_Y(\zeta(0),{\cal T}_\alpha^q \zeta(n))$ coincides with
$d_Y(\zeta(0), \zeta(n))$,
and if $Y$ is an annulus then 
$d_Y(\zeta(0),{\cal T}_\alpha^q\zeta(n))\geq 2m$ by construction. 
Thus any geodesic
in ${\cal C\cal G}(\Sigma)$ connecting $\zeta(0)=\zeta_\alpha(0)$
to $\zeta_\alpha(n)={\cal T}_\alpha^q(\zeta(n))$
has to pass through a curve disjoint from $Y$.

Let $a\in [0,b]$ be such that $\beta(a)$ is disjoint from 
$Y$. Then $\beta(a)$ is left fixed by ${\cal T}_\alpha^{-q}$
and therefore we can define 
an edge path $\hat \beta$ of length $b$ connecting $\zeta(0)$ to $\zeta(n)$ by
\[ \hat \beta(u)  =  \begin{cases} \beta(u)  & \text{ for }   u\leq a\\
{\cal T}_\alpha^{-q}\beta(u)  & \text{  for } u\geq a\end{cases}.\]
As $\zeta$ is a geodesic, the length $b$ of $\hat \beta$ is not smaller than the length 
$n$ of $\zeta$. Thus we have $b=n$ and consequently $\zeta_\alpha$ is a geodesic as claimed.
This shows the first part of the lemma. 

The second part of the lemma follows from the fact that the 
projection of a simple closed curve $c$ with an essential intersection with 
a proper essential non-annular subsurface $Y$ of $\Sigma$ equals the union of 
the intersection arcs $c\cap Y$. Thus if $c$ is replaced by $T_\alpha^q(c)$ for a 
boundary component $\alpha$ of $Y$, then these subsurface projections coincide.

To show the third part of the lemma, assume without loss of generality that
$\ell <j$, the case $\ell >j$ follows from an application of ${\cal T}_\alpha^q$. 
Since the distance in ${\cal C\cal G}(\Sigma)$ between $\partial Z$ and 
$\partial Y$ is at least 5, and a curve disjoint from $\partial Z,\partial Y$ has
distance at most 1 to $\partial Z,\partial Y$, 
we have $\vert j-\ell \vert\geq 3$ and hence 
Theorem \ref{subsurfaceprojection} shows that 
\[d_Z(\zeta(0),\zeta(j))\geq d_Z(\zeta(0),\zeta(n))- p\geq m+p.\] 
 Now the restriction of the geodesic $\zeta_\alpha$ to 
$[0,j]$ coincides with the restriction of the geodesic $\zeta$, and hence the same estimate
holds true for $\zeta_\alpha$ as well. 
As $\zeta_\alpha$ is a geodesic, and $\zeta_\alpha[0,j]$ passes through a curve disjoint from $Z$, 
the subsegment  
$\zeta_\alpha[j,n]$ does not pass through a curve disjoint from $Z$. 
Theorem \ref{subsurfaceprojection} then shows that  
$d_Z(\zeta_\alpha(0),\zeta_\alpha(n))\geq 
d_Z(\zeta_\alpha(0),\zeta_\alpha(j))- p\geq m-p.$ 
But 
$d_Z(\zeta_\alpha(0),\zeta_\alpha(j))=d_Z(\zeta(0),\zeta(j)) $ and consequently we have
$d_Z(\zeta_\alpha(0),\zeta_\alpha(n))\geq d_Z(\zeta(0),\zeta(n))-2p$ as claimed. 
\end{proof}
 
We are now ready to complete the proof of Theorem \ref{hyperbolizationfinal - introduction}
from the introduction. 

\begin{thm}\label{hyperbolizationfinal}
For every $g\geq 2$ there exist
numbers $R=R(g)>0$ and $C=C(g)>0$  with the following property. 
Let $M_f$ be a closed 3-manifold
with Heegaard surface $\Sigma$ of genus $g$, gluing map $f$ and disk sets 
${\cal D}_1\cup {\cal D}_2$, and assume that   
$d_{\cal C\cal G}({\cal D}_1,{\cal D}_2)\geq R$.
Then $M_f$ admits a hyperbolic metric, and the volume of $M_f$ for this 
metric is at least $Cd_{\cal C\cal G}({\cal D}_1,{\cal D}_2)$.
\end{thm}

\begin{proof}
Fix a Margulis constant $\mu$ for hyperbolic 3-manifolds and 
let $\xi\in (0,1/4)$ be a sufficiently small constant.
Let $L=L(\xi,1/100,2)>0$ be as in Theorem \ref{filling}.
By Theorem \ref{futer}, there exists a constant
$\varepsilon >0$ such that the following holds true.
If $M$ is any hyperbolic 3-manifold and if $N\subset M$ is a hyperbolic
solid torus whose core geodesic has length less than
$\varepsilon$, then the length of the meridian of the tube
on its boundary is at least $2L$.
For this number $\varepsilon$ let $k=k(\Sigma,\varepsilon)>0$
be as in Theorem \ref{minsky}.
For the number $k$ and the above number $\xi>0$ let
$b=b(\Sigma,k,\xi)>0$ be as in Theorem \ref{relative}.
Let $p\geq 3$ be as in \Cref{minsky}.

Consider   
a closed 3-manifold $M_f$, constructed from a gluing map
  $f:\partial H_1\to \partial H_2$. Assume that the Hempel distance 
  of $M_f$, that is, the distance in
  ${\cal C\cal G}(\Sigma)$ between the disk sets
  ${\cal D}_1,{\cal D}_2$, is larger than $2b+2p+3$.
  There are two possibilities.

  In the first case, a minimal geodesic $\zeta$ in ${\cal C\cal G}(\Sigma)$ 
  connecting ${\cal D }_1$ to ${\cal D}_2$ contains a subsegment of
    length at least $b$ whose endpoints do not have any subsurface
    projections of diameter at least $k$.
    Then $M_f$ fulfilles the assumptions in Proposition \ref{hyperbolization} and
    the existence of a hyperbolic metric on $M_f$ is an immediate consequence of
    \Cref{hyperbolization}.

    In the second case, no such subsegment exists. Then there exist at least
    two distinct proper essential incompressible subsurfaces $Y_1,Y_2$
    of $\Sigma$ whose boundaries have distance at least $5$ 
   in ${\cal C\cal G}(\Sigma)$, distance at least $p$ from
    ${\cal D}_1\cup {\cal D}_2$ 
    and such that the
    diameter of the subsurface projection of ${\cal D}_1\cup {\cal D}_2$ into
    $Y_1,Y_2$ is at least $k$. Namely, 
    in this case there exists such a subsurface $Y_1$, and there exist
    one or two points on the minimal geodesic $\zeta$ in ${\cal C\cal G}(\Sigma)$ 
    connecting ${\cal D}_1$ to ${\cal D}_2$ which are disjoint from $Y_1$, and these
    points are of distance one in ${\cal C\cal G}(\Sigma)$. 
    Such a point $\zeta(m)$ decomposes
    the geodesic into two subsegments, one of which has length at least 
    $b+p+1$. We then use this subsegment to find a second subsurface $Y_2$ of 
    $\Sigma$ with these properties and whose boundary if of distance at least $5$
    to the boundary of $Y_1$ in the curve graph of $\Sigma$.

    Let $\alpha_1,\alpha_2$ be a boundary component of $Y_1,Y_2$.
    Its distance in ${\cal C\cal G}(\Sigma)$ from ${\cal D}_1\cup {\cal D}_2$
    is at least $p\geq 3$ and hence by Proposition 3.1 of \cite{FSV19}, the 
3-manifold $M_i$ obtained by drilling $\alpha_i$ is 
irreducible atoroidal and Haken, with a single end $C_i$ which is homeomorphic
to $T\times [0,\infty)$ where $T=\partial C_i$  is a 2-torus.
The simple closed curve $\alpha_i\subset \partial Y_i$ determines
a distinguished free homotopy class $\beta_i$ 
on the boundary $\partial C_i$ of
$C_i$,  chosen so that the 3-manifold obtained by removing
$C_i$ and gluing a solid torus to the boundary $\partial C_i$ of
$M_i-C_i$ with 
meridian $\beta_i$ is just the manifold $M_f$.
We call the curve $\beta_i$ 
the \emph{meridian} of $M_f$ in the sequel.
By Theorem \ref{comar}, the manifold $M_i$ is obtained from
$M_f$ by $(1,\infty)$-Dehn surgery along a preferred
meridian-longitude pair for the boundary of a tube
about $\alpha_i$ in $M_f$.

By Thurston's hyperbolization theorem for irreducible atoroidal 
Haken manifolds (see \cite{T86}, \cite{Th86b}), $M_i$ admits 
a complete finite volume 
hyperbolic metric for which the end $C_i$ is a rank two cusp. 
Replace $M_i$ by a Dehn filling $\hat M_i$
which is obtained from $M_f$ by $(1,q)$-Dehn surgery
along a preferred meridian-longitude pair for the boundary
of a tube about $\alpha_i$ in $M_f$.
Theorem \ref{filling} shows that for sufficiently large $q$,
the manifold $\hat M_i$ admits a hyperbolic metric which is close
to the metric of $M_i$ away from the cusp $C_i$.

Since $\hat M_i$ is obtained from $M_f$ by $(1,q)$-surgery along 
along a preferred meridian-longitude pair for $M_f$, 
\Cref{comar} shows that the manifold $\hat M_i$ admits a Heegaard decomposition
with the same Heegaard surface $\Sigma$ as $M_f$. 
Let $\hat {\cal D}_1,\hat {\cal D}_2$ be the disk sets of $\hat M_i$ 
for this Heegaard decomposition. 
By Lemma \ref{deletesub}, the Heegaard distance of $\hat M_i$ coincides with 
the Heegaard distance of $M_f$, and the diameter 
of the subsurface projection of $\hat {\cal D}_1\cup \hat {\cal D}_2$
into $Y_{i+1}$ equals the diameter of the subsurface projection of 
${\cal D}_1\cup {\cal D}_2$ 
up to a uniformly bounded additive error
(indices are taken modulo 2).

As a consequence, the proper incompressible subsurface 
$Y_{i+1}$ fulfills the assumption of 
Theorem \ref{minsky} for the hyperbolic manifold $\hat M_i$ 
with Heegaard surface $\Sigma$. 
An application of Theorem \ref{minsky} shows that the length of 
the curve $\alpha_{i+1}$ in $\hat M_i$ is
less than $\varepsilon$. In particular, if we consider the meridian 
of the Margulis tube in $\hat M_i$ with core curve 
$\alpha_{i+1}$, viewed as a curve on the boundary of the Margulis tube 
defined by $\alpha_{i+1}$ in $\hat M_i$, then the length of this meridian 
is at least $2L$, independent
of the filling slope for the Dehn filling of $M_i$ which gives rise to $\hat M_i$.

Exchanging the roles of $M_1$ and $M_2$ in this argument shows the following.
For sufficiently large $q$, 
the manifold $N$ obtained from $M_f$ by $(1,q)$ Dehn surgery at both $\alpha_1,\alpha_2$
is hyperbolic, and the lengths of the
simple closed curves on the boundaries of the surgered Margulis tubes
which correspond to the meridians in $M_f$ (that is, which are obtained from
$N$ by $(1,-q)$-surgery) are at least $2L$.
Thus Theorem \ref{filling} shows that we can modify $N$ by Dehn surgery with slope
$(1,-q)$ at both $\alpha_1,\alpha_2$. The resulting
manifold is diffeomorphic to $M_f$, and it
carries a hyperbolic metric  whose restriction to the $\kappa_0$-thick part of 
$M_f$ is 
$\xi$-close in the $C^2$-topology to the restriction of the hyperbolic metric 
of $N$, where $\kappa_0>0$ is the constant with properties (P1),(P2) used in 
Section \ref{lengthbounds}. 
This completes the proof that $M_f$ admits a hyperbolic metric.

We are left with controlling the volume of this metric. 
Using the constant $k=k(\varepsilon)>0$ from Theorem \ref{minsky},
and for this number $k$ the integer   
$b=b(\Sigma,k,\varepsilon)$ from Theorem \ref{relative}, we find the following.
Denote by $n$ the Hempel distance $d_{\cal C\cal G}({\cal D}_1,{\cal D}_2)$ of 
$M_f$. 
Let $\zeta:[0,n]\to {\cal C\cal G}(\Sigma)$ be a shortest
geodesic in ${\cal C\cal G}(\Sigma)$
connecting ${\cal D}_1$ to ${\cal D}_2$. Let $Y_1,\dots,Y_s$
be the subsurfaces of $\Sigma$ with $d_{\cal C\cal G}(\partial Y_i,{\cal D}_1\cup {\cal D}_2)\geq p$
such that the diameter of the subsurface projection of 
${\cal D}_1\cup {\cal D}_2$ into $Y_i$ is at least $k$.
The geodesic $\zeta$ passes through
simple closed curves disjoint from $Y_i$.

Subdivide $\zeta$ into segments of length $b$. Let 
$\ell_0,\ell_1$, respectively, the smallest and largest integer
such that for the segment $[\ell_j b,(\ell_j +1)b]$, 
there exists no $u\in [\ell_j b,(\ell_j +1)b)$ so that
$\zeta(u)$ is disjoint from one of the surfaces $Y_i$ $(j=0,1)$.
There are now two possibilities. In the first case, 
we have $\ell_1-\ell_0\geq n/2b$. By Proposition \ref{hyperbolization} and its proof,
we conclude that in this case, the volume of $M_f$ is at least 
$v(d_{\cal C\cal G}(\zeta(b(\ell_0+1)),\zeta(b(\ell_1-1)))\geq 
vb(\ell_1-\ell_0-2)\geq vn/2-2vb$ which gives the required bound 
up to adjusting constants.

On the other hand, if $\ell_1-\ell_0\leq n/2b$ then 
each of the  segments $\zeta|_{[b\ell,b(\ell+1))}$ for $\ell <\ell_0$ or 
$\ell >\ell_1$ contains at least one curve which is disjoint from a 
subsurface with large subsurface projection. There are at least 
$\ell_0 +\lfloor \frac{n}{b}\rfloor -\ell_1\geq \lfloor n/2b\rfloor$ such segments. 
For each of these segments $[kb,(k+1)b)$, 
there exists at least one subsurface $Y_k$ such that 
$d_{Y_k}(\zeta(0),\zeta(n))$ is large, and such that $\zeta(u)$ is disjoint from 
$Y_k$ for some $u\in [kb,(k+1)b)$.
Now if $\zeta(u),\zeta(s)$ are both disjoint from $Y_k$, then 
$\vert u-s\vert \leq 2$ and hence if $Y_{k_1}=Y_{k_2}$ then 
$\vert k_1-k_2\vert \leq 2$. As a consequence,  
the number $s$ of such distinct subsurfaces is at least $\lfloor n/2b\rfloor/2$.

By Theorem \ref{minsky}, for each $i\leq s$ the total length of the geodesic representatives
of the boundary curves $\partial Y_i$ 
of the surface $Y_i$
is not bigger than $\varepsilon$ and therefore a boundary component of 
$Y_i$ is the core curve of a Margulis tube in $M_f$. These Margulis tubes are pairwise disjoint, 
and their volumes are bounded from below by a fixed number $w>0$ as this is already
true for the one-neighborhoods of their boundary tori. 
In other words, each of the tubes 
contributes at least the fixed amount $w$ to the
volume of $M_f$, independent of any choices. Adding up shows  
that the volume of $M_f$ and is at least $Cn$ where $C>0$ is a constant
only depending on $b$ and hence only depending on $\Sigma$.
\end{proof}

\begin{rem}\normalfont  
  Our construction for manifolds with Heegaard splitting $\Sigma$ and
  large subsurface projection of the disk sets ${\cal D}_1\cup {\cal D}_2$
into a proper essential non-annular subsurface $Y$ of $\Sigma$   
gives less information than the article \cite{FSV19}. Namely,
in contrast to these results, we do not obtain any information
on the shape of boundary tori of Margulis tubes arising
from such large subsurface projections which are reminiscent of 
the model manifold theorem for quasi-fuchsian groups in \cite{M10}.
\end{rem}

 

 
 \bigskip
 
 



\begin{thebibliography}{AFLMR07}

\bibitem[AR04]{AR04}
I.~R. Aitchison and J.~H. Rubinstein.
\newblock {Localising Dehn's lemma and the loop theorem in 3-manifolds}.
\newblock {\em Mathematical Proceedings of the Cambridge Philosophical
  Society}, 137:281 -- 292, 2004.


\bibitem[And90]{Anderson1990}
M.~T. Anderson.
\newblock {Convergence and rigidity of manifolds under Ricci curvature bounds}.
\newblock {\em Inventiones mathematicae}, 102(2):429--446, 1990.

\bibitem[And06]{Anderson2006}
M.~T. Anderson.
\newblock {Dehn Filling and Einstein Metrics in Higher Dimensions}.
\newblock {\em Journal of Differential Geometry}, 73(2):219 -- 261, 2006.


\bibitem[AFLMR07]{Azagra2007}
D.~Azagra, J.~Ferrera, F.~López-Mesas, and Y.~Rangel.
\newblock {Smooth approximation of Lipschitz functions on Riemannian
  manifolds}.
\newblock {\em Journal of Mathematical Analysis and Applications},
  326(2):1370--1378, 2007.
  
\bibitem[BGS85]{BGS85}
W.~Ballmann, M.~Gromov, and V.~Schroeder.
\newblock {\em {Manifolds of Nonpositive Curvature}}.
\newblock Progress in Mathematics. Birkh{\"a}user Boston, 1985.


\bibitem[Bam12]{Bamler2012}
R.~H. Bamler.
\newblock {Construction of Einstein metrics by generalized Dehn filling}.
\newblock {\em Journal of the European Mathematical Society}, 14:887--909,
  2012.
  
\bibitem[BP92]{BP92}
R.~Benedetti and C.~Petronio.
\newblock {\em {Lectures on Hyperbolic Geometry}}.
\newblock Universitext (Berlin. Print). Springer Berlin Heidelberg, 1992.
  

\bibitem[Ber66]{Berger1966}
M.~Berger.
\newblock Sur les vari\'{e}t\'{e}s d'{E}instein compactes.
\newblock In {\em Comptes {R}endus de la {III}e {R}\'{e}union du {G}roupement
  des {M}ath\'{e}maticiens d'{E}xpression {L}atine ({N}amur, 1965)}, pages
  35--55. Librairie Universitaire, Louvain, 1966.

\bibitem[Bes08]{Besse1987}
A.~L. Besse.
\newblock {\em {Einstein Manifolds}}.
\newblock Classics in Mathematics. Springer, 2008.
\newblock Reprint of the 1987 edition.


\bibitem[Biq00]{Biquard2000}
O.~Biquard.
\newblock {\em {M\'etriques d'Einstein asymptotiquement sym\'etriques}}.
\newblock Number 265 in Ast\'erisque. Soci\'et\'e math\'ematique de France,
  2000.
  
\bibitem[Bro03]{Br03}
J.~F. Brock.
\newblock {The Weil-Petersson Metric and Volumes of 3-Dimensional Hyperbolic
  Convex Cores}.
\newblock {\em Journal of the American Mathematical Society}, 16(3):495--535,
  2003.

  
\bibitem[BB02]{BB04}
J.~F. Brock and K.~W. Bromberg.
\newblock {On the density of geometrically finite Kleinian groups}.
\newblock {\em Acta Mathematica}, 192:33--93, 2002.

\bibitem[BCM12]{BCM12}
J.~F. Brock, R.~D. Canary, and Y.~N. Minsky.
\newblock {The classification of Kleinian surface groups, {II}: The Ending
  Lamination Conjecture}.
\newblock {\em Annals of Mathematics}, 176(3):1--149, 2012.

\bibitem[BMNS16]{BMNS16}
J.~F. Brock, Y.~N. Minsky, H.~Namazi, and J.~Souto.
\newblock {Bounded combinatorics and uniform models for hyperbolic
  3-manifolds}.
\newblock {\em Journal of Topology}, 9(2):451--501, 2016.


\bibitem[Can93]{C93}
R.~D. Canary.
\newblock {Algebraic Convergence of Schottky Groups}.
\newblock {\em Transactions of the American Mathematical Society},
  337(1):235--258, 1993.
  
\bibitem[CE08]{CheegerEbin}
J.~Cheeger and D.~G. Ebin.
\newblock {\em {Comparison Theorems in Riemannian Geometry}}.
\newblock AMS Chelsea Publishing, 2008.

\bibitem[CGT82]{CGT82}
J.~Cheeger, M.~Gromov, and M.~Taylor.
\newblock {Finite propagation speed, kernel estimates for functions of the
  Laplace operator, and the geometry of complete Riemannian manifolds}.
\newblock {\em Journal of Differential Geometry}, 17(1):15 -- 53, 1982.

\bibitem[Com96]{C96}
T.~Comar.
\newblock {Hyperbolic Dehn surgery and convergence of Kleinian groups}.
\newblock {\em University of Michigan}, 1996.
\newblock Dissertation.

\bibitem[dC92]{doCarmoRG}
M.~P. do~Carmo.
\newblock {\em {Riemannian Geometry}}.
\newblock Mathematics (Boston, Mass.). Birkh{\"a}user, 1992.

\bibitem[dC16]{doCarmoSurfaces}
M.~P. do~Carmo.
\newblock {\em {Differential Geometry of Curves and Surfaces: Revised and
  Updated Second Edition}}.
\newblock Dover Books on Mathematics. Dover Publications, 2016.

\bibitem[Esc87]{E87}
J.~H. Eschenburg.
\newblock {Comparison theorems and hypersurfaces}.
\newblock {\em manuscripta mathematica}, 59:295--323, 1987.

\bibitem[Eva10]{EvansPDE}
L.C. Evans.
\newblock {\em Partial Differential Equations}.
\newblock Graduate studies in mathematics. American Mathematical Society, 2010.

\bibitem[FSV19]{FSV19}
P.~Feller, A.~Sisto, and G.~Viaggi.
\newblock Uniform models and short curves for random 3-manifolds.
\newblock {\em arXiv preprint arXiv:1910.09486}, 2019.


\bibitem[FP20]{FP20}
J.~Fine and B.~Premoselli.
\newblock {Examples of compact Einstein four-manifolds with negative
  curvature}.
\newblock {\em Journal of the American Mathematical Society}, 33(4):991--1038,
  2020.

\bibitem[FPS19a]{FPS19b}
D.~Futer, J.~Purcell, and S.~Schleimer.
\newblock Effective bilipschitz bounds on drilling and filling.
\newblock {\em arXiv preprint arXiv:1907.13502}, 2019.

\bibitem[FPS19b]{FPS19}
D.~Futer, J.~Purcell, and S.~Schleimer.
\newblock {Effective distance between nested Margulis tubes}.
\newblock {\em Transactions of the American Mathematical Society},
  372(6):4211–4237, 2019.

\bibitem[FPS21]{FPS21}
D.~Futer, J.~Purcell, and S.~Schleimer.
\newblock {Effective drilling and filling of tame hyperbolic 3-manifolds}.
\newblock {\em arXiv preprint arXiv:2104.09983}, 2021.


\bibitem[Gaf54]{Gaffney1954}
M.~P. Gaffney.
\newblock {A Special Stokes's Theorem for Complete Riemannian Manifolds}.
\newblock {\em Annals of Mathematics}, 60(1):140--145, 1954.

\bibitem[GHL04]{GHL}
S.~Gallot, D.~Hulin, and J.~Lafontaine.
\newblock {\em {Riemannian Geometry}}.
\newblock Universitext. Springer Berlin Heidelberg, 3. edition, 2004.

\bibitem[GT01]{Gilbarg2001}
D.~Gilbarg and N.~S. Trudinger.
\newblock {\em {Elliptic Partial Differential Equations of Second Order}}.
\newblock Classics in Mathematics. Springer-Verlag, 2001.
\newblock Reprint of the 1998 edition.

\bibitem[Gro78]{G78}
M.~Gromov.
\newblock {Manifolds of negative curvature}.
\newblock {\em Journal of Differential Geometry}, 13(2):223 -- 230, 1978.

\bibitem[GKS07]{GromovMetricStructures07}
M.~Gromov, M.~Katz, and S.~Semmes.
\newblock {\em {Metric Structures for Riemannian and Non-Riemannian Spaces}}.
\newblock Modern Birkh{\"a}user Classics. Birkh{\"a}user Boston, 2007.



\bibitem[Ham10]{H10}
U.~Hamenst{\"a}dt.
\newblock {Stability of quasi-geodesics in Teichm{\"u}ller space}.
\newblock {\em Geometriae Dedicata}, 146:101--116, 2010.

\bibitem[HV22]{HV22}
U.~Hamenst{\"a}dt and G.~Viaggi.
\newblock {Small eigenvalues of random 3-manifolds}.
\newblock {\em Transactions of the American Mathematical Society}, 375(6):3795--3840, 2022.

\bibitem[Har82]{ODEHartman}
P.~Hartman.
\newblock {\em Ordinary Differential Equations}.
\newblock Classics in Applied Mathematics. Society for Industrial and Applied
  Mathematics, 2. edition, 1982.

\bibitem[Hat91]{H91}
A.~Hatcher.
\newblock {On triangulations of surfaces}.
\newblock {\em Topology and its Applications}, 40(2):189--194, 1991.

\bibitem[HIH77]{HIH77}
E.~Heintze and H.~C. Im~Hof.
\newblock {Geometry of horospheres}.
\newblock {\em Journal of Differential Geometry}, 12(4):481 -- 491, 1977.

\bibitem[Hem76]{He76}
J.~Hempel.
\newblock {\em {3-manifolds}}.
\newblock AMS Chelsea Publishing Series. Princeton University Press, 1976.

\bibitem[Hem01]{He01}
J.~Hempel.
\newblock {3-Manifolds as viewed from the curve complex}.
\newblock {\em Topology}, 40(3):631--657, 2001.


\bibitem[HK08]{HK08}
C.~Hodgson and S.~Kerckhoff.
\newblock {The shape of hyperbolic Dehn surgery space}.
\newblock {\em Geometry \& Topology}, 12(2):1033–1090, 2008.

\bibitem[HQS12]{HQS12}
X.~Hu, J.~Qing, and Y.~Shi.
\newblock {Regularity and rigidity of asymptotically hyperbolic manifolds}.
\newblock {\em Advances in Mathematics}, 230(4):2332--2363, 2012.

\bibitem[JMM10]{JMM10}
J.~Johnson, Y.~Minsky, and Y.~Moriah, 
\newblock {Heegaard splittings with large subsurface distance}, 
\newblock{\em Algebr. Geom. Topol.}, 10(4):2251--2275, 2010. 

\bibitem[JK82]{JK82}
J.~Jost and H.~Karcher.
\newblock {Geometrische Methoden zur Gewinnung von A-Priori-Schranken f{\"u}r
  harmonische Abbildungen}.
\newblock {\em Manuscripta Mathematica}, 40:27--77, 1982.




\bibitem[KY09]{KY09}
D.~Knopf and A.~Young.
\newblock {Asymptotic Stability of the Cross Curvature Flow at a Hyperbolic Metric}.
\newblock {\em Proceedings of the American Mathematical Society}, 137(2):699--709, 2009.


\bibitem[Koi78]{Koiso1978}
N.~Koiso.
\newblock {Nondeformability of Einstein metrics}.
\newblock {\em Osaka Journal of Mathematics}, 15(2):419 -- 433, 1978.

\bibitem[MM99]{MM99}
H.~A. Masur and Y.~N. Minsky.
\newblock {Geometry of the complex of curves I: Hyperbolicity}.
\newblock {\em Inventiones mathematicae}, 138:103--149, 1999.

\bibitem[MM00]{MM00}
H.~A. Masur and Y.~N. Minsky.
\newblock {Geometry of the complex of curves II: Hierarchical structure}.
\newblock {\em Geometric \& Functional Analysis GAFA}, 10:902--974, 2000.

\bibitem[MM04]{MM04}
H.~A. Masur and Y.~N. Minsky.
\newblock {Quasiconvexity in the curve complex}.
\newblock In {\em In the tradition of {A}hlfors and {B}ers, {III}}, volume 355
  of {\em Contemp. Math.}, pages 309--320. Amer. Math. Soc., Providence, RI,
  2004.

\bibitem[MO90]{MO90}
M.~Min-Oo.
\newblock {Almost Einstein manifolds of negative Ricci curvature}.
\newblock {\em Journal of Differential Geometry}, 32(2):457 -- 472, 1990.

\bibitem[Min00]{M00}
Y.~N. Minsky.
\newblock {Kleinian groups and the complex of curves}.
\newblock {\em Geometry \& Topology}, 4:117--148, 2000.

\bibitem[Min10]{M10}
Y.~N. Minsky.
\newblock {The classification of Kleinian surface groups, I: models and
  bounds}.
\newblock {\em Annals of Mathematics}, 171(1):1--107, 2010.



\bibitem[Pet16]{Petersen2016}
P.~Petersen.
\newblock {\em {Riemannian Geometry}}.
\newblock Springer, 3. edition, 2016.

\bibitem[PW97]{PetersenWei1997}
P.~Petersen and G.~Wei.
\newblock {Relative Volume Comparison with Integral Curvature Bounds}.
\newblock {\em Geometric And Functional Analysis}, 7(6):1031--1045, 1997.

\bibitem[Shc83]{Shcherbakov1983}
S.~A. Shcherbakov.
\newblock {The degree of smoothness of horospheres, radial fields, and
  horosphericalcoordinates on a Hadamard manifold}.
\newblock {\em Dokl. Akad. Nauk SSSR}, 271(5):1078--1082, 1983.

\bibitem[Thu86a]{T86}
W.~P. Thurston.
\newblock {Hyperbolic Structures on 3-manifolds, I: Deformation of acylindrical
  manifolds}.
\newblock {\em Annals of Mathematics}, 124:203 -- 246, 1986.

\bibitem[Thu86b]{Th86b}
W.~P. Thurston.
\newblock {Hyperbolic Structures on 3-manifolds, III: Deformation of 3-manifolds
 with incompressible boundary}.
\newblock {\em arXiv:math/9801058}, 1986.

\bibitem[Tia]{Tian}
G.~Tian.
\newblock {A Pinching Theorem on Manifolds with Negative Curvature}.
\newblock unpublished manuscript.

\bibitem[Top06]{topping_2006}
P.~Topping.
\newblock {\em {Lectures on the Ricci Flow}}.
\newblock London Mathematical Society Lecture Note Series. Cambridge University
  Press, 2006.

\bibitem[Via21]{V21}
G.~Viaggi.
\newblock {Volumes of random 3-manifolds}.
\newblock {\em Journal of Topology}, 14(2):504--537, 2021.


\end{thebibliography}

\bigskip

\noindent
MATH. INSTITUT DER UNIVERSIT\"AT BONN\\
ENDENICHER ALLEE 60, 53115 BONN, GERMANY\\
e-mail: ursula@math.uni-bonn.de

\bigskip

\noindent
MATH. INSTITUT DER UNIVERSIT\"AT BONN\\
ENDENICHER ALLEE 60, 53115 BONN, GERMANY\\
e-mail: fjaeckel@math.uni-bonn.de


\end{document}